\documentclass[12pt,a4paper]{article}
\usepackage{amssymb,amsthm,amsmath,epsfig,psfrag,subfigure}
\usepackage{multirow}
\usepackage{mathrsfs}
\usepackage{yfonts}[1998/10/03]
\usepackage{color}
\usepackage{ulem}
\date{today}
\numberwithin{equation}{section}

\newtheorem{Theorem}{Theorem}[section]

\newtheorem{Corollary}[Theorem]{Corollary}
\newtheorem{Lemma}[Theorem]{Lemma}

\theoremstyle{remark}
\newtheorem{Remark}{Remark}[section]
\theoremstyle{algorithm}

\theoremstyle{definition}
\newtheorem{Definition}{Definition}[section]
\newtheorem{Example}{Example}[section]

\title{Structure-Preserving Flows of Symplectic Matrix Pairs}

\author{\and Yueh-Cheng
Kuo\footnote{Department of Mathematics, National University of
Kaohsiung, Kaohsiung, 811, Taiwan \texttt{(yckuo@nuk.edu.tw)}} \and
Wen-Wei Lin\thanks{Department of Applied Mathematics, National Chiao Tung University, Hsinchu 300, Taiwan \texttt{ (wwlin@math.nctu.edu.tw)}}\and
Shih-Feng Shieh\thanks{Department of Mathematics, National Taiwan
Normal University, Taipei 116, Taiwan
\texttt{(sfshieh@ntnu.edu.tw)}}}

\begin{document}
\date{}



\maketitle

\tableofcontents

\addcontentsline{toc}{section}{Abstract}
\begin{abstract}
We construct a nonlinear differential equation of matrix pairs $(\mathcal{M}(t),\mathcal{L}(t))$ that is invariant (the \textbf{Structure-Preserving Property}) in the class of symplectic matrix pairs
\begin{align*}
\mathbb{S}_{\mathcal{S}_1,\mathcal{S}_2}=\left\{\left(\mathcal{M},\mathcal{L}\right)| \ \mathcal{M}=\left[%
\begin{array}{cc}
  X_{12} & 0 \\
 X_{22} & I \\
\end{array}%
\right]\mathcal{S}_2,\ \
\mathcal{L}=\left[%
\begin{array}{cc}
  I & X_{11} \\
 0 & X_{21} \\
\end{array}%
\right]\mathcal{S}_1\right.\nonumber\\ \left. \text{ and }X=\left[%
\begin{array}{cc}
  X_{11} & X_{12} \\
 X_{21} & X_{22} \\
\end{array}%
\right]\text{ is Hermitian}\right\}
\end{align*}
for certain fixed symplectic matrices $\mathcal{S}_1$ and $\mathcal{S}_2$. Its solution also preserves invariant subspaces on the whole orbit (the \textbf{Eigenvector-Preserving Property}).  Such a flow is called a \textit{structure-preserving flow} and is governed by a Riccati differential equation (RDE) having the form
\begin{align*}
\begin{array}{l}
\dot{W}(t)=[-W(t),I]\mathscr{H}[I ,W(t)^{\top} ]^{\top},\\
W(0)=W_0,
\end{array}
\end{align*}
for some suitable Hamiltonian matrix $\mathscr{H}$. In addition, Radon's lemma (\cite{Radon27} or see Theorem~\ref{thm3.8}) leads to the explicit form $W(t)=P(t)Q(t)^{-1}$ where $[
Q(t)^{\top},
P(t)^{\top}]^{\top}= e^{\mathscr{H}t}[
I,
W_0^{\top}]^{\top}$.
Therefore, blow-ups for the structure-preserving flows may happen at a finite $t$ whenever $Q(t)$ is singular. To continue, we then utilize the Grassmann manifolds to extend the domain of the structure-preserving flow to the whole $\mathbb{R}$ subtracting some isolated points.

On the other hand, the Structure-Preserving Doubling Algorithm (SDA) is an efficient numerical method for solving algebraic Riccati equations and nonlinear matrix equations. In conjunction with the structure-preserving flow, we consider the following two special classes of symplectic pairs: $\mathbb{S}_1=\mathbb{S}_{I_{2n},I_{2n}}$ and $\mathbb{S}_2=\mathbb{S}_{-I_{2n},\mathcal{J}}$ and the corresponding algorithms SDA-1 and SDA2. It is shown  that at $t=2^{k-1},k\in \mathbb{Z}$ this flow passes through the iterates generated by SDA-1 and SDA-2, respectively.  Therefore, the SDA and its corresponding structure-preserving flow have identical asymptotic behaviors, including the stability, instability, periodicity, and quasi-periodicity of the dynamics.

Taking advantage of the special structure and properties of the Hamiltonian matrix, we apply a symplectically similar transformation introduced by \cite{Lin1999469} to reduce $\mathscr{H}$ to a Hamiltonian Jordan canonical form $\mathfrak{J}$. The asymptotic analysis of the structure-preserving flows and RDEs is studied by using $e^{\mathfrak{J}t}$. The convergence of the SDA as well as its rate can thus result from the study of the structure-preserving flows. A complete asymptotic dynamics of the SDA is investigated, including the linear and quadratic convergence studied in the literature \cite{ccghlx09,GuoLin:2010,HuangLin:2009}.
\end{abstract}


\section{Introduction}\label{sec1}
We first introduce the algebraic structures that we consider in this paper.  Let
\begin{align*}
\mathcal{J}_n=\left[\begin{array}{cc}0 & I_n \\-I_n & 0\end{array}\right],
\end{align*}
where $I_n$ is the $n\times n$ identity matrix. For convenience, we use  $\mathcal{J}$ for $\mathcal{J}_n$ by dropping the subscript ``$n$" if the order of $\mathcal{J}_n$ is clear in the context.

\begin{Definition}\label{def1.1}
\begin{itemize}
\item[1.] A matrix $\mathcal{H}\in \mathbb{C}^{2n\times 2n}$ is  Hamiltonian if $\mathcal{H}\mathcal{J}=(\mathcal{H}\mathcal{J})^H$.
\item[2.] A matrix pair $(\mathcal{M}_h, \mathcal{L}_h)$ with $\mathcal{M}_h, \mathcal{L}_h\in \mathbb{C}^{2n\times 2n}$ is called a Hamiltonian pair if $\mathcal{M}_h\mathcal{J}\mathcal{L}_h^H=-\mathcal{L}_h\mathcal{J}\mathcal{M}_h^H$.
\item[3.] A matrix $\mathcal{S}\in \mathbb{C}^{2n\times 2n}$ is  symplectic if $\mathcal{S}\mathcal{J}\mathcal{S}^H=\mathcal{J}$.
\item[4.] A matrix pair $(\mathcal{M}_s, \mathcal{L}_s)$ with $\mathcal{M}_s, \mathcal{L}_s \in \mathbb{C}^{2n\times 2n}$ is called a symplectic pair if $\mathcal{M}_s\mathcal{J}\mathcal{M}_s^H=\mathcal{L}_s\mathcal{J}\mathcal{L}_s^H$.
\end{itemize}
\end{Definition}
 Denote by $Sp(n)$  the multiplicative group of all $2n\times2n$ symplectic matrices and by $\mathbb{H}(2n)$ the additive group of all $2n\times 2n$ Hermitian matrices. The matrix pairs $(A_1, B_1)$ and $(A_2,B_2)\in \mathbb{C}^{n\times n}\times \mathbb{C}^{n\times n}$ are said to be left equivalent, denoted by
\begin{align*}
(A_1, B_1)\overset{\textrm{l.e.}}{\sim} (A_2, B_2)
\end{align*}
if $A_1=CA_2$, $B_1=CB_2$ for some invertible matrix $C$. A matrix pair $(A, B)$ is said to be  \textit{regular} if det$(A-\lambda B)\neq 0$ for some $\lambda\in \mathbb{C}$. It is well-known that for a regular matrix pair $(A, B)$ there are invertible matrices $P$ and $Q$ which transform $(A, B)$ to the Kronecker canonical form \cite{Gantmacher59} as
\begin{align*}
PAQ=\left[%
\begin{array}{cc}
  J & 0 \\
 0 & I \\
\end{array}%
\right],\ \ \ \ PBQ=\left[%
\begin{array}{cc}
  I & 0 \\
 0 & N \\
\end{array}%
\right],
\end{align*}
where $J$ is a Jordan matrix corresponding to the finite eigenvalues of $(A, B)$ and $N$ is a nilpotent Jordan matrix corresponding to the infinity eigenvalues. The index of a matrix pair $(A, B)$ is the index of nilpotency of  $N$, i.e., the matrix pair $( A,B)$ is
of index $\nu$, denoted by $\nu={\rm ind}_{\infty}(A,B)$, if $N^{\nu-1}\neq 0$ and $N^{\nu}=0$. By convention,
if $B$ is invertible, the pair $(A, B)$ is said to be of index zero.

The following three types of Riccati-type equations appear in many fields of applied sciences.
\begin{itemize}
\item Continuous-time Algebraic Riccati Equation (CARE) \cite{Lancaster95,Mehrmann91}:
\begin{align}\label{eq1.1}
-XGX+A^HX+XA+H=0,
\end{align}
where $A,H,G\in \mathbb{C}^{n\times n}$ with $G=G^H\geqslant 0$, $H=H^H\geqslant 0$ (positive semi-definite).
\item Discrete-time Algebraic Riccati Equation (DARE) \cite{Lancaster95,Mehrmann91}:
\begin{align}\label{eq1.2}
X=A^HX(I+GX)^{-1}A+H,
\end{align}
where $A,H,G\in \mathbb{C}^{n\times n}$ with $G=G^H\geqslant 0$, $H=H^H\geqslant 0$.
\item Nonlinear Matrix Equation (NME) \cite{err93}:
\begin{align}\label{eq1.3}
X+A^HX^{-1}A=Q,
\end{align}
where $A,Q\in \mathbb{C}^{n\times n}$ with $Q=Q^H> 0$.
\end{itemize}
These classical Riccati-type matrix equations occur in many important applications (see \cite{Anderson90,Ferrante1996359,Lancaster95,Mehrmann91}
 and references therein). The CAREs and DAREs have been studied extensively (see \cite{Ammar199155,Bai_Demmel_Gu1997,Benner:1997,Benner_Byers:2001,Chu_Fan_Lin_Wang:2004,Gudmundsson:1992,Guo:1998,Hench_Laub:1994,Kimura:1988,Lainiotis1994,Lancaster95,Laub:1979,Lu_Lin1993,Lu_Lin_Pearce:1999,Mehrmann91,Paige1981,VanDooren1981}). The NMEs have been studied in \cite{Anderson90,Engwerda:1993,Ferrante1996359,Guo:2001}.
The solutions of the Riccati-type equations can be solved by iterative methods such as the fixed-point iteration, the Newton's method, and the Structure-Preserving Doubling Algorithms (SDAs) \cite{err93,gula99,Lancaster95,lin06,Mehrmann91,Mehrmann12}. Recently, SDAs for solving the stabilizing solutions of the three Riccati-type equations have been applied successfully in many industrial applications. For instance, in the vibration analysis of fast trains \cite{GuoLin:2010} and Green's function calculation in nano research \cite{GuoLin_Sci:2010}, $Q=Q^T$ and $A=A^T$ in (\ref{eq1.3}) instead of being Hermitian. In the $\text{H}_{\infty}$- optimal controls \cite{Francis:1987,Mehrmann91}, the Riccati-type equations used are (\ref{eq1.1}) and (\ref{eq1.2})
but with $G$ and $H$ being Hermitian but not definite.  Lack of positive semi-definiteness of $G$ and $H$ in general may cause possible breakdown in
iteration formula containing an $(I+GH)^{-1}$  term such as the one in  (\ref{eq1.4}) below with $G=G_k,H=H_k$, but in the above applications some extra physical properties were used to show that the breakdown would never happen. Since the SDAs developed in papers \cite{err93,gula99,Lancaster95,lin06,Mehrmann91,Mehrmann12} enjoy well-defined iterates and favorable convergence rates, it is tempting to design SDAs that can be applied to new Riccati-type matrix equations in which the matrices $G, H$ and $Q$ are just Hermitian. Indeed, we shall demonstrate that a class of SDAs can be designed to produce sequences of symplectic matrix pairs in special forms as in \eqref{eq1.6} below. Furthermore, their convergence behavior and general property can be studied by a related continuous dynamical system which is structure-preserving such that each symplectic pair generated by the SDA coincides with the solution of the structure-preserving flow at some time-step. We now describe these SDAs for solving DARE/CARE and NME {\it with the matrices
$G,H$ and $Q$ being Hermitian, not necessarily positive semi-definite.}
\begin{itemize}
\item For solving DAREs \eqref{eq1.2}, the symplectic pairs $(\mathcal{M}_k,\mathcal{L}_k)=\left(\left[\begin{array}{cc} A_k & 0 \\ -H_k & I\end{array}\right],\left[\begin{array}{cc} I & G_k \\ 0 & A_k^H\end{array}\right]\right)$ are generated by
    \begin{itemize}
    \item[]\textbf{Algorithm }SDA-1.
    \begin{align}\label{eq1.4}
    \begin{array}{l}
    A_0=A,~G_0=G,~H_0=H,\\
    A_{k+1}=A_k(I+G_kH_k)^{-1}A_k,\\
    G_{k+1}=G_k+A_kG_k(I+H_kG_k)^{-1}A_k^H,\\
    H_{k+1}=H_k+A_k^H(I+H_kG_k)^{-1}H_kA_k.
    \end{array}
    \end{align}
    \end{itemize}
    It has been shown in \cite{HuangLin:2009,lin06} that under some mild conditions, the sequence of symplectic pairs $(\mathcal{M}_k,\mathcal{L}_k)$ quadratically/linearly converges, in which, as $k\to\infty$
    \begin{align*}
     & H_k\to\text{ the stabilizing solution of \eqref{eq1.2}},\\
     & G_k\to\text{ the stabilizing solution of the dual equation of \eqref{eq1.2}},\\
     & A_k\to 0.
    \end{align*}
Here the dual equation of \eqref{eq1.2} is of the form
    \begin{align*}
    Y=AY(I+HY)^{-1}A^H+G.
    \end{align*}
\item For solving CAREs \eqref{eq1.1}, one can transform it into a DARE \eqref{eq1.2} by using a suitable Cayley transformation \cite{Mehrmann1996}. Then \textbf{Algorithm }SDA-1 can be employed to find the desired stabilizing solution of CAREs.
\item For solving NMEs \eqref{eq1.3}, the symplectic pairs $(\mathcal{M}_k,\mathcal{L}_k)=\left(\left[\begin{array}{cc} A_k & 0 \\ Q_k & -I\end{array}\right],\left[\begin{array}{cc} -P_k & I \\ A^H_k & 0\end{array}\right]\right)$ are generated by
\begin{itemize}
\item[] \textbf{Algorithm }SDA-2.
    \begin{align}\label{eq1.5}
    \begin{array}{l}
     A_0=A,~Q_0=Q,~P_0=0,\\
     A_{k+1}=A_k(Q_k-P_k)^{-1}A_k,\\
     Q_{k+1}=Q_k-A_k^H(Q_k-P_k)^{-1}A_k,\\
     P_{k+1}=P_k+A_k(Q_k-P_k)^{-1}A_k^H.
    \end{array}
    \end{align}
\end{itemize}
It has been shown in \cite{ccghlx09,lin06} that under some conditions, the sequence of symplectic pairs  $(\mathcal{M}_k,\mathcal{L}_k)$ quadratically/linearly converges, in which, as $k\to\infty$
   \begin{align*}
     & Q_k\to\text{ the maximal solution of \eqref{eq1.3}},\\
     & P_k\to\text{ the minimal solution of \eqref{eq1.3}},\\
     & A_k\to 0.
    \end{align*}
\end{itemize}

\begin{description}
\item[\textbf{Eigenvector-Preserving Property:}] For each case above, if $\mathcal{M}_kU=\mathcal{L}_kUS$ or $\mathcal{M}_kVT=\mathcal{L}_kV$, where $U,V\in \mathbb{C}^{2n\times r}$ and $S,T\in \mathbb{C}^{r\times r}$, then $\mathcal{M}_{k+1}U=\mathcal{L}_{k+1}US^2$ or  $\mathcal{M}_{k+1}VT^2=\mathcal{L}_{k+1}V$, i.e., the SDA preserves the invariant subspaces for each $k$ and the squares of eigenvalues;
\item[\textbf{Structure-Preserving Property:}] The sequences of symplectic pairs $\{(\mathcal{M}_k,\mathcal{L}_k)\}_{k=1}^{\infty}$ generated by \textbf{Algorithms }SDA-1 and SDA-2 are, respectively, invariant in the sets
\begin{subequations}\label{eq1.6}
\begin{align}\label{eq1.6a}
\mathbb{S}_1=\left\{\left(\left[%
\begin{array}{cc}
  A & 0 \\
 H & I \\
\end{array}%
\right],\left[%
\begin{array}{cc}
  I & G \\
  0 & A^{H} \\
\end{array}%
\right]\right)\ |\ A,\ H=H^{H},\ G=G^{H}\in
\mathbb{C}^{n\times n}\right\},
\end{align}
and
\begin{align}\label{eq1.6b}
\mathbb{S}_2=\left\{\left(\left[%
\begin{array}{cc}
  A & 0 \\
  Q & -I \\
\end{array}%
\right],\left[%
\begin{array}{cc}
  P& I \\
   A^{H} &0\\
\end{array}%
\right]\right)\ |\ A,\ Q=Q^{H},\ P=P^{H}\in
\mathbb{C}^{n\times n}\right\}.
\end{align}
\end{subequations}
\end{description}
To study the symplectic pairs, we first quote the following theorem in \cite{Mehrmann12} regarding a simple left equivalence for regular symplectic pairs.

\begin{Theorem}\label{thm1.1}(see \cite{Mehrmann12}) Let $(\mathcal{M}, \mathcal{L})$ be a regular symplectic pair with $\mathcal{M},\ \mathcal{L}\in\mathbb{C}^{2n\times 2n}$. Then there exist $\mathcal{S}_1$, $\mathcal{S}_2\in Sp(n)$ and a Hermitian matrix $X=\left[%
\begin{array}{cc}
  X_{11} & X_{12} \\
 X_{21} & X_{22} \\
\end{array}%
\right]$ such that
\begin{align*}
(\mathcal{M}, \mathcal{L})\overset{\rm{l.e.}}{\sim}\left(\left[%
\begin{array}{cc}
  X_{12} & 0 \\
 X_{22} & I \\
\end{array}%
\right]\mathcal{S}_2, \left[%
\begin{array}{cc}
  I & X_{11} \\
 0 & X_{21} \\
\end{array}%
\right]\mathcal{S}_1\right).
\end{align*}
\end{Theorem}
\noindent  Theorem~\ref{thm1.1} provides us a classification for symplectic pairs.  Specifically, let $\mathcal{S}_1,\mathcal{S}_2\in Sp(n)$. We denote the class of symplectic pairs generated by $\mathcal{S}_1,\mathcal{S}_2$ as
\begin{subequations}\label{eq1.7}
\begin{align}\label{eq1.7a}
\mathbb{S}_{\mathcal{S}_1,\mathcal{S}_2}=\left\{\left(\mathcal{M},\mathcal{L}\right)| \ \mathcal{M}=\left[%
\begin{array}{cc}
  X_{12} & 0 \\
 X_{22} & I \\
\end{array}%
\right]\mathcal{S}_2,\ \
\mathcal{L}=\left[%
\begin{array}{cc}
  I & X_{11} \\
 0 & X_{21} \\
\end{array}%
\right]\mathcal{S}_1\right.\nonumber\\ \left. \text{ and }X=\left[%
\begin{array}{cc}
  X_{11} & X_{12} \\
 X_{21} & X_{22} \\
\end{array}%
\right]\in \mathbb{H}(2n)\right\}.
\end{align}
It is easily seen that  each pair $(\mathcal{M},\mathcal{L})\in \mathbb{S}_{\mathcal{S}_1,\mathcal{S}_2}$ is symplectic.  The bijective correspondence between $\mathbb{H}(2n)$ and $\mathbb{S}_{\mathcal{S}_1,\mathcal{S}_2}$  can be constructed by the transformation $T_{\mathcal{S}_1,\mathcal{S}_2}:\mathbb{H}(2n)\to\mathbb{S}_{\mathcal{S}_1,\mathcal{S}_2}$ with
\begin{align}\label{eq1.7b}
T_{\mathcal{S}_1,\mathcal{S}_2}(X)=\left(\left[%
\begin{array}{cc}
  X_{12} & 0 \\
 X_{22} & I \\
\end{array}%
\right]\mathcal{S}_2,\left[%
\begin{array}{cc}
  I & X_{11} \\
 0 & X_{21} \\
\end{array}%
\right]\mathcal{S}_1\right).
\end{align}
\end{subequations}

Therefore, the invariant sets for SDA-1 and SDA-2, i.e., $\mathbb{S}_1$ and $\mathbb{S}_2$, respectively,  given in \eqref{eq1.6}, can be  rewritten as $\mathbb{S}_1= \mathbb{S}_{I_{2n},I_{2n}}$ and $\mathbb{S}_2= \mathbb{S}_{-I_{2n},\mathcal{J}}$. Note that  $\mathbb{S}_1\nsubseteq \mathbb{S}_2$,  $\mathbb{S}_2\nsubseteq \mathbb{S}_1$.  In \cite{Kuo12}, a parameterized curve is constructed in $\mathbb{S}_2$  passing through the iterates generated by the fixed-point iteration,  the SDA and the Newton's method with some additional conditions.  Finding a smooth curve with a specific structure that passes through a sequence of iterates generated by some numerical algorithm is a popular topic studied by many researchers, especially in the study of the so-called Toda flow that links matrices/matrix pairs generated by QR/QZ-algorithm \cite{Chu:1984,Chu:1984c,Chu:1985,Chu:1995,Chu:2008,Symes:1982}. The  Toda flow is the solution of a nonlinear ordinary differential matrix equation  in which the eigenvalues are preserved, but the eigenvectors are changed in $t$. Rather than the invariance property of Toda flows, in this paper we shall focus on the flows $(\mathcal{M}(t),\mathcal{L}(t))$ on a specified $\mathbb{S}_{\mathcal{S}_1,\mathcal{S}_2}$ (i.e., the \textbf{Structure-Preserving Property}) that has \textbf{Eigenvector-Preserving Property}. More precisely, for a flow $\{(\mathcal{M}(t),\mathcal{L}(t))\ | \ t\in \mathbb{R}\}\subseteq \mathbb{S}_{\mathcal{S}_1,\mathcal{S}_2}$ satisfying the initial value problem with an initial regular pair $(\mathcal{M}(1),\mathcal{L}(1))=(\mathcal{M}_1,\mathcal{L}_1)$, the \textbf{Eigenvector-Preserving Property} of this flow can be stated as follows:
\begin{itemize}
\item[] {\it Assume that
\begin{align}\label{eq1.8}
\mathcal{M}_1U_{0}=0,\ \ \mathcal{L}_1U_{\infty}=0,\ \ \mathcal{M}_1U_1=\mathcal{L}_1U_1S,\ \ \mathcal{M}_1V_1T=\mathcal{L}_1V_1,
\end{align}
where $[U_{0},U_\infty,U_1,V_1]\in \mathbb{C}^{2n\times 2n}$ is invertible, and  $S$ and $T$  have no semi-simple zero eigenvalues. Then
\begin{align}\label{eq1.9}
\mathcal{M}(t)U_{0}=0,\ \ \mathcal{L}(t)U_{\infty}=0,\ \ \mathcal{M}(t)U_1=\mathcal{L}(t)U_1S^t,\ \ \mathcal{M}(t)V_1T^t=\mathcal{L}(t)V_1
\end{align}
hold.}
\end{itemize}

Here in \eqref{eq1.9},  $S^{t}$ and $T^t$, for $t\in \mathbb{R}$, represent the matrix exponentials.  Because $z^t=\exp(t\log(z))$ for each
$z\in \mathbb{C}\setminus\{0\}$, it follows from \cite[Definition 1.11 and Theorem 1.17]{Higham08} that the matrix exponentials $S^{t}$ and $T^t$ are well-defined if $S$ and $T$ are invertible. On the other hand, if $S$ (or $T$) is singular, then $S^{t}$ (or $T^t$) for $t\in \mathbb{R}$ is undefined. Hence,  to make the \textbf{Eigenvector-Preserving Property} meaningful, we  assume that the matrices $S$ and $T$ in \eqref{eq1.8} are invertible. This coincides with the assumption that the regular symplectic pair $(\mathcal{M}_1,\mathcal{L}_1)$ has only semi-simple zero and infinite eigenvalues (if exists). We shall show in Theorem~\ref{thm2.3} that this assumption for $(\mathcal{M}_1,\mathcal{L}_1)$ can result from the condition ${\rm ind}_{\infty}(\mathcal{M}_1,\mathcal{L}_1)\leqslant 1$.  Throughout this paper, we assume that the initial matrix pair $(\mathcal{M}_1,\mathcal{L}_1)$ is  regular and symplectic  with ${\rm ind}_{\infty}(\mathcal{M}_1,\mathcal{L}_1)\leqslant 1$. Note that if a matrix pair $(A,B)$
is regular and of index at most one, the corresponding time-invariant continuous system
\begin{align*}
B\frac{dx}{dt}=Ax(t)+f(t)
\end{align*}
has a unique solution for all admissible $f(t)$ with consistent initial conditions  \cite{Geerts93, Mehrmann91}. However, if the index of $(A,B)$ is larger than one,  impulses can occur in the time-invariant continuous system \cite{Geerts93}.

This paper is organized as
follows. In Section~\ref{sec2}, we introduce some preliminary results. In Sections~\ref{sec3}, we construct a differential equation such that its solution is invariant in $\mathbb{S}_{\mathcal{S}_1,\mathcal{S}_2}$ and has \textbf{Eigenvector-Preserving Property} for certain fixed symplectic matrices $\mathcal{S}_1$, $\mathcal{S}_2$. Such a flow is called a \textit{structure-preserving flow}. On the other hand, we also study the algebraic equation that is determined by both the \textbf{Eigenvector-Preserving Property} and the \textbf{Structure-Preserving Property}, in which the solution curve is denoted by $\mathcal{C}_{\mathcal{M}_1,\mathcal{L}_1}$.  We further show that the phase portrait of the structure-preserving flow is  $\mathcal{C}_{\mathcal{M}_1,\mathcal{L}_1}$. In  Subsection~\ref{sec3.2}, it will be shown that structure-preserving flows are governed by the Riccati differential equations (RDE) of the form
\begin{align*}
\begin{array}{l}
\dot{W}(t)=[-W(t),I]\mathscr{H}[I, W(t)^{\top} ]^{\top},\\
W(0)=W_0,
\end{array}
\end{align*}
where   $\mathscr{H}$ is a suitable Hamiltonian matrix. In addition, Radon's lemma (\cite{Radon27} or see Theorem~\ref{thm3.8}) leads to the explicit form $W(t)=P(t)Q(t)^{-1}$, where $[
Q(t)^{\top},
P(t)^{\top}]^{\top}= e^{\mathscr{H}t}[%
I,
W_0^{\top}
]^{\top}$.
This important relationship between linear differential equations and Riccati differential equations will be used to obtain an explicit representation formula for all solutions of RDEs as well as the structure-preserving flows. Therefore, the blow-up can occur at some finite time $t$ whenever $Q(t)$ is singular. In Subsection~\ref{sec3.3}, we adopt the Grassmann manifold to extend the domain of the structure-preserving flow to the whole $\mathbb{R}$ except some isolated points.  For two special classes of symplectic pairs $\mathbb{S}_1=\mathbb{S}_{I_{2n},I_{2n}}$ and $\mathbb{S}_2=\mathbb{S}_{-I_{2n},\mathcal{J}}$, it is shown in Subsection~\ref{sec3.4} that the structure-preserving flow passes through the iterates generated by SDA-1 and SDA-2, respectively.  Therefore, the SDA and its associated structure-preserving flow have identical asymptotic behaviors, including the stability, instability, periodicity, and quasi-periodicity of the dynamics. In Section~\ref{sec4}, we investigate the asymptotic behavior of $[
Q(t)^{\top},
P(t)^{\top}]^{\top}$ and use the results to analyze the convergence of SDAs.  Due to the special structure and properties of the Hamiltonian matrix, we apply a symplectic similarity transformation introduced by \cite{Lin1999469} to reduce $\mathscr{H}$ to a Hamiltonian Jordan canonical form $\mathfrak{J}$.  In Subsections~\ref{sec4.1} and \ref{sec4.2}, we first study the structure of $e^{\mathfrak{J}t}$ and then the asymptotic behaviors of $W(t)$ and $Q(t)^{-1}$, as $t\to\pm\infty$, with $\mathfrak{J}$ being of elementary cases.  The results for general $\mathfrak{J}$ are given in Subsection~\ref{sec4.3}.  The asymptotic analysis of SDAs as well as its convergence rate by using the asymptotic behavior of RDEs are shown in Subsection~\ref{sec4.4}.  Complementary proofs in Sections \ref{sec2}  and \ref{sec4} are given in Appendix.

\section{Preliminaries}\label{sec2}
In this section, we introduce notation, definitions and some preliminary results. For a matrix $A\in \mathbb{C}^{n\times n}$,  $A^{H}$ and $A^{\top}$ are the conjugate
transpose and the transpose of $A$, respectively.  $\sigma(A)$ denotes the spectrum  of $A$.  For each $\lambda_0\in \sigma(A)$, $\mathcal{R}_{\lambda_0}(A)=\{x| (A-\lambda_0 I)^{\nu}x=0\text{ for some }\nu\in \mathbb{N}\}$ is the generalized eigenspace of $A$ corresponding to the eigenvalue $\lambda_0$.  For a regular matrix pair $(A, B)$ with $A,B\in\mathbb{C}^{n\times n}$,  $\sigma(A,B)$  denotes the spectrum of $(A, B)$. Note that the matrix pair $(A, B)$ is said to have eigenvalues at infinity if $B$ is singular.

\begin{Definition}\label{def2.1}
Two subspaces $\mathbb{U}$ and $\mathbb{V}$ of $\mathbb{C}^{2n}$ are called $\mathcal{J}$-orthogonal if $u^H\mathcal{J}v=0$ for each $u\in \mathbb{U}$ and $v\in \mathbb{V}$. A subspace $\mathbb{U}$ of $\mathbb{C}^{2n}$ is called isotropic
if $x^{H}\mathcal{J}y=0$ for any $x,y\in \mathbb{U}$. An $n$-dimensional isotropic subspace is called a Lagrangian subspace.
\end{Definition}
Suppose that $\mathcal{H}\in\mathbb{C}^{n\times n}$ is Hamiltonian. It is well-known that for $\lambda, \mu\in \sigma(\mathcal{H})$ with $\lambda\neq -\bar{\mu}$, the subspaces $\mathcal{R}_{\lambda}(\mathcal{H})$ and $\mathcal{R}_{\mu}(\mathcal{H})$ are $\mathcal{J}$-orthogonal. Similarly, for a symplectic matrix $\mathcal{S}\in Sp(n)$ and $\lambda, \mu\in \sigma(\mathcal{S})$ with $\lambda\neq 1/\bar{\mu}$,  $\mathcal{R}_{\lambda}(\mathcal{S})$ and $\mathcal{R}_{\mu}(\mathcal{S})$ are $\mathcal{J}$-orthogonal. The $\mathcal{J}$-orthogonality also holds for invariant subspaces of Hamiltonian pairs and symplectic pairs. To prove this, we need the following lemma.

\begin{Lemma}\label{lem2.1}
Suppose that $(A_1,B_1)$ and $(A_2,B_2)\in \mathbb{C}^{n\times n}\times \mathbb{C}^{n\times n}$ are regular matrix pairs. If $\sigma(A_{1},B_{1})\cap\sigma(-A_{2},B_{2})=\emptyset$, then the equation
\begin{align}\label{eq2.1}
A_2XB_1+B_2XA_1=0
\end{align}
 has only trivial solution.
\end{Lemma}
\begin{proof}
We first consider the case that both $B_1$ and $B_2$ are invertible. Since $\sigma(B_{1}^{-1}A_{1})\cap\sigma(-B_{2}^{-1}A_{2})=\sigma(A_{1},B_{1})\cap\sigma(-A_{2},B_{2})=\emptyset$, Eq. \eqref{eq2.1}  has only trivial solution.

For the general case, we may assume that $B_1$ is singular. Therefore $(A_1,B_1)$ has eigenvalues at infinity. Since $\sigma(A_{1},B_{1})\cap\sigma(-A_{2},B_{2})=\emptyset$,  $B_2$ must be nonsingular. Let $\widehat{A}_2=B_2^{-1}A_2$. Since $(A_1,B_1)$ is regular, there are nonsingular matrices $P$ and $Q$ such that
\begin{align*}
PA_1Q=\left[%
\begin{array}{cc}
  J_1 & 0 \\
 0 & I \\
\end{array}%
\right],\ \ \ \ PB_1Q=\left[%
\begin{array}{cc}
  I & 0 \\
 0 & N_1 \\
\end{array}%
\right],
\end{align*}
where $N_1$ is nilpotent. Then \eqref{eq2.1} can be transformed into
\begin{align*}
\begin{array}{l}
\widehat{A}_2\widehat{X}_1+\widehat{X}_1J_1=0,\\
\widehat{A}_2\widehat{X}_2N_1+\widehat{X}_2=0,
\end{array}
\end{align*}
where $[ \widehat{X}_1,\widehat{X}_2]:=XP^{-1}$. Since $\sigma(\widehat{A}_2)\cap\sigma(-J_1)=\emptyset$ and $\sigma(N_1^{\top}\otimes \widehat{A}_2)=\{0\}$, we have $\widehat{X}_1=0$ and $\widehat{X}_2=0$, then  $X=[ \widehat{X}_1,\widehat{X}_2]P=0$. Hence, Eq. \eqref{eq2.1} has only trivial solution.
\end{proof}

\begin{Theorem}\label{thm2.2}
Let $(\mathcal{M},\mathcal{L})$, $(R_1,T_1)$ and $(R_2,T_2)$ be regular pairs and
$U_1$ and $U_2$ be of full column rank satisfying
\begin{align}
\mathcal{M}U_1T_1=\mathcal{L}U_1R_1,\text{ and } \mathcal{M}U_2T_2=\mathcal{L}U_2R_2.\label{eq2.2}
\end{align}
\begin{itemize}
\item[(i)] If $(\mathcal{M},\mathcal{L})$ is Hamiltonian and $\sigma(R_1,T_1)\cap\sigma(-R^H_2,T^H_2)=\emptyset$, then $U_1$ and $U_2$ are $\mathcal{J}$-orthogonal.
\item[(ii)] If $(\mathcal{M},\mathcal{L})$ is symplectic and $\sigma(R_1,T_1)\cap\sigma(T^H_2,R^H_2)=\emptyset$, then $U_1$ and $U_2$ are $\mathcal{J}$-orthogonal.
\end{itemize}
\end{Theorem}
\begin{proof}
$(i)$ Since $(\mathcal{M},\mathcal{L})$ is a regular Hamiltonian pair, we have
\begin{align*}
\left[\mathcal{M}\mathcal{J},\mathcal{L}\right]\left[%
\begin{array}{c}
  \mathcal{L}^H \\
  \mathcal{J} \mathcal{M}^H\\
\end{array}%
\right]=0,
\end{align*}
and rank$[\mathcal{M}\mathcal{J},\mathcal{L}]=2n$. Hence, the column vectors of $\left[%
\begin{array}{c}
  \mathcal{L}^H \\
  \mathcal{J} \mathcal{M}^H\\
\end{array}%
\right]$ form a basis of
null space of $[\mathcal{M}\mathcal{J},\mathcal{L}]$. On the other hand, it follows from \eqref{eq2.2} that
\begin{align}
\left[\mathcal{M}\mathcal{J},\mathcal{L}\right]\left[%
\begin{array}{c}
  \mathcal{J}^HU_1T_1 \\
  -U_1R_1\\
\end{array}%
\right]=0,\label{eq2.3}\\
\left[R_2^HU_2^H,-T_2^HU_2^H\mathcal{J}^H\right]\left[%
\begin{array}{c}
  \mathcal{L}^H \\
  \mathcal{J} \mathcal{M}^H\\
\end{array}%
\right]=0.\label{eq2.4}
\end{align}
Therefore, by \eqref{eq2.3} there is a nonsingular matrix $W$ such that
\begin{align*}
\left[%
\begin{array}{c}
  \mathcal{L}^H \\
  \mathcal{J} \mathcal{M}^H\\
\end{array}%
\right]W=\left[%
\begin{array}{c}
  \mathcal{J}^HU_1T_1 \\
  -U_1R_1\\
\end{array}%
\right].
\end{align*}
Multiplying $W$ from the right of \eqref{eq2.4}, we have
\begin{align*}
0&=\left[R_2^HU_2^H,-T_2^HU_2^H\mathcal{J}^H\right]\left[%
\begin{array}{c}
  \mathcal{J}^HU_1T_1 \\
  -U_1R_1\\
\end{array}%
\right]\\&=R_2^H(U_2^H\mathcal{J}^HU_1)T_1+T_2^H(U_2^H\mathcal{J}^HU_1)R_1.
\end{align*}
Since $\sigma(R_1,T_1)\cap\sigma(-R^H_2,T^H_2)=\emptyset$, it follows from Lemma \ref{lem2.1} that $U_2^H\mathcal{J}U_1=0$.

$(ii)$ Similarly, if $(\mathcal{M},\mathcal{L})$ is a regular symplectic pair, then from equations of \eqref{eq2.2} we have
\begin{align*}
T_2^H(U_2^H\mathcal{J}U_1)T_1-R_2^H(U_2^H\mathcal{J}U_1)R_1=0.
\end{align*}
Since $\sigma(R_1,T_1)\cap\sigma(T^H_2,R^H_2)=\emptyset$, it follows from Lemma \ref{lem2.1} that $U_2^H\mathcal{J}U_1=0$.
\end{proof}

From now on, we assume that the condition ${\rm ind}_{\infty}(\mathcal{M},\mathcal{L})\leqslant 1$ holds for a regular symplectic pair $(\mathcal{M},\mathcal{L})$, i.e., either the matrix  pair $(\mathcal{M},\mathcal{L})$ has no eigenvalue at infinity or the Jordan block corresponding to the eigenvalues at infinity is a zero matrix.

\begin{Theorem}\label{thm2.3}
Suppose $(\mathcal{M}, \mathcal{L})$ is a regular symplectic pair with $\mathcal{M}, \mathcal{L}\in \mathbb{C}^{2n\times 2n}$ and ${\rm ind}_{\infty}(\mathcal{M},\mathcal{L})\leqslant 1$. Then there is $\hat{n}\leqslant n$ such that ${\rm rank}(\mathcal{M})={\rm rank}(\mathcal{L})=n+\hat{n}$.  In addition, there exist $U_{0}, U_{\infty}\in \mathbb{C}^{2n\times \ell}$, $U_1\in \mathbb{C}^{2n\times 2\hat{n}}$ with $\ell=n-\hat{n}$ and a symplectic matrix $\widehat{\mathcal{S}}\in \mathbb{C}^{2\hat{n}\times 2\hat{n}}$ such that
\begin{align}\label{eq2.5}
\mathbf{U}^H\mathcal{J}_{n}\mathbf{U}=\left[%
\begin{array}{c|c}
  \mathcal{J}_{\hat{n}}&0 \\\hline
  0&\mathcal{J}_{\ell}\\
\end{array}%
\right]
\end{align}
with $\mathbf{U}=[U_1|U_0,U_\infty]$, and
\begin{align}\label{eq2.6}
\mathcal{M}U_{0}=0,\ \ \mathcal{L}U_{\infty}=0,\ \ \mathcal{M}U_1=\mathcal{L}U_1\widehat{\mathcal{S}}.
\end{align}
\end{Theorem}

\begin{Remark}\label{rem2.1}
From \eqref{eq2.5}, it is easily seen that $\mathbf{U}^{-1}=(\mathcal{J}_{\hat{n}}\oplus\mathcal{J}_{\ell})^{H}\mathbf{U}^H\mathcal{J}$.
\end{Remark}

\begin{proof}[Proof of Theorem~\ref{thm2.3}]
From Theorem~\ref{thm1.1}, the pair  $(\mathcal{M}, \mathcal{L})$  is left equivalent to the pair of the form $\left(\left[%
\begin{array}{cc}
  X_{12} & 0 \\
 X_{22} & I \\
\end{array}%
\right]\mathcal{S}_2, \left[%
\begin{array}{cc}
  I & X_{11} \\
 0 & X_{21} \\
\end{array}%
\right]\mathcal{S}_1\right)$ for some $X=[X_{ij}]_{1\le i,j\le 2}\in\mathbb{H}(2n)$ and $\mathcal{S}_1$, $\mathcal{S}_2\in Sp(n)$.  Therefore, the relation $X_{12}^H=X_{21}$, and the nonsingularity of $\mathcal{S}_1$ and $\mathcal{S}_2$ imply that ${\rm rank}(\mathcal{M})={\rm rank}(\mathcal{L})$.  Since ${\rm ind}_{\infty}(\mathcal{M},\mathcal{L})\leqslant 1$, there exist $U_{0}, U_{\infty}\in \mathbb{C}^{2n\times \ell}$  and $U_1\in \mathbb{C}^{2n\times 2(n-\ell)}$ such that the invariances of \eqref{eq2.6} hold,
where $\widehat{\mathcal{S}}\in \mathbb{C}^{ 2(n-\ell)\times 2(n-\ell)}$ is nonsingular and the column spaces spanned by $U_{0}$, $U_{\infty}$ and $U_1$ are the eigenspaces of $(\mathcal{M},\mathcal{L})$ corresponding to zero, infinity and finite-nonzero eigenvalues, respectively.  Applying Theorem~\ref{thm2.2}(ii) by setting $T_1=T_2=I$, $R_1=R_2=0$ and $U_1=U_2=U_0$, respectively, we have $U_0^H\mathcal{J}_nU_0=0$. Similarly, $U_1^H\mathcal{J}_nU_0$, $U_1^H\mathcal{J}_nU_\infty$ and $U_\infty^H\mathcal{J}_nU_\infty$ are also zero matrices. In addition, noting that $\mathbf{U}$ is nonsingular, we have
\begin{align}\label{eq2.7}
[U_1|U_0,U_\infty]^H\mathcal{J}_{n}[U_1|U_0,U_\infty]=\left[
\begin{array}{c|cc}
K_1&0&0\\
\hline
0&0&K_2\\
0&-K_2^H&0
\end{array}
\right],
\end{align}
where  $K_1$ is nonsingular skew-Hermitian and  $K_2$ is nonsingular. Resetting  $U_\infty:=U_\infty K_2^{-1}$, we then have $\mathcal{L}U_\infty=0$ and $[U_0,U_\infty]^H\mathcal{J}_n[U_0,U_\infty]=\mathcal{J}_\ell$. From the congruence transformation of \eqref{eq2.7}, it is easily seen that Hermitian matrices $iK_1$ and  $i\mathcal{J}_{\hat{n}}$ have the same inertia. Hence, there exists an invertible matrix $W$ such that $W^HK_1W=\mathcal{J}_{\hat{n}}$. Resetting $U_1:=U_1W$ and $\widehat{\mathcal{S}}:=W^{-1}\widehat{\mathcal{S}}W$, we then have $\mathcal{M}U_1=\mathcal{L}U_1\widehat{\mathcal{S}}$ and \eqref{eq2.5}.

Now, we show that  $\widehat{\mathcal{S}}$ is symplectic.
Since $(\mathcal{M},\mathcal{L})$ is  a regular symplectic pair, as in the proof of Theorem \ref{thm2.2} above,
the columns of $\left[%
\begin{array}{c}
   \mathcal{J} \mathcal{M}^H \\
 - \mathcal{J} \mathcal{L}^H\\
\end{array}%
\right]$ form a basis of
null space of $[\mathcal{M},\mathcal{L}]$. From \eqref{eq2.6}, we have
$\left[\mathcal{M},\mathcal{L}\right]\left[%
\begin{array}{c}
  U_1 \\
  -U_1\widehat{\mathcal{S}}\\
\end{array}%
\right]=0$. Hence there is a matrix $W\in \mathbb{C}^{2n\times 2\hat{n}}$ of full column rank such that
\begin{align}\label{eq2.8}
\left[%
\begin{array}{c}
   \mathcal{J} \mathcal{M}^H \\
 - \mathcal{J} \mathcal{L}^H\\
\end{array}%
\right]W=
\left[%
\begin{array}{c}
  U_1 \\
  -U_1\widehat{\mathcal{S}}\\
\end{array}%
\right].
\end{align}
Taking the conjugate transpose of \eqref{eq2.6}, we obtain
\begin{align*}
0=\left[U_1^H,-\widehat{\mathcal{S}}^HU_1^H\right]\left[%
\begin{array}{c}
  \mathcal{M}^H \\
  \mathcal{L}^H\\
\end{array}%
\right]=\left[-U_1^H\mathcal{J},-\widehat{\mathcal{S}}^HU_1^H\mathcal{J}\right]\left[%
\begin{array}{c}
 \mathcal{J} \mathcal{M}^H \\
  -\mathcal{J}\mathcal{L}^H\\
\end{array}%
\right].
\end{align*}
Applying \eqref{eq2.5} and \eqref{eq2.8} to the last equation yields that
\begin{align*}
-\mathcal{J}_{\hat{n}}+\widehat{\mathcal{S}}^H\mathcal{J}_{\hat{n}}\widehat{\mathcal{S}}=\left[-U_1^H\mathcal{J},-\widehat{\mathcal{S}}^HU_1^H\mathcal{J}\right]\left[%
\begin{array}{c}
  U_1 \\
  -U_1\widehat{\mathcal{S}}\\
\end{array}%
\right]=0.
\end{align*}
Thus, $\widehat{\mathcal{S}}$ is a symplectic matrix.
\end{proof}

Note that the matrix $\widehat{\mathcal{S}}$ in Theorem~\ref{thm2.3} is symplectic. It is proven in Theorem~\ref{thm5.1} that  there is a Hamiltonian matrix $\widehat{\mathcal{H}}$ satisfying $e^{\widehat{\mathcal{H}}}=\widehat{\mathcal{S}}$. Using $\widehat{\mathcal{H}}$, we shall construct a Hamiltonian matrix $\mathcal{H}$ which has invariant subspaces spanned by $U_0$, $U_\infty$, and $U_1$.

\begin{Theorem}\label{thm2.4}
Suppose $(\mathcal{M},\mathcal{L})$ is a regular symplectic pair with $\mathcal{M},\mathcal{L}\in{\mathbb C}^{2n\times2n}$ and ${\rm ind}_{\infty}(\mathcal{M},\mathcal{L})\leqslant 1$. Let the matrices $\mathbf{U}$ and $\widehat{\mathcal{S}}$ be given as in Theorem~\ref{thm2.3}, and $\widehat{\mathcal{H}}\in \mathbb{C}^{2\hat{n}\times 2\hat{n}}$ be the Hamiltonian matrix such that
\begin{align}\label{eq2.9}
e^{\widehat{\mathcal{H}}}=\widehat{\mathcal{S}}.
\end{align}
Then the matrix \begin{align}\label{eq2.10}
\mathcal{H}=\mathbf{U}\left[%
\begin{array}{cc}
  \widehat{\mathcal{H}}&0 \\
 0&0\\
\end{array}%
\right](\mathcal{J}_{\hat{n}}\oplus\mathcal{J}_{\ell})^{H}\mathbf{U}^H\mathcal{J}
\end{align}
is Hamiltonian.
\end{Theorem}

\begin{proof}
Since $\widehat{\mathcal{H}}$ is Hamiltonian, we have
\begin{align*}
\mathcal{H}\mathcal{J}&=-\mathbf{U}\left[%
\begin{array}{cc}
 \widehat{\mathcal{H}}\mathcal{J}_{\hat{n}}^H&0 \\
 0&0\\
\end{array}%
\right]\mathbf{U}^H=-\mathbf{U}\left[%
\begin{array}{cc}
 \mathcal{J}_{\hat{n}}\widehat{\mathcal{H}}^H&0 \\
 0&0\\
\end{array}%
\right]\mathbf{U}^H\\
&=\mathcal{J}^H\mathcal{J}^H\mathbf{U}\left[%
\begin{array}{cc}
 \mathcal{J}_{\hat{n}}\widehat{\mathcal{H}}^H&0 \\
 0&0\\
\end{array}%
\right]\mathbf{U}^H=\mathcal{J}^H\mathcal{H}^H.
\end{align*}
Hence, $\mathcal{H}$ is Hamiltonian.
\end{proof}

\begin{Remark}\label{rem2.2}
Suppose that $(\mathcal{M},\mathcal{L})$ is a {\it real} regular symplectic pair. Then  $\mathbf{U}$ is real and  $\widehat{\mathcal{S}}$ is  real symplectic. In \cite{Dieci98}, under the assumptions
\begin{itemize}
\item[(i)] $\widehat{\mathcal{S}}$ has an even number of Jordan blocks of each size relative to every negative eigenvalue;
\item[(ii)] the size of two identical Jordan blocks corresponding to eigenvalue $-1$ is odd;
\end{itemize}
it is shown that there is a {\it real} Hamiltonian matrix $\widehat{\mathcal{H}}$ such that $e^{\widehat{\mathcal{H}}}=\widehat{\mathcal{S}}$. Hence, the Hamiltonian $\mathcal{H}$ defined in \eqref{eq2.10} is real.
\end{Remark}

Suppose that $\mathcal{L}$ is invertible. It follows from Theorem~\ref{thm2.3} that $\mathcal{M}$ is also invertible.  Therefore,  $U_0$ and $U_\infty$ in \eqref{eq2.6} are absent. On the other hand, the matrix $\mathcal{L}^{-1}\mathcal{M}$ is symplectic. From \eqref{eq2.6} and Theorem~\ref{thm2.4}, we have that $e^{\mathcal{H}}=\mathcal{L}^{-1}\mathcal{M}$ for some Hamiltonian matrix $\mathcal{H}$, that is, $\mathcal{M}=\mathcal{L}e^{\mathcal{H}}$. For the case that $\mathcal{L}$ is singular and $(\mathcal{M}, \mathcal{L})$ is a regular symplectic pair with ${\rm ind}_{\infty}(\mathcal{M}, \mathcal{L})\leqslant 1$, it is natural to ask whether there is a Hamiltonian matrix $\mathcal{H}$ such that $\mathcal{M}=\mathcal{L}e^{\mathcal{H}}$. To this end, we need the following lemma.

\begin{Lemma}\label{lem2.5}
Suppose that $(\mathcal{M}, \mathcal{L})$ is a regular symplectic pair. If $\mathcal{M}= \mathcal{L} W$ for some nonsingular $W$, then both $\mathcal{M}$ and $\mathcal{L}$ are invertible.
\end{Lemma}

\begin{proof}
From Theorem \ref{thm1.1}, there are two symplectic matrices $\mathcal{S}_1$ and $\mathcal{S}_2$, and a Hermitian matrix $X=[X_{ij}]_{1\le i, j\le2}$ such that
\begin{align*}
\mathcal{M}=C\left[%
\begin{array}{cc}
  X_{12} & 0 \\
 X_{22} & I \\
\end{array}%
\right]\mathcal{S}_2,\ \  \mathcal{L}=C\left[%
\begin{array}{cc}
  I & X_{11} \\
 0 & X_{21} \\
\end{array}%
\right]\mathcal{S}_1,
\end{align*}
where $C$ is nonsingular. Suppose that $\mathcal{M}= \mathcal{L} W$. Then we have
\begin{align*}
\left[%
\begin{array}{cc}
  X_{12} & 0 \\
 X_{22} & I \\
\end{array}%
\right]\mathcal{S}_2=\left[%
\begin{array}{cc}
  I & X_{11} \\
 0 & X_{21} \\
\end{array}%
\right]\mathcal{S}_1 W.
\end{align*}
Since $\mathcal{S}_1$, $\mathcal{S}_2$ and $W$ are nonsingular, it is easily seen that $X_{12}$ and $X_{21}$ are nonsingular. Thus,  $\mathcal{M}$ and $\mathcal{L}$ are invertible.
\end{proof}


\begin{Lemma}\label{lem2.6}
Suppose $(\mathcal{M},\mathcal{L})$ is a regular symplectic pair with $\mathcal{M},\mathcal{L}\in{\mathbb C}^{2n\times2n}$ and ${\rm ind}_{\infty}(\mathcal{M},\mathcal{L})\leqslant 1$. Let the matrices $\mathbf{U}$ and $\mathcal{H}$ be given as in Theorems~\ref{thm2.3} and \ref{thm2.4}, respectively. Let
\begin{align}\label{eq2.11}
\begin{array}{l}
\Pi_0=\mathbf{U}\left[%
\begin{array}{c|cc}
I_{2\hat{n}}&0 &0\\\hline
 0&I_{\ell}&0\\
  0&0&0\\
\end{array}%
\right](\mathcal{J}_{\hat{n}}\oplus\mathcal{J}_{\ell})^{H}\mathbf{U}^H\mathcal{J},\\
\Pi_\infty=\mathbf{U}\left[%
\begin{array}{c|cc}
I_{2\hat{n}}&0 &0\\\hline
 0&0&0\\
  0&0&I_{\ell}\\
\end{array}%
\right](\mathcal{J}_{\hat{n}}\oplus\mathcal{J}_{\ell})^{H}\mathbf{U}^H\mathcal{J},
\end{array}
\end{align}
Then we have
\begin{align}\label{eq2.12}
\mathcal{M}\Pi_0=\mathcal{L}\Pi_\infty e^{\mathcal{H}}.
\end{align}
\end{Lemma}

\begin{Remark}\label{rem2.1.5}
It follows from Remark~\ref{rem2.1} that both $\Pi_0$ and $\Pi_\infty$ are idempotent, i.e.,  $\Pi_0^2=\Pi_0$ and $\Pi_\infty^2=\Pi_\infty$. In addition, if ${\rm ind}_{\infty}(\mathcal{M},\mathcal{L})=0$, then both $\mathcal{M}$ and $\mathcal{L}$ are invertible, which implies $\hat{n}=n$.  In this case, $\Pi_0=\Pi_\infty=I$.  Therefore, $\mathcal{M}=\mathcal{L}e^{\mathcal{H}}$ with some appropriate Hamiltonian matrix $\mathcal{H}$.  This coincides with Lemma~\ref{lem2.5}.
\end{Remark}

\begin{proof}[Proof of Lemma~\ref{lem2.6}]
From \eqref{eq2.6}, \eqref{eq2.11} and Remark \ref{rem2.1}, we have
\begin{align}\label{eq2.13}
\mathcal{M}\Pi_0&=\left[\mathcal{L}U_1\widehat{\mathcal{S}}|0,0\right](\mathcal{J}_{\hat{n}}\oplus\mathcal{J}_{\ell})^{H}[U_1| U_0, U_{\infty}]^H\mathcal{J}\nonumber\\
&=\mathcal{L}[U_1| U_0, U_{\infty}]\left[%
\begin{array}{c|cc}
\widehat{\mathcal{S}}&0 &0\\\hline
 0&0&0\\
  0&0&I_{\ell}\\
\end{array}%
\right](\mathcal{J}_{\hat{n}}\oplus\mathcal{J}_{\ell})^{H}[U_1| U_0, U_{\infty}]^H\mathcal{J}\nonumber\\
&=\mathcal{L}\Pi_\infty[U_1| U_0, U_{\infty}]\left[%
\begin{array}{c|c}
\widehat{\mathcal{S}}&0 \\\hline
  0&I_{2\ell}\\
\end{array}%
\right](\mathcal{J}_{\hat{n}}\oplus\mathcal{J}_{\ell})^{H}[U_1| U_0, U_{\infty}]^H\mathcal{J}.
\end{align}
It follows from \eqref{eq2.9} and \eqref{eq2.10} that
\begin{align*}
e^{\mathcal{H}}=\mathbf{U}\left[%
\begin{array}{cc}
  e^{\widehat{\mathcal{H}}}&0 \\
 0&e^{0}\\
\end{array}%
\right](\mathcal{J}_{\hat{n}}\oplus\mathcal{J}_{\ell})^{H}\mathbf{U}^H\mathcal{J}
=\mathbf{U}\left[%
\begin{array}{cc}
 \widehat{\mathcal{S}}&0 \\
 0&I_{2\ell}\\
\end{array}%
\right](\mathcal{J}_{\hat{n}}\oplus\mathcal{J}_{\ell})^{H}\mathbf{U}^H\mathcal{J}.
\end{align*}
From \eqref{eq2.13}, Eq. \eqref{eq2.12} holds.
\end{proof}

To make the correspondence between the constructed matrices in the previous lemmas/theorems and the symplectic pairs $(\mathcal{M},\mathcal{L})$, we  use the following notations  throughout this paper.

\begin{Definition}\label{def2.2}
Suppose $(\mathcal{M},\mathcal{L})$ is a regular symplectic pair with $\mathcal{M},\mathcal{L}\in{\mathbb C}^{2n\times2n}$ and ${\rm ind}_{\infty}(\mathcal{M},\mathcal{L})\leqslant 1$. We define
\begin{align*}
& \hat{n}:= \hat{n}(\mathcal{M},\mathcal{L}),\mathbf{U}:=\mathbf{U}(\mathcal{M},\mathcal{L}),\text{ and }~\widehat{\mathcal{S}}:=\widehat{\mathcal{S}}(\mathcal{M},\mathcal{L})\text{ in Theorem~\ref{thm2.3}};\\
& \widehat{\mathcal{H}}:=\widehat{\mathcal{H}}(\mathcal{M},\mathcal{L})\text{ and }\mathcal{H}:=\mathcal{H}(\mathcal{M},\mathcal{L})\text{ in Theorem~\ref{thm2.4}};\\
& \Pi_0:= \Pi_0(\mathcal{M},\mathcal{L})\text{ and }\Pi_\infty:=\Pi_\infty(\mathcal{M},\mathcal{L})\text{ in Lemma~\ref{lem2.6}}.\\
\end{align*}
\end{Definition}

We now provide a perturbation theory for the symplectic pair $(\mathcal{M},\mathcal{L})$ that preserves the invariant subspaces spanned by $U_0$, $U_\infty$ and $U_1$, as well as all finite nonzero eigenvalues, but  perturbs the eigenvalues $0$'s and $\infty$'s to $O(\varepsilon)$ and $O(1/\varepsilon)$, respectively.

\begin{Theorem}\label{thm2.7}
Suppose $(\mathcal{M},\mathcal{L})$ is a regular symplectic pair with  $\mathcal{M},\mathcal{L}\in{\mathbb C}^{2n\times2n}$ and ${\rm ind}_{\infty}(\mathcal{M},\mathcal{L})= 1$. Let $\mathbf{U}=\mathbf{U}(\mathcal{M},\mathcal{L})$ and $\widehat{\mathcal{S}}=\widehat{\mathcal{S}}(\mathcal{M},\mathcal{L})$ be given as in Definition~\ref{def2.2} and let
$\Phi^\varepsilon\in{\mathbb C}^{\ell\times\ell}$ be a family of nonsingular matrices with $\|\Phi^\varepsilon\|\leqslant \varepsilon$ for each $\varepsilon>0$. If
\begin{align}\label{eq2.14}
\mathcal{M}^\varepsilon=\mathcal{M}+\Delta\mathcal{M}^\varepsilon,\ \ \mathcal{L}^\varepsilon=\mathcal{L}+\Delta\mathcal{L}^\varepsilon,
\end{align}
where
\begin{align}\label{eq2.15}
\Delta\mathcal{M}^\varepsilon=-\mathcal{L}U_0{\Phi^\varepsilon}^HU_{\infty}^H\mathcal{J},\  \Delta\mathcal{L}^\varepsilon=\mathcal{M}U_{\infty}\Phi^\varepsilon U_{0}^H\mathcal{J},
\end{align}
then
$(\mathcal{M}^\varepsilon, \mathcal{L}^\varepsilon)$ is a regular symplectic pair with  $\mathcal{L}^\varepsilon$ being invertible. Moreover, $\mathcal{M}^\varepsilon$ and $\mathcal{L}^\varepsilon$ satisfy
\begin{align}\label{eq2.16}
\begin{array}{l}
\mathcal{M}^\varepsilon U_0=\mathcal{L}^\varepsilon U_0{\Phi^\varepsilon}^H,\\
\mathcal{M}^\varepsilon U_{\infty}\Phi^\varepsilon=\mathcal{L}^\varepsilon U_{\infty},\\
\mathcal{M}^\varepsilon U_1= \mathcal{L}^\varepsilon U_1\widehat{\mathcal{S}},
\end{array}
\end{align}
and
\begin{align}\label{eq2.17}
(\mathcal{M}^{\varepsilon}, \mathcal{L}^{\varepsilon})\rightarrow (\mathcal{M}, \mathcal{L})\text{ as }\varepsilon\rightarrow 0.
\end{align}
\end{Theorem}

\begin{proof}
From \eqref{eq2.5}, it holds that
\begin{align}\label{eq2.18}
U_0^H\mathcal{J}U_{\infty}=I,\ U_{\infty}^H\mathcal{J}U_{0}=-I,\ U_{\infty}^H\mathcal{J}U_{\infty}=U_{0}^H\mathcal{J}U_{0}=0.
\end{align}
For each $\varepsilon>0$, from \eqref{eq2.14}, \eqref{eq2.15}  and \eqref{eq2.18} it holds that
\begin{align*}
\mathcal{M}^{\varepsilon}\mathcal{J}\mathcal{M}^{\varepsilon H}&=(\mathcal{M}+\Delta\mathcal{M}^{\varepsilon})\mathcal{J}(\mathcal{M}+\Delta\mathcal{M}^{\varepsilon})^H\\
&=\mathcal{M}\mathcal{J}\mathcal{M}^{H}+\mathcal{M}\mathcal{J}\Delta\mathcal{M}^{\varepsilon H}+\Delta\mathcal{M}^{\varepsilon }\mathcal{J}\mathcal{M}^H+\Delta\mathcal{M}^{\varepsilon }\mathcal{J}\Delta\mathcal{M}^{\varepsilon H}\\
&=\mathcal{L}\mathcal{J}\mathcal{L}^{H}-\mathcal{M}U_{\infty}\Phi^\varepsilon U_0^H\mathcal{L}^H+\mathcal{L}U_0{\Phi^\varepsilon}^HU_{\infty}^H\mathcal{M}^H\\
&=\mathcal{L}\mathcal{J}\mathcal{L}^{H}+\Delta\mathcal{L}^{\varepsilon }\mathcal{J}\mathcal{L}^H+\mathcal{L}\mathcal{J}\Delta\mathcal{L}^{\varepsilon H}+\Delta\mathcal{L}^{\varepsilon }\mathcal{J}\Delta\mathcal{L}^{\varepsilon H}\\
&=(\mathcal{L}+\Delta\mathcal{L}^{\varepsilon})\mathcal{J}(\mathcal{L}+\Delta\mathcal{L}^{\varepsilon})^H=\mathcal{L}^{\varepsilon}\mathcal{J}\mathcal{L}^{\varepsilon H}.
\end{align*}
That is, $(\mathcal{M}^{\varepsilon}, \mathcal{L}^{\varepsilon})$ forms a symplectic pair. Now, we show that $\mathcal{L}^{\varepsilon}$ is invertible. Since $(\mathcal{M},\mathcal{L})$ is a regular symplectic pair, there exists a nonzero constant $\lambda_0$ such that $\mathcal{M}-\lambda_0\mathcal{L}$ is invertible. Using the fact that $\mathbf{U}=[U_1,U_0,U_{\infty}]$ is nonsingular, it follows from \eqref{eq2.6} that
\begin{align*}
(\mathcal{M}-\lambda_0\mathcal{L})\mathbf{U}=\left[\mathcal{M}U_1-\lambda_0\mathcal{L}U_1,-\lambda_0\mathcal{L}U_0,\mathcal{M}U_{\infty}\right]
=\left[\mathcal{L}U_1(\widehat{\mathcal{S}}-\lambda_0I),-\lambda_0\mathcal{L}U_0,\mathcal{M}U_{\infty}\right]
\end{align*}
is nonsingular, and hence, $\widehat{\mathcal{S}}-\lambda_0I$ is also invertible.  Since $\Phi^{\varepsilon}$ is nonsingular, from \eqref{eq2.14}, \eqref{eq2.15}  and \eqref{eq2.18} together with the fact that $U_0^H\mathcal{J}U_1=0$, we have
\begin{align*}
\mathcal{L}^{\varepsilon}\mathbf{U}&=\left[\mathcal{L}^{\varepsilon}U_1,\mathcal{L}^{\varepsilon}U_0, \mathcal{L}^{\varepsilon}U_{\infty}\right]=\left[\mathcal{L}U_1,\mathcal{L}U_0,\mathcal{M}U_{\infty}\Phi^{\varepsilon}\right]\\&=(\mathcal{M}-\lambda_0\mathcal{L})\mathbf{U}\left((\widehat{\mathcal{S}}-\lambda_0I)^{-1}\oplus (-\lambda_0)^{-1}I\oplus \Phi^{\varepsilon}\right)
\end{align*}
is invertible and then
 $\mathcal{L}^{\varepsilon}$ is invertible. Hence, $(\mathcal{M}^{\varepsilon}, \mathcal{L}^{\varepsilon})$ is a regular symplectic pair.

From \eqref{eq2.6} and \eqref{eq2.18}, we have
\begin{align*}
\mathcal{M}^{\varepsilon}U_0&=(\mathcal{M}+\Delta\mathcal{M}^{\varepsilon})U_0=-\mathcal{L}U_0{\Phi^{\varepsilon}}^HU_{\infty}^H\mathcal{J}U_0\\&=\mathcal{L}U_0{\Phi^{\varepsilon}}^H=\mathcal{L}^{\varepsilon}U_0{\Phi^{\varepsilon}}^H,\\
\mathcal{L}^{\varepsilon}U_{\infty}&=(\mathcal{L}+\Delta\mathcal{L}^{\varepsilon})U_{\infty}=\mathcal{M}U_{\infty}\Phi^{\varepsilon}U_{0}^H\mathcal{J}U_{\infty}\\
&=\mathcal{M}U_{\infty}\Phi^{\varepsilon}=\mathcal{M}^{\varepsilon}U_{\infty}\Phi^{\varepsilon},\\
\mathcal{M}^{\varepsilon}U_1&=(\mathcal{M}+\Delta\mathcal{M}^{\varepsilon})U_1=\mathcal{M}U_1=\mathcal{L}U_1\widehat{\mathcal{S}}\\&=(\mathcal{L}+\Delta\mathcal{L}^{\varepsilon})U_1\widehat{\mathcal{S}}=\mathcal{L}^{\varepsilon}U_1\widehat{\mathcal{S}}.
\end{align*}
Thus, equations of \eqref{eq2.16} hold. Since  $\|\Phi^{\varepsilon}\|\leqslant \varepsilon$, \eqref{eq2.17} also holds.
\end{proof}

\begin{Corollary}\label{cor2.8}
Suppose $(\mathcal{M}, \mathcal{L})\in \mathbb{S}_{\mathcal{S}_1,\mathcal{S}_2}$ is a regular symplectic pair with ${\rm ind}_{\infty}(\mathcal{M}, \mathcal{L})\leqslant 1$. Let $\Phi^{\varepsilon}$ be nonsingular with  $\|\Phi^{\varepsilon}\|\leqslant \varepsilon$ for each $0<\varepsilon\ll 1$, and $\mathcal{M}^{\varepsilon}$, $ \mathcal{L}^{\varepsilon}$ be given as in Theorem~\ref{thm2.7}. Then there exists $(\widetilde{\mathcal{M}}^{\varepsilon}, \widetilde{\mathcal{L}}^{\varepsilon})\in \mathbb{S}_{\mathcal{S}_1,\mathcal{S}_2}$ for $0\leqslant\varepsilon\ll 1$, such that
\begin{align*}
\mathcal{M}^{\varepsilon}-\lambda \mathcal{L}^{\varepsilon}\overset{\rm{l.e.}}{\sim}\widetilde{\mathcal{M}}^{\varepsilon}-\lambda \widetilde{\mathcal{L}}^{\varepsilon}.
\end{align*}
Moreover, for each $0<\varepsilon\ll 1$, $\widetilde{\mathcal{M}}^{\varepsilon}$ and $\widetilde{\mathcal{L}}^{\varepsilon}$ are invertible satisfying \eqref{eq2.16} and \eqref{eq2.17}.
\end{Corollary}

\begin{proof}
Since $(\mathcal{M}, \mathcal{L})\in \mathbb{S}_{\mathcal{S}_1,\mathcal{S}_2}$, it holds that  $\mathcal{M}= \left[%
\begin{array}{cc}
  X_{12} & 0 \\
 X_{22} & I \\
\end{array}%
\right]\mathcal{S}_2$, $ \mathcal{L}=\left[%
\begin{array}{cc}
  I & X_{11} \\
 0 & X_{21} \\
\end{array}%
\right]\mathcal{S}_1$, where $\left[%
\begin{array}{cc}
  X_{11} & X_{12} \\
 X_{21} & X_{22} \\
\end{array}%
\right]$
is Hermitian. Since $\|\Phi^{\varepsilon}\|\leqslant \varepsilon$ for $0<\varepsilon\ll 1$, from \eqref{eq2.14} we have
\begin{align*}
\mathcal{M}^{\varepsilon}=\left[%
\begin{array}{cc}
  X_{12}+O(\varepsilon ) & O(\varepsilon )  \\
 X_{22}+O(\varepsilon )  & I+O(\varepsilon )  \\
\end{array}%
\right]\mathcal{S}_2,\
 \mathcal{L}^{\varepsilon}=\left[%
\begin{array}{cc}
  I+O(\varepsilon )  & X_{11}+O(\varepsilon )  \\
 O(\varepsilon )  & X_{21}+O(\varepsilon )  \\
\end{array}%
\right]\mathcal{S}_1,
\end{align*}
where $O(\varepsilon)$ is big O of $\varepsilon$. Applying row operations to $(\mathcal{M}^{\varepsilon},\mathcal{L}^{\varepsilon})$ yields
\begin{align*}
(\mathcal{M}^{\varepsilon},\mathcal{L}^{\varepsilon})&\overset{\rm{l.e.}}{\sim} \left(\left[%
\begin{array}{cc}
  X_{12}+O(\varepsilon ) & 0  \\
 X_{22}+O(\varepsilon )  & I  \\
\end{array}%
\right]\mathcal{S}_2,\
 \left[%
\begin{array}{cc}
  I+O(\varepsilon )  & X_{11}+O(\varepsilon )  \\
 O(\varepsilon )  & X_{21}+O(\varepsilon )  \\
\end{array}%
\right]\mathcal{S}_1\right)\\
&\overset{\rm{l.e.}}{\sim} \left(\left[%
\begin{array}{cc}
  \widetilde{X}_{12}(\varepsilon ) & 0  \\
 \widetilde{X}_{22}(\varepsilon )  & I  \\
\end{array}%
\right]\mathcal{S}_2,\
 \left[%
\begin{array}{cc}
  I  & \widetilde{X}_{11}(\varepsilon )  \\
 0  & \widetilde{X}_{21}(\varepsilon )  \\
\end{array}%
\right]\mathcal{S}_1\right)\equiv \left(\widetilde{\mathcal{M}}^{\varepsilon}, \widetilde{\mathcal{L}}^{\varepsilon}\right),
\end{align*}
where $ \widetilde{X}_{ij}(\varepsilon ) =X_{ij}+O(\varepsilon )$ for $1\leqslant i,j \leqslant2$. Hence,
 $(\widetilde{\mathcal{M}}^{\varepsilon},\widetilde{\mathcal{L}}^{\varepsilon})\rightarrow (\mathcal{M}, \mathcal{L})$ as $\varepsilon\rightarrow 0$. Using the fact that $(\mathcal{M}^{\varepsilon},\mathcal{L}^{\varepsilon})\overset{\rm{l.e.}}{\sim} (\widetilde{\mathcal{M}}^{\varepsilon}, \widetilde{\mathcal{L}}^{\varepsilon})$, it follows from Theorem \ref{thm2.7} that $\widetilde{\mathcal{M}}^{\varepsilon}$ and $\widetilde{\mathcal{L}}^{\varepsilon}$ are invertible, and satisfy the equalities of \eqref{eq2.16}.
Since $ (\widetilde{\mathcal{M}}^{\varepsilon}, \widetilde{\mathcal{L}}^{\varepsilon})$ is symplectic and  $\left[%
\begin{array}{cc}
  \widetilde{X}_{11}(\varepsilon )  & \widetilde{X}_{12}(\varepsilon )  \\
 \widetilde{X}_{21}(\varepsilon )  & \widetilde{X}_{22}(\varepsilon )  \\
\end{array}%
\right]$
is Hermitian, we have  $(\widetilde{\mathcal{M}}^{\varepsilon}, \widetilde{\mathcal{L}}^{\varepsilon})\in \mathbb{S}_{\mathcal{S}_1,\mathcal{S}_2}$ for $0\leqslant\varepsilon\ll 1$.
\end{proof}

\section{Structure-Preserving Flows}\label{sec3}

\subsection{Construction of Structure-Preserving Flows}\label{sec3.1}
Suppose that $(\mathcal{M}_1, \mathcal{L}_1)$ is a regular symplectic pair with ${\rm ind}_{\infty}(\mathcal{M}_1, \mathcal{L}_1)\leqslant 1$. From Theorem~\ref{thm1.1}, there exist two symplectic matrices $\mathcal{S}_1$ and $\mathcal{S}_2$ such that $(\mathcal{M}_1, \mathcal{L}_1)\in \mathbb{S}_{\mathcal{S}_1,\mathcal{S}_2}$. In this subsection we shall construct a differential equation with $(\mathcal{M}_1, \mathcal{L}_1)$ as an initial matrix pair such that the flow of this differential equation is invariant in $\mathbb{S}_{\mathcal{S}_1,\mathcal{S}_2}$.

We first consider the case that $\mathcal{L}_1$ is invertible.  We recall the class $\mathbb{S}_{\mathcal{S}_1,\mathcal{S}_2}$ of symplectic pairs and the transformation $T_{\mathcal{S}_1,\mathcal{S}_2}$ defined in \eqref{eq1.7a} and \eqref{eq1.7b}, respectively.

\begin{Theorem}\label{thm3.1}
Let $\mathcal{S}_1$, $\mathcal{S}_2\in Sp(n)$, $\mathcal{H}\in \mathbb{C}^{2n\times 2n}$ be Hamiltonian and $X_1=[X^1_{ij}]_{1\le i,j\le 2}\in{\mathbb H}(2n)$. Suppose $X(t)=[X_{ij}(t)]_{1\le i,j\le 2}$, for $t\in (t_0,t_1)$ and $t_0<1<t_1$, is the solution of the initial value problem (IVP):
\begin{align}\label{eq3.1}
\begin{array}{l}
\dot{X}(t)=\mathcal{M}(t)\mathcal{H}\mathcal{J}\mathcal{M}(t)^H,\\
X(1)=X_1,
\end{array}
\end{align}
where $(\mathcal{M}(t),\mathcal{L}(t))=T_{\mathcal{S}_1,\mathcal{S}_2}(X(t))$. If the initial pair $(\mathcal{M}_1, \mathcal{L}_1)\equiv(\mathcal{M}(1), \mathcal{L}(1))$ satisfies
\begin{align}\label{eq3.2}
\mathcal{M}_1= \mathcal{L}_1e^{\mathcal{H}_1}
\end{align}
for some Hamiltonian $\mathcal{H}_1\in \mathbb{C}^{2n\times 2n}$, then
\begin{align}\label{eq3.3}
\mathcal{M}(t)=\mathcal{L}(t)e^{\mathcal{H}_1}e^{\mathcal{H}(t-1)}
\end{align}
for all $t\in (t_0,t_1)$.
\end{Theorem}

\begin{proof}
Note that $e^{\mathcal{H}_1}$ is invertible. From \eqref{eq3.2} and Lemma \ref{lem2.5}, we see that both $\mathcal{M}_1$ and $ \mathcal{L}_1$ are invertible. On the other hand, the solution $X(t)$ of IVP \eqref{eq3.1} is continuous. Therefore, there exists an interval $(\tilde{t}_0,\tilde{t}_1)\subseteq(t_0,t_1)$ such that $1\in(\tilde{t}_0,\tilde{t}_1)$ and that both $\mathcal{M}(t)$ and $\mathcal{L}(t)$ are invertible for $t\in (\tilde{t}_0,\tilde{t}_1)$. We first show that assertion \eqref{eq3.3} holds for $t\in (\tilde{t}_0,\tilde{t}_1)$.
By the fact that
\begin{align*}
\mathcal{M}(t)=\left[%
\begin{array}{cc}
X_{12}(t)&0 \\
X_{22}(t)&I  \\
\end{array}%
\right]\mathcal{S}_2,\ \
\mathcal{L}(t)=\left[%
\begin{array}{cc}
I&X_{11}(t) \\
0&X_{21}(t) \\
\end{array}%
\right]\mathcal{S}_1,
\end{align*}
we have
\begin{align}
\dot{X}&=\left[%
\begin{array}{cccc}
  \dot{X}_{12} &0&0& \dot{X}_{11} \\
 \dot{X}_{22} &0&0& \dot{X}_{21} \\
\end{array}%
\right]\left[%
\begin{array}{rr}
  0 & I_n \\
 -X_{12}^H & -X_{22}^H \\
 -X_{11}^H & -X_{21}^H \\
 I_n & 0 \\
\end{array}%
\right]=[\dot{\mathcal{M}}\mathcal{S}_2^{-1},\dot{\mathcal{L}}\mathcal{S}_1^{-1}]\left[%
\begin{array}{c}
 \mathcal{J}(\mathcal{M}\mathcal{S}_2^{-1})^H \\
 -\mathcal{J}(\mathcal{L}\mathcal{S}_1^{-1})^H  \\
\end{array}%
\right]\notag\\
&=[\dot{\mathcal{M}},\dot{\mathcal{L}}]\left[%
\begin{array}{c}
 \mathcal{J}\mathcal{M}^H \\
 -\mathcal{J}\mathcal{L}^H  \\
\end{array}%
\right].\label{eq3.4}
\end{align}
Plugging \eqref{eq3.4} into the first equation of \eqref{eq3.1} and multiplying $\mathcal{M}^{-H}\mathcal{J}^H$ from the right to the resulting equation, we have
 \begin{align}\label{eq3.5}
 [\dot{\mathcal{M}},\dot{\mathcal{L}}]\left[%
\begin{array}{c}
 I \\
 \mathcal{J}\mathcal{L}^H \mathcal{M}^{-H}\mathcal{J} \\
\end{array}%
\right]=\mathcal{M}\mathcal{H}, \ \ t\in (\tilde{t}_0,\tilde{t}_1).
 \end{align}
Since $(\mathcal{M},\mathcal{L})$ forms a symplectic pair, and both $\mathcal{M}$ and $\mathcal{L}$ are invertible, the equality $\mathcal{M}\mathcal{J}\mathcal{M}^{H}=\mathcal{L}\mathcal{J}\mathcal{L}^{H}$ implies that $\mathcal{L}^{-1}\mathcal{M}=-\mathcal{J}\mathcal{L}^{H}\mathcal{M}^{-H}\mathcal{J}$. Thus, \eqref{eq3.5} becomes
\begin{align}\label{eq3.6}
\dot{\mathcal{M}}-\dot{\mathcal{L}}(\mathcal{L}^{-1}\mathcal{M})=\mathcal{M}\mathcal{H}.
\end{align}
Multiplying $\mathcal{L}^{-1}$ from the left of \eqref{eq3.6}, we thus obtain $$\mathcal{L}^{-1}\dot{\mathcal{M}}-(\mathcal{L}^{-1}\dot{\mathcal{L}}\mathcal{L}^{-1})\mathcal{M}=\mathcal{L}^{-1}\mathcal{M}\mathcal{H}.$$
This coincides with
\begin{align}\label{eq3.7}
\frac{d}{dt}(\mathcal{L}^{-1}\mathcal{M})=(\mathcal{L}^{-1}\mathcal{M})\mathcal{H}.
\end{align}
Using \eqref{eq3.7} together with the initial condition in  \eqref{eq3.1} and \eqref{eq3.2}, it follows that $\mathcal{L}(t)^{-1}\mathcal{M}(t)=e^{\mathcal{H}_1}e^{\mathcal{H}(t-1)}$ for $t\in (\tilde{t}_0, \tilde{t}_1)$. Hence, assertion \eqref{eq3.3} holds.

Now we claim that $\tilde{t}_0=t_0$ and $\tilde{t}_1=t_1$. We only prove the case $\tilde{t}_1=t_1$. Suppose that $\tilde{t}_1<t_1$. This implies that $\mathcal{M}(\tilde{t}_1)$ and $\mathcal{L}(\tilde{t}_1)$ are singular. Using \eqref{eq3.3} and taking the limit $t\rightarrow \tilde{t}_1^{-}$, we have
$\mathcal{M}(\tilde{t}_1)=\mathcal{L}(\tilde{t}_1)e^{\mathcal{H}_1}e^{\mathcal{H}(\tilde{t}_1-1)}$.
Since $e^{\mathcal{H}_1}e^{\mathcal{H}(\tilde{t}_1-1)}$ is invertible, $\mathcal{M}(\tilde{t}_1)$ and $\mathcal{L}(\tilde{t}_1)$ are invertible by Lemma~\ref{lem2.5}. This is a contradiction. Hence, $\tilde{t}_0=t_0$ and $\tilde{t}_1=t_1$.
\end{proof}

\begin{Remark}\label{rem3.1}
$(i)$ In Theorem~\ref{thm3.1}, since $X_1$ and $\mathcal{H}\mathcal{J}$ are Hermitian, it is easily seen that the solution, $X(t)=[X_{ij}(t)]_{1\leqslant i,j\leqslant 2}$ for $t\in (t_0,t_1)$, of IVP \eqref{eq3.1} is also Hermitian. From the definition that $(\mathcal{M}(t),\mathcal{L}(t))=T_{\mathcal{S}_1,\mathcal{S}_2}(X(t))$, we have that the curve $\{(\mathcal{M}(t),\mathcal{L}(t))|t\in (t_0,t_1)\}\subset\mathbb{S}_{\mathcal{S}_1,\mathcal{S}_2}$.\\
$(ii)$  Suppose that $(\mathcal{M}_1, \mathcal{L}_1)$ is a {\it real} symplectic pair.  If the Hamiltonian matrix $\mathcal{H}$ in \eqref{eq3.1} is also real,  then the curve $\{(\mathcal{M}(t),\mathcal{L}(t))|t\in (t_0,t_1)\}\subset\mathbb{S}_{\mathcal{S}_1,\mathcal{S}_2}$ is real.
\end{Remark}

In Theorem~\ref{thm3.1}, the assumption \eqref{eq3.2} implies that both $\mathcal{M}_1$ and $\mathcal{L}_1$ are invertible. It turns out that ${\rm ind}_{\infty} (\mathcal{M}_1,\mathcal{L}_1)=0$. We now show the invariance property of the flow \eqref{eq3.1} with the general assumption ${\rm ind}_{\infty} (\mathcal{M}_1,\mathcal{L}_1)\le 1$.

\begin{Theorem}\label{thm3.2}
Let $\mathcal{S}_1$, $\mathcal{S}_2\in Sp(n)$ and $X_1\in{\mathbb H}(2n)$ be given such that the symplectic pair $(\mathcal{M}_1, \mathcal{L}_1)=T_{\mathcal{S}_1,\mathcal{S}_2}(X_1)$ is regular with ${\rm ind}_{\infty} (\mathcal{M}_1,\mathcal{L}_1)\leqslant 1$. Let the idempotent matrices $\Pi_0=\Pi_0(\mathcal{M}_1,\mathcal{L}_1)$, $\Pi_\infty=\Pi_\infty(\mathcal{M}_1,\mathcal{L}_1)$ and the Hamiltonian matrix $\mathcal{H}=\mathcal{H}(\mathcal{M}_1,\mathcal{L}_1)$ be defined in Definition~\ref{def2.2} such that (from Lemma~\ref{lem2.6})
\begin{align}\label{eq3.8}
\mathcal{M}_1\Pi_0=\mathcal{L}_1\Pi_\infty e^{\mathcal{H}}.
\end{align}
If $X(t)=[X_{ij}(t)]_{1\leqslant i,j\leqslant 2}$, for $t\in (t_0,t_1)$, $t_0<1<t_1$, is the solution of the IVP
\begin{align}\label{eq3.9}
\begin{array}{l}
\dot{X}(t)=\mathcal{M}(t)\mathcal{H}\mathcal{J}\mathcal{M}(t)^H,\\
X(1)=X_1,
\end{array}
\end{align}
where $(\mathcal{M}(t),\mathcal{L}(t))=T_{\mathcal{S}_1,\mathcal{S}_2}(X(t))$, then
\begin{align}\label{eq3.10}
\mathcal{M}(t)\Pi_0=\mathcal{L}(t)\Pi_\infty e^{\mathcal{H}t}
\end{align}
for all $t\in (t_0,t_1)$.
\end{Theorem}

\begin{Remark}\label{rem3.2.5}
Note that (i) Eq. \eqref{eq3.8} holds true due to Lemma~\ref{lem2.6}; (ii) if the pair $(\mathcal{M}_1, \mathcal{L}_1)$ is real symplectic  and its Jordan blocks of negative eigenvalues satisfy the specified conditions mentioned in Remark~\ref{rem2.2}, then there exists a real Hamiltonian matrix $\mathcal{H}$ such that \eqref{eq3.8} holds;  (iii) if  $\mathcal{M}_1$ and $\mathcal{L}_1$ in \eqref{eq3.8} are invertible, i.e., $\Pi_o=\Pi_\infty=I$, then the result of  Theorem~\ref{thm3.2} is consistent with Theorem~\ref{thm3.1} in which $\mathcal{H}_1$ is replaced by $\mathcal{H}$; and (iv) from definitions of $\mathcal{H}$, $\Pi_0$ and $\Pi_\infty$, Eq. \eqref{eq3.10} can be rewritten as
\begin{align*}
\mathcal{M}(t)U_0=0,\ \ \mathcal{L}(t)U_\infty=0\text{ and }\mathcal{M}(t)U_1=\mathcal{L}(t)U_1e^{\widehat{\mathcal{H}}t}.
\end{align*}
This  shows that the flow $(\mathcal{M}(t), \mathcal{L}(t))=T_{\mathcal{S}_1,\mathcal{S}_2}(X(t))$ satisfies \textbf{Eigenvector-Preserving Property}, where $X(t)$ is the solution of IVP \eqref{eq3.9}. Actually, this flow $(\mathcal{M}(t), \mathcal{L}(t))$ is the structure-preserving flow with the initial $(\mathcal{M}_1,\mathcal{L}_1)$.
\end{Remark}

\begin{proof}[Proof of Theorem~\ref{thm3.2}]
Applying Corollary~\ref{cor2.8} with $\Phi^{\varepsilon}=\varepsilon I$, we see that $(\mathcal{M}_1,\mathcal{L}_1)$ is left equivalent to the symplectic pair
\begin{align*}
(\mathcal{M}_1^{\varepsilon},\mathcal{L}_1^{\varepsilon})\equiv(\widetilde{\mathcal{M}}^{\varepsilon},\widetilde{\mathcal{L}}^{\varepsilon})=\left(\left[%
\begin{array}{cc}
  X_{12}^{1\varepsilon} & 0 \\
 X_{22}^{1\varepsilon} & I \\
\end{array}%
\right]\mathcal{S}_2,\left[%
\begin{array}{cc}
  I & X_{11}^{1\varepsilon} \\
 0 & X_{21}^{1\varepsilon} \\
\end{array}%
\right]\mathcal{S}_1\right)\in \mathbb{S}_{\mathcal{S}_1,\mathcal{S}_2}
\end{align*}
for each $0\leqslant\varepsilon\ll 1$. In addition, $\mathcal{M}_1^{\varepsilon}$ and $\mathcal{L}_1^{\varepsilon}$ are invertible for $\varepsilon>0$ and
\begin{align*}
(\mathcal{M}_1^{\varepsilon},\mathcal{L}_1^{\varepsilon})\rightarrow (\mathcal{M}_1,\mathcal{L}_1) \text{ as }\varepsilon\rightarrow 0.
\end{align*}
Let $X^{\varepsilon}(t)=\left[%
\begin{array}{cc}
  X^{\varepsilon}_{11}(t) & X^{\varepsilon}_{12}(t) \\
 X^{\varepsilon}_{21}(t) & X^{\varepsilon}_{22}(t) \\
\end{array}%
\right]$ be the solution of  the IVP
\begin{align*}
\begin{array}{l}
\dot{X}^{\varepsilon}(t)=\mathcal{M}^{\varepsilon}(t)\mathcal{H}\mathcal{J}\mathcal{M}^{\varepsilon}(t)^H,\\
X^{\varepsilon}(1)=X^{\varepsilon}_1,
\end{array}
\end{align*}
where $X^{\varepsilon}_1=\left[%
\begin{array}{cc}
  X_{11}^{1\varepsilon} & X_{12}^{1\varepsilon} \\
 X_{21}^{1\varepsilon} & X_{22}^{1\varepsilon} \\
\end{array}%
\right]$ and $(\mathcal{M}^\varepsilon(t),\mathcal{L}^\varepsilon(t))=T_{\mathcal{S}_1,\mathcal{S}_2}(X^\varepsilon(t))$. By the continuous dependence of
 the solution on the initial condition of the IVP (see e.g. Section~8.4 in \cite{Hirsch_Smale:1974}), we have
\begin{align*}
(\mathcal{M}^{\varepsilon}(t),\mathcal{L}^{\varepsilon}(t))\rightarrow (\mathcal{M}(t),\mathcal{L}(t)) \text{ as }\varepsilon\rightarrow 0.
\end{align*}
On the other hand, it follows from Theorem \ref{thm3.1} that $\mathcal{M}^{\varepsilon}(t)=\mathcal{L}^{\varepsilon}(t)({\mathcal{L}^{\varepsilon}_1}^{-1}\mathcal{M}^{\varepsilon}_1)e^{\mathcal{H}(t-1)}$.
Consequently,
\begin{align}\label{eq3.11}
\mathcal{M}^{\varepsilon}(t)e^{-\mathcal{H}t}e^{\mathcal{H}}=\mathcal{L}^{\varepsilon}(t)({\mathcal{L}^{\varepsilon}_1}^{-1}\mathcal{M}^{\varepsilon}_1).
\end{align}
Let $\mathbf{U}=\mathbf{U}(\mathcal{M}_1,\mathcal{L}_1)=[U_1|U_0,U_{\infty}]$ satisfy \eqref{eq2.5} and \eqref{eq2.6} in which $(\mathcal{M},\mathcal{L})$ is replaced by $(\mathcal{M}_1,\mathcal{L}_1)$. From \eqref{eq2.16}, we have
\begin{align}\label{eq3.12}
\mathcal{M}_1^{\varepsilon}[U_1,U_0,U_{\infty}](I_{2\hat{n}}\oplus I_{\ell}\oplus \varepsilon I_{\ell})=\mathcal{L}_1^{\varepsilon}[U_1,U_0,U_{\infty}](\widehat{\mathcal{S}} \oplus \varepsilon  I_{\ell}\oplus I_{\ell}).
\end{align}
From the definition of $\mathcal{H}$ in \eqref{eq2.10}, we have
\begin{align}\label{eq3.13}
e^{\mathcal{H}}=\mathbf{U}(\widehat{\mathcal{S}} \oplus   I_{\ell}\oplus I_{\ell})\mathbf{U}^{-1}.
\end{align}
Plugging  \eqref{eq3.12} and \eqref{eq3.13}  into  \eqref{eq3.11}, we have
\begin{align*}
\mathcal{M}^{\varepsilon}(t)e^{-\mathcal{H}t}\mathbf{U}(I_{2\hat{n}}\oplus I_{\ell}\oplus \varepsilon I_{\ell})\mathbf{U}^{-1}=\mathcal{L}^{\varepsilon}(t)\mathbf{U}(I_{2\hat{n}}\oplus \varepsilon I_{\ell}\oplus  I_{\ell})\mathbf{U}^{-1}.
\end{align*}
When $\varepsilon$ approaches $0$, it follows from \eqref{eq2.11} that
\begin{align*}
\mathcal{M}(t)e^{-\mathcal{H}t}\Pi_{0}=\mathcal{L}(t)\Pi_{\infty}.
\end{align*}
Since $e^{-\mathcal{H}t}$ commutes with $\Pi_0$, we obtain assertion \eqref{eq3.10}.
\end{proof}

\begin{Corollary}\label{cor3.3}
Theorem \ref{thm3.2} holds true if Eq.  \eqref{eq3.9}  is replaced by
\begin{align*}
\begin{array}{l}
\dot{X}(t)=\mathcal{L}(t)\mathcal{H}\mathcal{J}\mathcal{L}(t)^H,\\
X(1)=X_1.
\end{array}
\end{align*}
\end{Corollary}

\begin{proof}
It suffices to show that $\mathcal{M}(t)\mathcal{H}\mathcal{J}\mathcal{M}(t)^H=\mathcal{L}(t)\mathcal{H}\mathcal{J}\mathcal{L}(t)^H$.
Using definitions of  $\Pi_0=\Pi_0(\mathcal{M}_1,\mathcal{L}_1)$ and $\Pi_{\infty}=\Pi_\infty(\mathcal{M}_1,\mathcal{L}_1)$ in  \eqref{eq2.11}, we have $\mathcal{M}(t)=\mathcal{M}(t)\Pi_0$, $\mathcal{L}(t)=\mathcal{L}(t)\Pi_\infty$. It follows from \eqref{eq3.10} and the symplecticity of $e^{\mathcal{H}t}$ that
\begin{align*}
\mathcal{M}(t)\mathcal{H}\mathcal{J}\mathcal{M}(t)^H&=\mathcal{M}(t)\Pi_0\mathcal{H}\mathcal{J}\Pi_0^H\mathcal{M}(t)^H\\&=\mathcal{L}(t)\Pi_\infty e^{\mathcal{H}t}\mathcal{H}\mathcal{J}(e^{\mathcal{H}t})^H\Pi_\infty^H\mathcal{L}(t)^H\\
&=\mathcal{L}(t)\Pi_\infty \mathcal{H}e^{\mathcal{H}t}\mathcal{J}(e^{\mathcal{H}t})^H\Pi_\infty^H\mathcal{L}(t)^H\\
&=\mathcal{L}(t)\Pi_\infty \mathcal{H}\mathcal{J}\Pi_\infty^H\mathcal{L}(t)^H=\mathcal{L}(t)\mathcal{H}\mathcal{J}\mathcal{L}(t)^H.
\end{align*}
\end{proof}

Now, we study the invariance property \eqref{eq3.10}. To this end, for given $\mathcal{S}_1$, $\mathcal{S}_2\in Sp(n)$, we let $(\mathcal{M}_1,\mathcal{L}_1)\in \mathbb{S}_{\mathcal{S}_1,\mathcal{S}_2}$ with ${\rm ind}_{\infty} (\mathcal{M}_1,\mathcal{L}_1)\leqslant 1$. Let the idempotent matrices $\Pi_0=\Pi_0(\mathcal{M}_1,\mathcal{L}_1)$, $\Pi_\infty=\Pi_\infty(\mathcal{M}_1,\mathcal{L}_1)$ and $\mathcal{H}=\mathcal{H}(\mathcal{M}_1,\mathcal{L}_1)$ be defined as in Definition~\ref{def2.2}. Consider the linear system
\begin{align}\label{eq3.14}
\begin{cases}
\mathcal{M}(t)\Pi_0=\mathcal{L}(t)\Pi_\infty e^{\mathcal{H}t},&\\
(\mathcal{M}(t),\mathcal{L}(t))\in \mathbb{S}_{\mathcal{S}_1,\mathcal{S}_2},&
\end{cases}
\end{align}
where $t\in \mathbb{R}$ and $(\mathcal{M}(t),\mathcal{L}(t))$ are unknowns. The first and second equations of \eqref{eq3.14} mean that the matrix pair $(\mathcal{M}(t),\mathcal{L}(t))$ has the \textbf{Eigenvector-Preserving Property} and the \textbf{Structure-Preserving Property}, respectively. It is clear from Theorem~\ref{thm3.2} that the solution $(\mathcal{M}(t),\mathcal{L}(t))$ of IVP \eqref{eq3.9} is invariant in the manifold described by \eqref{eq3.14}. In the following, we shall show that the consistency of Eq.  \eqref{eq3.14} implies the uniqueness of the solution $(\mathcal{M}(t),\mathcal{L}(t))$, for which the pair $(\mathcal{M}(t),\mathcal{L}(t))$ is  regular.
\begin{Lemma}\label{lem3.4}
Let $(A, B)$ be a regular pair with $A, B\in \mathbb{C}^{n\times n}$. Suppose that
\begin{align}\label{eq3.15}
[C,D]\left[%
\begin{array}{c}
 A \\
 B\\
\end{array}%
\right]=0,
\end{align}
and $[C,D]\in \mathbb{C}^{n\times 2n}$ is of full row rank. Then $(D,C)$ is  regular.
\end{Lemma}
\begin{proof}
Since $(A,B)$ is  regular, there exists $\lambda_0\in \mathbb{C}$ such that $A-\lambda_0 B$ is invertible and $[A^{\top},B^{\top}]^{\top}$ is of full column rank. From \eqref{eq3.15}, we have
\begin{align}\label{eq3.16}
0&=[C,D]\left[%
\begin{array}{c}
 A \\
 B\\
\end{array}%
\right](A-\lambda_0 B)^{-1}=[C,D]\left[%
\begin{array}{cc}
 I &\lambda_0 I\\
 o&I\\
\end{array}%
\right]\left[%
\begin{array}{cc}
 I &-\lambda_0 I\\
 o&I\\
\end{array}%
\right]\left[%
\begin{array}{c}
 A \\
 B\\
\end{array}%
\right](A-\lambda_0 B)^{-1}\nonumber\\&=[C,D+\lambda_0 C]\left[%
\begin{array}{c}
 I \\
 B(A-\lambda_0 B)^{-1}\\
\end{array}%
\right].
\end{align}
It is easily seen that  rank$[C,D+\lambda_0 C]={\rm rank}[C,D]=n$. It follows from \eqref{eq3.16} that there is a nonsingular  matrix $W$ such that
\begin{align*}
[C,D+\lambda_0 C]=W[-B(A-\lambda_0 B)^{-1},I].
\end{align*}
Then $D+\lambda_0 C$ is invertible and hence $(D,C)$ is  regular.
\end{proof}

Let $\mathbf{U}\equiv[U_1,U_0,U_\infty]=\mathbf{U}(\mathcal{M}_1,\mathcal{L}_1)$ be defined in Definition~\ref{def2.2}.  From definitions of $\Pi_0$, $\Pi_\infty$ and $\mathcal{H}$ in \eqref{eq2.11} and \eqref{eq2.10}, respectively, the linear system \eqref{eq3.14} can be rewritten as
\begin{align}\label{eq3.17}
\left[%
\begin{array}{cc}
 X_{12} (t) & 0 \\
  X_{22}(t) &   I\\
\end{array}%
\right]\mathcal{S}_2\mathbf{U}( I_{2\hat{n}}\oplus  I_{\ell}\oplus 0)=\left[%
\begin{array}{cc}
 I & X_{11}(t)   \\
  0&   X_{21}(t)  \\
\end{array}%
\right]\mathcal{S}_1\mathbf{U}( e^{\widehat{\mathcal{H}}t} \oplus 0\oplus  I_{\ell}).
\end{align}
The following lemma can be obtained by direct calculations.

\begin{Lemma}\label{lem3.5}
Let
\begin{subequations}\label{eq3.18}
\begin{align}
&E_{11}=(I_{\ell}\oplus 0),\ \ E_{22}=(0\oplus I_{\ell}),\label{eq3.18a}\\
&\mathbf{V}_1\equiv\left[%
\begin{array}{c}
\mathbf{V}_1^1  \\
 \mathbf{V}_2^1 \\
\end{array}%
\right]=\mathcal{S}_1\mathbf{U},\ \ \  \mathbf{V}_2\equiv\left[%
\begin{array}{c}
\mathbf{V}_1^2  \\
 \mathbf{V}_2^2 \\
\end{array}%
\right]=\mathcal{S}_2\mathbf{U},\label{eq3.18b}
\end{align}
\end{subequations}
where $\mathbf{V}_i^j\in \mathbb{C}^{n\times 2n}$ for each $1\leqslant i,j \leqslant 2$. Then the linear system \eqref{eq3.14} is equivalent to the alternative form:
\begin{align}\label{eq3.19}
\left[%
\begin{array}{cc}
X_{11}(t)&X_{12}(t)  \\
X_{21}(t)&X_{22}(t) \\
\end{array}%
\right] \left[%
\begin{array}{r}
-\mathbf{V}_2^1
 (e^{\widehat{\mathcal{H}}t} \oplus
E_{22} )\\
 \mathbf{V}_1^2
(I_{2\hat{n}}\oplus E_{11} )\\
\end{array}%
\right]= \left[%
\begin{array}{r}
\mathbf{V}_1^1
 (e^{\widehat{\mathcal{H}}t} \oplus
E_{22} )\\
 -\mathbf{V}_2^2
(I_{2\hat{n}}\oplus E_{11} )\\
\end{array}%
\right].
\end{align}
\end{Lemma}

\begin{Theorem}\label{thm3.6}
Let $(\mathcal{M}_1,\mathcal{L}_1)\in \mathbb{S}_{\mathcal{S}_1,\mathcal{S}_2}$ be a regular symplectic pair with ${\rm ind}_{\infty} (\mathcal{M}_1,\mathcal{L}_1)\leqslant 1$ and $\mathbf{U}\equiv[U_1,U_0,U_\infty]=\mathbf{U}(\mathcal{M}_1,\mathcal{L}_1)$. Suppose $(\mathcal{M}(t),\mathcal{L}(t))$ is a solution of \eqref{eq3.14} at some $t\in \mathbb{R}$. Then
\begin{itemize}
\item[(i)] $(\mathcal{M}(t), \mathcal{L}(t))$ is regular;
\item[(ii)]   $(\mathcal{M}(t),\mathcal{L}(t))$ is the unique solution of \eqref{eq3.14};
\item[(iii)]  It holds that
\begin{align}\label{eq3.20}
\mathcal{M}(t)U_0=0,\ \  \mathcal{L}(t)U_\infty=0,\ \  \mathcal{M}(t)U_1=\mathcal{L}(t)U_1e^{\widehat{\mathcal{H}}t}.
\end{align}
Conversely, if $(\mathcal{M}(t),\mathcal{L}(t))\in \mathbb{S}_{\mathcal{S}_1,\mathcal{S}_2}$ satisfies  \eqref{eq3.20}, then $(\mathcal{M}(t),\mathcal{L}(t))$ is a solution of \eqref{eq3.14}.
\end{itemize}
\end{Theorem}

\begin{proof}
First, we write
\begin{align*}
(\mathcal{M}(t),\mathcal{L}(t))=\left(\left[%
\begin{array}{cc}
 X_{12}(t)  & 0 \\
  X_{22}(t)&   I\\
\end{array}%
\right]\mathcal{S}_2, \left[%
\begin{array}{cc}
 I & X_{11}(t)  \\
  0&   X_{21}(t) \\
\end{array}%
\right]\mathcal{S}_1\right)\in \mathbb{S}_{\mathcal{S}_1,\mathcal{S}_2}.
\end{align*}
Then $X_{ij}(t)$ for $1\leqslant i,j\leqslant 2$ satisfy \eqref{eq3.17}. Consequently,
\begin{align*}
[-\mathcal{L}(t),\mathcal{M}(t)]\left[%
\begin{array}{c}
\mathbf{U}(e^{\widehat{\mathcal{H}}t} \oplus E_{22} ) \\
 \mathbf{U}(I_{2\hat{n}}\oplus E_{11} ) \\
\end{array}%
\right]=0.
\end{align*}
Since the matrix $[-\mathcal{L}(t),\mathcal{M}(t)]\in \mathbb{C}^{2n\times 4n}$ is of full row rank and $ \left((e^{\widehat{\mathcal{H}}t} \oplus
E_{22} ),(I_{2\hat{n}}\oplus E_{11} )\right)$ is regular, it follows from Lemma \ref{lem3.4}  that $(\mathcal{M}(t),\mathcal{L}(t))$ is regular. Hence, assertion $(i)$ holds.

Next, we show that the linear system \eqref{eq3.14}  has a unique solution. From Lemma \ref{lem3.5}, it suffices to show that  the matrix $\left[%
\begin{array}{r}
-\mathbf{V}_2^1
 (e^{\widehat{\mathcal{H}}t} \oplus
E_{22} )\\
 \mathbf{V}_1^2
(I_{2\hat{n}}\oplus E_{11} )\\
\end{array}%
\right]$ in \eqref{eq3.19} is invertible. Suppose that $y\in \mathbb{C}^{2n}$ satisfying $\left[%
\begin{array}{r}
-\mathbf{V}_2^1
 (e^{\widehat{\mathcal{H}}t} \oplus
E_{22} )\\
 \mathbf{V}_1^2
(I_{2\hat{n}}\oplus E_{11} )\\
\end{array}%
\right]y=0$. Let
\begin{align}\label{eq3.21}
z_1=(e^{\widehat{\mathcal{H}}t} \oplus
E_{22} )y,\ \ z_2=(I_{2\hat{n}}\oplus E_{11} )y.
\end{align}
Then we have $\mathbf{V}_2^1 z_1=0$ and $\mathbf{V}_1^2z_2=0$. Since the linear system \eqref{eq3.19} is consistent, we obtain that
$\mathbf{V}_1^1 z_1=0$ and $\mathbf{V}_2^2z_2=0$. Hence, $\mathbf{V}_1z_1=0$ and $\mathbf{V}_2z_2=0$. It follows from \eqref{eq3.18b} that  $z_1=z_2=0$. From \eqref{eq3.21}, it is easily seen that $y=0$. Thus, $\left[%
\begin{array}{r}
-\mathbf{V}_2^1
 (e^{\widehat{\mathcal{H}}t} \oplus
E_{22} )\\
 \mathbf{V}_1^2
(I_{2\hat{n}}\oplus E_{11} )\\
\end{array}%
\right]$ is invertible. This proves assertion $(ii)$.

Assertion $(iii)$ can be obtained by \eqref{eq3.17} directly.
\end{proof}

\begin{Remark}\label{rem3.2}
Given two symplectic matrices $\mathcal{S}_1$ and $\mathcal{S}_2$, the linear system \eqref{eq3.14} may have no solution in $\mathbb{S}_{\mathcal{S}_1,\mathcal{S}_2}$. We consider a simple example. Let $\mathcal{S}_1=\mathcal{S}_2=I_2$,  $\mathcal{H}=\left[%
\begin{array}{cc}
0&\pi/2\\
-\pi/2&0\\
\end{array}%
\right]$ and $t=1$. Then $e^{\mathcal{H}t}=\left[%
\begin{array}{cc}
0&1\\
-1&0\\
\end{array}%
\right]$. It is easily seen that  \eqref{eq3.14} has no solution in $\mathbb{S}_{\mathcal{S}_1,\mathcal{S}_2}$.
\end{Remark}

Let $(\mathcal{M}_1, \mathcal{L}_1)\in \mathbb{S}_{\mathcal{S}_1,\mathcal{S}_2}$ be a regular symplectic pair with ${\rm ind}_{\infty} (\mathcal{M}_1,\mathcal{L}_1)\leqslant 1$.
From Lemma \ref{lem2.6} there are a Hamiltonian $\mathcal{H}\in \mathbb{C}^{2n\times 2n}$, two idempotent matrices $\Pi_{0}$ and $\Pi_{\infty}$ such that \eqref{eq3.8} holds.
Let
\begin{align}\label{eq3.22}
\mathcal{C}_{\mathcal{M}_1,\mathcal{L}_1}=\{(\mathcal{M}(t),\mathcal{L}(t))\  |\ (\mathcal{M}(t),\mathcal{L}(t)) \text{ is a solution of \eqref{eq3.14} at }t\in \mathbb{R}\}.
\end{align}
It follows from Theorem~\ref{thm3.6}(ii) that the set $\mathcal{C}_{\mathcal{M}_1,\mathcal{L}_1}$ can be parameterized by $t$ on the set
\begin{align}\label{eq3.23}
\mathcal{T}_X=\{t\in\mathbb{R}\ |\ \text{\eqref{eq3.14} has a solution at } t\}.
\end{align}

\begin{Remark}\label{rem3.3.5}
Let  $X(t)=[X_{ij}(t)]_{1\leqslant i,j \leqslant 2}=T_{\mathcal{S}_1,\mathcal{S}_2}^{-1}(\mathcal{M}(t),\mathcal{L}(t))$ for $(\mathcal{M}(t),\mathcal{L}(t))\in \mathcal{C}_{\mathcal{M}_1,\mathcal{L}_1}$ and $t\in\mathcal{T}_X$. We obtain that  $X(t)$ is continuously differentiable for each $t\in \mathcal{T}_X$. In this case, $\left[%
\begin{array}{r}
-\mathbf{V}_2^1
 (e^{\widehat{\mathcal{H}}t} \oplus
E_{22} )\\
 \mathbf{V}_1^2
(I_{2\hat{n}}\oplus E_{11} )\\
\end{array}%
\right]$ is invertible. Consequently, $\mathcal{T}_X$ is open.
\end{Remark}
Next, we show that $X(t)=T_{\mathcal{S}_1,\mathcal{S}_2}^{-1}(\mathcal{M}(t),\mathcal{L}(t))$ for $t\in (\tilde{t}_0,\tilde{t}_1)\subseteq\mathcal{T}_X$ is the solution of IVP \eqref{eq3.9}.
\begin{Theorem}\label{thm3.7}
Suppose that $(\mathcal{M}(t),\mathcal{L}(t))\in\mathcal{C}_{\mathcal{M}_1,\mathcal{L}_1}$ for $t\in(\tilde{t}_0,\tilde{t}_1)\subseteq\mathcal{T}_X$, where $\tilde{t}_0< 1< \tilde{t}_1$. Then $X(t)=T_{\mathcal{S}_1,\mathcal{S}_2}^{-1}(\mathcal{M}(t),\mathcal{L}(t))$ for $t\in (\tilde{t}_0,\tilde{t}_1)$ is the solution of IVP \eqref{eq3.9}.
\end{Theorem}

\begin{proof}
It follows from Theorem \ref{thm3.6} $(ii)$ that the solution of  \eqref{eq3.14} for each $t\in (\tilde{t}_0,\tilde{t}_1)$ is unique.  Define  the curve
\begin{align*}
\mathcal{C}_{(\tilde{t}_0,\tilde{t}_1)}\equiv\{(\mathcal{M}(t),\mathcal{L}(t))\ |\ t\in (\tilde{t}_0,\tilde{t}_1)\}\subseteq\mathcal{C}_{\mathcal{M}_1,\mathcal{L}_1}.
\end{align*}
Let $Y(t)=T_{\mathcal{S}_1,\mathcal{S}_2}^{-1}(\mathcal{M}(t),\mathcal{L}(t))$ for $t\in (\tilde{t}_0,\tilde{t}_1)$. From Remark \ref{rem3.3.5}, $Y(t)$ is continuously differentiable. Suppose that $X(t)=[X_{ij}(t)]_{1\leqslant i,j \leqslant 2}$ for $t\in (t_0,t_1)$ is the solution of IVP \eqref{eq3.9}, where $(t_0,t_1)$ is the maximal interval. It follows from Theorem \ref{thm3.2} that $\{T_{\mathcal{S}_1,\mathcal{S}_2}(X(t)) \ | t\in (t_0,t_1)\}\subset \mathcal{C}_{\mathcal{M}_1,\mathcal{L}_1}$. If $(\tilde{t}_0,\tilde{t}_1)\subseteq(t_0,t_1)$, then the uniqueness of the solution of \eqref{eq3.14} implies that $Y(t)=X(t)$ for $t\in (\tilde{t}_0,\tilde{t}_1)$, and hence  $X(t)=T_{\mathcal{S}_1,\mathcal{S}_2}^{-1}(\mathcal{M}(t),\mathcal{L}(t))$, for $t\in (\tilde{t}_0,\tilde{t}_1)$, is the solution of IVP \eqref{eq3.9}. Now we claim that $(\tilde{t}_0,\tilde{t}_1)\subseteq(t_0,t_1)$. We prove the case $\tilde{t}_1\leqslant t_1$. On the contrary, suppose that $\tilde{t}_1> t_1$. Then $t_1\in (\tilde{t}_0,\tilde{t}_1)\subseteq \mathcal{T}_X$ and hence  $(\mathcal{M}(t_1),\mathcal{L}(t_1))\in \mathcal{C}_{\mathcal{M}_1,\mathcal{L}_1}$. By the uniqueness of solution of \eqref{eq3.14}, we have $X(t)=Y(t)$ for $t\in (t_0,t_1)$. We also note that $\dot{Y}(t)$ is continuous at $t_1\in (\tilde{t}_0,\tilde{t}_1)$. Therefore,
\begin{align*}
\dot{Y}(t_1)-\mathcal{M}(t_1)\mathcal{H}\mathcal{J}\mathcal{M}(t_1)^H=\lim_{t\rightarrow {t_1}^{-}}\dot{Y}(t)-\mathcal{M}(t)\mathcal{H}\mathcal{J}\mathcal{M}(t)^H=0.
\end{align*}
Hence, the solution $X(t)$ of IVP \eqref{eq3.9} can be extended to $t_1$.  This is a  contradiction because $(t_0,t_1)$ is the maximal interval of IVP \eqref{eq3.9}.
\end{proof}
\begin{Remark}\label{rem3.4}
Theorem  \ref{thm3.7} shows that the connected component of $\mathcal{T}_X$ cotaining $1$ coincides with the maximal interval of IVP \eqref{eq3.9}. The flow of IVP \eqref{eq3.9} can be extended to whole $\mathcal{T}_X$ by using the so-called Grassmann manifold which will be studied in Subsection \ref{sec3.3} for details.
\end{Remark}

\subsection{Structure-Preserving Flow vs. Riccati Equation}\label{sec3.2}

In this subsection, we investigate an explicit representation of IVP \eqref{eq3.9}. Since $\mathcal{S}_2$ is symplectic and  $\mathcal{H}$ is Hamiltonian,  $\mathcal{S}_2\mathcal{H}\mathcal{S}_2^{-1}$ is also Hamiltonian, say
\begin{align}\label{eq3.24}
\mathcal{S}_2\mathcal{H}\mathcal{S}_2^{-1}=\left[%
\begin{array}{cc}
A&S\\
D&-A^H\\
\end{array}%
\right],
\end{align}
where $A, S,D\in \mathbb{C}^{n\times n}$ with $S^H=S$ and $D^H=D$. Suppose that $X(t)=[X_{ij}(t)]_{1\leqslant i,j \leqslant 2}$, for $t\in (t_0,t_1)$ and $t_0<1<t_1$, is the solution of \eqref{eq3.9}. We then have
\begin{align}\label{eq3.25}
\left[%
\begin{array}{cc}
\dot{X}_{11}&\dot{X}_{12}\\
\dot{X}_{21}&\dot{X}_{22}\\
\end{array}%
\right]&=\left[%
\begin{array}{cc}
X_{12}&0\\
X_{22}&I\\
\end{array}%
\right]\mathcal{S}_2\mathcal{H}\mathcal{S}_2^{-1}\mathcal{J}\left[%
\begin{array}{cc}
X^H_{12}&X^H_{22}\\
0&I\\
\end{array}%
\right]\\&=\left[%
\begin{array}{cc}
-X_{12}SX_{12}^H&-X_{12}SX^H_{22}+X_{12}A\\
-X_{22}SX^H_{12}+A^HX_{12}^H&-X_{22}SX_{22}^H+X_{22}A+A^HX_{22}^H+D\\
\end{array}%
\right],\nonumber\\
X_{ij}(1)&=X_{ij}^1\text{ for }1\leqslant i,j \leqslant 2.\nonumber
\end{align}
That is, $X_{ij}(t)$ for $1\leqslant i,j \leqslant 2$ satisfy the coupled differential equations
\begin{subequations}\label{eq3.26}
\begin{align}
\dot{X}_{11}&=-X_{12}SX_{12}^H, \label{eq3.26a}\\
\dot{X}_{12}&=-X_{12}SX^H_{22}+X_{12}A, \label{eq3.26b}\\
\dot{X}_{21}&=-X_{22}SX^H_{12}+A^HX_{12}^H,\label{eq3.26c}\\
\dot{X}_{22}&=-X_{22}SX_{22}^H+X_{22}A+A^HX_{22}^H+D,\label{eq3.26d}
\end{align}
\end{subequations}
with $X_{ij}(1)=X_{ij}^1$, where $A,\ D$ and $S$ are given in \eqref{eq3.24}. Note that $S$, $D$ and the initial matrix $X_{22}^1$  are Hermitian. From \eqref{eq3.26d},  $X_{22}(t)$ is Hermitian for $t\in(t_0,t_1)$. Therefore, by taking a time shift,  $W(t)=X_{22}(t+1)$, $t\in (t_0-1,t_1-1)$, is the solution of the {\it Riccati differential equation (RDE):}
\begin{align}\label{eq3.27}
\begin{array}{l}
\dot{W}(t)=-W(t)SW(t)+W(t)A+A^HW(t)+D,\\
W(0)=W_0,
\end{array}
\end{align}
with $W_0=X_{22}^1$.

\begin{Remark}\label{rem3.5}
Suppose that  $W(t)$, for $t\in (t_0-1,t_1-1)$ and $t_0-1<0<t_1-1$, is a solution of the Riccati differential equation \eqref{eq3.27}. Using the fact $X_{22}(t)=W(t-1)$, $t\in (t_0,t_1)$, we can get $X_{12}(t)$  for $t\in (t_0,t_1)$ by solving the linear differential equation \eqref{eq3.26b} with $X_{12}(1)=X_{12}^1$. Since $X_{21}^1=X_{12}^{1H}$, it follows from \eqref{eq3.26b} and \eqref{eq3.26c} that $X_{21}(t)=X_{12}(t)^H$, for $t\in(t_0,t_1)$. Finally, $X_{11}(t)$ for $t\in (t_0,t_1)$ can be obtained directly from \eqref{eq3.26a}. So, solving IVP \eqref{eq3.9} is equivalent to solving the Riccati differential equation \eqref{eq3.27}.
\end{Remark}

Riccati differential equations arise frequently throughout applied mathematics, science and engineering.  In particular, they play an important role in optimal controls \cite{Davison_Maki:1973,Kenney_Leipnik:1985,Kwakernaak_Sivan:1972,Lainiotis:1976,Lainiotis:1976b,Laub:1982} and in two-point boundary value problems \cite{Ascher_Mattheij_Russell:1988,Babuska_Majer:1987,Dieci_Osborne_Russell:1988,Dieci_Osborne_Russell:1988b}. Theoretical analysis as well as the monotonicity property of RDEs have been widely investigated in \cite{Abou-Kandil03,Freiling:2002,Reid:1970}. A family of unconventional numerical methods for solving matrix Riccati differential equations is developed in \cite{Li:2012} that can produce meaningful numerical results even if there are poles in the solution. An important tool in the literature mentioned above is the use of a relationship between linear differential equations and Riccati differential equations. This relation has been known at least since the work of  Radon \cite{Radon27}.

\begin{Theorem}\label{thm3.8}\cite[Radon's Lemma]{Abou-Kandil03}
Let $A,\ S,\ D\in \mathbb{C}^{n\times n}$ with $S^H=S$ and  $D^H=D$, then the following statements hold.
\begin{itemize}
\item[(i)] Let $W(t)$ be a solution of RDE \eqref{eq3.27} in the interval $(t_0-1,t_1-1)$ containing zero. If $Q(t)$ is a solution of the IVP
\begin{align}\label{eq3.28}
\dot{Q}(t)=(SW(t)-A)Q(t),\ \ Q(0)=I_n
\end{align}
and $P(t):=W(t)Q(t)$, then $Y(t)\equiv [Q(t)^{\top},P(t)^{\top}]^{\top}$ is the solution of the linear IVP
\begin{subequations}\label{eq3.29}
\begin{align}\label{eq3.29a}
\dot{Y}(t)=\widetilde{\mathcal{H}}Y(t),\ \ Y(0)=\left[%
\begin{array}{c}
I\\
W_0\\
\end{array}%
\right]=\left[%
\begin{array}{c}
I\\
X_{22}^1\\
\end{array}%
\right],
\end{align}
where
\begin{align}\label{eq3.29b}
\widetilde{\mathcal{H}}=\left[%
\begin{array}{cc}
-A&S\\
D&A^H\\
\end{array}%
\right].
\end{align}
\end{subequations}

\item[(ii)] Let $Y(t)\equiv[Q(t)^{\top},P(t)^{\top}]^{\top}$ be the solution of \eqref{eq3.29}. If $Q(t)$ is invertible for $t\in(t_0-1,t_1-1)\subset \mathbb{R}$, then $W(t)\equiv P(t)Q(t)^{-1}$ is a solution of RDE \eqref{eq3.27}.
\end{itemize}
\end{Theorem}

\begin{Remark}\label{rem3.6}
Using the definition of $\widetilde{\mathcal{H}}$ in  \eqref{eq3.29b}, it  follows from \eqref{eq3.24} that
\begin{align}\label{eq3.30}
\widetilde{\mathcal{H}}=-\left[%
\begin{array}{cc}
I&0\\
0&-I\\
\end{array}%
\right]\mathcal{S}_2\mathcal{H}\mathcal{S}_2^{-1}\left[%
\begin{array}{cc}
I&0\\
0&-I\\
\end{array}%
\right].
\end{align}
Therefore, if $\mathcal{S}_2\mathcal{H}\mathcal{S}_2^{-1}\left[%
\begin{array}{c}
U_1\\
U_2\\
\end{array}%
\right]=\left[%
\begin{array}{c}
U_1\\
U_2\\
\end{array}%
\right]\Lambda$, then $\widetilde{\mathcal{H}}\left[%
\begin{array}{r}
U_1\\
-U_2\\
\end{array}%
\right]=\left[%
\begin{array}{r}
U_1\\
-U_2\\
\end{array}%
\right](-\Lambda)$.
\end{Remark}

\begin{Corollary}\label{cor3.9}
Let $Y(t)\equiv[Q(t)^{\top},P(t)^{\top}]^{\top}$ and $W(t)$ be the solution of \eqref{eq3.29} and \eqref{eq3.27}, respectively,  with $W_0=X_{22}^1$. If $Q(t)$ is invertible, for $t\in(t_0-1,t_1-1)$ and $t_0-1<0<t_1-1$, then the solutions of \eqref{eq3.26d} and \eqref{eq3.26b} are
\begin{align*}
X_{22}(t)&=W(t-1)=P(t-1)Q(t-1)^{-1},\\
X_{12}(t)&=X_{12}^1Q(t-1)^{-1},
\end{align*}
respectively, for $t\in(t_0,t_1)$. In addition, $X_{21}(t)=X_{12}(t)^H=Q(t-1)^{-H}X_{21}^1$.
\end{Corollary}

\begin{proof}
From Radon's Lemma, we obtain that $W(t)=P(t)Q(t)^{-1}$ for $t\in (t_0-1,t_1-1)$  is the solution of RDE \eqref{eq3.27}. Hence, we have $X_{22}(t)=W(t-1)=P(t-1)Q(t-1)^{-1}$ for $t\in (t_0,t_1)$ by comparing \eqref{eq3.26d} and \eqref{eq3.27}. Note that $Q(t)$ satisfies \eqref{eq3.28} and $\frac{d}{dt}Q(t)^{-1}=-Q(t)^{-1}\dot{Q}(t)Q(t)^{-1}$.
Multiplying $Q(t)^{-1}$ from both sides of Eq. \eqref{eq3.28}, it is easily seen that $Q(t)^{-1}$ is the fundamental solution of the equation
\begin{align}\label{eq3.31}
\dot{R}(t)=R(t)(A-SW(t)),\ \ R(0)=I_n.
\end{align}
Comparing \eqref{eq3.31} and \eqref{eq3.26b}, we thus have $X_{12}(t)=X_{12}^1Q(t-1)^{-1}$. Assertion for $X_{21}(t)$ follows from the fact that $X(t)$ is Hermitian.
\end{proof}

Let $\mathcal{S}_1\mathcal{H}\mathcal{S}_1^{-1}=\left[%
\begin{array}{cc}
A_{\star}&S_{\star}\\
D_{\star}&-A_{\star}^H\\
\end{array}%
\right]$.
From Corollary \ref{cor3.3} and a similar calculation as \eqref{eq3.25}, we obtain that $X_{11}(t)$ and $X_{21}(t)$ satisfy
\begin{align}\label{eq3.32}
\begin{array}{l}
\dot{X}_{11}=X_{11}D_{\star}X_{11}^H+X_{11}A_{\star}^H+A_{\star}X_{11}^H-S_{\star},\\
\dot{X}_{21}=X_{21}D_{\star}X_{11}^H+X_{21}A_{\star}^H,
\end{array}
\end{align}
with $X_{11}(1)=X_{11}^1$ and $X_{21}(1)=X_{21}^1$.
Similarly, by using the fact that the solution $X_{11}(t)$ is Hermitian and
taking the time shift, $t\rightarrow t+1$, we see that $W_{\star}(t)=X_{11}(t+1)$ is the solution of the RDE
\begin{align}\label{eq3.33}
\begin{array}{l}
\dot{W}_{\star}(t)=W_{\star}(t)D_{\star}W_{\star}(t)+W_{\star}(t)A_{\star}^H+A_{\star}W_{\star}(t)-S_{\star},\\
W_{\star}(0)=X_{11}^1.
\end{array}
\end{align}
Let $Y_{\star}(t)=\left[%
\begin{array}{c}
Q_{\star}(t)\\
P_{\star}(t)\\
\end{array}%
\right]$ be  the solution of the  linear differential equation
\begin{subequations}\label{eq3.34}
\begin{align}\label{eq3.34a}
\dot{Y}_{\star}(t)=\widetilde{\mathcal{H}}_{\star}Y_{\star}(t),\ \
Y_{\star}(0)=\left[%
\begin{array}{c}
I\\
X_{11}^1\\
\end{array}%
\right],
\end{align}
where
\begin{align}\label{eq3.34b}
\widetilde{\mathcal{H}}_{\star}\equiv\left[%
\begin{array}{cc}
-A_{\star}^H&-D_{\star}\\
-S_{\star}&{A_{\star}}\\
\end{array}%
\right]=\mathcal{J}^{-1}\mathcal{S}_1\mathcal{H}\mathcal{S}_1^{-1}\mathcal{J}.
\end{align}
\end{subequations}
Suppose that $Q_{\star}(t)$ is invertible for $t\in (t_0^{\star}-1,t_1^{\star}-1)$ and $t_0^{\star}<1<t_1^{\star}$. By Radon's Lemma and Corollary \ref{cor3.9}, the solution $X_{11}(t)$, $X_{21}(t)$ of \eqref{eq3.32}  can be formulated by
\begin{align}\label{eq3.35}
\begin{array}{l}
X_{11}(t)=W_{\star}(t-1)=P_{\star}(t-1)Q_{\star}(t-1)^{-1},\\
X_{21}(t)=X_{21}^{1}Q_{\star}(t-1)^{-1},
\end{array}
\end{align}
respectively, for $t\in (t_0^{\star},t_1^{\star})$.
Comparing \eqref{eq3.30} and \eqref{eq3.34b} yields that $\widetilde{\mathcal{H}}_{\star}$ and $-\widetilde{\mathcal{H}}$ are similar.

The nonsingularity of $Q(t)$ and $Q_\star(t)$ plays an important role to determine whether $X_{22}(t)$ and $X_{11}(t)$ exist, respectively.  The following theorem claims that both $Q(t)$ and $Q_\star(t)$ are invertible simultaneously.
\begin{Theorem}\label{thm3.10}
Let $Q(t)$, $P(t)$, $Q_\star(t)$ and $P_\star(t)$ be the matrix functions given in \eqref{eq3.29} and \eqref{eq3.34}, respectively. Then we have
\begin{align}\label{eq3.36}
\{t\in \mathbb{R}| \ {\rm det}(Q(t))\ne 0\}=\{t\in \mathbb{R}| \ {\rm det}(Q_{\star}(t))\ne 0\}.
\end{align}
In addition, if $\hat{t}\in{\mathbb R}$ such that ${\rm det}(Q(\hat{t}))= 0$, then
\begin{align*}
\begin{array}{l}
\lim_{t\rightarrow \hat{t}}\|P(t)Q(t)^{-1}\|=\lim_{t\rightarrow \hat{t}}\|X_{12}^{1}Q(t)^{-1}\|=\infty,\\
\lim_{t\rightarrow \hat{t}}\|P_{\star}(t)Q_{\star}(t)^{-1}\|=\lim_{t\rightarrow \hat{t}}\|X_{21}^{1}Q_{\star}(t)^{-1}\|=\infty.
\end{array}
\end{align*}
\end{Theorem}

\begin{proof}
Let  $\Pi_0$, $\Pi_\infty$, $\mathbf{U}\equiv[U_1|U_0,U_\infty]$ and $\mathcal{H}$ be defined in Definition~\ref{def2.2} that satisfy \eqref{eq3.8}. Using the facts that $\mathcal{S}_1$, $\mathcal{S}_2$ and $e^{\mathcal{H}t}$ are symplectic and  applying \eqref{eq3.29}, \eqref{eq3.30} and \eqref{eq3.34}, we have
\begin{align}
Q(t)&=[I,0]\left[%
\begin{array}{c}
Q(t)\\
P(t)\\
\end{array}%
\right]=[I,0]\mathcal{S}_2e^{-\mathcal{H}t}\mathcal{S}_2^{-1}\left[%
\begin{array}{cc}
I&0\\
0&-I\\
\end{array}%
\right]\left[%
\begin{array}{c}
I\\
X_{22}^{1}\\
\end{array}%
\right]\nonumber\\&=[I,0]\mathcal{S}_2e^{-\mathcal{H}t}\mathcal{S}_2^{-1}\mathcal{J}\left[%
\begin{array}{c}
X_{22}^{1}\\
I\\
\end{array}%
\right]=[I,0]\mathcal{S}_2\mathcal{J}{e^{\mathcal{H}t}}^H\mathcal{S}_2^{H}\left[%
\begin{array}{c}
X_{22}^{1}\\
I\\
\end{array}%
\right],\label{eq3.37}\\
Q_{\star}(t)&=[I,0]\left[%
\begin{array}{c}
Q_{\star}(t)\\
P_{\star}(t)\\
\end{array}%
\right]=[0,-I]\mathcal{S}_1e^{\mathcal{H}t}\mathcal{S}_1^{-1}\mathcal{J}\left[%
\begin{array}{c}
I\\
X_{11}^{1}\\
\end{array}%
\right]\nonumber\\&=-[0,I]\mathcal{S}_1e^{\mathcal{H}t}\mathcal{J}\mathcal{S}_1^{H}\left[%
\begin{array}{c}
I\\
X_{11}^{1}\\
\end{array}%
\right].\label{eq3.38}
\end{align}

Suppose that $\hat{t}\in{\mathbb R}$ such that ${\rm det}(Q(\hat{t}))= 0$. We first claim that
\begin{align*}
\lim_{t\rightarrow \hat{t}}\|P(t)Q(t)^{-1}\|=\lim_{t\rightarrow \hat{t}}\|X_{12}^{1}Q(t)^{-1}\|=\infty.
\end{align*}
Since $[Q(t)^{\top},P(t)^{\top}]^{\top}=e^{\widetilde{\mathcal{H}}\hat{t}}[I, X_{22}^1]^{\top}$ is of full column rank and $Q(\hat{t})$ is singular, it is easily seen that  $\lim_{t\rightarrow \hat{t}}\|P(t)Q(t)^{-1}\|=\infty$. Now, we show that $\lim_{t\rightarrow \hat{t}}\|X_{12}^{1}Q(t)^{-1}\|=\infty$. Since $Q(t)$ is continuous and $Q(\hat{t})$ is singular, it suffices to show that $X_{12}^{1}x_0\neq 0$, where $Q(\hat{t})x_0=0$ with  $x_0\neq 0$. We prove it by contradiction. Suppose that $X_{12}^{1}x_0= 0$. Since $(\mathcal{M}_1, \mathcal{L}_1)\in\mathbb{S}_{\mathcal{S}_1,\mathcal{S}_2}$, Eq. \eqref{eq3.8} can be written in the form
\begin{align}\label{eq3.39}
\left[%
\begin{array}{cc}
X_{12}^1&0\\
X_{22}^1&I\\
\end{array}%
\right]\mathcal{S}_2\Pi_0=\left[%
\begin{array}{cc}
I&X_{11}^1\\
0&X_{21}^1\\
\end{array}%
\right]\mathcal{S}_1\Pi_{\infty}e^{\mathcal{H}}.
\end{align}
Using the facts that $X_{12}^1=X_{21}^{1^H}$ and $X_{12}^{1}x_0= 0$, it follows from the second row of \eqref{eq3.39} that
$x_0^H[X_{22}^1,I]\mathcal{S}_2\Pi_0=0$.
Since $\Pi_0[U_1,U_0]=[U_1,U_0]$, we have
$x_0^H[X_{22}^1,I]\mathcal{S}_2[U_1,U_0]=0$.
Using the definition of $\mathcal{H}$ in \eqref{eq2.10} yields
\begin{align*}
x_0^H[X_{22}^1,I]\mathcal{S}_2e^{\mathcal{H}t}&=x_0^H[X_{22}^1,I]\mathcal{S}_2[U_1|U_0,U_\infty]\left[%
\begin{array}{c|c}
e^{\widehat{\mathcal{H}}t}&0\\\hline
0&I_{2\ell}\\
\end{array}%
\right][U_1|U_0,U_\infty]^{-1}\\
&=\left[0,0,x_0^H[X_{22}^1,I]\mathcal{S}_2U_{\infty}\right][U_1|U_0,U_\infty]^{-1},
\end{align*}
which is  independent of the parameter $t$. Therefore, we may denote $z_0^H=x_0^H[X_{22}^1,I]\mathcal{S}_2e^{\mathcal{H}t}$.
Multiplying $x_0$ from the right of \eqref{eq3.37}, it follows that
\begin{align*}
Q(t)x_0&=[I,0]\mathcal{S}_2\mathcal{J}\left({e^{\mathcal{H}t}}^H\mathcal{S}_2^{H}\left[%
\begin{array}{c}
X_{22}^{1}\\
I\\
\end{array}%
\right]x_0\right)=[I,0]\mathcal{S}_2\mathcal{J}z_0
\end{align*}
which is  independent of the parameter $t$. Because $Q(0)=I$ and $x_0\neq 0$, we have $Q(\hat{t})x_0=Q(0)x_0\neq 0$. This contradicts that $Q(\hat{t})x_0=0$.

Now, we show that
\begin{align*}
\lim_{t\rightarrow \hat{t}}\|P_{\star}(t)Q_{\star}(t)^{-1}\|=\lim_{t\rightarrow \hat{t}}\|X_{21}^{1}Q_{\star}(t)^{-1}\|=\infty.
\end{align*}
Using the fact that $X_{21}^{1}Q_{\star}(t)^{-1}=X_{21}(t+1)=X_{12}(t+1)^H=(X_{12}^{1}Q(t)^{-1})^{H}$, we have $\lim_{t\rightarrow \hat{t}}\|X_{21}^{1}Q_{\star}(t)^{-1}\|=\infty$. Consequently, $Q_{\star}(\hat{t})$ is singular. Then $\lim_{t\rightarrow \hat{t}}\|P_{\star}(t)Q_{\star}(t)^{-1}\|=\infty$ can be proven by the similar argument for $\lim_{t\rightarrow \hat{t}}\|P(t)Q(t)^{-1}\|=\infty$. This  proves the inclusion
\begin{align}\label{eq3.40}
\{t\in \mathbb{R}| \ {\rm det}(Q(t))= 0\}\subseteq\{t\in \mathbb{R}| \ {\rm det}(Q_{\star}(t))= 0\}.
\end{align}
The conclusion for Eq. \eqref{eq3.40} can be shown accordingly by \eqref{eq3.38}. Hence, \eqref{eq3.36} holds true.
\end{proof}

Now, let
\begin{align}\label{eq3.41}
\mathcal{T}_W=\{t\in{\mathbb R}|~Q(t)\text{ is invertible}\}.
\end{align}
Theorem~\ref{thm3.10} enables us to write the set $\mathcal{T}_W$ in an alternative form $\mathcal{T}_W=\{t\in{\mathbb R}|~Q_\star(t)\text{ is invertible}\}$.
Since det$(Q(t))$ is analytic, the zeros of det$(Q(t))$ are isolated. It follows $\mathcal{T}_W$ is the set that $\mathbb{R}$ subtracts some isolated points, and hence, $\mathcal{T}_W$ is a union of open intervals, say
\begin{align}\label{eq3.42}
\mathcal{T}_W=\bigcup_{k\in \mathbb{Z}}(\hat{t}_k,\hat{t}_{k+1}).
\end{align}
Here det$Q(\hat{t}_k)=0$ for each $k$ and $\cdots<\hat{t}_{-1}<\hat{t}_0<\hat{t}_1<\cdots$. Since $Q(0)=I$, it implies that $0\in\mathcal{T}_W$.  For convenience, we may say $0\in(\hat{t}_0,\hat{t}_1)$. Therefore, from Radon's Lemma follows that $(\hat{t}_0,\hat{t}_1)$ is the maximal interval of the RDEs~\eqref{eq3.27} and \eqref{eq3.33}. Later in Subsection~\ref{sec3.3}, we shall extend the domain of $W(t)$ and $W_\star(t)$ to whole $\mathcal{T}_W$.

\subsection{The Extension of Structure-Preserving Flow: the Phase Portrait on Grassmann Manifolds}\label{sec3.3}

Let $G^n(\mathbb{C}^{2n})$ be the Grassmann manifold that consists of $n$-dimensional subspaces of a $2n$-dimensional space, equipped with an appropriate
topology (see e.g., \cite{Abou-Kandil03}). Intrinsically, $G^n(\mathbb{C}^{2n})$ can be written as
\[
G^n(\mathbb{C}^{2n})=\left\{\text{Im}\left(\left[\begin{array}{l}A\\B\end{array}\right]\right)|~A,B\in{\mathbb C}^{n\times n}\text{ and }\text{rank}\left(\left[\begin{array}{l}A\\B\end{array}\right]\right)=n\right\}.
\]
Here Im$\left([A^{\top},B^{\top}]^{\top}\right)$ denotes the column space spanned by $[A^{\top},B^{\top}]^{\top}$. It is easily seen that  $\mathbb{C}^{n\times n}$ can be embedded into $G^n(\mathbb{C}^{2n})$ by
\[
\psi(W)=\text{Im}\left(\left[\begin{array}{c}I\\W\end{array}\right]\right).
\]
Let
$G^n_0(\mathbb{C}^{2n})=\left\{\text{Im}\left([A^{\top},B^{\top}]^{\top}\right)\in G^n(\mathbb{C}^{2n})|~A \in{\mathbb C}^{n\times n}\text{ is invertible}\right\}$.
Then $G^n_0(\mathbb{C}^{2n})=\psi(\mathbb{C}^{n\times n})$ is the image of $\psi$. Note that the Grassmann manifold $G^n(\mathbb{C}^{2n})$ is a compact analytic manifold of dimension $n^2$ and that $G^n_0(\mathbb{C}^{2n})$ is an open dense subset of $G^n(\mathbb{C}^{2n})$ (see e.g., \cite{Abou-Kandil03}).

Radon's Lemma leads us to consider a natural extension of the flow defined by the RDE \eqref{eq3.27} in  $\mathbb{C}^{n\times n}$ to a flow on the Grassmann manifold  $G^n(\mathbb{C}^{2n})$, via the process by the embedding
\[
\psi(W(t))=\text{Im}\left(\left[\begin{array}{c}I\\W(t)\end{array}\right]\right)=\text{Im}\left(\left[\begin{array}{c}Q(t)\\P(t)\end{array}\right]\right).
\]
Hence, a flow of RDE \eqref{eq3.27} on $G^n(\mathbb{C}^{2n})$ is just the linear flow of \eqref{eq3.29}.
Note that the maximal interval of the linear flow of \eqref{eq3.29} is $\mathbb{R}$.
In addition, the representation of Theorem~\ref{thm3.8}(ii) holds not only for all $t\in(\hat{t}_0,\hat{t}_1)$ but also for $t\in \mathcal{T}_W$ defined in \eqref{eq3.41}. Hence, the \textit{extended solution} of RDE \eqref{eq3.27} is
\begin{align*}
W(t)=P(t)Q(t)^{-1},\ \ \text{ for }t\in \mathcal{T}_W
\end{align*}
where $[Q(t)^{\top},P(t)^{\top}]^{\top}$ is the solution of \eqref{eq3.29}. Here,  $\psi(W(t))\in G^n_0(\mathbb{C}^{2n})$ for $t\in \mathcal{T}_W$. In the case $t\not\in \mathcal{T}_W$, i.e., $t=\hat{t}_k$ for some $k\in{\mathbb Z}$, $W(t)$ does not exist but $\text{Im}\left([Q(t)^{\top},P(t)^{\top}]^{\top}\right)\in G^{n}(\mathbb{C}^{2n})\setminus G^{n}_0(\mathbb{C}^{2n})$. Since det$(Q(t))$ is an analytic function of $t$,  $W(t)$ is meromorphic. We note that the unboundedness of $\mathcal{T}_W$ implies that the limit, $\lim_{t\rightarrow \infty }W(t)$, is meaningful.
The asymptotic phenomena of the phase portrait of RDE \eqref{eq3.27} can be investigated by using the extended solution of RDE. This will be done in Section~\ref{sec4}.

Theorem \ref{thm3.10} shows that $Q(t)$ and $Q_{\star}(t)$ are simultaneously invertible, where $Y(t)=[Q(t)^{\top},P(t)^{\top}]^{\top}$ and $Y_{\star}(t)=[Q(t)_{\star}^{\top},P(t)_{\star}^{\top}]^{\top}$ are the solutions of \eqref{eq3.29} and \eqref{eq3.34}, respectively. From Corollary \ref{cor3.9} and \eqref{eq3.35}, the  {\it extended solution},
$X(t)=[X_{ij}(t)]_{1\leqslant i,j\leqslant 2}$,
of IVP \eqref{eq3.9} can be defined as
\begin{align}\label{eq3.43}
\begin{array}{l}
X_{11}(t)=P_{\star}(t-1)Q_{\star}(t-1)^{-1},\\
X_{21}(t)=X_{21}^1Q_{\star}(t-1)^{-1},\\
X_{12}(t)=X_{12}^1Q(t-1)^{-1},\\
X_{22}(t)=P(t-1)Q(t-1)^{-1},
\end{array}
\end{align}
for $t\in \mathcal{T}_W+1$, where $\mathcal{T}_W+1$ denotes the set
\begin{align}\label{eq3.44}
\mathcal{T}_W+1\equiv\left\{t+1|t\in \mathcal{T}_W\right\}=\{t\in \mathbb{R}|~Q(t-1)\text{ is invertible}\}.
\end{align}
In Remark \ref{rem3.4}, we demonstrate that the maximal interval of IVP \eqref{eq3.9}, i.e., the maximal interval of $\mathcal{T}_W+1$ containing 1, coincides with the connected component of $\mathcal{T}_X$ containing 1.  In the following theorem we will show that $\mathcal{T}_W+1=\mathcal{T}_X$ and  $(\mathcal{M}(t), \mathcal{L}(t))=T_{\mathcal{S}_1, \mathcal{S}_2}(X(t))$ satisfies \eqref{eq3.14} for $t\in \mathcal{T}_W+1$, where $X(t)$ is the extended solution of IVP \eqref{eq3.9}, and vice versa.

\begin{Theorem}\label{thm3.11}
Suppose the assumptions of Theorem \ref{thm3.2} hold.

\begin{itemize}
\item[(i)] If $X(t)$, for $t\in \mathcal{T}_W+1$, is the extended solution of IVP \eqref{eq3.9}, then $(\mathcal{M}(t), \mathcal{L}(t))=T_{\mathcal{S}_1, \mathcal{S}_2}(X(t))$ satisfies \eqref{eq3.14} for $t\in \mathcal{T}_W+1$;
\item[(ii)] $\mathcal{T}_W+1=\mathcal{T}_X$ where $\mathcal{T}_X$ is defined in \eqref{eq3.23};
\item[(iii)] if $(\mathcal{M}(t), \mathcal{L}(t))$ is the solution of \eqref{eq3.14} for $t\in\mathcal{T}_X$, then $X(t)=T^{-1}_{\mathcal{S}_1, \mathcal{S}_2}(\mathcal{M}(t), \mathcal{L}(t))$ is the extended solution of IVP \eqref{eq3.9}.
\end{itemize}
\end{Theorem}

\begin{proof}
We first prove assertion $(i)$. Suppose that $X(t)=[X_{ij}(t)]_{1\leqslant i,j\leqslant 2}$  for $t\in\mathcal{T}_W+1$, defined in \eqref{eq3.43}, is the extended solution of IVP \eqref{eq3.9}. Since $X_{22}(t)$ is Hermitian and $X_{21}(t)=X_{12}(t)^H$, it holds that
\begin{align}\label{eq3.45}
[X_{21}(t),X_{22}(t)]=Q(t-1)^{-H}[X_{12}^{1^H}, P(t-1)^H],
\end{align}
where $[Q(t)^{\top},P(t)^{\top}]^{\top}$ is the solution of IVP \eqref{eq3.29}. Using the definitions of $\mathcal{H}$ and $\widetilde{\mathcal{H}}$ in \eqref{eq2.10} and  \eqref{eq3.30}, respectively, we have
\begin{align}\label{eq3.46}
\left[%
\begin{array}{c}
Q(t-1)\\
P(t-1)\\
\end{array}%
\right]&=e^{\widetilde{\mathcal{H}}(t-1)}\left[%
\begin{array}{c}
I\\
X_{22}^1\\
\end{array}%
\right]\nonumber\\&=\left[%
\begin{array}{cc}
I&0\\
0&-I\\
\end{array}%
\right]\mathcal{S}_2\mathbf{U}(e^{-\widehat{\mathcal{H}}(t-1)}\oplus I_{2\ell})\mathbf{U}^{-1}\mathcal{S}_2^{-1}\left[%
\begin{array}{c}
I\\
-X_{22}^1\\
\end{array}%
\right].
\end{align}
Since $\mathcal{S}_2$ and  $e^{-\mathcal{H}(t-1)}=\mathbf{U}(e^{-\widehat{\mathcal{H}}(t-1)}\oplus I_{2\ell})\mathbf{U}^{-1}$  are symplectic, we have
\begin{align*}
\mathcal{J}\mathcal{S}_2\mathbf{U}(e^{-\widehat{\mathcal{H}}(t-1)}\oplus I_{2\ell})\mathbf{U}^{-1}\mathcal{S}_2^{-1}=\mathcal{S}_2^{-H}\mathbf{U}^{-H}(e^{\widehat{\mathcal{H}}(t-1)}\oplus I_{2\ell})^H\mathbf{U}^{H}\mathcal{S}_2^{H}\mathcal{J}.
\end{align*}
Applying the last equation to \eqref{eq3.46} it follows that
\begin{align}\label{eq3.47}
\mathbf{U}^{H}\mathcal{S}_2^{H}\left[%
\begin{array}{c}
P(t-1)\\
Q(t-1)\\
\end{array}%
\right]&=-(e^{\widehat{\mathcal{H}}(t-1)}\oplus I_{2\ell})^H\mathbf{U}^{H}\mathcal{S}_2^{H}\mathcal{J}\left[%
\begin{array}{c}
I\\
-X_{22}^1\\
\end{array}%
\right]\nonumber\\
&=(e^{\widehat{\mathcal{H}}(t-1)}\oplus I_{2\ell})^H\mathbf{U}^{H}\mathcal{S}_2^{H}\left[%
\begin{array}{c}
X_{22}^1\\
I\\
\end{array}%
\right].
\end{align}
Using the fact that $(\mathcal{M}_1,\mathcal{L}_1)\in \mathbb{S}_{\mathcal{S}_1, \mathcal{S}_2}$ satisfying \eqref{eq3.8},  definitions of $\Pi_0$ and $\Pi_{\infty}$ in  \eqref{eq2.11}, we have
\begin{align}\label{eq3.48}
\left[%
\begin{array}{cc}
X_{12}^1&0\\
X_{22}^1&I\\
\end{array}%
\right]\mathcal{S}_2\mathbf{U}(I_{2\hat{n}}\oplus E_{11})=\left[%
\begin{array}{cc}
I&X_{11}^1\\
0&X_{21}^1\\
\end{array}%
\right]\mathcal{S}_1\mathbf{U}(e^{\widehat{\mathcal{H}}}\oplus E_{22}),
\end{align}
where $E_{11}$ and $E_{22}$ are defined in \eqref{eq3.18a} and $\hat{n}=n-\ell$. Since $X_{22}^{1^H}=X_{22}^1$ and $X_{21}^{1^H}=X_{12}^1$, it follows from \eqref{eq3.47} and \eqref{eq3.48} that
\begin{align*}
(I_{2\hat{n}}\oplus E_{11})&[{\mathbf{V}_{2}^2}^{H}, {\mathbf{V}_{1}^2}^H]\left[%
\begin{array}{c}
Q(t-1)\\
P(t-1)\\
\end{array}%
\right]=(I_{2\hat{n}}\oplus E_{11})\mathbf{U}^{H}\mathcal{S}_2^{H}\left[%
\begin{array}{c}
P(t-1)\\
Q(t-1)\\
\end{array}%
\right]\nonumber\\
&=(e^{\widehat{\mathcal{H}}(t-1)}\oplus  I_{2\ell})^H(I_{2\hat{n}}\oplus E_{11})^H\mathbf{U}^{H}\mathcal{S}_2^{H}\left[%
\begin{array}{c}
X_{22}^1\\
I\\
\end{array}%
\right]\nonumber\\
&=(e^{\widehat{\mathcal{H}}(t-1)}\oplus  I_{2\ell})^H(e^{\widehat{\mathcal{H}}}\oplus E_{22})^H\mathbf{U}^H\mathcal{S}_1^H\left[%
\begin{array}{c}
0\\
X_{21}^{1^H}\\
\end{array}%
\right]\nonumber\\
&=(e^{\widehat{\mathcal{H}}t}\oplus  E_{22})^H{\mathbf{V}_{2}^1}^HX_{12}^1,
\end{align*}
where $\mathbf{V}_{1}$ and $\mathbf{V}_{2}$ are defined in \eqref{eq3.18b}. We then have
\begin{align*}
[-(e^{\widehat{\mathcal{H}}t}\oplus  E_{22})^H{\mathbf{V}_{2}^1}^H,(I_{2\hat{n}}\oplus E_{11}) {\mathbf{V}_{1}^2}^H]\left[%
\begin{array}{c}
X_{12}^1\\
P(t-1)\\
\end{array}%
\right]=-(I_{2\hat{n}}\oplus E_{11}) {\mathbf{V}_{2}^2}^HQ(t-1).
\end{align*}
Combining the last equation and \eqref{eq3.45}, we obtain
\begin{align*}
[X_{21}(t),X_{22}(t)]\left[%
\begin{array}{r}
-\mathbf{V}_{2}^1(e^{\widehat{\mathcal{H}}t}\oplus  E_{22})\\
\mathbf{V}_{1}^2(I_{2\hat{n}}\oplus E_{11})\\
\end{array}%
\right]=-\mathbf{V}_{2}^2(I_{2\hat{n}}\oplus E_{11}).
\end{align*}
Therefore, the equality of the second row of \eqref{eq3.19} holds.  The equality of the first row can be accordingly obtained by using the formulas for $X_{11}(t)$ and $X_{12}(t)=X_{21}(t)^H$ in \eqref{eq3.43} and the solution $Y_{\star}(t)=[Q(t)_{\star}^{\top},P(t)_{\star}^{\top}]^{\top}$ of the linear differential equation \eqref{eq3.34}.  Since \eqref{eq3.19} is equivalent to \eqref{eq3.14} by Lemma~\ref{lem3.5}, this proves assertion $(i)$.

Now we prove assertion $(ii)$. From assertion $(i)$, we have $\mathcal{T}_W+1\subseteq \mathcal{T}_X$. From \eqref{eq3.42} and \eqref{eq3.44}, we obtain that  $\mathcal{T}_W+1=\bigcup_{k\in \mathbb{Z}}(\hat{t}_k+1,\hat{t}_{k+1}+1)\subseteq \mathcal{T}_X$. For each $k\in \mathbb{Z}$, we have $(\hat{t}_k+1,\hat{t}_{k+1}+1)\subseteq \mathcal{T}_X$. Choosing a point $t_{k+1/2}\in (\hat{t}_k+1,\hat{t}_{k+1}+1)$, it follows from assertion $(i)$ that  $(\mathcal{M}(t_{k+1/2}), \mathcal{L}(t_{k+1/2}))=T_{\mathcal{S}_1, \mathcal{S}_2}(X(t_{k+1/2}))$ is the solution of \eqref{eq3.14} at $t=t_{k+1/2}$.  A similar argument to Theorem \ref{thm3.7} and Remark \ref{rem3.4} shows that $(\hat{t}_k+1,\hat{t}_{k+1}+1)$ is the connected component of $\mathcal{T}_X$ containing $t_{k+1/2}$. Hence, $\mathcal{T}_W+1=\mathcal{T}_X$.

Now we prove assertion $(iii)$. From Theorem~\ref{thm3.6} it follows that the solution $(\mathcal{M}(t), \mathcal{L}(t))$ of \eqref{eq3.14} is unique for each $t\in\mathcal{T}_X$. Therefore, assertions $(i)$ and $(ii)$ lead to the fact that $X(t)=T^{-1}_{\mathcal{S}_1, \mathcal{S}_2}(\mathcal{M}(t), \mathcal{L}(t))$ is the extended solution. This completes the proof.
\end{proof}

\subsection{Structure-Preserving Flow vs. SDA}\label{sec3.4}
Suppose that $(\mathcal{M}_1, \mathcal{L}_1)\in \mathbb{S}_{\mathcal{S}_1,\mathcal{S}_2}$ is a regular symplectic pair with ${\rm ind}_{\infty}(\mathcal{M}_1, \mathcal{L}_1)\leqslant 1$. Then
the structure-preserving flow $(\mathcal{M}(t), \mathcal{L}(t))=T_{\mathcal{S}_1, \mathcal{S}_2}(X(t))\in \mathbb{S}_{\mathcal{S}_1, \mathcal{S}_2}$ with the initial $(\mathcal{M}(1), \mathcal{L}(1))=(\mathcal{M}_1, \mathcal{L}_1)$  has been constructed in Theorem \ref{thm3.2}, where $X(t)$ for $t\in \mathcal{T}_X$ is the extended solution of IVP \eqref{eq3.9}. This flow satisfies both the \textbf{Eigenvector-Preserving Property} and the \textbf{Structure-Preserving Property}. In addition, Theorem \ref{thm3.11} shows the phase portrait of this flow is actually the solution curve of \eqref{eq3.14}, i.e., the curve $\mathcal{C}_{\mathcal{M}_1, \mathcal{L}_1}$ in \eqref{eq3.22}.

The structure-preserving doubling algorithm (SDA) is a powerful tool for solving CAREs \eqref{eq1.1}, DAREs \eqref{eq1.2} and NMEs \eqref{eq1.3}. In \cite{lin06}, two special classes of symplectic pairs, $\mathbb{S}_{1}=\mathbb{S}_{I_{2n},I_{2n}}$ and $\mathbb{S}_{2}=\mathbb{S}_{-I_{2n},\mathcal{J}}$ as in \eqref{eq1.6}, are considered and SDAs (SDA-1 and SDA-2 shown in \eqref{eq1.4} and \eqref{eq1.5}, respectively) are developed for solving CAREs, DAREs and NMEs such that the iterates, $(\mathcal{M}_k, \mathcal{L}_k)$ for $k=1,2,\ldots$, generated by SDA-1 and SDA-2 are in $\mathbb{S}_{1}$ and in $\mathbb{S}_{2}$, respectively. In addition, it has been shown that the iterate $(\mathcal{M}_k, \mathcal{L}_k)$ satisfies \begin{align}\label{eq3.49}
\mathcal{M}_kU_0=0,\ \ \mathcal{L}_kU_\infty=0\ \text{ and }\mathcal{M}_kU_1=\mathcal{L}_kU_1e^{\widehat{\mathcal{H}}2^{k-1}}
\end{align}
for each $k\in\mathbb{N}$, where the initial  pair $(\mathcal{M}_1, \mathcal{L}_1)$ satisfies \eqref{eq2.6} with $\widehat{\mathcal{S}}=e^{\widehat{\mathcal{H}}}$. By applying Theorem~\ref{thm3.6} $(iii)$ to \eqref{eq3.49}, we have $(\mathcal{M}_k, \mathcal{L}_k)=(\mathcal{M}(2^{k-1}), \mathcal{L}(2^{k-1}))\in \mathcal{C}_{{\mathcal M}_1,{\mathcal L}_1}$ defined in \eqref{eq3.22}.
Applying Theorem~\ref{thm3.11}, we have the following consequence immediately.

\begin{Theorem}\label{thm3.12}
Let $(\mathcal{M}_1,\mathcal{L}_1)\in \mathbb{S}_1$ or $\mathbb{S}_2$ with ${\rm ind}_{\infty} (\mathcal{M}_1,\mathcal{L}_1)\leqslant 1$.  Suppose $(\mathcal{M}_k, \mathcal{L}_k)$, $k=1,2,\ldots$, is the sequence generated by the SDA.  Then, for each $k\in\mathbb{N}$, $(\mathcal{M}_k, \mathcal{L}_k)=(\mathcal{M}(2^{k-1}), \mathcal{L}(2^{k-1}))$, where $(\mathcal{M}(t), \mathcal{L}(t))=T_{\mathcal{S}_1, \mathcal{S}_2}(X(t))$ and $X(t)$ is the extended solution of IVP \eqref{eq3.9}. Here, $(\mathcal{S}_1, \mathcal{S}_2)=(I,I)$ or $(-I,\mathcal{J})$ if $(\mathcal{M}_1,\mathcal{L}_1)\in \mathbb{S}_1$ or $\mathbb{S}_2$, respectively.
\end{Theorem}

\section{Asymptotic Analysis of Structure-Preserving Flows Using $e^{\mathscr{H}t}$}\label{sec4}

In this section, we consider  the solution of the IVP:
\begin{align}\label{eq4.1}
\dot{Y}(t)=\mathscr{H}Y(t),\ \ Y(0)= \left[%
\begin{array}{c}
I\\
W_0\\
\end{array}%
\right],
\end{align}
where $Y(t)\in \mathbb{C}^{2n\times n}$, $W_0=W_0^{H}$ and $\mathscr{H}\in \mathbb{C}^{2n\times 2n}$ is a Hamiltonian matrix. It is well-known that the solution of IVP \eqref{eq4.1} is
\begin{align}\label{eq4.2}
Y(t;\mathscr{H},W_0)\equiv\left[%
\begin{array}{c}
Q(t;\mathscr{H},W_0)\\
P(t;\mathscr{H},W_0)\\
\end{array}%
\right]= e^{\mathscr{H}t}\left[%
\begin{array}{c}
I\\
W_0\\
\end{array}%
\right].
\end{align}
Radon's Lemma shows that $P(t;\mathscr{H},W_0)Q(t;\mathscr{H},W_0)^{-1}$, $t\in\mathcal{T}_W$, is the extended solution of the RDE
\begin{align}\label{eq4.3}
\begin{array}{l}
\dot{W}(t)=[-W(t),I]\mathscr{H}\left[\begin{array}{c}I \\W(t) \end{array}\right],\\
W(0)=W_0.
\end{array}
\end{align}
Here $\mathcal{T}_W$ is defined in \eqref{eq3.41} depending on $\mathscr{H}$ and $W_0$.

We denote by $W(t;\mathscr{H},W_0)$ the solution of RDE \eqref{eq4.3} with $\mathscr{H}$ and the initial  $W_0$ being parameters of the system.  Then, the Hamiltonian matrix $\mathscr{H}$ plays the role that governs how $W(t)$ in \eqref{eq4.3} behaves. Fist, we consider the case $W_0=X_{22}^1$ and
$\mathscr{H}=\widetilde{\mathcal{H}}$,
where $\widetilde{\mathcal{H}}$ is given in \eqref{eq3.30}. Applying the relation between $\widetilde{\mathcal{H}}$ and $\mathcal{S}_2\mathcal{H}\mathcal{S}_2^{-1}$ in \eqref{eq3.30} together with  Corollary~\ref{cor3.9}, the extended solutions of \eqref{eq3.26d} and \eqref{eq3.26b}, respectively, are of the forms
\begin{subequations}\label{eq4.4}
\begin{align}
X_{22}(t)&=W(t-1;\widetilde{\mathcal{H}},X_{22}^1)=-W(t-1;-\mathcal{S}_2\mathcal{H}\mathcal{S}_2^{-1},-X_{22}^1)\notag\\
&=-P(t-1;-\mathcal{S}_2\mathcal{H}\mathcal{S}_2^{-1},-X_{22}^1)Q(t-1;-\mathcal{S}_2\mathcal{H}\mathcal{S}_2^{-1},-X_{22}^1)^{-1},\notag\\
&=-P(-t+1;\mathcal{S}_2\mathcal{H}\mathcal{S}_2^{-1},-X_{22}^1)Q(-t+1;\mathcal{S}_2\mathcal{H}\mathcal{S}_2^{-1},-X_{22}^1)^{-1},\label{eq4.4a}\\
X_{12}(t)&=X_{12}^1Q(-t+1;\mathcal{S}_2\mathcal{H}\mathcal{S}_2^{-1},-X_{22}^1)^{-1},\label{eq4.4b}
\end{align}
\end{subequations}
for $t\in \mathcal{T}_W+1$. On the other hand, if   $W_0=X_{11}^1$ and
$\mathscr{H}=\widetilde{\mathcal{H}}_{\star}\equiv \mathcal{J}^{-1}\mathcal{S}_1\mathcal{H}\mathcal{S}_1^{-1}\mathcal{J}$,
then from Theorem \ref{thm3.10} and \eqref{eq3.35},  the extended solutions of \eqref{eq3.26a} and \eqref{eq3.26c}, respectively, are of the forms
\begin{subequations}\label{eq4.5}
\begin{align}
X_{11}(t)&=W(t-1;\widetilde{\mathcal{H}}_{\star},X_{11}^1)=P(t-1;\widetilde{\mathcal{H}}_{\star},X_{11}^1)Q(t-1;\widetilde{\mathcal{H}}_{\star},X_{11}^1)^{-1},\label{eq4.5a}\\
X_{21}(t)&=X_{21}^1Q(t-1;\widetilde{\mathcal{H}}_{\star},X_{11}^1)^{-1}, \label{eq4.5b}
\end{align}
\end{subequations}
for $t\in \mathcal{T}_W+1$.  Connections between RDEs and SDAs can be made by the terminology of \eqref{eq4.4} and \eqref{eq4.5}.

\begin{Lemma}\label{lem4.1}
Let $(\mathcal{M}_1,\mathcal{L}_1)\in \mathbb{S}_1$ or $\mathbb{S}_2$ with ${\rm ind}_{\infty} (\mathcal{M}_1,\mathcal{L}_1)\leqslant 1$. Suppose $(\mathcal{M}_k, \mathcal{L}_k)$, $k=1,2,\ldots$, is the sequence generated by the SDA and denote $X_k=[X^k_{ij}]_{1\le i,j\le 2}\equiv T_{\mathcal{S}_1,\mathcal{S}_2}^{-1}(\mathcal{M}_k, \mathcal{L}_k)$. Here, $(\mathcal{S}_1, \mathcal{S}_2)=(I,I)$ or $(-I,\mathcal{J})$ if $(\mathcal{M}_1,\mathcal{L}_1)\in \mathbb{S}_1$ or $\mathbb{S}_2$, respectively. Then
\begin{align*}
X^k_{22}&=-W(-2^{k-1}+1;\mathcal{S}_2\mathcal{H}\mathcal{S}_2^{-1},-X_{22}^1)\\
&=-P(-2^{k-1}+1;\mathcal{S}_2\mathcal{H}\mathcal{S}_2^{-1},-X_{22}^1)Q(-2^{k-1}+1;\mathcal{S}_2\mathcal{H}\mathcal{S}_2^{-1},-X_{22}^1)^{-1},\\
X^k_{12}&=X_{12}^1Q(-2^{k-1}+1;\mathcal{S}_2\mathcal{H}\mathcal{S}_2^{-1},-X_{22}^1)^{-1},\\
X^k_{11}&=W(2^{k-1}-1;\widetilde{\mathcal{H}}_{\star},X_{11}^1)=P(2^{k-1}-1;\widetilde{\mathcal{H}}_{\star},X_{11}^1)Q(2^{k-1}-1;\widetilde{\mathcal{H}}_{\star},X_{11}^1)^{-1},\\
X^k_{21}&=X_{21}^1Q(2^{k-1}-1;\widetilde{\mathcal{H}}_{\star},X_{11}^1)^{-1},
\end{align*}
for all $k=1,2,\ldots$.
\end{Lemma}

\begin{proof}
Denote $X(t)$ the solution of \eqref{eq3.9} and
$
(\mathcal{M}(t),\mathcal{L}(t))=T_{\mathcal{S}_1,\mathcal{S}_2}(X(t)).
$
By Theorem~\ref{thm3.12}, we see that
$
(\mathcal{M}_k, \mathcal{L}_k)=(\mathcal{M}(2^{k-1}), \mathcal{L}(2^{k-1})).
$
The fact of $X_k\equiv T_{\mathcal{S}_1,\mathcal{S}_2}^{-1}(\mathcal{M}_k, \mathcal{L}_k)$ implies that $X_k=X(2^{k-1})$. Applying \eqref{eq4.4} and \eqref{eq4.5} to the resulting equation leads to the assertion.
\end{proof}
From \eqref{eq4.4} and \eqref{eq4.5}, we conclude that
\begin{itemize}
\item[$(i)$] the large time behaviors of $X_{22}(t)$, $X_{12}(t)$ as $t\rightarrow \infty$ are determined by $W(t; \mathcal{S}_2\mathcal{H}\mathcal{S}_2^{-1},-X_{22}^1)$ and $Q(t; \mathcal{S}_2\mathcal{H}\mathcal{S}_2^{-1},-X_{22}^1)^{-1}$ as $t\rightarrow -\infty$;
\item[$(ii)$] the large time behaviors of $X_{11}(t)$, $X_{21}(t)$ as $t\rightarrow \infty$ are determined by  $W(t;\widetilde{\mathcal{H}}_{\star},X_{11}^1)$ and $Q(t;\widetilde{\mathcal{H}}_{\star},X_{11}^1)^{-1}$ as $t\rightarrow \infty$.
\end{itemize}
Note that Hamiltonian matrices  $\mathcal{S}_2\mathcal{H}\mathcal{S}_2^{-1}$ and  $\widetilde{\mathcal{H}}_{\star}$ are symplectically similar.
By assertions $(i)$ and $(ii)$ above, we see that the asymptotic behaviors of $X_{22}(t)$, $X_{12}(t)$ and $X_{11}(t)$, $X_{21}(t)$ as $t\rightarrow \infty$ are governed by
\begin{subequations}\label{eq4.6}
\begin{align}\label{eq4.6a}
Y(t; \mathcal{S}_2\mathcal{H}\mathcal{S}_2^{-1},-X_{22}^1)=\mathcal{S}_2e^{\mathcal{H}t}\mathcal{S}_{2}^{-1}\left[%
\begin{array}{c}
I\\
-X_{22}^1\\
\end{array}%
\right] (\text{as }t\rightarrow -\infty),
\end{align}
and
\begin{align}\label{eq4.6b}
Y(t;\widetilde{\mathcal{H}}_{\star},X_{11}^1)=\mathcal{J}^{-1}\mathcal{S}_1e^{\mathcal{H}t}\mathcal{S}_1^{-1}\mathcal{J}\left[%
\begin{array}{c}
I\\
X_{11}^1\\
\end{array}%
\right] (\text{as }t\rightarrow \infty),
\end{align}
\end{subequations}
respectively. For both cases in \eqref{eq4.6a} and \eqref{eq4.6b}, $e^{\mathcal{H}t}$ is involved. Therefore,  for a given Hamiltonian matrix $\mathscr{H}$, we are interested in the study of the asymptotic behavior of the solution, $W(t)=P(t)Q(t)^{-1}$, of RDE \eqref{eq4.3} and $Q(t)^{-1}$ as $t\rightarrow \pm\infty$.

The convergence results of RDEs, including time variant/invariant as well as Hermitian/non-Hermitian types, have been studied and generalized in many research works \cite{Abou-Kandil03,Callier_Willems:1981,Callier:1995,CALLIER_WINKIN_WILLEMS:1994,DeNICOLAO:1992,Freiling_Hochhaus:2002,Freiling_Jank:1991,Freiling_Jank:1995,Park_Kailath:1997}. In \cite{Freiling_Jank:1991} an analogous (asymptotic) formula for $W(t)$ has been derived for RDEs with polynomial coefficients and in \cite{Freiling_Hochhaus:2002} the representation formula and the comparison theorem have been used to derive convergence results in an elegant way for Hermitian RDEs. The influence of the initial value $W_0$ and of the Jordan structure of $\mathscr{H}$ on the corresponding Riccati flow is studied in \cite{Freiling_Jank:1995} by using Cramer's rule for the explicit representation of $P(t)Q(t)^{-1}$.

Due to the dependence on the Hamiltonian matrix $\mathscr{H}$, rather than applying a Jordan canonical form to $\mathscr{H}$,  we shall adopt the Hamiltonian Jordan canonical form for studying the asymptotic behavior of RDEs. The asymptotic formula for $P(t)Q(t)^{-1}$ can thus be obtained by the column space of $Y(t)$ in \eqref{eq4.6}.  A canonical form of a Hamiltonian matrix under symplectic similarity transformations has been investigated in \cite{Lin1999469}. For the description of this canonical form, we introduce some notations. Denote $\mathbb{C}_{>}:=\{z\in \mathbb{C}|\ \Re(z)>0\}$, where $\Re(z)$ is the real part of the complex number $z$. Let
\begin{align}\label{eq4.7}
N_{k}=\left[\begin{array}{cccc}0 & 1 &  &  \\ & \ddots & \ddots &  \\ &  & \ddots & 1 \\ &  &  & 0\end{array}\right]\in \mathbb{R}^{k\times k}, \ \ N_{k}(\lambda)=\lambda I_{k}+N_{k}
\end{align}
be the $k\times k$ nilpotent matrix and the Jordan block of size $k$ with the eigenvalue $\lambda$, respectively, and $e_k$ be the $k$th unit vector.

\begin{Theorem}\label{thm4.2}\cite[Hamiltonian Jordan canonical form]{Lin1999469}
Given a complex Hamiltonian matrix $\mathscr{H}$, there exists a complex symplectic matrix $\mathcal{S}$ such that
\begin{align}\label{eq4.8}
\mathfrak{J}:=\mathcal{S}^{-1}\mathscr{H}\mathcal{S}=\left[\begin{array}{cccc|cccc}R_r &  &  &  & 0 &  &  &  \\ & R_e &  &  &  & D_e &  &  \\ &  & R_c &  &  &  & D_c &  \\ &  &  & R_d &  &  &  & D_d \\\hline0 &  &  &  & -R_r^H &  &  &  \\ & 0 &  &  &  & -R_e^H &  &  \\ &  & 0 &  &  &  & -R_c^H &  \\ &  &  & G_d &  &  &  & -R_d^H\end{array}\right],
\end{align}
where the different blocks have the following structures.
\begin{itemize}
\item[1.] The blocks with index $r$ have the form
\begin{align*}
R_r ={\rm diag}(R_1^r,\ldots, R_{\mu_r}^r),\ \  R_k^r={\rm diag}(N_{d_{k,1}}(\lambda_k),\ldots,N_{d_{k,p_k}}(\lambda_k)),\ \ k=1,\ldots,\mu_r,
\end{align*}
where $\lambda_k\in \mathbb{C}_{>}$ are distinct.
\item[2.] The blocks with index $e$  have the form
\begin{align*}
\begin{array}{ll}
R_e ={\rm diag}(R_1^e,\ldots, R_{\mu_e}^e),& R_k^e={\rm diag}(N_{l_{k,1}}(i \alpha_k),\ldots,N_{l_{k,q_k}}(i \alpha_k)),\\
D_e ={\rm diag}(D_1^e,\ldots, D_{\mu_e}^e),& D_k^e={\rm diag}(\beta_{k,1}^ee_{l_{k,1}}e_{l_{k,1}}^H,\ldots,\beta_{k,q_k}^ee_{l_{k,q_k}}e_{l_{k,q_k}}^H),
\end{array}
\end{align*}
where for $k=1,\ldots,\mu_e$ and $j=1,\ldots,q_k$ we have $ \alpha_k\in \mathbb{R}$ are distinct and $\beta^e_{k,j}\in \{-1,1\}$.
\item[3.] The blocks with index $c$  have the form
\begin{align*}
\begin{array}{ll}
R_c ={\rm diag}(R_1^c,\ldots, R_{\mu_c}^c),& R_k^c={\rm diag}(B_{k,1},\ldots,B_{k,r_k}),\\
D_c ={\rm diag}(D_1^c,\ldots, D_{\mu_c}^c),& D_k^c={\rm diag}(D_{k,1},\ldots,D_{k,r_k}),
\end{array}
\end{align*}
where for $k=1,\ldots,\mu_c$ and $j=1,\ldots,r_k$ we have
\begin{align*}
B_{k,j}&=\left[\begin{array}{ccc}N_{m_{k,j}} (i \eta_k)& 0 & -\frac{\sqrt{2}}{2} e_{m_{k,j}}\\0 & N_{n_{k,j}}(i \eta_k)  & -\frac{\sqrt{2}}{2} e_{n_{k,j}} \\0 & 0 & i \eta_k\end{array}\right],\\
D_{k,j}&=\frac{\sqrt{2}}{2}i\beta_{k,j}^c\left[\begin{array}{ccc}0& 0 &  e_{m_{k,j}}\\0 & 0  & - e_{n_{k,j}} \\-e_{m_{k,j}}^H & e_{n_{k,j}}^H & 0\end{array}\right],
\end{align*}
$ \eta_k\in \mathbb{R}$ are distinct and $\beta^c_{k,j}\in \{-1,1\}$.
\item[4.] The blocks with index $d$  have the form
\begin{align*}
R_d ={\rm diag}(R_1^d,\ldots, R_{\mu_d}^d),\ \  G_d ={\rm diag}(G_1^d,\ldots, G_{\mu_d}^d),\ \
D_d ={\rm diag}(D_1^d,\ldots, D_{\mu_d}^d),
\end{align*}
where for $k=1,\ldots,\mu_d$, we have
\begin{align*}
R_{k}^d&=\left[\begin{array}{ccc}N_{s_{k}} (i \gamma_k)& 0 & -\frac{\sqrt{2}}{2} e_{s_{k}}\\0 & N_{t_{k}}(i \delta_k)  & -\frac{\sqrt{2}}{2} e_{t_{k}} \\0 & 0 & \frac{i}{2}(\gamma_k+\delta_k)\end{array}\right],\ \ G_k^d=\beta_{k}^d\left[\begin{array}{ccc}0& 0 &  0\\0 & 0  & 0\\0 & 0 & -\frac{1}{2}(\gamma_k-\delta_k)\end{array}\right],\\
D_{k}^d&=\frac{\sqrt{2}}{2}i\beta_{k}^d\left[\begin{array}{ccc}0& 0 &  e_{s_{k}}\\0 & 0  & - e_{t_{k}} \\-e_{s_{k}}^H & e_{t_{k}}^H & -i\frac{\sqrt{2}}{2}(\gamma_k-\delta_k)\end{array}\right],
\end{align*}
$\gamma_k\neq \delta_k$ and $\beta_{k}^d\in \{-1,1\}$.
\end{itemize}
\end{Theorem}

Suppose that the Hamiltonian matrix $\mathscr{H}$ in \eqref{eq4.1} has Hamiltonian Jordan canonical form
$\mathfrak{J}$
in \eqref{eq4.8}. Then the solution $Y(t)$ in \eqref{eq4.2} can be reformulated as
\begin{align}\label{eq4.9}
Y(t)= \mathcal{S}e^{\mathfrak{J}t}\mathcal{S}^{-1}\left[%
\begin{array}{c}
I\\
W_0\\
\end{array}%
\right]=\mathcal{S}e^{\mathfrak{J}t}\left[%
\begin{array}{c}
W_1\\
W_2\\
\end{array}%
\right],
\end{align}
where $[W_1^{\top},W_2^{\top}]^{\top}=\mathcal{S}^{-1}[I,W_0^{\top}]^{\top}$.

\subsection{The Structure of $e^{\mathfrak{J}t}$}\label{sec4.1}

In this subsection, we will describe the structure of  $e^{\mathfrak{J}t}$, where $\mathfrak{J}$ has form in \eqref{eq4.8}. Since $\mathfrak{J}$ is Hamiltonian, it is shown in Theorem~\ref{thm5.1} that $e^{\mathfrak{J}t}$ is symplectic for each $t\in \mathbb{R}$. Let
\begin{align}\label{eq4.10}
\begin{array}{l}
 P_k=\left[\begin{array}{cccc} 0&  &  &-1 \\ &  & (-1)^2 &  \\ &  &  &  \\ (-1)^{k}&  &  &0 \end{array}\right],\\
\Phi_k\equiv  \Phi_k(t)=e^{N_kt}=\left[\begin{array}{ccccc}1 & t &  \frac{t^2}{2!} & \cdots & \frac{t^{k-1}}{(k-1)! }\\ & 1 & t & \ddots& \vdots \\ &  & \ddots & \ddots &   \frac{t^2}{2!}\\  & &  & 1 & t\\& &  &  & 1\end{array}\right],\\ \phi_k\equiv \phi_k(t)=\left[\begin{array}{c} \frac{t^k}{k!}\\  \vdots\\  \frac{t^2}{2!}\\t\end{array}\right],\ \ \psi_k\equiv \psi_k(t)=\left[\begin{array}{c} t\\  \frac{t^2}{2!}\\  \vdots\\\frac{t^k}{k!}\end{array}\right],\\
\Gamma_{k_1}^{k_2}\equiv \Gamma_{k_1}^{k_2}(t)=\left[\begin{array}{cccc}\frac{t^{k_1}}{{k_1}!} & \frac{t^{(k_1+1)}}{{(k_1+1)}!} & \cdots & \frac{t^{k_2}}{{k_2}!} \\\frac{t^{(k_1-1)}}{{(k_1-1)}!} & \frac{t^{k_1}}{{k_1}!} & \ddots & \frac{t^{(k_2-1)}}{{(k_2-1)}!} \\\vdots & \ddots & \ddots & \vdots \\ \frac{t^{(2k_1-k_2)}}{{(2k_1-k_2)}!} & \cdots & \cdots & \frac{t^{k_1}}{{k_1}!}\end{array}\right],\\
\widehat{\Phi}_k\equiv \widehat{\Phi}_k(t)=P_k^{-1}\Phi_{k}P_k,\ \ \ \widehat{\Gamma}_k^{2k-1}\equiv \widehat{\Gamma}_k^{2k-1}(t)=\Gamma_k^{2k-1}P_k,
\end{array}
\end{align}
where $\Phi_k,P_k\in \mathbb{R}^{k\times k}$,
 $\Gamma_{k_1}^{k_2}\in \mathbb{R}^{(k_2-k_1+1)\times (k_2-k_1+1)}$ with $2k_1\geqslant k_2>k_1$ and $\phi_k,\psi_k\in \mathbb{R}^k$.

\begin{Lemma}\label{lem4.3}
Let $N_k$ and $\Phi_{k}$, $P_k$, $\widehat{\Phi}_k$ be as in \eqref{eq4.7} and  \eqref{eq4.10}, respectively.
Then
\begin{itemize}
\item[(i)] $P_k^{-1}=P_k^H=(-1)^{k-1}P_k$, $P_k^{-1}N_kP_k=-N_k^H$;
\item[(ii)] $\widehat{\Phi}_k\equiv P_k^{-1}\Phi_{k}P_k=e^{-N_k^Ht}=\Phi_k^{-H}$;
\item[(iii)] for each $\lambda\in \mathbb{C}$, we have $e^{\lambda t}\Phi_{k}=e^{N_k(\lambda)t}$ and $e^{-\bar{\lambda} t}\widehat{\Phi}_{k}=e^{-N_k(\lambda)^Ht}$.
\end{itemize}
\end{Lemma}
\begin{proof}
The proof is straightforward by direct calculations.
\end{proof}


\begin{Lemma}\label{lem4.4}
Let $A$ denote the Hamiltonian matrix $\left[\begin{array}{c|c}N_{k}(i\alpha)&\beta e_ke_k^H\\\hline 0&-N_{k}(i\alpha)^H\end{array}\right]\in \mathbb{C}^{2k\times 2k}$, where $\beta\in\{-1,1\}$ and $\alpha\in \mathbb{R}$. Then for each $t\in \mathbb{R}$,  $e^{At}$ has the form
\begin{align*}
e^{At}=\left[\begin{array}{c|c}e^{i\alpha t}\Phi_{k}&-e^{i\alpha t}\beta \widehat{\Gamma}_{k}^{2k-1}\\\hline 0&(e^{i\alpha t}\Phi_{k})^{-H}\end{array}\right],
\end{align*}
where $\Phi_{k}$, $\Gamma_{k}^{2k-1}$, $\widehat{\Gamma}_{k}^{2k-1}$, $P_k$ are defined in \eqref{eq4.10}.
\end{Lemma}

\begin{proof}
Let $\Theta=I_k\oplus (-\beta P_k)\in \mathbb{R}^{2k\times 2k}$. From Lemma \ref{lem4.3} it follows that
\begin{align*}
\Theta^{-1} N_{2k}(i\alpha)\Theta=\left[\begin{array}{c|c}N_{k}(i\alpha)&- \beta e_ke_1^HP_k\\\hline 0&P_k^{-1}N_{k}(i\alpha)P_k\end{array}\right]=\left[\begin{array}{c|c}N_{k}(i\alpha)&\beta e_ke_k^H\\\hline 0&-N_{k}(i\alpha)^H\end{array}\right]=A.
\end{align*}
Therefore,
\begin{align*}
e^{At}&=\Theta^{-1} e^{N_{2k}(i\alpha)t}\Theta=e^{i\alpha t}\Theta^{-1} \left[\begin{array}{c|c}\Phi_{k}&\Gamma_{k}^{2k-1}\\\hline 0&\Phi_{k}\end{array}\right]\Theta\\
&=e^{i\alpha t}\left[\begin{array}{c|c}\Phi_{k}&-\beta \Gamma_{k}^{2k-1}P_{k}\\\hline 0&P_{k}^{-1}\Phi_{k}P_{k}\end{array}\right]
=e^{i\alpha t}\left[\begin{array}{c|c}\Phi_{k}&-\beta \Gamma_{k}^{2k-1}P_{k}\\\hline 0&\widehat{\Phi}_{k}\end{array}\right]\\
&=\left[\begin{array}{c|c}e^{i\alpha t}\Phi_{k}&-e^{i\alpha t}\beta \widehat{\Gamma}_{k}^{2k-1}\\\hline 0&e^{i\alpha t}\Phi_{k}^{-H}\end{array}\right]=\left[\begin{array}{c|c}e^{i\alpha t}\Phi_{k}&-e^{i\alpha t}\beta \widehat{\Gamma}_{k}^{2k-1}\\\hline 0&(e^{i\alpha t}\Phi_{k})^{-H}\end{array}\right].
\end{align*}
\end{proof}

\begin{Lemma}\label{lem4.5}
Let $A$ denote the Hamiltonian matrix $\left[\begin{array}{c|c}B&D\\\hline G&-B^H\end{array}\right]$, where
\begin{align*}
&B=\left[\begin{array}{ccc}N_{m} (i \gamma)& 0 & -\frac{\sqrt{2}}{2} e_{m}\\0 & N_{n}(i \delta)  & -\frac{\sqrt{2}}{2} e_{n} \\0 & 0 & \frac{i}{2} (\gamma+\delta)\end{array}\right],\ \ G=\beta \left[\begin{array}{ccc}0& 0 &0  \\0 & 0  & 0 \\ 0&  0& -\frac{1}{2} (\gamma-\delta)\end{array}\right],\\ &D=\frac{\sqrt{2}}{2}i\beta \left[\begin{array}{ccc}0& 0 &  e_{m}\\0 & 0  & - e_{n} \\-e_{m}^H & e_{n}^H & -i\frac{\sqrt{2}}{2} (\gamma-\delta)\end{array}\right],
\end{align*}
$\beta\in\{-1,1\}$ and $\gamma,\ \delta\in \mathbb{R}$. Then for each $t\in \mathbb{R}$, $e^{At}$ has the form
\begin{align}\label{eq4.11}
e^{At}&=\left[\begin{array}{c|c}\mathbf{ B}&\mathbf{ D}\\\hline \mathbf{ G}&\mathbf{ E}\end{array}\right]\equiv\left[\begin{array}{c|c}\mathbf{ B}(t)&\mathbf{ D}(t)\\\hline \mathbf{ G}(t)&\mathbf{ E}(t)\end{array}\right]\nonumber\\&=\left[\begin{array}{l|l}
\left[\begin{array}{cc}\Phi_{m,n}& \phi_{m,n}^{1}  \\0&\omega_{11}\end{array}\right]&
\left[\begin{array}{cc}\widehat{\Gamma}_{m+1,n+1}^{2m,2n}& \phi_{m,n}^{2}  \\\widehat{\psi}_{m,n}^{1^H} &\omega_{12}\end{array}\right]\\ \hline
\left[\begin{array}{cc}0& 0 \\0 &\omega_{21}\end{array}\right]&\left[\begin{array}{cc}\widehat{\Phi}_{m,n}& 0 \\\widehat{\psi}_{m,n}^{2^H}&\omega_{22}\end{array}\right]
\end{array}\right],
\end{align}
where
\begin{align}\label{eq4.12}
\begin{array}{l}
\Phi_{m,n}\equiv \Phi_{m,n}(t)=e^{i\gamma t}\Phi_{m}(t)\oplus e^{i\delta t}\Phi_{n}(t),\\
\widehat{\Phi}_{m,n}\equiv\widehat{\Phi}_{m,n}(t) =e^{i\gamma t}\Phi_{m}^{-H}(t)\oplus e^{i\delta t}\Phi_{n}^{-H}(t)\\ \hspace{0.9cm} :=e^{i\gamma t}P_m^{-1}\Phi_{m}(t)P_m\oplus e^{i\delta t}P_n^{-1}\Phi_{n}(t)P_n,\\
\phi_{m,n}^1\equiv \phi_{m,n}^1(t)=-\frac{\sqrt{2}}{2}\left[\begin{array}{c}e^{i\gamma t}\phi_{m}(t)\\e^{i\delta t}\phi_{n}(t)\end{array}\right],\\ \phi_{m,n}^2\equiv \phi_{m,n}^2(t)=\frac{\sqrt{2}}{2}i\beta\left[\begin{array}{c}e^{i\gamma t}\phi_{m}(t)\\-e^{i\delta t}\phi_{n}(t)\end{array}\right],\\
\widehat{\psi}_{m,n}^{1^H}\equiv\widehat{\psi}_{m,n}^{1^H}(t)=\frac{\sqrt{2}}{2}i\beta\left[e^{i\gamma t}\widehat{\psi}_{m}^H(t), -e^{i\delta t}\widehat{\psi}_{n}^H(t)\right]\\\hspace{0.9cm}:=\frac{\sqrt{2}}{2}i\beta\left[e^{i\gamma t}\psi_{m}^H(t)P_m, -e^{i\delta t}\psi_{n}^H(t)P_n\right],\\
\widehat{\psi}_{m,n}^{2^H}\equiv\widehat{\psi}_{m,n}^{2^H}(t)=-\frac{\sqrt{2}}{2}\left[e^{i\gamma t}\widehat{\psi}_{m}^H(t), e^{i\delta t}\widehat{\psi}_{n}^H(t)\right]\\ \hspace{0.9cm}:=-\frac{\sqrt{2}}{2}\left[e^{i\gamma t}\psi_{m}^H(t)P_m, e^{i\delta t}\psi_{n}^H(t)P_n\right],\\
\widehat{\Gamma}_{m+1,n+1}^{2m,2n}\equiv \widehat{\Gamma}_{m+1,n+1}^{2m,2n}(t)=i\beta\left(-e^{i\gamma t}\widehat{\Gamma}_{m+1}^{2m}(t)\oplus e^{i\delta t}\widehat{\Gamma}_{n+1}^{2n}(t)\right)\\ \hspace{1.6cm}:=i\beta\left(-e^{i\gamma t}\Gamma_{m+1}^{2m}(t)P_m\oplus e^{i\delta t}\Gamma_{n+1}^{2n}(t)P_n\right),\\
\left[\begin{array}{cc}\omega_{11}& \omega_{12} \\\omega_{21} &\omega_{22}\end{array}\right]\equiv \left[\begin{array}{cc}\omega_{11}(t)& \omega_{12}(t) \\\omega_{21}(t) &\omega_{22}(t)\end{array}\right]=\frac{1}{2}\left[\begin{array}{cc}e^{i\gamma t}+e^{i\delta t}&- i\beta(e^{i\gamma t}-e^{i\delta t}) \\i\beta(e^{i\gamma t}-e^{i\delta t})&e^{i\gamma t}+e^{i\delta t}\end{array}\right].
\end{array}
\end{align}
\end{Lemma}

\begin{proof}
Let
\begin{align*}
\mathbf{N}_{2m+1}(i\gamma)=\left[\begin{array}{cc}N_{m+1}(i\gamma) & i\beta e_{m+1}e_m^H \\0 & -N_{m}(i\gamma)^H\end{array}\right],\
\mathbf{N}_{2n+1}(i\delta)=\left[\begin{array}{cc}N_{n+1}(i\delta) & -i\beta e_{n+1}e_n^H \\0 & -N_{n}(i\delta)^H\end{array}\right],
\end{align*}
and $\Theta=\Theta_1\Theta_2$, where
\begin{align*}
\begin{array}{l}
\Theta_1=\left[\begin{array}{c|c}I_{m+1}\oplus (-i\beta P_{m}) & 0 \\\hline0 & I_{n+1}\oplus i\beta P_{n} \end{array}\right],\\\Theta_2=\left[\begin{array}{ccc|ccc}I_m & 0 & 0 & 0 & 0 & 0 \\0 & 0 & -\frac{\sqrt{2}}{2} & 0 & 0 & \frac{\sqrt{2}}{2}i\beta \\0 & 0 & 0 & I_m & 0 & 0 \\\hline 0 & I_n & 0 & 0 & 0 & 0 \\0 & 0 & -\frac{\sqrt{2}}{2} & 0 & 0 & -\frac{\sqrt{2}}{2}i\beta  \\0 & 0 & 0 & 0 & I_n & 0\end{array}\right]
\end{array}
\end{align*}
are unitary matrices.
Then we have
\begin{align*}
\Theta^{-1}\left[\begin{array}{c|c}N_{2m+1}(i\gamma) & 0 \\\hline0 & N_{2n+1}(i\delta)\end{array}\right]\Theta&=\Theta_2^{-1}\left[\begin{array}{c|c}\mathbf{N}_{2m+1}(i\gamma)& 0\\\hline 0&\mathbf{N}_{2n+1}(i\delta)\end{array}\right]\Theta_2\\&=\left[\begin{array}{c|c}B&D\\\hline G&-B^H\end{array}\right]=A.
\end{align*}
Since
\begin{align*}
e^{N_{2m+1}(i\gamma)t}=e^{i\gamma t}\left[\begin{array}{ccc}\Phi_{m}  &\phi_m& \Gamma_{m+1}^{2m} \\0 &1& \psi_m^H\\0&0&\Phi_{m} \end{array}\right], \ \
e^{N_{2n+1}(i\delta)t}=e^{i\delta t}\left[\begin{array}{ccc}\Phi_{n}  &\phi_n& \Gamma_{n+1}^{2n} \\0 &1& \psi_n^H\\0&0&\Phi_{n} \end{array}\right],
\end{align*}
we have
\begin{align*}
\Psi_{2m+1}&:=[I_{m+1}\oplus (-i\beta P_{m})]^{-1}e^{N_{2m+1}(i\gamma)t}[I_{m+1}\oplus (-i\beta P_{m})]\\
 &=e^{i\gamma t}\left[\begin{array}{ccc}\Phi_{m}  &\phi_m& -i\beta \widehat{\Gamma}_{m+1}^{2m} \\0 &1& -i\beta\widehat{\psi}_m^H\\0&0&\Phi_{m}^{-H} \end{array}\right], \\
\Psi_{2n+1}&:=[I_{n+1}\oplus (i\beta P_{n})]^{-1}e^{N_{2n+1}(i\delta)t}[I_{n+1}\oplus (i\beta P_{n})] \\&=e^{i\delta t}\left[\begin{array}{ccc}\Phi_{n}  &\phi_n& i\beta \widehat{\Gamma}_{n+1}^{2n} \\0 &1& i\beta\widehat{\psi}_n^H\\0&0&\Phi_{n}^{-H} \end{array}\right],
\end{align*}
where $\widehat{\Gamma}_{j+1}^{2j}=\Gamma_{j+1}^{2j}P_j$ and $\widehat{\psi}_j^H=\psi_j^HP_j$ for $j=m,n$. Hence, we obtain
\begin{align*}
e^{At}&=\Theta^{-1}e^{N_{2m+1}(i\gamma)\oplus N_{2n+1}(i\delta)t}\Theta=\Theta_2^{-1}\left[\begin{array}{c|c}\Psi_{2m+1} & 0 \\\hline0 & \Psi_{2n+1}\end{array}\right]\Theta_2\\&=\left[\begin{array}{ccc|ccc}e^{i\gamma t}\Phi_m&0&-\frac{\sqrt{2}}{2}e^{i\gamma t}\phi_m&-i\beta e^{i\gamma t}\widehat{\Gamma}_{m+1}^{2m}&0& \frac{\sqrt{2}}{2}i \beta e^{i\gamma t}\phi_m\\ 0&e^{i\delta t}\Phi_n&-\frac{\sqrt{2}}{2}e^{i\delta t}\phi_n&0& i\beta e^{i\delta t}\widehat{\Gamma}_{n+1}^{2n}& -\frac{\sqrt{2}}{2}i \beta e^{i\delta t}\phi_n\\0&0&\frac{1}{2}(e^{i\gamma t}+e^{i\delta t})&\frac{\sqrt{2}}{2}i\beta e^{i\gamma t}\widehat{\psi}_m^H&-\frac{\sqrt{2}}{2}i\beta e^{i\delta t}\widehat{\psi}_n^H&-\frac{1}{2}i\beta(e^{i\gamma t}-e^{i\delta t}) \\\hline 0&0&0&e^{i\gamma t}\Phi_m^{-H}&0&0 \\ 0&0&0&0&e^{i\delta t}\Phi_n^{-H}&0 \\  0&0&\frac{1}{2}i\beta(e^{i\gamma t}-e^{i\delta t})&-\frac{\sqrt{2}}{2} e^{i\gamma t}\widehat{\psi}_m^H&-\frac{\sqrt{2}}{2} e^{i\delta t}\widehat{\psi}_n^H&\frac{1}{2}(e^{i\gamma t}+e^{i\delta t}) \end{array}\right]\\
&=\left[\begin{array}{c|c}\mathbf{ B}&\mathbf{ D}\\\hline \mathbf{ G}&\mathbf{ E}\end{array}\right],
\end{align*}
where $\mathbf{ B}$, $\mathbf{ D}$, $\mathbf{ G}$ and $\mathbf{ E}$ are given in \eqref{eq4.11}.
\end{proof}

\begin{Lemma}\label{lem4.6}
It holds that
$\phi_k^{H}\Phi_{k}^{-H}+\psi_k^{H}P_k=0$, where $\Phi_k$, $\phi_k$, $\psi_k$ and $P_k$ are defined in \eqref{eq4.10}.
\end{Lemma}

\begin{proof}
Using definitions of $\Phi_k$, $\phi_k$, $\psi_k$, $P_k$  and $\widehat{\Phi}_{k}$ in \eqref{eq4.10}  yield
\begin{align*}
\Phi_{k+1}&=\left[\begin{array}{cc}\Phi_{k}& \phi_k  \\ 0&1\end{array}\right]=\left[\begin{array}{cc}1& \psi_k^H  \\ 0&\Phi_{k}\end{array}\right],\\
\widehat{\Phi}_{k+1}&=P_{k+1}^{-1}\Phi_{k+1}P_{k+1}=\left[\begin{array}{c|c}0& -P_k^{-1}  \\\hline -1&0\end{array}\right]\left[\begin{array}{c|c}1& \psi_k^H  \\\hline 0&\Phi_{k}\end{array}\right]\left[\begin{array}{c|c}0& -1  \\\hline -P_k&0\end{array}\right]\\
&=\left[\begin{array}{cc}\widehat{\Phi}_{k}&0  \\ \psi_k^HP_k&1\end{array}\right].
\end{align*}
From Lemma \ref{lem4.3}, we have $\Phi_{k+1}^H\widehat{\Phi}_{k+1}=I_{k+1}$ and $\widehat{\Phi}_k=\Phi_k^{-H}$. Hence, it holds that $\phi_k^{H}\Phi_{k}^{-H}+\psi_k^{H}P_k=0$.
\end{proof}

In Lemma \ref{lem4.5}, if $\gamma=\delta=: \eta\in \mathbb{R}$, we have the corollary.

\begin{Corollary}\label{cor4.7}
Let $A=\left[\begin{array}{c|c}B&D\\\hline 0&-B^H\end{array}\right]$, where
\begin{align*}
B=\left[\begin{array}{ccc}N_{m} (i \eta)& 0 & -\frac{\sqrt{2}}{2} e_{m}\\0 & N_{n}(i \eta)  & -\frac{\sqrt{2}}{2} e_{n} \\0 & 0 & i\eta\end{array}\right],\ \ D=\frac{\sqrt{2}}{2}i\beta \left[\begin{array}{ccc}0& 0 &  e_{m}\\0 & 0  & - e_{n} \\-e_{m}^H & e_{n}^H & 0\end{array}\right],
\end{align*}
$\beta\in\{-1,1\}$ and $\eta\in \mathbb{R}$. Then for each $t\in \mathbb{R}$, $e^{At}$ has the form
\begin{align*}
e^{At}=\left[\begin{array}{c|c}\mathbf{ B}&\mathbf{ D}\\\hline 0&\mathbf{ B}^{-H}\end{array}\right],
\end{align*}
where
\begin{align*}
\begin{array}{l}
\mathbf{ B}\equiv\mathbf{ B}(t)=\left[\begin{array}{cc}\Phi_{m,n}& \phi_{m,n}^{1}  \\0&e^{i\eta t}\end{array}\right],\ \
\mathbf{ D}\equiv\mathbf{ D}(t)=\left[\begin{array}{cc}\widehat{\Gamma}_{m+1,n+1}^{2m,2n}& \phi_{m,n}^{2}  \\\widehat{\psi}_{m,n}^{1^H} &0\end{array}\right],
\end{array}
\end{align*}
and $\Phi_{m,n}$, $\phi_{m,n}^{1}$, $\phi_{m,n}^{2}$, $\widehat{\Gamma}_{m+1,n+1}^{2m,2n}$ and $\widehat{\psi}_{m,n}^{1^H}$ are defined in \eqref{eq4.12} in which $\gamma=\delta$ is  replaced  by  $\eta$.
\end{Corollary}

\begin{proof}
From \eqref{eq4.12}, if  $\eta:=\gamma=\delta \in \mathbb{R}$, then $\omega_{11}=\omega_{22}=e^{i\eta t}$, $\omega_{12}=\omega_{21}=0$. Then the matrix $\mathbf{ G}$ in \eqref{eq4.11} is a zero matrix. Now, we show that $\mathbf{ B}^{-H}=\mathbf{ E}$, where $\mathbf{ B}$ and $\mathbf{ E}$ are defined in \eqref{eq4.11}. Using definitions of $\Phi_{m,n}$ and $\widehat{\Phi}_{m,n}$ in \eqref{eq4.12} and Lemma \ref{lem4.3} $(ii)$ yield $\widehat{\Phi}_{m,n}=\Phi_{m,n}^{-H}$. From \eqref{eq4.12} and Lemma \ref{lem4.6}, it is easily seen that $\phi_{m,n}^{1^H}\widehat{\Phi}_{m,n}+e^{-i\eta t}\widehat{\psi}_{m,n}^{2^H}=0$. Hence, we have
\begin{align*}
\mathbf{ B}^H\mathbf{ E}=\left[\begin{array}{cc}\Phi_{m,n}^H& 0  \\\phi_{m,n}^{1^H}&e^{-i\eta t}\end{array}\right]\left[\begin{array}{cc}\widehat{\Phi}_{m,n}& 0  \\\widehat{\psi}_{m,n}^{2^H}&e^{i\eta t}\end{array}\right]=I,
\end{align*}
i.e.,  $\mathbf{ B}^{-H}=\mathbf{ E}$. From Lemma \ref{lem4.5}, we complete the proof.
\end{proof}

Combining the previous lemmas and corollary, in Theorem \ref{thm4.8}, we can arrive at the structure of $e^{\mathfrak{J}t}$ as in the form of Theorem~\ref{thm4.2}, where $\mathfrak{J}$ in \eqref{eq4.8} is a Hamiltonian Jordan canonical form.

\begin{Theorem}\label{thm4.8}
Given a Hamiltonian Jordan canonical form  $\mathfrak{J}$ as in \eqref{eq4.8}, then
\begin{align}\label{eq4.13}
e^{\mathfrak{J}t}=\left[\begin{array}{cccc|cccc}\mathcal{R}_r &  &  &  & 0 &  &  &  \\ & \mathcal{ R}_e &  &  &  & \mathcal{D}_e &  &  \\ &  & \mathcal{R}_c &  &  &  & \mathcal{D}_c &  \\ &  &  & \mathcal{R}_d &  &  &  & \mathcal{D}_d \\\hline0 &  &  &  & \mathcal{R}_r^{-H} &  &  &  \\ & 0 &  &  &  & \mathcal{R}_e^{-H} &  &  \\ &  & 0 &  &  &  & \mathcal{R}_c^{-H} &  \\ &  &  & \mathcal{G}_d &  &  &  & \mathcal{E}_d\end{array}\right],
\end{align}
where the different blocks, $\mathcal{R}_r\equiv\mathcal{R}_r(t)$, $\mathcal{R}_x\equiv \mathcal{R}_x(t)$, $\mathcal{D}_x\equiv \mathcal{D}_x(t)$ for $x=e,c,d$, $\mathcal{G}_d\equiv\mathcal{G}_d(t)$ and $\mathcal{E}_d\equiv\mathcal{E}_d(t)$ are dependent of $t$ and have the following structures.
\begin{itemize}
\item[1.] The blocks with index $r$ have the form
\begin{align*}
\mathcal{R}_r ={\rm diag}(\mathcal{R}_1^r,\ldots, \mathcal{R}_{\mu_r}^r),\ \  \mathcal{R}_k^r=e^{\lambda_kt}{\rm diag}(\Phi_{d_{k,1}},\ldots,\Phi_{d_{k,p_k}}),\ \ k=1,\ldots,\mu_r,
\end{align*}
where $\lambda_k\in \mathbb{C}_{>}$ are distinct and $\Phi_{d_{k,j}}$, $j=1,\ldots,p_k$, are defined in \eqref{eq4.10}.
\item[2.] The blocks with index $e$  have the form (see Lemma \ref{lem4.4})
\begin{align*}
\begin{array}{ll}
\mathcal{R}_e ={\rm diag}(\mathcal{R}_1^e,\ldots, \mathcal{R}_{\mu_e}^e),& \mathcal{R}_k^e=e^{i \alpha_kt}{\rm diag}(\Phi_{l_{k,1}},\ldots,\Phi_{l_{k,q_k}}),\\
\mathcal{D}_e ={\rm diag}(\mathcal{D}_1^e,\ldots, \mathcal{D}_{\mu_e}^e),& \mathcal{D}_k^e=-e^{i \alpha_kt}{\rm diag}(\beta_{k,1}^e\widehat{\Gamma}_{l_{k,1}}^{2l_{k,1}-1},\ldots,\beta_{k,q_k}^e\widehat{\Gamma}_{l_{k,q_k}}^{2l_{k,q_k}-1}),
\end{array}
\end{align*}
where for $k=1,\ldots,\mu_e$,  $j=1,\ldots,q_k$, $ \alpha_k\in \mathbb{R}$ are distinct,   $\Phi_{l_{k,j}}$, $\Gamma_q^{2{l_{k,j}}-1}$, $\widehat{\Gamma}_{l_{k,j}}^{2{l_{k,j}}-1}$, $P_{l_{k,j}}$ are defined in \eqref{eq4.10} and $\beta^e_{k,j}\in \{-1,1\}$.
\item[3.] The blocks with index $c$  have the form  (see Corollary \ref{cor4.7})
\begin{align*}
\begin{array}{ll}
\mathcal{R}_c ={\rm diag}(\mathcal{R}_1^c,\ldots, \mathcal{R}_{\mu_c}^c),& \mathcal{R}_k^c={\rm diag}(\mathbf{ B}_{k,1},\ldots,\mathbf{ B}_{k,r_k}),\\
\mathcal{D}_c ={\rm diag}(\mathcal{D}_1^c,\ldots, \mathcal{D}_{\mu_c}^c),& \mathcal{D}_k^c={\rm diag}(\mathbf{ D}_{k,1},\ldots,\mathbf{ D}_{k,r_k}),
\end{array}
\end{align*}
where for $k=1,\ldots,\mu_c$, $j=1,\ldots,r_k$,
\begin{align*}
\begin{array}{l}
\mathbf{ B}_{k,j}=\left[\begin{array}{cc}\Phi_{m_{k,j},n_{k,j}}& \phi_{m_{k,j},n_{k,j}}^{1}  \\0&e^{i\eta_k t}\end{array}\right],\ \
\mathbf{ D}_{k,j}=\left[\begin{array}{cc}\widehat{\Gamma}_{m_{k,j}+1,n_{k,j}+1}^{2m_{k,j},2n_{k,j}}& \phi_{m_{k,j},n_{k,j}}^{2}  \\\widehat{\psi}_{m_{k,j},n_{k,j}}^{1^H} &0\end{array}\right],
\end{array}
\end{align*}
with $ \eta_k\in \mathbb{R}$  distinct,  $\Phi_{m_{k,j},n_{k,j}}$, $ \phi_{m_{k,j},n_{k,j}}^{1}$,  $\phi_{m_{k,j},n_{k,j}}^{2}$, $\widehat{\Gamma}_{m_{k,j}+1,n_{k,j}+1}^{2m_{k,j},2n_{k,j}}$ and $\widehat{\psi}_{m_{k,j},n_{k,j}}^{1^H}$ being defined in \eqref{eq4.12}, in which $\gamma$ and $\delta$ are  replaced by $\eta_k$, and $\beta$ is replaced by $\beta^c_{k,j}\in \{-1,1\}$.
\item[4.] The blocks with index $d$  have the form  (see Lemma \ref{lem4.5})
\begin{align*}
\begin{array}{ll}
\mathcal{R}_d ={\rm diag}(\mathbf{B}_1^d,\ldots, \mathbf{B}_{\mu_d}^d),&\mathcal{D}_d ={\rm diag}(\mathbf{ D}_1^d,\ldots, \mathbf{ D}_{\mu_d}^d),\\  \mathcal{G}_d ={\rm diag}(\mathbf{ G}_1^d,\ldots, \mathbf{ G}_{\mu_d}^d),&\mathcal{E}_d ={\rm diag}(\mathbf{ E}_1^d,\ldots, \mathbf{ E}_{\mu_d}^d),
\end{array}
\end{align*}
where for $k=1,\ldots,\mu_d$,
\begin{align*}
\begin{array}{ll}
\mathbf{ B}_k^d=\left[\begin{array}{cc}\Phi_{s_k,t_k}& \phi_{s_k,t_k}^{1}  \\0&\omega_{11}\end{array}\right],&
\mathbf{ D}_k^d=\left[\begin{array}{cc}\widehat{\Gamma}_{s_k+1,t_k+1}^{2s_k,2t_k}& \phi_{s_k,t_k}^{2}  \\\widehat{\psi}_{s_k,t_k}^{1^H} &\omega_{12}\end{array}\right],\\
\mathbf{ G}_k^d=\left[\begin{array}{cc}0& 0 \\0 &\omega_{21}\end{array}\right],&\mathbf{ E}_k^d=\left[\begin{array}{cc}\widehat{\Phi}_{s_k,t_k}& 0 \\\widehat{\psi}_{s_k,t_k}^{2^H}&\omega_{22}\end{array}\right],
\end{array}
\end{align*}
with $\Phi_{s_k,t_k}$, $\widehat{\Gamma}_{s_k+1,t_k+1}^{2s_k,2t_k}$, $\phi_{s_k,t_k}^{1}$, $\phi_{s_k,t_k}^{2}$, $\widehat{\psi}_{s_k,t_k}^{1^H}$, $\widehat{\psi}_{s_k,t_k}^{2^H}$, $\omega_{11}$, $\omega_{12}$, $\omega_{21}$ and $\omega_{22}$ being defined in \eqref{eq4.12}, in which  $\gamma$ and $\delta$ are  replaced by $\gamma_k$ and $\delta_k$, respectively, and $\beta$ is replaced by $\beta^d_{k}\in \{-1,1\}$. Note that in this case, $\gamma_k\neq \delta_k$.
\end{itemize}
\end{Theorem}

\subsection{Asymptotic Analysis of RDE with Elementary Hamiltonian Jordan Blocks}\label{sec4.2}

It follows from the Radon's Lemma that the extended solution $W(t)$ of a RDE \eqref{eq4.3} can be obtained by taking $W(t)=P(t)Q(t)^{-1}$ for $t\in \mathcal{T}_W$, where $[Q(t)^{\top},  P(t)^{\top}]^{\top}$ is the solution of IVP \eqref{eq4.1} and $\mathcal{T}_W$ is defined in \eqref{eq3.41}. Suppose that the Hamiltonian matrix $\mathscr{H}$  in IVP \eqref{eq4.1} is symplectically similar to a Hamiltonian Jordan canonical form $\mathfrak{J}$ as in \eqref{eq4.8}. Therefore, the solution $Y(t)$ described in \eqref{eq4.9} involves the matrix $e^{\mathfrak{J}t}$. However the structure of $e^{\mathfrak{J}t}$ in \eqref{eq4.13} is complicated, in this subsection, we first consider four elementary cases. The general cases will be discussed in Subsection~\ref{sec4.3}. We assume that the Hamiltonian Jordan canonical form $\mathfrak{J}$ in \eqref{eq4.8}  is one of  the following four elementary cases
\begin{align}\label{eq4.14}
\mathfrak{J}_x\equiv \mathfrak{J}=\left[\begin{array}{c|c}R_x & D_x \\\hline G_x & -R_x^H\end{array}\right]\in \mathbb{C}^{2n\times 2n},\ \ x=r,e,c,d,
\end{align}
where
\begin{itemize}
\item[1.] if $x=r$, then $R_r=N_{n}(\lambda)$, $D_r=G_r=0$ and $\lambda\in \mathbb{C}_{>}$.
\item[2.] if $x=e$, then $R_e=N_{n}(i\alpha)$, $D_e=\beta e_{n}e_{n}^H$, $G_e=0$ and $\alpha\in \mathbb{R}$, $\beta\in\{-1,1\}$.
\item[3.] if $x=c$, then $n=n_1+n_2+1$, $ \eta\in \mathbb{R}$, $\beta\in \{-1,1\}$, $G_c=0$ and
\begin{align*}
R_c=\left[\begin{array}{ccc}N_{n_1} (i \eta)& 0 & -\frac{\sqrt{2}}{2} e_{n_1}\\0 & N_{n_2}(i \eta)  & -\frac{\sqrt{2}}{2} e_{n_2} \\0 & 0 & i \eta\end{array}\right],\ \
D_c=\frac{\sqrt{2}}{2}i\beta\left[\begin{array}{ccc}0& 0 &  e_{n_1}\\0 & 0  & - e_{n_2} \\-e_{n_1}^H & e_{n_2}^H & 0\end{array}\right].
\end{align*}
\item[4.] if $x=d$, then $n=n_1+n_2+1$,  $\gamma,\delta\in \mathbb{R}$ with $\gamma\neq\delta$, $\beta\in \{-1,1\}$ and
\begin{align*}
\begin{array}{l}
R_d=\left[\begin{array}{ccc}N_{n_1} (i \gamma)& 0 & -\frac{\sqrt{2}}{2} e_{n_1}\\0 & N_{n_2}(i \delta)  & -\frac{\sqrt{2}}{2} e_{n_2} \\0 & 0 & \frac{i}{2} (\gamma+\delta)\end{array}\right],\ \
G_d=\beta\left[\begin{array}{ccc}0& 0 & 0\\0 & 0  & 0 \\0 & 0 & - \frac{1}{2} (\gamma-\delta)\end{array}\right]\\
D_d=\frac{\sqrt{2}}{2}i\beta\left[\begin{array}{ccc}0& 0 &  e_{n_1}\\0 & 0  & - e_{n_2} \\-e_{n_1}^H & e_{n_2}^H & -i \frac{\sqrt{2}}{2} (\gamma-\delta)\end{array}\right].
\end{array}
\end{align*}
\end{itemize}

The asymptotic analysis of $W(t)$ and $Q(t)^{-1}$ is given as follows whenever $\mathscr{H}$ is one of these four cases.

\begin{Theorem}\label{thm4.9}
Suppose that $\mathscr{H}$ in \eqref{eq4.1} has one of Hamiltonian Jordan canonical forms $\mathfrak{J}_x$ as in \eqref{eq4.14}. Let $Y(t)=[Q(t)^{\top},  P(t)^{\top}]^{\top}$ and $W(t)=P(t)Q(t)^{-1}$ for $t\in \mathcal{T}_W$ be the solution of IVP \eqref{eq4.1} and the extended solution of RDE \eqref{eq4.3}, respectively.
Note that $[W_1^{\top},W_2^{\top}]^{\top}=\mathcal{S}^{-1}[I,W_0]^{\top}$ by \eqref{eq4.9}.
\begin{itemize}
\item[(i)] Suppose that the symplectic matrix $\mathcal{S}$ in \eqref{eq4.8} is partitioned as
\begin{align}\label{eq4.15}
\mathcal{S}=\left[\begin{array}{c|c}U_1 &V_1\\\hline U_2 & V_2\end{array}\right],
\end{align}
 where $U_1,U_2,V_1,V_1\in \mathbb{C}^{n\times n}$.
\begin{itemize}
\item[1.] If $x=r$, $\Re(\lambda)>0$ and $U_1$, $W_1$ are invertible, then
\begin{align*}
W(t)=U_2U_1^{-1}+O(e^{-2\Re(\lambda) t}t^{2(n-1)})\ \text{ and }\   Q(t)^{-1}=O(e^{-\Re(\lambda) t}t^{n-1}),
\end{align*}
as $t\to\infty$. On the other hand, if  $V_1$ and $W_2$ are invertible, then
\begin{align*}
W(t)=V_2V_1^{-1}+O(e^{-2\Re(\lambda) |t|}|t|^{2(n-1)})\ \text{ and }\ Q(t)^{-1}=O(e^{-\Re(\lambda) |t|}|t|^{n-1}),
\end{align*}
as $t\to-\infty$.
\item[2.] If  $x=e$ and $U_1$, $W_2$ are invertible, then
\begin{align*}
W(t)= U_2U_1^{-1}+O(t^{-1})\ \text{ and }\  Q(t)^{-1}=O(t^{-1}),
\end{align*}
as $t\rightarrow \pm \infty$.
\end{itemize}
\item[(ii)] Suppose that the symplectic matrix $\mathcal{S}$  in \eqref{eq4.8} is further partitioned as
\begin{align}\label{eq4.16}
\mathcal{S}=\left[\begin{array}{cc|cc}U_1 &u_1 &V_1&v_1\\\hline U_2&u_2 & V_2&v_2\end{array}\right]\in \mathbb{C}^{2n\times 2n},
\end{align}
where $U_1,U_2$, $V_1,V_2\in \mathbb{C}^{n\times (n_1+n_2)}$, $u_1,u_2,v_1,v_2\in \mathbb{C}^{n}$ and $n=n_1+n_2+1$.
\begin{itemize}
\item[3.]  If $x=c$ and $W_2$ is invertible, then there exist constants  $\tilde{f}_u$, $\tilde{f}_v\in{\mathbb C}$ with $\mathbf{U}_{1,0}=\left[U_1,\tilde{f}_uu_1+\tilde{f}_vv_1\right]$, $\mathbf{U}_{2,0}=\left[U_2,\tilde{f}_uu_2+\tilde{f}_vv_2\right]\in \mathbb{C}^{n\times n}$ and   a rank-one matrix $K_Q=W_2^{-1}e_ne_n^{H}\mathbf{U}_{1,0}^{-1}$ such that
\begin{align*}
W(t)=\mathbf{U}_{2,0}\mathbf{U}_{1,0}^{-1}+O(t^{-1})\ \text{ and }\
Q(t)^{-1}=e^{-i\eta t}K_Q+O(t^{-1}),
\end{align*}
as $t\rightarrow \pm\infty$, provided that $\mathbf{U}_{1,0}$ is invertible.
\item[4.] If $x=d$ and $W_2$ is invertible, then there exist constants $f_u$, $f_v$ with  $\mathbf{U}_1=[U_1, f_uu_1+f_vv_1]$, $\mathbf{U}_2=[U_2, f_uu_2+f_vv_2]\in \mathbb{C}^{n\times n}$, two rank-one matrices $K_W,\ K_Q\in{\mathbb C}^{n\times n}$ and $c\in{\mathbb C}$ with $|c|=1$ such that
\begin{align*}
&W(t)=\mathbf{U}_2\mathbf{U}_1^{-1}+\frac{e^{i\theta t}}{1+e^{i\theta t}c+O(t^{-1})}\left[K_W+O(t^{-1})\right]+O(t^{-1}),\\
&Q(t)^{-1}=\frac{e^{-i\gamma t}}{1+e^{i\theta t}c+O(t^{-1})}\left[K_Q+O(t^{-1})\right]+O(t^{-1}),
\end{align*}
as $t\rightarrow \pm\infty$, provided that $\mathbf{U}_1$ is invertible. Here $\theta=\delta-\gamma$.
\end{itemize}
\end{itemize}
\end{Theorem}

The proof of assertions $1$, $2$, $3$ and $4$ of Theorem~\ref{thm4.9} are given below in Theorems~\ref{thm4.10}, \ref{thm4.11}, \ref{thm4.15}, and \ref{thm4.13}, respectively. In assertion $4$, those two rank-one matrices $K_W$ and $K_Q$ will be explicitly expressed  in Theorem \ref{thm4.13}.
\begin{Remark}\label{rem4.1.1}
In assertion 1, we see that the extended solution $W(t)$ forms a hetroclinic orbit starting from the equilibrium $V_2V_1^{-1}$ to the equilibrium $U_2U_1^{-1}$.  In assertion 2, the equilibrium $V_2V_1^{-1}$ collapses and $W(t)$ becomes a homoclinic orbit that links $U_2U_1^{-1}$ itself. In assertion 3, $W(t)$ is also a homoclinic orbit but, $Q(t)^{-1}$ tends to a limit circle.  In assertion $4$, the extended solution $W(t)$ of the RDE  converges with the rate $O(t^{-1})$ to a periodic orbit, say $W_\infty(t)$, with period $2\pi/\theta$ whenever $\theta\neq 0$. Here
$
W_{\infty}(t)=\mathbf{U}_2\mathbf{U}_1^{-1}+\frac{e^{i\theta t}}{1+e^{i\theta t}c}K_W.
$
We shall prove in Theorem~\ref{thm4.14} that $W_\infty(t)$ blows up periodically.
\end{Remark}

\begin{Theorem}\label{thm4.10}
Suppose assumptions in Theorem~\ref{thm4.9} hold. Let $\mathfrak{J}_x=\mathfrak{J}_r$ and the symplectic matrix $\mathcal{S}$ have the form in \eqref{eq4.15}. If $U_1$ and $W_1$ are invertible, then $W(t)=U_2U_1^{-1}+O(e^{-2\Re(\lambda) t}t^{2(n-1)})$  and $Q(t)^{-1}=O(e^{-\Re(\lambda) t}t^{n-1})$ as $t\to\infty$, where $\Re(\lambda)>0$ and $U_2U_1^{-1}$ is Hermitian. On the other hand, if  $V_1$ and $W_2$ are invertible, then  $W(t)=V_2V_1^{-1}+O(e^{-2\Re(\lambda) |t|}|t|^{2(n-1)})$ and $Q(t)^{-1}=O(e^{-\Re(\lambda) |t|}|t|^{n-1})$ as $t\to-\infty$, where $V_2V_1^{-1}$  is Hermitian.
\end{Theorem}

\begin{proof}
Since $\mathfrak{J}_r=N_n(\lambda)\oplus(-N_n(\lambda)^H)$, we have
$e^{\mathfrak{J}_rt}=(e^{\lambda t}\Phi_n)\oplus (e^{-\bar{\lambda} t}\Phi_n^{-H})$, where $\lambda\in \mathbb{C}_{>}$ and $\Phi_n$ is defined in \eqref{eq4.10}.  From \eqref{eq4.9}, we have
\begin{align}\label{eq4.17}
Y(t)&\equiv \left[\begin{array}{c}Q(t) \\P(t)\end{array}\right]=\left[\begin{array}{c|c}U_1 &V_1\\\hline U_2 & V_2\end{array}\right]\left[\begin{array}{c|c}e^{\lambda t}\Phi_n &0\\\hline 0& e^{-\bar{\lambda} t}\Phi_n^{-H}\end{array}\right]\left[\begin{array}{c}W_1 \\W_2\end{array}\right]\nonumber\\&=\left[\begin{array}{c|c}U_1 &V_1\\\hline U_2 & V_2\end{array}\right]\left[\begin{array}{c}e^{\lambda t}\Phi_nW_1 \\e^{-\bar{\lambda} t}\Phi_n^{-H}W_2\end{array}\right].
\end{align}
From Lemma \ref{lem4.3} $(ii)$, we have $\|\Phi_n^{-H}\|=\|\Phi_n^{-1}\|=O(t^{n-1})$. Since $\Re(\lambda)>0$ and  $Q(t)=e^{\lambda t}(U_1\Phi_nW_1+e^{-2\Re(\lambda) t}V_1\Phi_n^{-H}W_2)$,  if $U_1$ and $W_1$ are invertible, then we have $Q(t)^{-1}=O(e^{-\Re(\lambda) t}t^{n-1})$ as  $t\to\infty$. Using the fact that $\Phi_n=e^{N_nt}$ is invertible, we obtain that for $t\in \mathcal{T}_W$,
\begin{align*}
W(t)&=P(t)Q(t)^{-1}\\&=(U_2+e^{-2\Re(\lambda) t}V_2\Phi_n^{-H}W_2W_1^{-1}\Phi_n^{-1})(U_1+e^{-2\Re(\lambda) t}V_1\Phi_n^{-H}W_2W_1^{-1}\Phi_n^{-1})^{-1}.
\end{align*}
Therefore, $W(t)=U_2U_1^{-1}+O(e^{-2\Re(\lambda) t}t^{2(n-1)})$ as  $t\to\infty$.
The matrix $U_2U_1^{-1}$ is Hermitian because $\mathcal{S}$ is symplectic.

Similarly, if  $V_1$ and $W_2$ are invertible,  it follows from \eqref{eq4.17} again that
\begin{align*}
Q(t)&=e^{-\bar{\lambda} t}(V_1\Phi_n^{-H}W_2+e^{2\Re(\lambda) t}U_1\Phi_nW_1),\\
W(t)&=(V_2+e^{2\Re(\lambda) t}U_2\Phi_nW_1W_2^{-1}\Phi_n^{H})(V_1+e^{2\Re(\lambda) t}U_2\Phi_nW_1W_2^{-1}\Phi_n^{H})^{-1}.
\end{align*}
Since $\Re(\lambda)>0$, we have $Q(t)^{-1}=O(e^{-\Re(\lambda) |t|}|t|^{n-1})$ and $W(t)=V_2V_1^{-1}+O(e^{-2\Re(\lambda) |t|}|t|^{2(n-1)})$ as $t\to-\infty$.
\end{proof}

For given integers $k$, $\ell$, $k_1$ and $k_2$ satisfying $0\leqslant k$, $0\leqslant \ell$, $k\neq \ell$ and $0<k_1<k_2\leqslant 2k_1$, we denote
\begin{align}\label{eq4.18}
\begin{array}{l}
\Xi_{k,\ell}\equiv \Xi_{k,\ell}(t)=\left\{\begin{array}{cc}{\rm diag}(t^k,t^{k+1},\cdots,t^\ell)&\text{ if }k<\ell,\\
{\rm diag}(t^k,t^{k-1},\cdots,t^\ell)&\text{ if } k>\ell,
\end{array}\right.\\
\digamma_{k_1}^{k_2}=\left[\begin{array}{cccc}\frac{1}{{k_1}!} & \frac{1}{{(k_1+1)}!} & \cdots & \frac{1}{{k_2}!} \\\frac{1}{{(k_1-1)}!} & \frac{1}{{k_1}!} & \ddots & \frac{1}{{(k_2-1)}!} \\\vdots & \ddots & \ddots & \vdots \\ \frac{1}{{(2k_1-k_2)}!} & \cdots & \cdots & \frac{1}{{k_1}!}\end{array}\right].
\end{array}
\end{align}
The matrix $\digamma_{k_1}^{k_2}$ is invertible (the detailed proof is shown in Theorem \ref{thm5.2}).
The matrix $\Gamma_{k_1}^{k_2}$ defined in \eqref{eq4.10} can be rewritten in terms of $\Xi_{k,\ell}$ and $\digamma_{k_1}^{k_2}$ as
\begin{align}\label{eq4.19}
\begin{array}{l}
\Gamma_{k_1}^{k_2}=t^{2k_1-k_2}\Xi_{k_2-k_1,0}\digamma_{k_1}^{k_2}\Xi_{0,k_2-k_1},\\
(\Gamma_{k_1}^{k_2})^{-1}=t^{-2k_1+k_2}(\Xi_{0,k_2-k_1})^{-1}(\digamma_{k_1}^{k_2})^{-1}(\Xi_{k_2-k_1,0})^{-1},
\end{array}
\end{align}
for each $t\neq 0$.
In order to investigate the asymptotic behaviors of  $W(t)$ and $Q(t)^{-1}$ when $\mathscr{H}$ is symplectically similar to one of  $\mathfrak{J}_x$ in \eqref{eq4.14} for $x=e, c$ and $d$, we need the useful Tables \ref{tab1} and \ref{tab2} (the detailed proofs of each item in Tables \ref{tab1} and \ref{tab2} are given in Lemmas \ref{lem5.4} and \ref{lem5.5}, respectively).   Note that  $\Gamma_{n}^{2n-1}$, $\Phi_n$, $P_n$, $\widehat{\Gamma}_{n}^{2n-1}$, $\widehat{\Phi}_n$ $\widehat{\Gamma}_{n_1+1,n_2+1}^{2n_1,2n_2}$, $\Phi_{n_1,n_2}$, $\widehat{\Phi}_{n_1,n_2}$, $\phi_{n_1,n_2}^{j}$ and $\widehat{\psi}_{n_1,n_2}^{j^H}$ for  $j=1,2$ are defined in \eqref{eq4.10} and  \eqref{eq4.12}.

\begin{table}[htdp]
\begin{center}\begin{tabular}{| l || l |}\hline $(\widehat{\Gamma}_{n}^{2n-1})^{-1}=O(t^{-1})$ & $(\widehat{\Gamma}_{n}^{2n-1})^{-1}\Phi_n=O(t^{-1})$ \\\hline $\widehat{\Phi}_n(\widehat{\Gamma}_{n}^{2n-1})^{-1}=O(t^{-1})$ & $\widehat{\Phi}_n(\widehat{\Gamma}_{n}^{2n-1})^{-1}\Phi_n=O(t^{-1})$ \\\hline\hline
\multicolumn{2}{|c|}{$\widehat{\Phi}_n(\Phi_nW\pm \widehat{\Gamma}_{n}^{2n-1})^{-1}=O(t^{-1})$}\\\hline
\end{tabular}
\end{center}
\caption{The asymptotic behaviors as $t\rightarrow \pm \infty$.}\label{tab1}
\end{table}
\begin{table}[htdp]
\begin{center}\begin{tabular}{| l || l |}\hline  $(\Upsilon+\widehat{\Gamma}_{n_1+1,n_2+1}^{2n_1,2n_2})^{-1}=O(t^{-2})$  & $(\Upsilon+\widehat{\Gamma}_{n_1+1,n_2+1}^{2n_1,2n_2})^{-1}\phi_{n_1,n_2}^{j}=O(t^{-1})$ \\\hline $(\Upsilon+\widehat{\Gamma}_{n_1+1,n_2+1}^{2n_1,2n_2})^{-1}\Phi_{n_1,n_2}=O(t^{-2})$& $\widehat{\psi}_{n_1,n_2}^{j^H}(\Upsilon+\widehat{\Gamma}_{n_1+1,n_2+1}^{2n_1,2n_2})^{-1}=O(t^{-1})$\\\hline $\widehat{\Phi}_{n_1,n_2}(\Upsilon+\widehat{\Gamma}_{n_1+1,n_2+1}^{2n_1,2n_2})^{-1}=O(t^{-2})$ & $\widehat{\psi}_{n_1,n_2}^{j^H}(\Upsilon+\widehat{\Gamma}_{n_1+1,n_2+1}^{2n_1,2n_2})^{-1}\Phi_{n_1,n_2}=O(t^{-1})$\\\hline $\widehat{\Phi}_{n_1,n_2}(\Upsilon+\widehat{\Gamma}_{n_1+1,n_2+1}^{2n_1,2n_2})^{-1}\Phi_{n_1,n_2}=O(t^{-2})$ & $\widehat{\Phi}_{n_1,n_2}(\Upsilon+\widehat{\Gamma}_{n_1+1,n_2+1}^{2n_1,2n_2})^{-1}\phi_{n_1,n_2}^{j}=O(t^{-1})$\\\hline\hline
\multicolumn{2}{|c|}{$\widehat{\psi}_{n_1,n_2}^{j^H}(\Upsilon+\widehat{\Gamma}_{n_1+1,n_2+1}^{2n_1,2n_2})^{-1}\phi_{n_1,n_2}^{k}=\widehat{\psi}_{n_1,n_2}^{j^H}(\widehat{\Gamma}_{n_1+1,n_2+1}^{2n_1,2n_2})^{-1}\phi_{n_1,n_2}^{k}+O(t^{-1})$}\\\hline
\end{tabular}
\end{center}
\caption{The asymptotic behaviors as $t\rightarrow \pm \infty$, where $j,k\in \{1, 2\}$, $\Upsilon=\Phi_{n_1,n_2}W+\phi^1_{n_1,n_2}w^H$ and $W$ and $w$ are arbitrary constant matrix and vector, respectively.}\label{tab2}
\end{table}

In Theorem \ref{thm4.11}, we analyze the asymptotic behaviors of  $W(t)$ and $Q(t)^{-1}$ when  $\mathscr{H}$ is symplectically similar to $\mathfrak{J}_e$.  This proves assertion 2 in Theorem~\ref{thm4.9}.

\begin{Theorem}\label{thm4.11}
Suppose assumptions in Theorem~\ref{thm4.9} hold. Let $\mathfrak{J}_x=\mathfrak{J}_e$ and the symplectic matrix $\mathcal{S}$ have the form in \eqref{eq4.15}. If $U_1$ and $W_2$ are invertible, then $W(t)= U_2U_1^{-1}+O(t^{-1})$ and $Q(t)^{-1}=O(t^{-1})$ as $t\rightarrow \pm \infty$. Here $U_2U_1^{-1}$ is Hermitian.
\end{Theorem}

\begin{proof}
Using the structure of $\mathfrak{J}_e$ in \eqref{eq4.14} and Lemma \ref{lem4.4}, it follows from \eqref{eq4.9} that
\begin{align*}
Y(t)&\equiv \left[\begin{array}{c}Q(t) \\P(t)\end{array}\right]=\left[\begin{array}{c|c}U_1 &V_1\\\hline U_2 & V_2\end{array}\right]
\left[\begin{array}{c|c}e^{i\alpha t}\Phi_{n}&-e^{i\alpha t}\beta \widehat{\Gamma}_{n}^{2n-1}\\\hline 0&(e^{i\alpha t}\Phi_{n})^{-H}\end{array}\right]\left[\begin{array}{c}W_1 \\W_2\end{array}\right]\\
&=e^{i\alpha t}\left[\begin{array}{c|c}U_1 &V_1\\\hline U_2 & V_2\end{array}\right]\left[\begin{array}{c}\Phi_nW_1-\beta \widehat{\Gamma}_{n}^{2n-1}W_2 \\\Phi^{-H}_nW_2\end{array}\right],\ \ \beta\in\{-1,1\}.
\end{align*}
Since $W_2$ is invertible and $\Phi^{-H}_n=\widehat{\Phi}_n$ (see Lemma \ref{lem4.3}), using Table \ref{tab1} we see that
\begin{align*}
\left[\begin{array}{c}Q(t) \\P(t)\end{array}\right]W_2^{-1}&(\Phi_nW_1W_2^{-1}-\beta \widehat{\Gamma}_{n}^{2n-1})^{-1}
=e^{i\alpha t}\left[\begin{array}{c|c}U_1 &V_1\\\hline U_2 & V_2\end{array}\right]\left[\begin{array}{c}I \\O(t^{-1})\end{array}\right]\\&=e^{i\alpha t}\left[\begin{array}{c}U_1+O(t^{-1}) \\U_2+O(t^{-1})\end{array}\right], \ \ \text{ as } t\rightarrow \pm \infty.
\end{align*}
From Table \ref{tab1}, we have $$W_2^{-1}(\Phi_nW_1W_2^{-1}-\beta \widehat{\Gamma}_{n}^{2n-1})^{-1}=W_2^{-1}(O(t^{-1})-\beta I)^{-1}(\widehat{\Gamma}_{n}^{2n-1})^{-1}=O(t^{-1}).$$
Then it holds that
\begin{align*}
Q(t)^{-1}=e^{-i\alpha t}W_2^{-1}(\Phi_nW_1W_2^{-1}-\beta \widehat{\Gamma}_{n}^{2n-1})^{-1}(U_1+O(t^{-1}))^{-1}=O(t^{-1}),
\end{align*}
as $t\rightarrow \pm \infty$. Consequently, we obtain that
\begin{align*}
W(t)=(U_2+O(t^{-1}))(U_1+O(t^{-1}))^{-1}=U_2U_1^{-1}+O(t^{-1}), \ \ \text{ as } t\rightarrow \pm \infty.
\end{align*}
Since $\mathcal{S}$ in \eqref{eq4.15} is symplectic, this implies that $U_2U_1^{-1}$ is Hermitian.
\end{proof}

Now, we consider the case that the Hamiltonian matrix $\mathscr{H}$ has a Hamiltonian Jordan canonical form $\mathfrak{J}_d$ or $\mathfrak{J}_c$ in \eqref{eq4.14}.
We first prove assertion 4 in Theorem~\ref{thm4.9}, i.e., the case $\mathfrak{J}=\mathfrak{J}_d$.  Accordingly, assertion 3 in Theorem~\ref{thm4.9}, i.e., the case $\mathfrak{J}=\mathfrak{J}_c$, is a quick consequence of assertion 4.  To this end, we need the following estimates.

\begin{Lemma}\label{lem4.12}
Suppose that $\mathscr{H}$ in \eqref{eq4.1} has a Hamiltonian Jordan canonical form $\mathfrak{J}_d$  in \eqref{eq4.14} and that the symplectic matrix $\mathcal{S}$ in \eqref{eq4.8} is of the form in \eqref{eq4.16}.
Let $Y(t)=[Q(t)^{\top},  P(t)^{\top}]^{\top}$ be the solution of IVP \eqref{eq4.1} and $[W_1^{\top}, W_2^{\top}]^{\top}=\mathcal{S}^{-1}[I,W_0]^{\top}$. Suppose that $W_2\in \mathbb{C}^{n\times n}$ is invertible and
\begin{align}\label{eq4.20}
\mathbf{W}:=W_1W_2^{-1}=\left[\begin{array}{c|c}\mathbf{W}_{1,1}&\mathbf{w}_{1,2}\\\hline \mathbf{w}_{2,1}&\mathbf{w}_{2,2}\end{array}\right],
\end{align}
where $\mathbf{W}_{1,1}\in \mathbb{C}^{(n_1+n_2)\times (n_1+n_2)}$, $\mathbf{w}_{1,2}, \mathbf{w}_{2,1}^H\in \mathbb{C}^{n_1+n_2}$ and $\mathbf{w}_{2,2}\in \mathbb{C}$. Let $i\gamma, i\delta$ be eigenvalues of $\mathscr{H}$,
\begin{subequations}\label{eq4.21}
\begin{align}\label{eq4.21a}
\left[\begin{array}{ll}f_u&g_u\\f_v&g_v\end{array}\right]=\frac{1}{2}\left[\begin{array}{ll}(-1)^{n_1}(\mathbf{w}_{2,2}-i\beta)&(-1)^{n_2}(\mathbf{w}_{2,2}+i\beta)\\(-1)^{n_1}(i\beta\mathbf{w}_{2,2}+1)&(-1)^{n_2}(-i\beta\mathbf{w}_{2,2}+1)\end{array}\right],
\end{align}
and let
\begin{align}\label{eq4.21b}
\mathbf{U}(t)=\left[\begin{array}{c|cc}U_1 &u_1&v_1\\\hline U_2& u_2&v_2\end{array}\right]\left[\begin{array}{c|c}I&0\\\hline 0&f_ue^{i\gamma t}+g_ue^{i\delta t}\\ 0&f_ve^{i\gamma t}+g_ve^{i\delta t} \end{array}\right]\in \mathbb{C}^{2n\times n}
\end{align}
be a quasiperiodic matrix.
\end{subequations}
Then there exists a nonsingular matrix $\Omega(t)$ of the form
\begin{align}\label{eq4.22}
\Omega(t)=\left[\begin{array}{c|c}O(t^{-2}) & O(t^{-1}) \\\hline 0 & 1\end{array}\right]
\end{align}
satisfying
\begin{align}\label{eq4.23}
Y(t)W_2^{-1}\Omega(t)=\mathbf{U}(t)+O(t^{-1}),
\end{align}
as $t\rightarrow\pm \infty$. Furthermore, if $\mathbf{w}_{2,2}$ is real then $\mathbf{U}(t)$ is $\mathcal{J}$-orthogonal for each $t$.
\end{Lemma}

\begin{proof}
From Lemma \ref{lem4.5} and \eqref{eq4.9}, we have
\begin{align}\label{eq4.24}
Y(t)&\equiv \left[\begin{array}{c}Q(t) \\P(t)\end{array}\right]=\mathcal{S}
\left[\begin{array}{c|c}\mathbf{ B}&\mathbf{ D}\\\hline \mathbf{ G}&\mathbf{ E}\end{array}\right]\left[\begin{array}{c}W_1 \\W_2\end{array}\right],
\end{align}
where
\begin{align*}
\begin{array}{ll}
\mathbf{ B}=\left[\begin{array}{cc}\Phi_{n_1,n_2}& \phi_{n_1,n_2}^{1}  \\0&\omega_{11}\end{array}\right],&
\mathbf{ D}=\left[\begin{array}{cc}\widehat{\Gamma}_{n_1+1,n_2+1}^{2n_1,2n_2}& \phi_{n_1,n_2}^{2}  \\\widehat{\psi}_{n_1,n_2}^{1^H} &\omega_{12}\end{array}\right],\\
\mathbf{ G}=\left[\begin{array}{cc}0& 0 \\0 &\omega_{21}\end{array}\right],&\mathbf{ E}=\left[\begin{array}{cc}\widehat{\Phi}_{n_1,n_2}& 0 \\\widehat{\psi}_{n_1,n_2}^{2^H}&\omega_{22}\end{array}\right],
\end{array}
\end{align*}
and $\Phi_{n_1,n_2}$, $\widehat{\Phi}_{n_1,n_2}$, $\widehat{\Gamma}_{n_1+1,n_2+1}^{2n_1,2n_2}$, $\phi_{n_1,n_2}^{j} $, $\widehat{\psi}_{n_1,n_2}^{j^H}$ and $\omega_{ij}$ for $i,j\in \{1,2\}$ are given in \eqref{eq4.12}.  Denote
\begin{align}\label{eq4.25}
\begin{array}{ll}
\Upsilon_{n_1,n_2}&=\Phi_{n_1,n_2}\mathbf{W}_{1,1}+\phi^1_{n_1,n_2}\mathbf{w}_{2,1},\\
p_{n_1,n_2}&=\Phi_{n_1,n_2}\mathbf{w}_{1,2}+\phi^1_{n_1,n_2}\mathbf{w}_{2,2}+\phi_{n_1,n_2}^2,
\end{array}
\end{align}
where $\mathbf{W}_{1,1}$, $\mathbf{w}_{2,1}$, $\mathbf{w}_{1,2}$ and $\mathbf{w}_{2,2}$ are defined in \eqref{eq4.20}.
From \eqref{eq4.16}, \eqref{eq4.24} and \eqref{eq4.25}, we have
\begin{align*}
\left[\begin{array}{c}Q(t) \\P(t)\end{array}\right]W_2^{-1}=&\left[\begin{array}{c|c}U_1 &V_1\\\hline U_2& V_2\end{array}\right]
\left[\begin{array}{c|c}\Upsilon_{n_1,n_2}+\widehat{\Gamma}_{n_1+1,n_2+1}^{2n_1,2n_2}& p_{n_1,n_2}\\\hline \widehat{\Phi}_{n_1,n_2}&0\end{array}\right]\nonumber\\
&+\left[\begin{array}{c|c}u_1 &v_1\\\hline u_2& v_2\end{array}\right]
\left[\begin{array}{c|c}\widehat{\psi}_{n_1,n_2}^{1^H}+\omega_{11}\mathbf{w}_{2,1}& \omega_{11}\mathbf{w}_{2,2}+\omega_{12}\\\hline \widehat{\psi}_{n_1,n_2}^{2^H}+\omega_{21}\mathbf{w}_{2,1}&\omega_{21}\mathbf{w}_{2,2}+\omega_{22}\end{array}\right].
\end{align*}
Let
\begin{align*}
\Omega(t)=\left[\begin{array}{c|c}(\Upsilon_{n_1,n_2}+\widehat{\Gamma}_{n_1+1,n_2+1}^{2n_1,2n_2})^{-1}& -(\Upsilon_{n_1,n_2}+\widehat{\Gamma}_{n_1+1,n_2+1}^{2n_1,2n_2})^{-1}p_{n_1,n_2}\\\hline 0&1\end{array}\right].
\end{align*}
By a direct computation from Table \ref{tab2} and \eqref{eq4.25}, we obtain that
$\Omega(t)=\left[\begin{array}{c|c}O(t^{-2}) & O(t^{-1}) \\\hline 0 & 1\end{array}\right]$
and
\begin{align}\label{eq4.26}
\left[\begin{array}{c}Q(t) \\P(t)\end{array}\right]W_2^{-1}\Omega(t)=&\left[\begin{array}{c|c}U_1 &V_1\\\hline U_2& V_2\end{array}\right]
\left[\begin{array}{c|c}I&0\\\hline O(t^{-2})&O(t^{-1})\end{array}\right]\nonumber\\
&+\left[\begin{array}{c|c}u_1 &v_1\\\hline u_2& v_2\end{array}\right]
\left[\begin{array}{c|c}O(t^{-1})& \omega_{11}\mathbf{w}_{2,2}+\omega_{12}-\xi_1+O(t^{-1})\\\hline O(t^{-1})&\omega_{21}\mathbf{w}_{2,2}+\omega_{22}-\xi_2+O(t^{-1})\end{array}\right]\nonumber\\
=&\left[\begin{array}{c|cc}U_1 &u_1&v_1\\\hline U_2& u_2&v_2\end{array}\right]\left[\begin{array}{c|c}I&0\\\hline 0&\omega_{11}\mathbf{w}_{2,2}+\omega_{12}-\xi_1\\ 0&\omega_{21}\mathbf{w}_{2,2}+\omega_{22}-\xi_2 \end{array}\right]+O(t^{-1}),
\end{align}
as $t\rightarrow\pm \infty$, where $\xi_j=\widehat{\psi}_{n_1,n_2}^{j^H}(\widehat{\Gamma}_{n_1+1,n_2+1}^{2n_1,2n_2})^{-1}(\phi^1_{n_1,n_2}\mathbf{w}_{2,2}+\phi_{n_1,n_2}^2)$ for $j=1,2$. Let
\begin{align}\label{eq4.27}
\Theta=\frac{\sqrt{2}}{2}\left[\begin{array}{cc}-1& i\beta\\-1&-i\beta\end{array}\right],
\end{align}
where $\beta\in\{-1,1\}$. Then $\Theta$ is unitary. From \eqref{eq4.12} we have
\begin{align*}
\left[\phi^1_{n_1,n_2}\ |\ \phi^2_{n_1,n_2}\right]=\left[\begin{array}{c|c}e^{i\gamma t}\phi_{n_1}&0\\\hline0&e^{i\delta t}\phi_{n_2}\end{array}\right]\Theta,\
\left[\begin{array}{c}\widehat{\psi}_{n_1,n_2}^{1^H}\\\hline\widehat{\psi}_{n_1,n_2}^{2^H}\end{array}\right]=\Theta^H\left[\begin{array}{c|c}-i\beta e^{i\gamma t}\widehat{\psi}_{n_1}^{H}&0\\\hline0&i\beta e^{i\delta t}\widehat{\psi}_{n_2}^{H}\end{array}\right].
\end{align*}
Then
\begin{align}\label{eq4.28}
&\left[\begin{array}{c}\xi_1 \\\xi_2\end{array}\right]=\Theta^H\left[\begin{array}{c|c}\widehat{\psi}_{n_1}^{H}&0\\\hline0&\widehat{\psi}_{n_2}^{H}\end{array}\right]\left[\begin{array}{c|c}\widehat{\Gamma}_{n_1+1}^{2n_1}&0\\\hline0&\widehat{\Gamma}_{n_2+1}^{2n_2}\end{array}\right]^{-1}
\left[\begin{array}{c|c}e^{i\gamma t}\phi_{n_1}&0\\\hline0&e^{i\delta t}\phi_{n_2}\end{array}\right]\Theta\left[\begin{array}{c}\mathbf{w}_{2,2} \\1\end{array}\right]\nonumber\\
&\ \ \ \ \ \ \ \ \ =\Theta^H\left[\begin{array}{c|c}\kappa_{n_1}&0\\\hline0&\kappa_{n_2}\end{array}\right]
\left[\begin{array}{c|c}e^{i\gamma t}&0\\\hline0&e^{i\delta t}\end{array}\right]\Theta\left[\begin{array}{c}\mathbf{w}_{2,2} \\1\end{array}\right],\\
&\left[\begin{array}{c}\omega_{11}\mathbf{w}_{2,2}+\omega_{12} \\\omega_{21}\mathbf{w}_{2,2}+\omega_{22}\end{array}\right]=\Theta^H\left[\begin{array}{c|c}e^{i\gamma t}&0\\\hline0&e^{i\delta t}\end{array}\right]\Theta\left[\begin{array}{c}\mathbf{w}_{2,2} \\1\end{array}\right],\nonumber
\end{align}
where
$\kappa_{n_j}=\widehat{\psi}_{n_j}^{H}(\widehat{\Gamma}_{n_j+1}^{2n_j})^{-1}\phi_{n_j}$ for $j=1,2$. From Theorem~\ref{thm5.3}, we have $1-\kappa_{n_j}=(-1)^{n_j}$ for $j=1,2$.
It follows from \eqref{eq4.27} and \eqref{eq4.28} that
\begin{align}\label{eq4.29}
\left[\begin{array}{c}\omega_{11}\mathbf{w}_{2,2}+\omega_{12}-\xi_1 \\\omega_{21}\mathbf{w}_{2,2}+\omega_{22}-\xi_2\end{array}\right]&=\Theta^H\left[\begin{array}{c|c}(-1)^{n_1}e^{i\gamma t}&0\\\hline0&(-1)^{n_2}e^{i\delta t}\end{array}\right]\Theta\left[\begin{array}{c}\mathbf{w}_{2,2} \\1\end{array}\right]\nonumber\\
&=\frac{1}{2}\left[\begin{array}{cc}-1& -1\\-i\beta&i\beta\end{array}\right] \left[\begin{array}{c|c}(-1)^{n_1}e^{i\gamma t}&0\\\hline0&(-1)^{n_2}e^{i\delta t}\end{array}\right]\left[\begin{array}{cc}-1& i\beta\\-1&-i\beta\end{array}\right]\left[\begin{array}{c}\mathbf{w}_{2,2} \\1\end{array}\right]\nonumber\\
&=\frac{1}{2}\left[\begin{array}{l}(-1)^{n_1}(\mathbf{w}_{2,2}-i\beta)e^{i\gamma t}+(-1)^{n_2}(\mathbf{w}_{2,2}+i\beta)e^{i\delta t}\\(-1)^{n_1}(i\beta\mathbf{w}_{2,2}+1)e^{i\gamma t}+(-1)^{n_2}(-i\beta\mathbf{w}_{2,2}+1)e^{i\delta t}\end{array}\right]\nonumber\\
&\equiv \left[\begin{array}{ll}f_u&g_u\\f_v&g_v\end{array}\right]\left[\begin{array}{c} e^{i\gamma t}\\e^{i\delta t}\end{array}\right].
\end{align}
Assertion \eqref{eq4.23} follows from \eqref{eq4.26} and \eqref{eq4.29}.

Suppose that $\mathbf{w}_{2,2}$ is real, we show that $\mathbf{U}(t)$ is $\mathcal{J}$-orthogonal. Let
\begin{align*}
z(t)=\left[\begin{array}{c}(f_ue^{i\gamma t}+g_ue^{i\delta t})u_1+(f_ve^{i\gamma t}+g_ve^{i\delta t})v_1 \\(f_ue^{i\gamma t}+g_ue^{i\delta t})u_2+(f_ve^{i\gamma t}+g_ve^{i\delta t})v_2\end{array}\right].
\end{align*}
Since the matrix $\mathcal{S}$ given in \eqref{eq4.16} is symplectic, it suffices to show that  $z(t)^H\mathcal{J}z(t)=0$. From  \eqref{eq4.29}, we have
\begin{align*}
z(t)^H&\mathcal{J}z(t)=\left[\begin{array}{c}f_ue^{i\gamma t}+g_ue^{i\delta t} \\f_ve^{i\gamma t}+g_ve^{i\delta t}\end{array}\right]^H \left[\begin{array}{cc}u_1 & v_1 \\u_2 & v_2\end{array}\right]^H \mathcal{J}\left[\begin{array}{cc}u_1 & v_1 \\u_2 & v_2\end{array}\right]\left[\begin{array}{c}f_ue^{i\gamma t}+g_ue^{i\delta t} \\f_ve^{i\gamma t}+g_ve^{i\delta t}\end{array}\right]\\
&=[\mathbf{w}_{2,2},1] \Theta^H\left[\begin{array}{ll}(-1)^{n_1}e^{-i\gamma t}&0\\0&(-1)^{n_2}e^{-i\delta t}\end{array}\right]\Theta \left[\begin{array}{cc}0&1\\-1&0\end{array}\right]\\ &\ \ \ \ \ \Theta^H\left[\begin{array}{ll}(-1)^{n_1}e^{i\gamma t}&0\\0&(-1)^{n_2}e^{i\delta t}\end{array}\right]\Theta\left[\begin{array}{c}\mathbf{w}_{2,2} \\1\end{array}\right]\\
&=i \beta  [\mathbf{w}_{2,2},1] \Theta^H\left[\begin{array}{cc}1&0\\0&-1\end{array}\right]\Theta\left[\begin{array}{c}\mathbf{w}_{2,2} \\1\end{array}\right]=i \beta[\mathbf{w}_{2,2},1] \left[\begin{array}{cc}0&-i\beta\\i\beta&0\end{array}\right]\left[\begin{array}{c}\mathbf{w}_{2,2} \\1\end{array}\right]\\&=[\mathbf{w}_{2,2},1]\left[\begin{array}{cc}0&1\\-1&0\end{array}\right]\left[\begin{array}{c}\mathbf{w}_{2,2} \\1\end{array}\right]=0,
\end{align*}
for each $t$. Hence, the quasiperiodic matrix $\mathbf{U}(t)$ is $\mathcal{J}$-orthogonal for each $t$.
\end{proof}

\begin{Remark}\label{rem4.1}
Since the initial matrix $W_0$ of the RDE is Hermitian,
\begin{align*}
\left[\begin{array}{c}W_1 \\W_2\end{array}\right]^H\mathcal{J}\left[\begin{array}{c}W_1 \\W_2\end{array}\right]=\left[I, W_0\right]\mathcal{S}^{-H}\mathcal{J}\mathcal{S}^{-1}\left[\begin{array}{c}I \\W_0\end{array}\right]=\left[I, W_0\right]\mathcal{J}\left[\begin{array}{c}I \\W_0\end{array}\right]=0,
\end{align*}
that is, the column space of $[W_1^{\top}, W_2^{\top}]^{\top}$  is a Lagrangian subspace. Hence, $\mathbf{W}$ defined in \eqref{eq4.20} is Hermitian and $\mathbf{w}_{2,2}$ is real. Notice that it follows from \eqref{eq4.21a}  that $f_u$, $f_v$, $g_u$ and $g_v$ are constants and $|f_u|=|f_v|=|g_u|=|g_v|$.
\end{Remark}

Partition $\mathbf{U}(t)$ in \eqref{eq4.21b} as $\mathbf{U}(t)=[\mathbf{U}_1(t)^{\top}, \mathbf{U}_2(t)^{\top}]^{\top}$, where
\begin{align}\label{eq4.30}
\begin{array}{l}
\mathbf{U}_1(t)=\left[U_1|(f_ue^{i\gamma t}+g_ue^{i\delta t})u_1+(f_ve^{i\gamma t}+g_ve^{i\delta t})v_1\right],\\
\mathbf{U}_2(t)=\left[U_2|(f_ue^{i\gamma t}+g_ue^{i\delta t})u_2+(f_ve^{i\gamma t}+g_ve^{i\delta t})v_2\right].
\end{array}
\end{align}
Now we are ready to prove assertion 4 in Theorem~\ref{thm4.9}.

\begin{Theorem}\label{thm4.13}
Suppose assumptions in Theorem \ref{thm4.9} hold and $W_2$ is invertible. Let $\mathfrak{J}_x=\mathfrak{J}_d$ and partion the symplectic matrix $\mathcal{S}$ as the form in \eqref{eq4.16}.
Denote
\begin{align}\label{eq4.31}
\begin{array}{ll}
\mathbf{U}_1=[U_1, f_uu_1+f_vv_1],& \mathbf{U}_2=[U_2, f_uu_2+f_vv_2]\in \mathbb{C}^{n\times n},\\
\zeta_1=g_uu_1+g_vv_1,&\zeta_2=g_uu_2+g_vv_2\in \mathbb{C}^{n},
\end{array}
\end{align}
where $U_j$, $u_j$, $v_j$, for $j=1,2$, are given in \eqref{eq4.16} and $f_u$, $f_v$, $g_u$, $g_v$ are defined in  \eqref{eq4.21a}. Then $\mathbf{U}_j$ and $\zeta_j$ for $j=1,2$  are independent of $t$. If $\mathbf{U}_1$ is invertible, then
\begin{align*}
&W(t)=\mathbf{U}_2\mathbf{U}_1^{-1}+\frac{e^{i\theta t}}{1+e^{i\theta t}e_n^H\mathbf{U}_1^{-1}\zeta_1+O(t^{-1})}\left[\left(\zeta_2-\mathbf{U}_2\mathbf{U}_1^{-1}\zeta_1\right)e_n^H\mathbf{U}_1^{-1}+O(t^{-1})\right]+O(t^{-1}),\\
&Q(t)^{-1}=\frac{e^{-i\gamma t}}{1+e^{i\theta t}e_n^H\mathbf{U}_1^{-1}\zeta_1+O(t^{-1})}\left[W_2^{-1}e_ne_n^{H}\mathbf{U}_1^{-1}+O(t^{-1})\right]+O(t^{-1}),
\end{align*}
as $t\rightarrow \pm \infty$, where $\theta=\delta-\gamma$.
\end{Theorem}

\begin{proof}
Let $\theta=\delta-\gamma$. From \eqref{eq4.30} and \eqref{eq4.31}, we have
\begin{align*}
\mathbf{U}_1(t)&=\left[U_1|(f_u+g_ue^{i\theta t})u_1+(f_v+g_ve^{i\theta t})v_1\right](I\oplus e^{i\gamma t})\\&=(\mathbf{U}_1+e^{i\theta t}\zeta_1e_n^{H})(I\oplus e^{i\gamma t}),\\
\mathbf{U}_2(t)&=\left[U_2|(f_u+g_ue^{i\theta t})u_2+(f_v+g_ve^{i\theta t})v_2\right](I\oplus e^{i\gamma t})
\\&=(\mathbf{U}_2+e^{i\theta t}\zeta_2e_n^{H})(I\oplus e^{i\gamma t}).
\end{align*}
From \eqref{eq4.23} it follows that there exist matrix functions $M_1^{\varepsilon}(t)$ and $M_2^{\varepsilon}(t)$ such that
\begin{align}\label{eq4.32}
\begin{array}{l}
Q(t)W_2^{-1}\Omega(t)(I\oplus e^{-i\gamma t})=(\mathbf{U}_1+e^{i\theta t}\zeta_1e_n^{H})+M_1^{\varepsilon}(t),\\
P(t)W_2^{-1}\Omega(t)(I\oplus e^{-i\gamma t})=(\mathbf{U}_2+e^{i\theta t}\zeta_2e_n^{H})+M_2^{\varepsilon}(t),
\end{array}
\end{align}
where $M_1^{\varepsilon}(t)=O(t^{-1})$ and $M_2^{\varepsilon}(t)=O(t^{-1})$ as $t\rightarrow \pm \infty$. Then
\begin{align}
W(t)=P(t)Q(t)^{-1}&=\left(\mathbf{U}_2+e^{i\theta t}\zeta_2e_n^{H}\right)\left[(\mathbf{U}_1+M_1^{\varepsilon}(t))+e^{i\theta t}\zeta_1e_n^{H}\right]^{-1}\notag\\
&\ \ \ \ +M_2^{\varepsilon}(t)\left[(\mathbf{U}_1+M_1^{\varepsilon}(t))+e^{i\theta t}\zeta_1e_n^{H}\right]^{-1}\label{eq4.33}.
\end{align}
Let
\begin{align}\label{eq4.34}
\mathbf{U}^{\varepsilon}_1(t)\equiv \mathbf{U}_1+M_1^{\varepsilon}(t).
\end{align}
Since $\mathbf{U}_1$ is invertible and $M_1^{\varepsilon}(t)=O(t^{-1})$ as $t\rightarrow \pm \infty$, $\mathbf{U}^{\varepsilon}_1(t)$ is invertible for sufficiently large $|t|$ and $\mathbf{U}^{\varepsilon}_1(t)^{-1}=\mathbf{U}_1^{-1}+O(t^{-1})$.
Applying  the Sherman-Morrison-Woodbury formula, we have
\begin{align}\label{eq4.35}
\left(\mathbf{U}^{\varepsilon}_1(t)+e^{i\theta t}\zeta_1e_n^{H}\right)^{-1}=\mathbf{U}^{\varepsilon}_1(t)^{-1}-\frac{e^{i\theta t}}{1+e^{i\theta t}e_n^H\mathbf{U}^{\varepsilon}_1(t)^{-1}\zeta_1}\mathbf{U}^{\varepsilon}_1(t)^{-1}\zeta_1e_n^H\mathbf{U}^{\varepsilon}_1(t)^{-1}.
\end{align}
Since  $M_2^{\varepsilon}(t)=O(t^{-1})$ as $t\rightarrow \pm \infty$,
\begin{align}\label{eq4.36}
M_2^{\varepsilon}(t)\left[(\mathbf{U}_1+M_1^{\varepsilon}(t))+e^{i\theta t}\zeta_1e_n^{H}\right]^{-1}=O\left(\frac{t^{-1}e^{i\theta t}}{1+e^{i\theta t}e_n^H\mathbf{U}^{\varepsilon}_1(t)^{-1}\zeta_1}\right).
\end{align}
Plugging \eqref{eq4.35} and \eqref{eq4.36} into \eqref{eq4.33}, it turns out
\begin{align*}
W(t)&=\mathbf{U}_2\mathbf{U}^{\varepsilon}_1(t)^{-1}+e^{i\theta t}(\zeta_2e_n^{H}\mathbf{U}^{\varepsilon}_1(t)^{-1})\\&\ \ \ -\frac{e^{i\theta t}}{1+e^{i\theta t}e_n^H\mathbf{U}^{\varepsilon}_1(t)^{-1}\zeta_1}e^{i\theta t}\zeta_2e_n^{H}\left(\mathbf{U}^{\varepsilon}_1(t)^{-1}\zeta_1e_n^H\mathbf{U}^{\varepsilon}_1(t)^{-1}\right)\\&\ \ \ -\frac{e^{i\theta t}}{1+e^{i\theta t}e_n^H\mathbf{U}^{\varepsilon}_1(t)^{-1}\zeta_1}\mathbf{U}_2\left(\mathbf{U}^{\varepsilon}_1(t)^{-1}\zeta_1e_n^H\mathbf{U}^{\varepsilon}_1(t)^{-1}\right)+O\left(\frac{t^{-1}e^{i\theta t}}{1+e^{i\theta t}e_n^H\mathbf{U}^{\varepsilon}_1(t)^{-1}\zeta_1}\right)\\
&=\mathbf{U}_2\mathbf{U}^{\varepsilon}_1(t)^{-1}+\frac{e^{i\theta t}}{1+e^{i\theta t}e_n^H\mathbf{U}^{\varepsilon}_1(t)^{-1}\zeta_1}\left[\left(\zeta_2-\mathbf{U}_2\mathbf{U}^{\varepsilon}_1(t)^{-1}\zeta_1\right)e_n^{H}\mathbf{U}^{\varepsilon}_1(t)^{-1}+O(t^{-1})\right]\\
&=\mathbf{U}_2\mathbf{U}_1^{-1}+\frac{e^{i\theta t}}{1+e^{i\theta t}e_n^H\mathbf{U}_1^{-1}\zeta_1+O(t^{-1})}\left[\left(\zeta_2-\mathbf{U}_2\mathbf{U}_1^{-1}\zeta_1\right)e_n^H\mathbf{U}_1^{-1}+O(t^{-1})\right] +O(t^{-1}),
\end{align*}
as $t\rightarrow \pm \infty$.

From \eqref{eq4.32}, we have $Q(t)^{-1}=W_2^{-1}\Omega(t)(I\oplus e^{-i\gamma t})(\mathbf{U}^{\varepsilon}_1(t)+e^{i\theta t}\zeta_1e_n^{H})^{-1}$, where $\mathbf{U}^{\varepsilon}_1(t)$ is defined in \eqref{eq4.34}. From \eqref{eq4.22}, we obtain that  $\Omega(t)(I\oplus e^{-i\gamma t})=e^{-i\gamma t}e_ne_n^H+O(t^{-1})$ as $t\rightarrow \pm \infty$. Therefore, by  \eqref{eq4.35} and \eqref{eq4.36}, we have
\begin{align*}
Q(t)^{-1}&=e^{-i\gamma t}W_2^{-1}\left(e_ne_n^{H}+O(t^{-1})\right)\left(\mathbf{U}^{\varepsilon}_1(t)+e^{i\theta t}\zeta_1e_n^{H}\right)^{-1}\\
=&e^{-i\gamma t}W_2^{-1}e_ne_n^{H}\left[\mathbf{U}^{\varepsilon}_1(t)^{-1}-\frac{e^{i\theta t}}{1+e^{i\theta t}e_n^H\mathbf{U}^{\varepsilon}_1(t)^{-1}\zeta_1}\left(\mathbf{U}^{\varepsilon}_1(t)^{-1}\zeta_1e_n^H\mathbf{U}^{\varepsilon}_1(t)^{-1}\right)\right]\\&+O\left(\frac{t^{-1}e^{i\theta t}}{1+e^{i\theta t}e_n^H\mathbf{U}^{\varepsilon}_1(t)^{-1}\zeta_1}\right)\\
=&e^{-i\gamma t}W_2^{-1}e_ne_n^{H}\left[\mathbf{U}_1^{-1}-\frac{e^{i\theta t}}{1+e^{i\theta t}e_n^H\mathbf{U}_1^{-1}\zeta_1+O(t^{-1})}\left(\mathbf{U}_1^{-1}\zeta_1e_n^H\mathbf{U}_1^{-1}+O(t^{-1})\right)+O(t^{-1})\right]\\
=&e^{-i\gamma t}W_2^{-1}e_ne_n^{H}\mathbf{U}_1^{-1}-\frac{e^{-i\gamma t}e^{i\theta t}e_n^{H}\mathbf{U}_1^{-1}\zeta_1}{1+e^{i\theta t}e_n^H\mathbf{U}_1^{-1}\zeta_1+O(t^{-1})}\left[W_2^{-1}e_ne_n^{H}\mathbf{U}_1^{-1}+O(t^{-1})\right]+O(t^{-1})\\
=&\frac{e^{-i\gamma t}}{1+e^{i\theta t}e_n^H\mathbf{U}_1^{-1}\zeta_1+O(t^{-1})}\left[W_2^{-1}e_ne_n^{H}\mathbf{U}_1^{-1}+O(t^{-1})\right]+O(t^{-1}),
\end{align*}
as $t\rightarrow \pm\infty$.
\end{proof}

Roughly speaking, Theorem \ref{thm4.13} shows that  if $\theta\neq 0$, then $W(t)$ and $Q(t)^{-1}$ will converge in the rate $O(t^{-1})$, as $t\rightarrow \pm\infty$, to a periodic orbit $W_{\infty}(t)$ with period $2\pi/\theta$ and to a quasiperiodic orbit $Q_{\infty}^{-1}(t)$,
\begin{align}\label{eq4.37}
W_{\infty}(t)&=\mathbf{U}_2(t)\mathbf{U}_1(t)^{-1}\nonumber\\&=\mathbf{U}_2\mathbf{U}_1^{-1}+\frac{e^{i\theta t}}{1+e^{i\theta t}e_n^H\mathbf{U}_1^{-1}\zeta_1}\left[\left(\zeta_2-\mathbf{U}_2\mathbf{U}_1^{-1}\zeta_1\right)e_n^H\mathbf{U}_1^{-1}\right],\\
Q_{\infty}^{-1}(t)&=\frac{e^{-i\gamma t}}{1+e^{i\theta t}e_n^H\mathbf{U}_1^{-1}\zeta_1}W_2^{-1}e_ne_n^{H}\mathbf{U}_1^{-1},\nonumber
\end{align}
respectively,  where $\mathbf{U}_1$, $\mathbf{U}_2$, $\zeta_1$ and $\zeta_2$ are defined in \eqref{eq4.31} and $t\in \{t\in \mathbb{R}| 1+e^{i\theta t}e_n^H\mathbf{U}_1^{-1}\zeta_1\neq 0\}$. More precisely, for each $0<\rho\ll1$, this convergence in the rate $O(t^{-1})$ is taking $t\rightarrow \pm\infty$ along the unbounded set $\{t\in{\mathbb R}|~|1+e^{i\theta t}e_n^H\mathbf{U}_1^{-1}\zeta_1|>\rho\}$. Note that  the matrices $\left(\zeta_2-\mathbf{U}_2\mathbf{U}_1^{-1}\zeta_1\right)e_n^H\mathbf{U}_1^{-1}$ and $W_2^{-1}e_ne_n^{H}\mathbf{U}_1^{-1}$ are constant (independent of $t$) and are of rank one. In addition, the periodic obit $W_{\infty}(t)$ is Hermitian because Lemma \ref{lem4.12} shows that $\mathbf{U}(t)=[\mathbf{U}_1(t)^{\top}, \mathbf{U}_2(t)^{\top}]^{\top}$ is $\mathcal{J}$-orthogonal. The orbit $W_{\infty}(t)$ blows up  when $\mathbf{U}_1(t)$ in \eqref{eq4.30} is singular (or $1+e^{i\theta t}e_n^H\mathbf{U}_1^{-1}\zeta_1= 0$). In the following theorem, we will show that if $\theta\neq 0$ then $W_{\infty}(t)$ will periodically blow-up.

\begin{Theorem}\label{thm4.14}
With the same notations of Theorem \ref{thm4.13}, suppose that $[U_1|u_1]$ is invertible, where $U_1$ and $u_1$ are given in \eqref{eq4.16}. If $\theta=\delta-\gamma\neq 0$, then the periodic matrix $\mathbf{U}_1(t)$ in \eqref{eq4.30} is singular with period $2\pi/|\theta|$.
\end{Theorem}

\begin{proof}
From \eqref{eq4.16}, $\mathcal{S}$ is symplectic and $[U_1,u_1|V_1,v_1]\mathcal{J}[U_1,u_1|V_1,v_1]^{H}=0$ . Since $[U_1|u_1]$ is invertible, we have
\begin{align*}
[Z_1|z_1]=[U_1|u_1]^{-1}[V_1|v_1]\in \mathbb{C}^{n\times n}
\end{align*}
is Hermitian. Here, $z_1\equiv [z_{11}^{\top},z_{12}^{\top}]^{\top}=[U_1|u_1]^{-1}v_1$, where $z_{11}\in \mathbb{C}^{n-1}$ and $z_{12}\in \mathbb{R}$. It follows from \eqref{eq4.30} that
\begin{align*}
[U_1|u_1]^{-1}\mathbf{U}_1(t)(I\oplus e^{-i\gamma t})=\left[\begin{array}{cc}I & (f_v+g_ve^{i\theta t})z_{11} \\ 0& (f_u+g_ue^{i\theta t})+z_{12}(f_v+g_ve^{i\theta t})\end{array}\right].
\end{align*}
Since $W_0$ is Hermitian,  we have $\mathbf{w}_{2,2}$ is real by Remark \ref{rem4.1}.  From \eqref{eq4.21a}, we obtain that $f_u=\bar{g}_u$ and $f_v=\bar{g}_v$. Since $z_{12}$ is real and $\theta\neq 0$, $f_u+z_{12}f_v=\overline{g_u+z_{12}g_v}$ and there exists $t_*\in [0,2\pi)$ such that $f_u+z_{12}f_v=(g_u+z_{12}g_v)e^{i\theta t_*}$. Hence, $\mathbf{U}_1(t_*+\frac{2\pi}{\theta}k)$ is singular for each  $k\in \mathbb{Z}$.
\end{proof}

We now consider the case that $\theta=0$, i.e., $\eta:=\gamma=\delta$ and $\mathscr{H}$ has Hamiltonian Jordan canonical form $\mathfrak{J}_c$ in \eqref{eq4.14}. In this case, we show that  $W(t)$ and $Q(t)^{-1}$ will converge in the rate $O(t^{-1})$ to a constant matrix and a periodic orbit, respectively. This proves assertion 3 in Theorem~\ref{thm4.9}.

\begin{Theorem}\label{thm4.15}
Suppose assumptions in Theorem \ref{thm4.9} hold. Let $\mathfrak{J}_x=\mathfrak{J}_c$ and the symplectic matrix $\mathcal{S}$ have the form in \eqref{eq4.16}.  Suppose that $W_2\in \mathbb{C}^{n\times n}$ is invertible and $\mathbf{W}:=W_1W_2^{-1}$ has the form in \eqref{eq4.20}. Let $\mathbf{U}_{1,0}=\mathbf{U}_{1}(0)$ and $\mathbf{U}_{2,0}=\mathbf{U}_{2}(0)$,
where $\mathbf{U}_{1}(t)$ and $\mathbf{U}_{2}(t)$ are defined in \eqref{eq4.30}.
If $\mathbf{U}_{1,0}$ is invertible, then
\begin{align*}
&W(t)=\mathbf{U}_{2,0}\mathbf{U}_{1,0}^{-1}+O(t^{-1}),\\
&Q(t)^{-1}=e^{-i\eta t}W_2^{-1}e_ne_n^{H}\mathbf{U}_{1,0}^{-1}+O(t^{-1}),
\end{align*}
as $t\rightarrow \pm\infty.$ Here,  $\mathbf{U}_{2,0}\mathbf{U}_{1,0}^{-1}$ is Hermitian.
\end{Theorem}

\begin{proof}
From \eqref{eq4.30}, \eqref{eq4.22} and \eqref{eq4.23} with $\eta:=\gamma=\delta$, we have
\begin{align*}
Y(t)W_2^{-1}\Omega(t)(I\oplus e^{-i\eta t})=\left[\begin{array}{c}\mathbf{U}_{1,0}\\\mathbf{U}_{2,0}\end{array}\right]+O(t^{-1}) \text{ and }\Omega(t)=e_ne_n^{H}+O(t^{-1}),
\end{align*}
as $t\rightarrow \pm\infty$,  where $Y(t)=[Q(t)^{\top},  P(t)^{\top}]^{\top}$. Since $\mathbf{U}_{1,0}$ is invertible,  we have
\begin{align*}
W(t)&=P(t)Q(t)^{-1}=(\mathbf{U}_{2,0}+O(t^{-1}))(\mathbf{U}_{1,0}+O(t^{-1}))^{-1}=\mathbf{U}_{2,0}\mathbf{U}_{1,0}^{-1}+O(t^{-1}),\\
Q(t)^{-1}&=W_2^{-1}\Omega(t)(I\oplus e^{-i\eta t})(\mathbf{U}_{1,0}+O(t^{-1}))^{-1}=e^{-i\eta t}W_2^{-1}e_ne_n^{H}\mathbf{U}_{1,0}^{-1}+O(t^{-1}),
\end{align*}
as $t\rightarrow \pm\infty.$ Using the fact that  $W_0$ is Hermitian, it follows from Remark \ref{rem4.1} and Lemma \ref{lem4.12} that $\mathbf{U}(0)=[\mathbf{U}_{1,0}^{\top},\mathbf{U}_{2,0}^{\top}]^{\top}$ is $\mathcal{J}$-orthogonal. Hence, $\mathbf{U}_{2,0}\mathbf{U}_{1,0}^{-1}$ is Hermitian.
\end{proof}

\begin{Example}\label{ex4.1}
In this example, we show some numerical experiments to demonstrate above theorems.
Consider the Hamiltonian matrix $\mathscr{H}$ has a Jordan canonical form $\mathfrak{J}_x=\left[\begin{array}{c|c}R_x & D_x\\\hline G_x& -R_x^{-H}\end{array}\right]$. Assume $\mathscr{H}=\mathcal{S}\mathfrak{J}_x\mathcal{S}^{-1}$, where the symplectic matrix $\mathcal{S}$ is randomly generated and
\begin{align*}
R_x&=\left[\begin{array}{cccc}
                 i\gamma &  1          &        0        &          0          \\
            0           &       i\gamma   &     0          &  -\frac{\sqrt{2}}{2}          \\
             0           &       0             &     i\delta  & -\frac{\sqrt{2}}{2}             \\
             0           &       0              &    0               &    \frac{i}{2}(\gamma+\delta)\\
\end{array}\right],\ \
D_x=\frac{\sqrt{2}i}{2}  \left[\begin{array}{cccc}
               0          &         0           &        0          &         0        \\
               0         &          0           &        0          &          1\\
             0          &         0             &      0           &         -1\\
              0          &       -1      &   1   &- \frac{\sqrt{2}i}{2} (\gamma-\delta)   \\
\end{array}\right],\\ G_x&=-\frac{1}{2}(\gamma-\delta) e_4e_4^{\top}.
\end{align*}
We also randomly generate a complex Hermitian matrix $W_0$ as the initial matrix of RDE \eqref{eq4.3}. Then the solution $Y(t)= [Q(t)^{\top},  P(t)^{\top}]^{\top}$ of IVP \eqref{eq4.1} can be computed by the formula $Y(t)=\mathcal{S}e^{\mathfrak{J}_xt}\mathcal{S}^{-1}[I,W_0]^{\top}$. The extended solution of the RDE can be obtained by the formula $W(t)=P(t)Q(t)^{-1}$ for $t\in \mathcal{T}_W$, where $\mathcal{T}_W$ is defined in \eqref{eq3.41}.
\begin{itemize}
\item[1.] We consider the case $x=d$ with $\gamma= 7.8340$ and $\delta=7.2888$. Then $|\theta|=|\delta-\gamma|=0.5452$.  We note from \eqref{eq4.37} that $W_{\infty}(t)=\mathbf{U}_2(t)\mathbf{U}_1(t)^{-1}$, where $\mathbf{U}_1(t)$ and $\mathbf{U}_2(t)$ are given in \eqref{eq4.30}. In Figure \ref{fig1},  we show  the smallest singular value of $\mathbf{U}_1(t)$ and $\|W_{\infty}(t)\|_F$, $\|W(t)\|_F$, $\|Q_{\infty}^{-1}(t)\|_F$ and $\|Q^{-1}(t)\|_F$  plotted by the log scale for $900\leqslant t\leqslant 1000$. This figure shows that the  periodic matrix $\mathbf{U}_1(t)$  is singular with period $2\pi/|\theta|=11.5248$ which coincides with assertions of Theorem \ref{thm4.14}. Therefore, $\|W_{\infty}(t)\|_F$ blows up at each $t$ at which $\mathbf{U}_1(t)$ is singular.  The asymptotic behaviors of $\|W(t)\|_F$ and $\|Q(t)^{-1}\|_F$ are similar to that of  $\|W_{\infty}(t)\|_F$ and $\|Q_{\infty}^{-1}(t)\|_F$, respectively. The differences,  $\|W(t)-W_{\infty}(t)\|_F$ and $\|Q(t)^{-1}-Q_{\infty}^{-1}(t)\|_F$,  for  $-1000\leqslant t\leqslant 0$  and for $0\leqslant t\leqslant 1000$, are shown in Figure \ref{fig2}. We can see that for each $0<\rho\ll1$, as $t\rightarrow \pm\infty$ along the set $\{t\in{\mathbb R}|~|1+e^{i\theta t}e_n^H\mathbf{U}_1^{-1}\zeta_1|>\rho\}$, $W(t)$ and $Q(t)^{-1}$ converges to $W_\infty(t)$ and $Q_\infty^{-1}(t)$, respectively, with the rate $O(t^{-1})$. It turns out that the curves $\|W(t)-W_\infty(t)\|_F$ and $\|Q(t)^{-1}-Q_\infty^{-1}(t)\|_F$ match the curve $y=C/t$ on this set.  However, as $t\rightarrow \pm\infty$ along the set $\{t\in{\mathbb R}|~1+e^{i\theta t}e_n^H\mathbf{U}_1^{-1}\zeta_1=0\}$, that is the poles of $W_\infty(t)$ and $Q_\infty^{-1}(t)$, $W(t)$ and $Q(t)^{-1}$ tend to infinity. This leads to the peaks appear periodically in Figure~\ref{fig2}.  Therefore, the curves $\|W(t)-W_\infty(t)\|_F$ and $\|Q(t)^{-1}-Q_\infty^{-1}(t)\|_F$  blow-up on this set.
\begin{figure}[htb]
\centering \resizebox{5in}{!}{\includegraphics{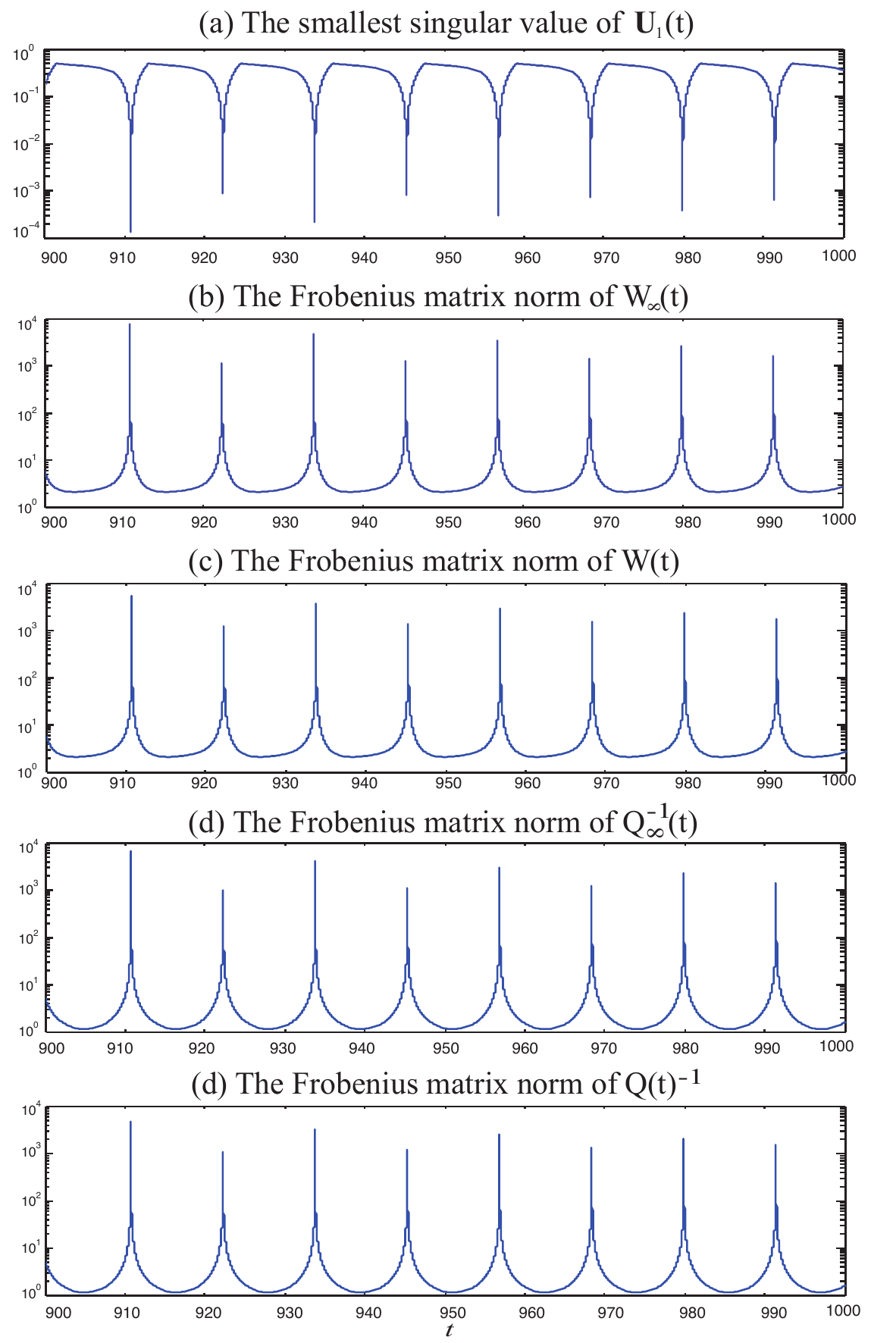}}
\caption{The smallest singular value of $\mathbf{U}_1(t)$, $\|W_{\infty}(t)\|_F$, $\|W(t)\|_F$, $\|Q_{\infty}^{-1}(t)\|_F$ and $\|Q^{-1}(t)\|_F$  plotted by the log scale.}\label{fig1}
\end{figure}
\begin{figure}[htb]
\centering \resizebox{6in}{!}{\includegraphics{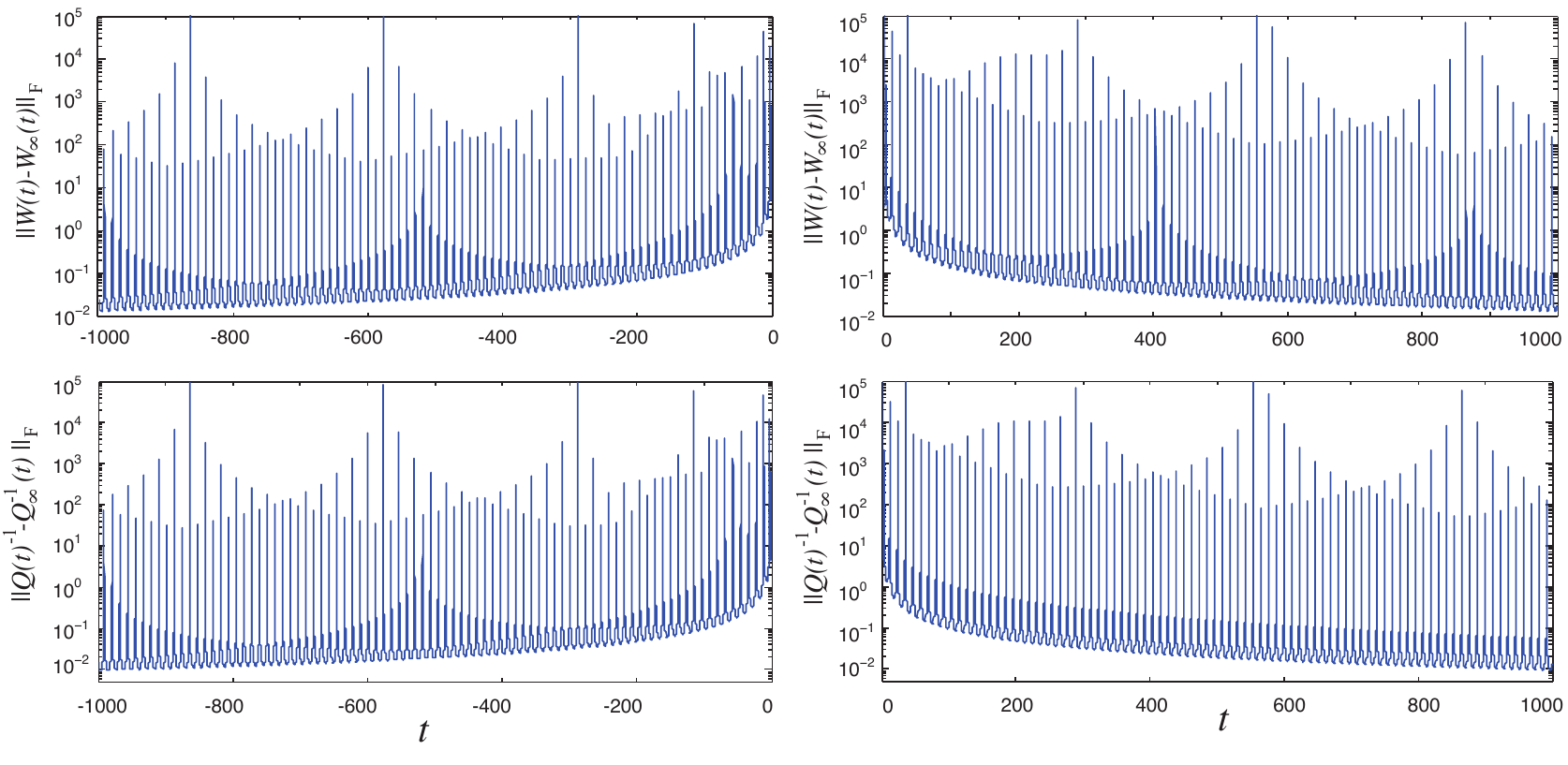}}
\caption{$\|W(t)-W_{\infty}(t)\|_F$ and $\|Q(t)^{-1}-Q_{\infty}^{-1}(t)\|_F$ plotted by the log scale for  $-1000\leqslant t\leqslant 0$ and for $0\leqslant t\leqslant 1000$.}\label{fig2}
\end{figure}
\item[2.] Let $\eta:=\gamma=\delta= 7.8340$, i.e., we consider the case $x=c$. In this case we have shown in Theorem \ref{thm4.15} that
\begin{itemize}
\item[(i)]  $W(t)$ converges at the rate $O(t^{-1})$ to a constant matrix $\mathbf{U}_{2,0}\mathbf{U}_{1,0}^{-1}$;
\item[(ii)] $Q(t)^{-1}$ converges at the rate $O(t^{-1})$ to a periodic orbit, $e^{-i\eta t}W_2^{-1}e_ne_n^H\mathbf{U}_{1,0}^{-1}$, with period $2\pi/|\eta|$.
\end{itemize}
In Figure \ref{fig3}, we  show the difference between $W(t)$ and the constant matrix $\mathbf{U}_{2,0}\mathbf{U}_{1,0}^{-1}$  for   $-1000\leqslant t\leqslant 0$  and for $0\leqslant t\leqslant 1000$. In Figure \ref{fig4},  we show the Frobenius norm of $Q(t)^{-1}$ and the difference between $Q(t)^{-1}$ and $Q_{\infty}^{-1}(t)=e^{-i\eta t}W_2^{-1}e_ne_n^H\mathbf{U}_{1,0}^{-1}$ for  $-1000\leqslant t\leqslant 0$  and for $0\leqslant t\leqslant 1000$. We can also see in Figures~\ref{fig3} and \ref{fig4} that $W(t)$ and $Q(t)^{-1}$ have a blow-up at $t\approx -10$,  even though we have shown in Theorem~\ref{thm4.15} that $W(t)$ and $Q(t)^{-1}$ have convergence with the rate $O(t^{-1})$.
\begin{figure}[htb]
\centering \resizebox{6in}{!}{\includegraphics{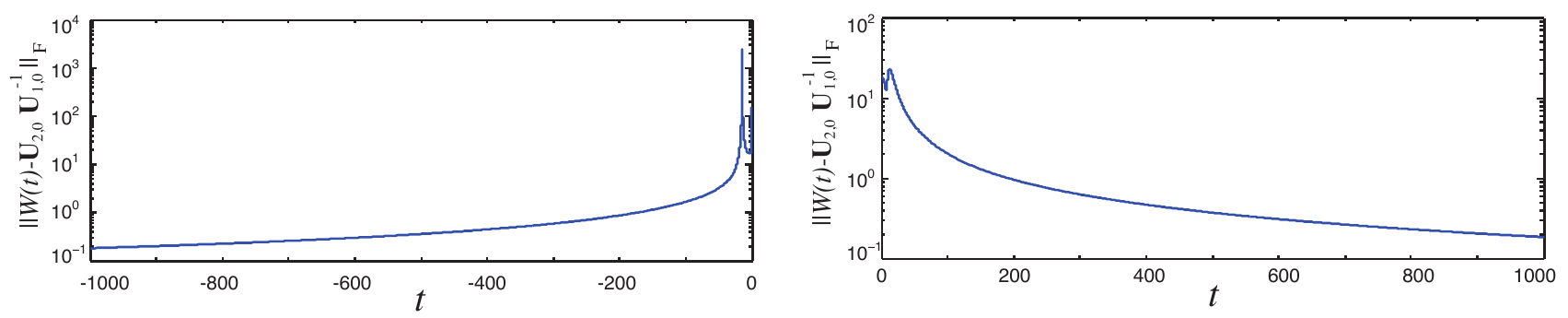}}
\caption{$\|W(t)-\mathbf{U}_{2,0}\mathbf{U}_{1,0}^{-1}\|_F$   for  $-1000\leqslant t\leqslant 0$ and for $0\leqslant t\leqslant 1000$.}\label{fig3}
\end{figure}
\begin{figure}[htb]
\centering \resizebox{6in}{!}{\includegraphics{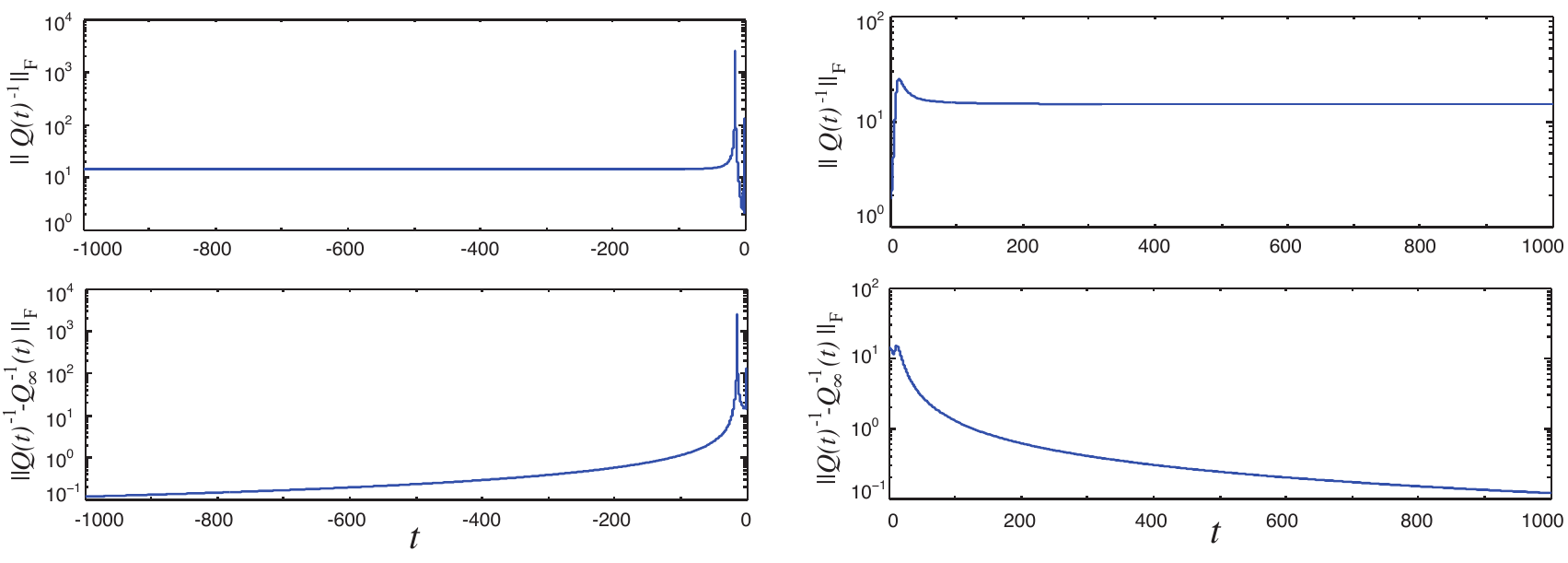}}
\caption{$\|Q(t)^{-1}\|_F$ and $\|Q(t)^{-1}-Q_{\infty}^{-1}(t)\|_F$  for  $-1000\leqslant t\leqslant 0$ and for $0\leqslant t\leqslant 1000$.}\label{fig4}
\end{figure}
\end{itemize}
\end{Example}

Note that the matrix $\mathfrak{J}$ in Example \ref{ex4.1} is given, and hence, the solution $W(t)$ can be computed with a good accuracy. For the general $\mathscr{H}$, a number of algorithms have been proposed for solving RDEs \eqref{eq4.1} numerically. These include conventional Runge-Kutta methods and linear multi-step methods \cite{Choi_Laub:1990,Dieci:1992} if blow-ups are not in the solution. If the solutions have blow-ups, an efficient numerical method developed by \cite{Li:2012} can be used for solving the RDEs.

\subsection{Asymptotic Analysis of RDE}\label{sec4.3}
In this subsection, we will investigate the asymptotic behaviors of $W(t)=P(t)Q(t)^{-1}$ and $Q(t)^{-1}$ as $t\rightarrow \pm \infty$,  where $Y(t)=[Q(t)^{\top},  P(t)^{\top}]^{\top}$ is the solution of IVP \eqref{eq4.1}.  Suppose that  $\mathcal{S}$ is symplectic such that $\mathcal{S}^{-1}\mathscr{H}\mathcal{S}=\mathfrak{J}$ is of the form in \eqref{eq4.8}, where  $\mathscr{H}\in \mathbb{C}^{2n\times 2n}$ is the Hamiltonian matrix in \eqref{eq4.1}. Since $\mathfrak{J}$ is the combination of the four elementary cases in Subsection~\ref{sec4.2}, detailed calculations for the asymptotic analysis are much tedious. However, the procedure is similar to what we have done in Subsection~\ref{sec4.2}.  Therefore, we only state the asymptotic analysis for the general cases and leave the proofs in Appendix.

Denote $n_r$, $n_e$, $ n_c$ and $ n_d$ be the sizes of $R_r$, $R_e$, $ R_c$ and $ R_d$ in \eqref{eq4.8}, respectively. It holds that $n_r+n_e+n_c+n_d=n$. Let  $n_{cd}=n_c+n_d$ and $n_{ecd}=n_e+n_c+n_d$.
Partitioning $W_j$ for $j=1,2$ in \eqref{eq4.9} as
\begin{align*}
\begin{array}{cc}W_j=\left[\begin{array}{cc}W^j_{1,1} & W^j_{1,2}   \\W^j_{2,1} & W^j_{2,2} \end{array}\right]&\!\!\!\!\!\!\!\!\begin{array}{l}\}n_r\\\} n_{ecd}\end{array}\vspace{-0.3cm}\\
\begin{array}{cc}\ \ \ \ \ \ \ \    \underbrace{}_{n_r}  &\underbrace{}_{n_{ecd}} \end{array}&\end{array}.
\end{align*}

We first make two assumptions:
 \begin{description}
 \item[{\bf Assumption} $\mathscr{A}_{+}$.] Assume that $\mathbf{ Z}_{1,+}^{-1}=\left[\begin{array}{cc}W^1_{1,1} & W^1_{1,2}   \\W^2_{2,1} & W^2_{2,2}  \end{array}\right]$ is invertible.
 \item[{\bf Assumption} $\mathscr{A}_{-}$.] Assume that $\mathbf{ Z}_{1,-}^{-1}=W_2$ is invertible.
 \end{description}
 \begin{Remark}\label{rem4.2}
When we consider four elementary cases mentioned in Subsection~\ref{sec4.2}, it follows from Theorem \ref{thm4.9} that  {\bf Assumptions} $\mathscr{A}_{+}$ and $\mathscr{A}_{-}$ are one of the necessary conditions for the asymptotic analysis as $t\rightarrow \infty$ and $t\rightarrow -\infty$, respectively.
 \end{Remark}

Under the {\bf Assumptions} $\mathscr{A}_{+}$ and $\mathscr{A}_{-}$, the matrices $\mathbf{ Z}_{1,+}$ and $\mathbf{ Z}_{1,-}$ exist. We partition $\mathbf{ Z}_{1,+}$ and $\mathbf{ Z}_{1,-}$ as the block forms
\begin{align}\label{eq4.38}
\begin{array}{c}
\mathbf{ Z}_{1,+}=\left[\begin{array}{cc}\mathbf{ Z}^1_{r,+} & \mathbf{ Z}^1_{ecd,+}  \end{array}\right]\}n,\vspace{-0.3cm}\\
\begin{array}{cc}\ \  \ \  \underbrace{}_{n_r}  &\  \underbrace{}_{n_{ecd}} \end{array}
\end{array}\ \
\begin{array}{c}
\mathbf{ Z}_{1,-}=\left[\begin{array}{cc}\mathbf{ Z}^1_{r,-} & \mathbf{ Z}^1_{ecd,-}  \end{array}\right]\}n.\vspace{-0.3cm}\\
\begin{array}{cc}\ \ \ \  \underbrace{}_{n_r}  &\  \underbrace{}_{n_{ecd}} \end{array}
\end{array}
\end{align}
From \eqref{eq4.9}, we have
\begin{align}\label{eq4.39}
Y(t)\mathbf{ Z}_{1,+}=\mathcal{S}e^{\mathfrak{J}t}\left[%
\begin{array}{c}
\mathbf{W}_1^{+}\\
\mathbf{W}_2^{+}\\
\end{array}%
\right],\ \ \ \ \ Y(t)\mathbf{ Z}_{1,-}=\mathcal{S}e^{\mathfrak{J}t}\left[%
\begin{array}{c}
\mathbf{W}_1^{-}\\
\mathbf{W}_2^{-}\\
\end{array}%
\right],
\end{align}
where $\mathbf{W}_j^{+}=W_j\mathbf{ Z}_{1,+}$ and $\mathbf{W}_j^{-}=W_j\mathbf{ Z}_{1,-}$, for $j=1,2$, are of the forms
\begin{subequations}\label{eq4.40}
\begin{align}
\mathbf{W}_1^+&=\left[\begin{array}{cc}I_{n_r} & 0   \\\mathbf{W}^+_{2,1} & \mathbf{W}^+_{ecd}  \end{array}\right],\ \ \ \ \mathbf{W}_2^{+}=\left[\begin{array}{cc}\mathbf{W}^{+}_{1,1} & \mathbf{W}^+_{1,2}  \\ 0 & I_{n_{ecd}}   \end{array}\right],\label{eq4.40a}\\
\mathbf{W}_1^-&=\left[\begin{array}{cc}\mathbf{W}^-_{1,1} & \mathbf{W}^-_{1,2}   \\\mathbf{W}^-_{2,1} & \mathbf{W}^-_{ecd}  \end{array}\right],\ \ \ \ \mathbf{W}_2^{-}=\left[\begin{array}{cc}I_{n_r}  & 0 \\ 0 & I_{n_{ecd}}  \end{array}\right].\label{eq4.40b}
\end{align}
\end{subequations}
In order to investigate the asymptotic behavior of $Y(t)$, we partition the symplectic matrix $\mathcal{S}$ in \eqref{eq4.8} and $e^{\mathfrak{J}t}$ in \eqref{eq4.13}, respectively, as
\begin{align}\label{eq4.41}
\mathcal{S}=\left[\begin{array}{cc|cc}U_r&U_{ecd}&V_r&V_{ecd}\end{array}\right]=\left[\begin{array}{cc|cc}U^r_1&U^{ecd}_1&V^r_1&V^{ecd}_1\\\hline U^r_2&U^{ecd}_2&V^r_2&V^{ecd}_2\end{array}\right]
\end{align}
and
\begin{align}\label{eq4.42}
e^{\mathfrak{J}t}=\left[\begin{array}{cc|cc}\mathcal{R}_r&0&0&0\\ 0&\mathcal{R}_{ecd}&0&\mathcal{D}_{ecd}\\\hline 0&0&\mathcal{R}_r^{-H}&0\\0&\mathcal{G}_{ecd}&0&\mathcal{E}_{ecd}\end{array}\right]\equiv \left[\begin{array}{cc|cc}\mathcal{R}_r(t)&0&0&0\\ 0&\mathcal{R}_{ecd}(t)&0&\mathcal{D}_{ecd}(t)\\\hline 0&0&\mathcal{R}_r(t)^{-H}&0\\0&\mathcal{G}_{ecd}(t)&0&\mathcal{E}_{ecd}(t)\end{array}\right],
\end{align}
where  $\mathcal{R}_r=e^{R_rt}$ and $R_r$ is defined in \eqref{eq4.8}.

Let $\mathfrak{r}=\min\{\Re({\rm diag}(R_r))\}>0$, where $\Re(z)$ is the real part of  $z\in \mathbb{C}$. Then
\begin{align}\label{eq4.43}
\begin{array}{l}
\mathcal{R}_r^{-1}=o(e^{-\mathfrak{r}t}t^{n_r}),\ \  \ \ \text{ as }t\rightarrow \infty,\\
\mathcal{R}_r^{H}=o(e^{-\mathfrak{r}|t|}|t|^{n_r}),\ \ \text{ as }t\rightarrow -\infty,
\end{array}
\end{align}
i.e., $\lim_{t\rightarrow \infty} t^{-n_r}e^{\mathfrak{r}t}\mathcal{R}_r^{-1}=\lim_{t\rightarrow -\infty} |t|^{-n_r}e^{\mathfrak{r}|t|}\mathcal{R}_r^{H}=0$.
\begin{Theorem}\label{thm4.16}
Assume that  $\mathscr{H}$ in \eqref{eq4.1} has Hamiltonian Jordan canonical form $\mathfrak{J}$ in \eqref{eq4.8} and the symplectic matrix $\mathcal{S}$ in \eqref{eq4.8} is of the form in \eqref{eq4.41}. Then
\begin{itemize}
\item[(i)] if {\bf Assumption} $\mathscr{A}_{+}$  holds, then there is a nonsingular matrix
\begin{subequations}\label{eq4.44}
\begin{align}\label{eq4.44a}
\mathbf{Z}_{2,+}(t)= \mathbf{Z}_{1,+}(\mathcal{R}_r^{-1}\oplus I_{n_{ecd}}),
\end{align}
such that
\begin{align}\label{eq4.44b}
Y(t)\mathbf{Z}_{2,+}(t)=\left[U_r,[U_{ecd}|V_{ecd}]\left[\begin{array}{c|c} \mathcal{R}_{ecd}&\mathcal{D}_{ecd} \\\hline
\mathcal{G}_{ecd}&\mathcal{E}_{ecd}\end{array}\right] \left[\begin{array}{c} \mathbf{W}^+_{ecd}\\\hline I_{n_{ecd}}\end{array}\right] \right]+o(e^{-\mathfrak{r}t}t^{n}),
\end{align}
\end{subequations}
as $t\rightarrow \infty$. In particular, if $\mathfrak{J}=\left[\begin{array}{c|c}R_r & 0   \\\hline 0  & -R_r^H \end{array}\right]$, then
\begin{align}\label{eq4.45}
\mathbf{Z}_{2,+}(t)=o(e^{-\mathfrak{r}t}t^{n}),\ \ Y(t)\mathbf{Z}_{2,+}(t)=U_r+o(e^{-2\mathfrak{r}t}t^{2n}),\ \text{ as }t\rightarrow \infty.
\end{align}
\item[(ii)] if {\bf Assumption} $\mathscr{A}_{-}$  holds,  then there exists an invertible matrix
\begin{subequations}\label{eq4.46}
\begin{align}\label{eq4.46a}
\mathbf{Z}_{2,-}(t)= \mathbf{Z}_{1,-}(\mathcal{R}_r^{H}\oplus I_{n_{ecd}}),
\end{align}
such that
\begin{align}\label{eq4.46b}
Y(t)\mathbf{Z}_{2,-}(t)=\left[V_r,[U_{ecd}|V_{ecd}]\left[\begin{array}{c|c} \mathcal{R}_{ecd}&\mathcal{D}_{ecd} \\\hline
\mathcal{G}_{ecd}&\mathcal{E}_{ecd}\end{array}\right] \left[\begin{array}{c} \mathbf{W}^-_{ecd}\\\hline I_{n_{ecd}}\end{array}\right] \right]+o(e^{-\mathfrak{r}|t|}|t|^{n}),
\end{align}
\end{subequations}
as $t\rightarrow -\infty$. In particular, if $\mathfrak{J}=\left[\begin{array}{c|c}R_r & 0   \\\hline 0  & -R_r^H \end{array}\right]$, then
\begin{align}\label{eq4.47}
\mathbf{Z}_{2,-}(t)=o(e^{-\mathfrak{r}|t|}|t|^{n}),\ \ Y(t)\mathbf{Z}_{2,-}(t)=V_r+o(e^{-2\mathfrak{r}|t|}|t|^{2n}),\ \text{ as }t\rightarrow -\infty.
\end{align}
\end{itemize}
\end{Theorem}
\begin{proof}
Suppose that {\bf Assumption} $\mathscr{A}_{+}$  holds. Since $\mathcal{R}_r=e^{R_rt}$ is invertible, the matrix $\mathbf{Z}_{2,+}(t)$ defined in \eqref{eq4.44a} is invertible. Plugging \eqref{eq4.40a}, \eqref{eq4.41}, \eqref{eq4.42} and \eqref{eq4.44a} into \eqref{eq4.39}, it follows from  \eqref{eq4.43} that
\begin{align}\label{eq4.48}
Y(t)\mathbf{Z}_{2,+}(t)&=\mathcal{S}\left[\begin{array}{cc|cc}\mathcal{R}_r&0&0&0\\ 0&\mathcal{R}_{ecd}&0&\mathcal{D}_{ecd}\\\hline 0&0&\mathcal{R}_r^{-H}&0\\0&\mathcal{G}_{ecd}&0&\mathcal{E}_{ecd}\end{array}\right]\left[\begin{array}{cc}\mathcal{R}_r^{-1} & 0   \\\mathbf{W}^+_{2,1}\mathcal{R}_r^{-1}  & \mathbf{W}^+_{ecd}\\\hline \mathbf{W}^{+}_{1,1}\mathcal{R}_r^{-1}  & \mathbf{W}^+_{1,2}  \\ 0 & I_{n_{ecd}}   \end{array}\right]\nonumber\\
&=\mathcal{S}\left[\begin{array}{cc}I & 0   \\o(e^{-\mathfrak{r}t}t^{n}) & \mathcal{R}_{ecd}\mathbf{W}^+_{ecd}+\mathcal{D}_{ecd}\\\hline o(e^{-2\mathfrak{r}t}t^{2n_r})  & o(e^{-\mathfrak{r}t}t^{n_r})  \\ o(e^{-\mathfrak{r}t}t^{n_r})  & \mathcal{G}_{ecd}\mathbf{W}^+_{ecd}+\mathcal{E}_{ecd}   \end{array}\right]\\
&=\left[U_r,[U_{ecd}|V_{ecd}]\left[\begin{array}{c|c} \mathcal{R}_{ecd}&\mathcal{D}_{ecd} \\\hline
\mathcal{G}_{ecd}&\mathcal{E}_{ecd}\end{array}\right] \left[\begin{array}{c} \mathbf{W}^+_{ecd}\\\hline I_{n_{ecd}}\end{array}\right] \right]+o(e^{-\mathfrak{r}t}t^{n}),\nonumber
\end{align}
as $t\rightarrow \infty$. Hence we obtain \eqref{eq4.44b}. In  particular, if $\mathfrak{J}=R_r\oplus( -R_r^H)$, then \eqref{eq4.45} can be obtained from \eqref{eq4.48} directly.

Suppose that {\bf Assumption} $\mathscr{A}_{-}$  holds. Since  $\mathcal{R}_r$ is invertible, the matrix $\mathbf{Z}_{2,-}(t)$ defined in \eqref{eq4.46a} is invertible. Plugging \eqref{eq4.40b}, \eqref{eq4.41}, \eqref{eq4.42} and \eqref{eq4.46a} into \eqref{eq4.39}, it follows from  \eqref{eq4.43} that
\begin{align}\label{eq4.49}
Y(t)\mathbf{Z}_{2,-}(t)&=\mathcal{S}\left[\begin{array}{cc|cc}\mathcal{R}_r&0&0&0\\ 0&\mathcal{R}_{ecd}&0&\mathcal{D}_{ecd}\\\hline 0&0&\mathcal{R}_r^{-H}&0\\0&\mathcal{G}_{ecd}&0&\mathcal{E}_{ecd}\end{array}\right]\left[\begin{array}{cc}\mathbf{W}^{-}_{1,1}\mathcal{R}_r^{H} & \mathbf{W}^{-}_{1,2}   \\\mathbf{W}^-_{2,1}\mathcal{R}_r^{H}  & \mathbf{W}^-_{ecd}\\\hline \mathcal{R}_r^{H}  & 0 \\ 0 & I_{n_{ecd}}   \end{array}\right]\nonumber\\
&=\mathcal{S}\left[\begin{array}{cc}o(e^{-2\mathfrak{r}|t|}|t|^{2n_r}) & o(e^{-\mathfrak{r}|t|}|t|^{n_r})   \\o(e^{-\mathfrak{r}|t|}|t|^{n}) & \mathcal{R}_{ecd}\mathbf{W}^-_{ecd}+\mathcal{D}_{ecd}\\\hline I  & 0 \\ o(e^{-\mathfrak{r}|t|}|t|^{n_r})  & \mathcal{G}_{ecd}\mathbf{W}^-_{ecd}+\mathcal{E}_{ecd}   \end{array}\right]\\
&=\left[V_r,[U_{ecd}|V_{ecd}]\left[\begin{array}{c|c} \mathcal{R}_{ecd}&\mathcal{D}_{ecd} \\\hline
\mathcal{G}_{ecd}&\mathcal{E}_{ecd}\end{array}\right] \left[\begin{array}{c} \mathbf{W}^-_{ecd}\\\hline I_{n_{ecd}}\end{array}\right] \right]+o(e^{-\mathfrak{r}|t|}|t|^{n}),\nonumber
\end{align}
as $t\rightarrow -\infty$. Hence we obtain \eqref{eq4.46b}. In  particular, if $\mathfrak{J}=R_r\oplus( -R_r^H)$, then \eqref{eq4.47} can be obtained from \eqref{eq4.49} directly.
\end{proof}

By \eqref{eq4.45} and \eqref{eq4.47}, we have the following consequence.
\begin{Corollary}\label{cor4.17}
With the same notations of Theorem \ref{thm4.16}, suppose that  {\bf Assumptions} $\mathscr{A}_{+}$ and  $\mathscr{A}_{-}$ hold and $\mathfrak{J}=\left[\begin{array}{c|c}R_r & 0   \\\hline 0  & -R_r^H \end{array}\right]$.
Let $\mathbf{U}_j=U_j^r$, $\mathbf{V}_j=V_j^r$ for $j=1,2$ and $W(t)= P(t)Q(t)^{-1}$, where $Y(t)=[Q(t)^{\top},  P(t)^{\top}]^{\top}$ is the solution of IVP \eqref{eq4.1}.
If $\mathbf{U}_1$ and $\mathbf{V}_1$ are invertible, then
\begin{align*}
\begin{array}{lll}
W(t)=\mathbf{U}_2\mathbf{U}_1^{-1}+O(e^{-2\mathfrak{r}t}t^{2n}),& Q(t)^{-1}=O(e^{-\mathfrak{r}t}t^{n}),& \text{ as }t\rightarrow \infty,\\
W(t)=\mathbf{V}_2\mathbf{V}_1^{-1}+O(e^{-2\mathfrak{r}|t|}|t|^{2n}),& Q(t)^{-1}=O(e^{-\mathfrak{r}|t|}|t|^{n}),& \text{ as }t\rightarrow -\infty,
\end{array}
\end{align*}
where $\mathfrak{r}=\min\{\Re({\rm diag}(R_r))\}>0$. Here, $\mathbf{U}_2\mathbf{U}_1^{-1}$ and $\mathbf{V}_2\mathbf{V}_1^{-1}$ are Hermitian.
\end{Corollary}

Let
\begin{align}\label{eq4.50}
Y_{ecd,\pm}(t)=[U_{ecd}|V_{ecd}]\left[\begin{array}{c|c} \mathcal{R}_{ecd}&\mathcal{D}_{ecd} \\\hline
\mathcal{G}_{ecd}&\mathcal{E}_{ecd}\end{array}\right] \left[\begin{array}{c} \mathbf{W}^{\pm}_{ecd}\\\hline I_{n_{ecd}}\end{array}\right].
\end{align}
From \eqref{eq4.44b} and \eqref{eq4.46b}, we need to simplify $Y_{ecd,\pm}(t)$ for checking the linear independence of its column space, as  $t\rightarrow \pm\infty$.  Plugging \eqref{eq4.38} into \eqref{eq4.44a} and  \eqref{eq4.46a}, it follows from \eqref{eq4.43} that
\begin{align}\label{eq4.51}
\begin{array}{ll}
\mathbf{Z}_{2,+}(t)=\left[o(e^{-\mathfrak{r}t}t^{n_r}),\mathbf{Z}_{ecd,+}^1\right], &\text{ as }t\rightarrow \infty,\\
\mathbf{Z}_{2,-}(t)=\left[o(e^{-\mathfrak{r}|t|}|t|^{n_r}),\mathbf{Z}_{ecd,-}^1\right], &\text{ as }t\rightarrow -\infty.
\end{array}
\end{align}
Partition $\mathbf{Z}^1_{ecd,\pm}$ in \eqref{eq4.38}, $\mathbf{W}_{ecd}^{\pm}$ in \eqref{eq4.40} and $U_{ecd}$, $V_{ecd}$ in \eqref{eq4.41}, respectively, as
\begin{align}\label{eq4.52}
\begin{array}{l}
\mathbf{ Z}^1_{ecd,\pm}=\left[\begin{array}{cc}\mathbf{ Z}^1_{e,\pm} & \mathbf{ Z}^1_{cd,\pm}  \end{array}\right]\}n,\vspace{-0.3cm}\\
\begin{array}{cc}\hspace{1.8cm}  \underbrace{}_{n_e}  &\  \underbrace{}_{n_{cd}} \end{array}\\
\mathbf{W}_{ecd}^{\pm}=\left[\begin{array}{cc}\mathbf{W}_{2,2}^{\pm} & \mathbf{W}_{2,3}^{\pm} \\\mathbf{W}_{3,2}^{\pm}&\mathbf{W}_{cd}^{\pm}\end{array}\right]\hspace{-0.2cm}\begin{array}{l}\}n_{e}\\ \}n_{cd}\end{array}\vspace{-0.3cm}\\
\begin{array}{cc}\hspace{1.8cm}  \underbrace{}_{n_e}  &\  \underbrace{}_{n_{cd}} \end{array}
\end{array}
\end{align}
and
\begin{align}\label{eq4.53}
U_{ecd}=\left[U_{e},U_{cd}\right]=\left[\begin{array}{cc}U^e_1 & U^{cd}_1 \\\hline U^e_2 & U^{cd}_2\end{array}\right],\ \
V_{ecd}=\left[V_{e},V_{cd}\right]=\left[\begin{array}{cc}V^e_1 & V^{cd}_1 \\\hline V^e_2 & V^{cd}_2\end{array}\right].
\end{align}
Let
\begin{align}\label{eq4.54}
\left[\begin{array}{c|c}\mathcal{R}_{cd} & \mathcal{D}_{cd} \\\hline \mathcal{G}_{cd} & \mathcal{E}_{cd}\end{array}\right]\equiv \left[\begin{array}{c|c}\mathcal{R}_{cd}(t) & \mathcal{D}_{cd}(t)  \\\hline \mathcal{G}_{cd}(t)  & \mathcal{E}_{cd}(t) \end{array}\right]=\left[\begin{array}{c|c}\mathcal{R}_c\oplus\mathcal{R}_d&\mathcal{D}_c\oplus\mathcal{D}_d\\\hline 0\oplus\mathcal{G}_d&\mathcal{R}_c^{-H}\oplus\mathcal{E}_d\end{array}\right],
\end{align}
where $\mathcal{R}_c$, $\mathcal{R}_d$, $\mathcal{D}_c$, $\mathcal{D}_d$, $\mathcal{G}_d$ and $\mathcal{E}_d$ are shown in Theorem \ref{thm4.8}. Then
\begin{align*}
\left[\begin{array}{c|c}\mathcal{R}_{ecd} & \mathcal{D}_{ecd} \\\hline \mathcal{G}_{ecd} & \mathcal{E}_{ecd}\end{array}\right]=\left[\begin{array}{c|c}\mathcal{R}_e\oplus\mathcal{R}_{cd}&\mathcal{D}_e\oplus\mathcal{D}_{cd}\\\hline 0\oplus\mathcal{G}_{cd}&\mathcal{R}_e^{-H}\oplus\mathcal{E}_{cd}\end{array}\right],
\end{align*}
where $\mathcal{R}_e\equiv \mathcal{R}_e(t)$ and $\mathcal{D}_e\equiv \mathcal{D}_e(t)$ are shown in \eqref{eq4.13}. Denote
\begin{align}\label{eq4.55}
 \mathcal{T}_{e,\pm}\equiv \mathcal{T}_{e,\pm}(t)=\mathcal{R}_e\mathbf{W}^{\pm}_{2,2}+\mathcal{D}_e,
\end{align}
where $\mathbf{W}^{\pm}_{2,2}$ is given in \eqref{eq4.52}. The proof of the following lemma is left in Appendix.

\begin{Lemma}\label{lem4.18}
Let $\mathcal{R}_{cd}$, $\mathcal{D}_{cd}$, $\mathcal{G}_{cd}$,  $\mathcal{E}_{cd}$ and $ \mathcal{T}_{e,\pm}$ be of the  forms in \eqref{eq4.54} and \eqref{eq4.55}, respectively.  Let
\begin{align}\label{eq4.56}
\mathcal{I}_{\pm}=\{t\in \mathbb{R}| \mathcal{T}_{e,\pm}  \text{ and } \mathcal{R}_{cd}(\mathbf{W}^{\pm}_{cd}-\mathbf{W}^{\pm}_{3,2} \mathcal{T}_{e,\pm}^{-1}\mathcal{R}_{e}\mathbf{W}^{\pm}_{2,3})+\mathcal{D}_{cd} \text{ are invertible}\},
\end{align}
where  $\mathbf{W}^{\pm}_{cd}$, $\mathbf{W}^{\pm}_{3,2}$ and $\mathbf{W}^{\pm}_{2,3}$ are given in \eqref{eq4.52}. Then there are nonsingular matrices, $Z_{ecd,+}(t)$ for $t\in \mathcal{I}_{+}$ and $Z_{ecd,-}(t)$ for $t\in \mathcal{I}_{-}$, of the forms
\begin{subequations}\label{eq4.57}
\begin{align}\label{eq4.57a}
\begin{array}{l}
Z_{ecd,\pm}(t)=\left[\begin{array}{cc}O(t^{-1})&O(t^{-1})\\O(t^{-1})&I_{n_{cd}}  \end{array}\right],\ \ \ \text{ as }t\rightarrow \pm\infty,\vspace{-0.3cm}\\
\begin{array}{cc} \hspace{2.4cm}      \underbrace{}_{n_e}& \ \ \ \  \underbrace{}_{n_{cd}} \end{array}
\end{array}
\end{align}
such that
\begin{align}\label{eq4.57b}
Y_{ecd,\pm}(t)Z_{ecd,\pm}(t)=\left[U_e,[U_{cd}|V_{cd}]\left[\begin{array}{c|c} \mathcal{R}_{cd}&\mathcal{D}_{cd} \\\hline
\mathcal{G}_{cd}&\mathcal{E}_{cd}\end{array}\right] \left[\begin{array}{c} \mathbf{W}^{\pm}_{cd}+O(t^{-1})\\\hline I\end{array}\right] \right]+O(t^{-1}),
\end{align}
\end{subequations}
as $t\rightarrow \pm\infty$, respectively, where $U_e$, $U_{cd}$ and $V_{cd}$ are given in \eqref{eq4.53} and $Y_{ecd,\pm}(t)$ and   $\mathbf{W}^{\pm}_{cd}$ are given in \eqref{eq4.50} and \eqref{eq4.52}, respectively.
\end{Lemma}

Denote
\begin{align*}
\mathbf{ Z}_{3,\pm}(t)=\mathbf{ Z}_{2,\pm}(t)(I_{n_r}\oplus Z_{ecd,\pm}(t)), \text{ for }t\in \mathcal{I}_{\pm}.
\end{align*}
From Theorem \ref{thm4.16}, Lemma \ref{lem4.18} and \eqref{eq4.51}, we then have the theorem.

\begin{Theorem}\label{thm4.19}
With the same notations of Theorem \ref{thm4.16}, where $U_{ecd}$ and $V_{ecd}$ are of the forms in \eqref{eq4.53}. Let $\mathfrak{r}=\min\{\Re({\rm diag}(R_r))\}>0$.
Then
\begin{itemize}
\item[(i)] if {\bf Assumption} $\mathscr{A}_{+}$  holds, then there is a nonsingular matrix $\mathbf{Z}_{3,+}(t)$ with
\begin{subequations}\label{eq4.58}
\begin{align}\label{eq4.58a}
\begin{array}{l}
\mathbf{ Z}_{3,+}(t)=\left[\begin{array}{ccc}o(e^{-\mathfrak{r}t}t^{n_r})&O(t^{-1})&\mathbf{ Z}^1_{cd,+} +O(t^{-1})   \end{array}\right]\vspace{-0.3cm}\\
\begin{array}{ccc}\hspace{2cm}    \underbrace{\ \ \ \ \ \ \ \ \ }_{n_r}  & \ \ \ \  \underbrace{}_{n_e}& \ \ \underbrace{\ \ \ \ \ \ \ \ \ \ \ \ \ \ \ }_{n_{cd}} \end{array}\end{array}
\end{align}
for $t\in \mathcal{I}_{+}$  such that
\begin{align}\label{eq4.58b}
Y(t)\mathbf{Z}_{3,+}(t)=\left[U_r+o(e^{-\mathfrak{r}t}t^{n}),\left[U_e,[U_{cd}|V_{cd}]\left[\begin{array}{c|c} \mathcal{R}_{cd}&\mathcal{D}_{cd} \\\hline
\mathcal{G}_{cd}&\mathcal{E}_{cd}\end{array}\right] \left[\begin{array}{c} \mathbf{W}^{+}_{cd}+O(t^{-1})\\\hline I\end{array}\right] \right]+O(t^{-1})\right],
\end{align}
\end{subequations}
as $t\rightarrow \infty$, where $\mathbf{ Z}^1_{cd,+}$ and $\mathbf{W}^{+}_{cd}$ are given in \eqref{eq4.52};
\item[(ii)] if {\bf Assumption} $\mathscr{A}_{-}$  holds,  then there is a nonsingular matrix $\mathbf{Z}_{3,-}(t)$ with
\begin{subequations}\label{eq4.59}
\begin{align}\label{eq4.59a}
\begin{array}{l}
\mathbf{ Z}_{3,-}(t)=\left[\begin{array}{ccc}o(e^{-\mathfrak{r}|t|}|t|^{n_r})&O(t^{-1})&\mathbf{ Z}^1_{cd,-} +O(t^{-1})   \end{array}\right]\vspace{-0.3cm}\\
\begin{array}{ccc}\hspace{2cm}    \underbrace{\ \ \ \ \ \ \ \  \ \ \ }_{n_r}  & \ \ \ \  \underbrace{}_{n_e}& \  \ \ \underbrace{\ \ \ \ \ \ \ \ \ \ \ \ \ \ \ }_{n_{cd}} \end{array}\end{array}
\end{align}
for $t\in \mathcal{I}_{-}$ such that
\begin{align}\label{eq4.59b}
Y(t)\mathbf{Z}_{3,-}(t)=\left[V_r+o(e^{-\mathfrak{r}|t|}|t|^{n}),\left[U_e,[U_{cd}|V_{cd}]\left[\begin{array}{c|c} \mathcal{R}_{cd}&\mathcal{D}_{cd} \\\hline
\mathcal{G}_{cd}&\mathcal{E}_{cd}\end{array}\right] \left[\begin{array}{c} \mathbf{W}^{-}_{cd}+O(t^{-1})\\\hline I\end{array}\right] \right]+O(t^{-1})\right],
\end{align}
\end{subequations}
as $t\rightarrow -\infty$, where $\mathbf{ Z}^1_{cd,-}$ and $\mathbf{W}^{-}_{cd}$ are given in \eqref{eq4.52}.
\end{itemize}
\end{Theorem}

In the case that $\mathcal{R}_{cd}$, $\mathcal{D}_{cd}$, $\mathcal{G}_{cd}$ and $\mathcal{E}_{cd}$  are absent, we have an immediate consequence.
\begin{Corollary}\label{cor4.20}
With the same notations of Theorem \ref{thm4.19},  suppose that  {\bf Assumptions} $\mathscr{A}_{+}$ and  $\mathscr{A}_{-}$ hold and
\begin{align}\label{eq4.60}
\mathfrak{J}=\mathfrak{J}_{re}\equiv\left[\begin{array}{cc|cc}R_r & 0 & 0 & 0 \\ 0 & R_e & 0 & D_e \\\hline 0 & 0 & -R_r^H & 0 \\ 0 & 0 & 0 & -R_e^H\end{array}\right].
\end{align}
Let $\mathbf{U}_{j,+}=[U_j^r,U_j^e]$, $\mathbf{U}_{j,-}=[V_j^r,U_j^e]$ for $j=1,2$ and $W(t)= P(t)Q(t)^{-1}$, where  $Y(t)=[Q(t)^{\top} ,P(t)^{\top} ]^{\top}$ is the solution of IVP \eqref{eq4.1}. If $\mathbf{U}_{1,+}$ and $\mathbf{U}_{1,-}$ are  invertible, then
\begin{align*}
\begin{array}{lll}
W(t)=\mathbf{U}_{2,+}\mathbf{U}^{-1}_{1,+}+O(t^{-1}), & Q(t)^{-1}=O(t^{-1}),&\text{ as }t\rightarrow \infty,\\
W(t)=\mathbf{U}_{2,-}\mathbf{U}^{-1}_{1,-}+O(t^{-1}), & Q(t)^{-1}=O(t^{-1}),&\text{ as }t\rightarrow -\infty.
\end{array}
\end{align*}
Here, $\mathbf{U}_{2,\pm}\mathbf{U}^{-1}_{1,\pm}$ is Hermitian.
\end{Corollary}

Suppose that  those submatrices $R_{d}$, $D_{d}$ and $G_{d}$ of $\mathfrak{J}$ in \eqref{eq4.8}  are absent and that $\mathfrak{J}_c=\left[\begin{array}{cc}R_{c}&D_{c} \\0&-R_{c}^H\end{array}\right]$ is of the elementary case with the form in \eqref{eq4.14}, where $R_{c}$ and $D_{c}$ are submatrices of $\mathfrak{J}$. Then $U_{cd}=U_c$ and $V_{cd}=V_c$. Partition
\begin{align}\label{eq4.61}
U_c\equiv\left[\begin{array}{c}U^c_1 \\\hline U^c_2\end{array}\right]=\left[\begin{array}{cc}U^c_{1,1}&u^c_{1,2} \\\hline U^c_{2,1}&u^c_{2,2}\end{array}\right],\ \ V_c\equiv\left[\begin{array}{c}V^c_1 \\\hline V^c_2\end{array}\right]=\left[\begin{array}{cc}V^c_{1,1}&v^c_{1,2} \\\hline V^c_{2,1}&v^c_{2,2}\end{array}\right].
\end{align}
We state the corollary but omit its proof, as it is an easy combination of Theorems  \ref{thm4.15} and \ref{thm4.19}.
\begin{Corollary}\label{cor4.21}
With the same notations of Theorem \ref{thm4.19},  suppose that  {\bf Assumptions} $\mathscr{A}_{+}$ and  $\mathscr{A}_{-}$ hold, and
\begin{subequations}\label{eq4.62}
\begin{align}\label{eq4.62a}
\mathfrak{J}=\left[\begin{array}{ccc|ccc}R_r & 0 & 0 & 0&0 & 0 \\ 0 & R_e & 0 & 0& D_e& 0\\ 0& 0 & R_c  & 0& 0& D_c \\\hline 0 &0 & 0 & -R_r^H & 0 & 0\\ 0 &0 & 0 & 0 & -R_e^H &0\\0 &0 & 0 & 0 &0& -R_c^H \end{array}\right],
\end{align}
where
\begin{align}\label{eq4.62b}
 \left[\begin{array}{cc}R_{c}&D_{c} \\0&-R_{c}^H\end{array}\right] \text{ is of the elementary case with }\sigma(R_{c})=\{i\eta\}.
\end{align}
\end{subequations}
Let $f_u^{\pm}$, $g_u^{\pm}$, $f_v^{\pm}$ and  $g_v^{\pm}$ be the constants defined in \eqref{eq4.21a} with $\mathbf{W}$ being replaced by $\mathbf{W}_{cd}^\pm$, where $\mathbf{W}_{cd}^\pm$ is given in \eqref{eq4.52}. Denote
\begin{align*}
\begin{array}{l}
\mathbf{U}_{1,0}^{c,\pm}=[U_{1,1}^c,(f_u^{\pm}+g_u^{\pm})u_{1,2}^c+(f_v^{\pm}+g_v^{\pm})v_{1,2}^c],\\ \mathbf{U}_{2,0}^{c,\pm}=[U_{2,1}^c,(f_u^{\pm}+g_u^{\pm})u_{2,2}^c+(f_v^{\pm}+g_v^{\pm})v_{2,2}^c],
\end{array}
\end{align*}
where $U_{1,1}^c$, $U_{2,1}^c$, $u_{1,2}^c$, $u_{2,2}^c$, $v_{1,2}^c$ and  $v_{2,2}^c$ are defined in \eqref{eq4.61}. Let
\begin{align}\label{eq4.63}
\mathbf{U}_{j,+}=[U_j^r,U_j^e,\mathbf{U}_{j,0}^{c,+}],\ \ \mathbf{U}_{j,-}=[V_j^r,U_j^e,\mathbf{U}_{j,0}^{c,-}],
\end{align}
for $j=1,2$.   If $\mathbf{U}_{1,+}$ and $\mathbf{U}_{1,-}$ are  invertible, then
\begin{align*}
\begin{array}{lll}
W(t)=\mathbf{U}_{2,+}\mathbf{U}^{-1}_{1,+}+O(t^{-1}), & Q(t)^{-1}=e^{-i\eta t}\mathbf{Z}_{cd,+}^1e_{n_{cd}}e_n^{H}\mathbf{U}^{-1}_{1,+}+O(t^{-1}),&\text{ as }t\rightarrow \infty,\\
W(t)=\mathbf{U}_{2,-}\mathbf{U}^{-1}_{1,-}+O(t^{-1}), & Q(t)^{-1}=e^{-i\eta t}\mathbf{Z}_{cd,-}^1e_{n_{cd}}e_n^{H}\mathbf{U}^{-1}_{1,-}+O(t^{-1}),&\text{ as }t\rightarrow -\infty,
\end{array}
\end{align*}
where $\mathbf{Z}_{cd,\pm}^1$ is defined in \eqref{eq4.52}. Here $\mathbf{U}_{2,\pm}\mathbf{U}^{-1}_{1,\pm}$ is Hermitian.
\end{Corollary}

From \eqref{eq4.58b} and \eqref{eq4.59b}, we need to simplify the linear independence of the column space of
\begin{align}\label{eq4.64}
Y_{cd,\pm}(t)=\left[U_{cd}|V_{cd}\right]\left[\begin{array}{c|c} \mathcal{R}_{cd}&\mathcal{D}_{cd} \\\hline
\mathcal{G}_{cd}&\mathcal{E}_{cd}\end{array}\right] \left[\begin{array}{c} \mathbf{W}^{\pm}_{cd}+O(t^{-1})\\\hline I\end{array}\right],
\end{align}
as $t\rightarrow \pm \infty$, respectively. Let
\begin{align}\label{eq4.65}
\begin{array}{ll}
\mathcal{R}_{cd}\equiv \mathcal{R}_{cd}(t)={\rm diag}(\mathbf{B}_{1},\ldots,\mathbf{B}_{\mu}),&\mathcal{D}_{cd}\equiv \mathcal{D}_{cd}(t)={\rm diag}(\mathbf{D}_{1},\ldots,\mathbf{D}_{\mu}),\\
\mathcal{G}_{cd}\equiv \mathcal{G}_{cd}(t)={\rm diag}(\mathbf{G}_{1},\ldots,\mathbf{G}_{\mu}),&\mathcal{E}_{cd}\equiv \mathcal{E}_{cd}(t)={\rm diag}(\mathbf{E}_{1},\ldots,\mathbf{E}_{\mu}),
\end{array}
\end{align}
where $\mu=\mu_c+\mu_d$,
\begin{align}\label{eq4.66}
\begin{array}{ll}
\mathbf{ B}_\ell\equiv \mathbf{ B}_\ell(t)=\left[\begin{array}{cc}\Phi_{m_{\ell},n_{\ell}}& \phi_{m_{\ell},n_{\ell}}^{1}  \\0&\omega_{11}^\ell\end{array}\right],&
\mathbf{ D}_\ell\equiv \mathbf{ D}_\ell(t)=\left[\begin{array}{cc}\widehat{\Gamma}_{m_{\ell}+1,n_{\ell}+1}^{2m_{\ell},2n_{\ell}}& \phi_{m_{\ell},n_{\ell}}^{2}  \\\widehat{\psi}_{m_{\ell},n_{\ell}}^{1^H} &\omega_{12}^\ell\end{array}\right],\\
\mathbf{ G}_\ell\equiv \mathbf{ G}_\ell(t)=\left[\begin{array}{cc}0& 0 \\0 &\omega_{21}^\ell\end{array}\right],&\mathbf{ E}_\ell\equiv \mathbf{ E}_\ell(t)=\left[\begin{array}{cc}\widehat{\Phi}_{m_{\ell},n_{\ell}}& 0 \\\widehat{\psi}_{m_{\ell},n_{\ell}}^{2^H}&\omega_{22}^\ell\end{array}\right],
\end{array}
\end{align}
and $\Phi_{m_{\ell},n_{\ell}}$, $\widehat{\Gamma}_{m_{\ell}+1,n_{\ell}+1}^{2m_{\ell},2n_{\ell}}$, $\phi_{m_{\ell},n_{\ell}}^{1}$, $\phi_{m_{\ell},n_{\ell}}^{2}$, $\widehat{\psi}_{m_{\ell},n_{\ell}}^{1^H}$, $\widehat{\psi}_{m_{\ell},n_{\ell}}^{2^H}$, $\omega^\ell_{11}$, $\omega_{12}^\ell$, $\omega_{21}^\ell$ and $\omega_{22}^\ell$ are defined in \eqref{eq4.12} in which  $\gamma$ and $\eta$ are  replaced by $\gamma_\ell$ and $\delta_\ell$, respectively,  and $\beta$ is replaced by $\beta^{cd}_{\ell}\in \{-1,1\}$ for $\ell=1,\ldots,\mu$. Note that  $\gamma_\ell=\delta_\ell$, $\beta^{cd}_{\ell}=\beta^{c}_{\ell}$ when $\ell\leqslant \mu_c$ and $\gamma_\ell\neq \delta_\ell$, $\beta^{cd}_{\ell}=\beta^{d}_{\ell}$ when $\mu_c< \ell\leqslant \mu$ . Let $\varkappa_\ell=m_\ell+n_\ell+1$.
Denote
\begin{align}\label{eq4.67}
\begin{array}{l}
\hat{\omega}^{\ell}_{11}\equiv \hat{\omega}^{\ell}_{11}(t)=\frac{1}{2}[(-1)^{m_{\ell}}e^{i\gamma_\ell t}+(-1)^{n_{\ell}}e^{i\delta_\ell t}],\\
\hat{\omega}^{\ell}_{12}\equiv \hat{\omega}^{\ell}_{12}(t)=\frac{1}{2}[-i\beta^{cd}_\ell ((-1)^{m_{\ell}}e^{i\gamma_\ell t}-(-1)^{n_{\ell}}e^{i\delta_\ell t})],\\
\hat{\omega}^{\ell}_{21}\equiv \hat{\omega}^{\ell}_{21}(t)=\frac{1}{2}[i\beta^{cd}_\ell ((-1)^{m_{\ell}}e^{i\gamma_\ell t}-(-1)^{n_{\ell}}e^{i\delta_\ell t})],\\
\hat{\omega}^{\ell}_{22}\equiv \hat{\omega}^{\ell}_{22}(t)=\frac{1}{2}[(-1)^{m_{\ell}}e^{i\gamma_\ell t}+(-1)^{n_{\ell}}e^{i\delta_\ell t}],\\
\end{array}
\end{align}
where $\ell\in\{1,2,\ldots,\mu\}$.

In the following lemma, we consider the special case with $\mu=2$, i.e.,  $\mathcal{R}_{cd}$, $\mathcal{D}_{cd}$, $\mathcal{G}_{cd}$ and $\mathcal{E}_{cd}$ in \eqref{eq4.65} have 2 diagonal blocks. The proof is left in Appendix. For the general case, a similar result can be obtained by using the same procedure of proof.  

\begin{Lemma}\label{lem4.22}
Suppose that
\begin{align}\label{eq4.68}
\mathcal{R}_{cd}=\mathbf{B}_{1}\oplus\mathbf{B}_{2},\ \  \mathcal{D}_{cd}=\mathbf{D}_{1}\oplus\mathbf{D}_{2},\ \  \mathcal{G}_{cd}=\mathbf{G}_{1}\oplus\mathbf{G}_{2},\ \ \mathcal{E}_{cd}=\mathbf{E}_{1}\oplus\mathbf{E}_{2},
\end{align}
where $\mathbf{ B}_j,\mathbf{ D}_j,\mathbf{ G}_j,\mathbf{ E}_j\in \mathbb{C}^{\varkappa_j\times \varkappa_j}$ are defined in \eqref{eq4.66}.
Let
\begin{align}\label{eq4.69}
\begin{array}{cr}
\mathbf{ W}=\left[\begin{array}{c|c}\mathbf{ W}_{11} & \mathbf{ W}_{12}  \\\hline \mathbf{ W}_{21}  & \mathbf{ W}_{22} \end{array}\right]&\!\!\!\!\!\!\!\!\begin{array}{c}\}\varkappa_1\\
\}\varkappa_2\end{array}\vspace{-0.3cm}\\
\ \ \ \ \ \ \  \underbrace{}_{\varkappa_1}\ \ \underbrace{}_{\varkappa_2}&
\end{array}
\end{align}
be a constant matrix.  Then there is a nonsingular matrix $\Omega(t)$ with $$\Omega(t)=\left[\begin{array}{c|c}e_{\varkappa_1}e_{\varkappa_1}^{H}&0\\\hline0&e_{\varkappa_2}e_{\varkappa_2}^{H}\end{array}\right]+O(t^{-1})$$  such that
\begin{align}\label{eq4.70}
\left[\begin{array}{c|c}\mathcal{R}_{cd} &\mathcal{D}_{cd} \\\hline \mathcal{G}_{cd} & \mathcal{E}_{cd}\end{array}\right]\left[\begin{array}{c}\mathbf{ W}\\\hline  I \end{array}\right]\Omega(t)=\left[\begin{array}{cc|cc}I&0&0&0 \\0&\hat{\omega}^{1}_{11} \mathbf{w}_{11}+\hat{\omega}^{1}_{12}&0&\hat{\omega}^{1}_{11} \mathbf{ w}_{12}\\ 0&0&I&0\\0&\hat{\omega}^2_{11}\mathbf{w}_{21}&0&\hat{\omega}^{2}_{11} \mathbf{w}_{22}+\hat{\omega}^{2}_{12}\\\hline
0&0&0&0 \\0&\hat{\omega}^{1}_{21} \mathbf{w}_{11}+\hat{\omega}^{1}_{22}&0&\hat{\omega}^{1}_{21} \mathbf{ w}_{12}\\ 0&0&0&0\\0&\hat{\omega}^{2}_{21} \mathbf{w}_{21}&0&\hat{\omega}^{2}_{21} \mathbf{w}_{22}+\hat{\omega}^{2}_{22} \end{array}\right]+O(t^{-1}),
\end{align}
as $t\rightarrow \pm\infty$, where $\mathbf{w}_{jk}=\mathbf{W}_{jk}(\varkappa_j,\varkappa_k)\in \mathbb{C}$ and $\hat{\omega}^{\ell}_{jk}$ are given in \eqref{eq4.67} for $\ell,j,k=1,2$.
\end{Lemma}

Partition $\mathbf{ W}_{cd}^{\pm}$, $\mathbf{Z}_{cd,\pm}^1$ in \eqref{eq4.52}, $U_{cd}$, $V_{cd}$ in \eqref{eq4.53} and identity matrix $I_{n_{cd}}$, respectively, as
\begin{align*}
\begin{array}{cr}
\mathbf{ W}_{cd}^\pm=\left[\begin{array}{cccc}\mathbf{ W}_{cd,11}^{\pm} &\mathbf{ W}_{cd,12}^{\pm}   &\cdots  & \mathbf{ W}_{cd,1\mu}^{\pm} \\
\mathbf{ W}_{cd,21}^{\pm} &\mathbf{ W}_{cd,22}^{\pm}   &\cdots  & \mathbf{ W}_{cd,2\mu}^{\pm}\\\vdots&\vdots&&\vdots\\
\mathbf{ W}_{cd,\mu1}^{\pm} &\mathbf{ W}_{cd,\mu2}^{\pm}  &\cdots  & \mathbf{ W}_{cd,\mu\mu}^{\pm}\end{array}\right]&\!\!\!\!\!\!\!\!\begin{array}{c}\}\varkappa_1\\
\}\varkappa_2\\\vdots
\\\}\varkappa_{\mu}\end{array}\vspace{-0.3cm}\\
\ \ \ \ \ \ \ \ \ \   \underbrace{\ \ \ \ \ \ \  }_{\varkappa_1}\ \ \  \ \underbrace{\ \ \ \ \ \ \  }_{\varkappa_2}\ \hspace{1.3cm}  \underbrace{\ \ \ \ \ \ \ }_{\varkappa_\mu}&
\end{array}
\end{align*}
\begin{align*}
\begin{array}{l}
\mathbf{Z}_{cd,\pm}^1=\left[\begin{array}{cc|cc|c|cc}\mathbf{Z}_{cd,1}^{\pm} &\mathbf{z}_{cd,1}^{\pm} &\mathbf{Z}_{cd,2}^{\pm}&\mathbf{z}_{cd,2}^{\pm}     &\cdots  & \mathbf{Z}_{cd,\mu}^{\pm} &\mathbf{z}_{cd,\mu}^{\pm} \end{array}\right]\}n,
\vspace{-0.2cm}\\
\hspace{1.8cm} \underbrace{}_{\varkappa_1-1}\  \ \underbrace{}_{1}\ \ \ \underbrace{}_{\varkappa_2-1}\ \ \underbrace{}_{1}\ \hspace{1.1cm}  \underbrace{}_{\varkappa_\mu-1}\ \ \underbrace{}_{1} \\
U_{cd}\equiv \left[\begin{array}{c}U^{cd}_{1}\\\hline U^{cd}_2\end{array}\right]=\left[\begin{array}{cc|cc|c|cc}U^{cd}_{1,1} &u^{cd}_{1,1} &U^{cd}_{2,1}&u^{cd}_{2,1}     &\cdots  & U^{cd}_{\mu,1} &u^{cd}_{\mu,1} \\\hline U^{cd}_{1,2}&u^{cd}_{1,2} &U^{cd}_{2,2} &u^{cd}_{2,2}   &\cdots  & U^{cd}_{\mu,2}&u^{cd}_{\mu,2} \end{array}\right],
\vspace{-0.2cm}\\
\hspace{3.5cm}   \underbrace{}_{\varkappa_1-1}\ \underbrace{}_{1}\ \  \underbrace{}_{\varkappa_2-1}\ \underbrace{}_{1}\ \hspace{1.1cm}  \underbrace{}_{\varkappa_\mu-1}\ \underbrace{}_{1} \\
V_{cd}\equiv \left[\begin{array}{c}V^{cd}_{1}\\\hline V^{cd}_2\end{array}\right]=\left[\begin{array}{cc|cc|c|cc}V^{cd}_{1,1} &v^{cd}_{1,1} &V^{cd}_{2,1}&v^{cd}_{2,1}     &\cdots  & V^{cd}_{\mu,1} &v^{cd}_{\mu,1} \\\hline V^{cd}_{1,2}&v^{cd}_{1,2} &V^{cd}_{2,2} &v^{cd}_{2,2}   &\cdots  & V^{cd}_{\mu,2}&v^{cd}_{\mu,2} \end{array}\right],
\vspace{-0.2cm}\\
\hspace{3.5cm}     \underbrace{}_{\varkappa_1-1}\ \underbrace{}_{1}\   \underbrace{}_{\varkappa_2-1}\ \underbrace{}_{1}\ \hspace{0.9cm}  \underbrace{}_{\varkappa_\mu-1}\ \underbrace{}_{1}\\
I_{n_{cd}}=\left[\begin{array}{cc|cc|c|cc}\mathbf{I}^{cd}_{1} &e_{\varkappa_1} &\mathbf{I}^{cd}_{2}&e_{\varkappa_1+\varkappa_2}     &\cdots  & \mathbf{I}^{cd}_{\mu}&e_{n_{cd}} \end{array}\right]\}n_{cd}.
\vspace{-0.2cm}\\
\hspace{1.3cm} \underbrace{}_{\varkappa_1-1} \ \underbrace{}_{1} \ \underbrace{}_{\varkappa_2-1}\ \ \underbrace{}_{1}\ \hspace{1.1cm}  \underbrace{}_{\varkappa_\mu-1}\ \underbrace{}_{1}
\end{array}
\end{align*}
Then we denote some constant matrices
\begin{align}\label{eq4.71}
\begin{array}{l}
\widehat{U}_{cd}\equiv \left[\begin{array}{c}\widehat{U}^{cd}_1\\\hline\widehat{U}^{cd}_2 \end{array}\right]=\left[\begin{array}{c|c|c|c}U^{cd}_{1,1}  &U^{cd}_{2,1}   &\cdots  & U^{cd}_{\mu,1}  \\\hline U^{cd}_{1,2}&U^{cd}_{2,2}    &\cdots  & U^{cd}_{\mu,2}\end{array}\right],\\
\mathfrak{U}_{cd}\equiv \left[\begin{array}{c|c}\mathfrak{U}^{cd}_{u,1}& \mathfrak{U}^{cd}_{v,1}\\\hline \mathfrak{U}^{cd}_{u,2}& \mathfrak{U}^{cd}_{v,2}\end{array}\right]=\left[\begin{array}{cccc|cccc}u^{cd}_{1,1} &u^{cd}_{2,1}    &\cdots  & u^{cd}_{\mu,1} &v^{cd}_{1,1} &v^{cd}_{2,1}    &\cdots  & v^{cd}_{\mu,1} \\\hline u^{cd}_{1,2} &u^{cd}_{2,2}    &\cdots  & u^{cd}_{\mu,2} &v^{cd}_{1,2} &v^{cd}_{2,2}    &\cdots  & v^{cd}_{\mu,2} \end{array}\right],\\
\mathfrak{U}^{cd}_u=\left[\begin{array}{c}\mathfrak{U}^{cd}_{u,1} \\\mathfrak{U}^{cd}_{u,2}\end{array}\right],\ \ \ \ \mathfrak{U}^{cd}_v=\left[\begin{array}{c}\mathfrak{U}^{cd}_{v,1} \\\mathfrak{U}^{cd}_{v,2}\end{array}\right],\\
\mathsf{E}_{\mu}\equiv \left[\begin{array}{c}0 \\I_{\mu}\end{array}\right]\in \mathbb{C}^{n\times \mu},\\
\mathfrak{Z}_{cd}^{\pm}\equiv \left[\begin{array}{cccc}\mathbf{z}_{cd,1}^{\pm}&\mathbf{z}_{cd,2}^{\pm}&\cdots&\mathbf{z}_{cd,\mu}^{\pm}\end{array}\right]\in \mathbb{C}^{n\times \mu},\\
\mathfrak{W}_{cd}^{\pm}\equiv\left[\begin{array}{cccc}\mathbf{w}_{cd,11}^{\pm}
&\mathbf{w}_{cd,12}^{\pm}&\cdots&\mathbf{w}_{cd,1\mu}^{\pm}\\
\mathbf{w}_{cd,21}^{\pm}
&\mathbf{w}_{cd,22}^{\pm}&\cdots&\mathbf{w}_{cd,2\mu}^{\pm}\\
\vdots
&&\ddots&\vdots\\
\mathbf{w}_{cd,\mu1}^{\pm}
&\mathbf{w}_{cd,\mu2}^{\pm}&\cdots&\mathbf{w}_{cd,\mu\mu}^{\pm}\end{array}\right],\\
P_{cd}\equiv \left[\begin{array}{c|c|c|c|cccc}\mathbf{I}^{cd}_{1}&\mathbf{I}^{cd}_{2}&\cdots&\mathbf{I}^{cd}_{\mu}&e_{\varkappa_1} &e_{\varkappa_1+\varkappa_2} &\cdots&e_{n_{cd}} \end{array}\right]\in \mathbb{C}^{n_{cd}\times n_{cd}},
\end{array}
\end{align}
where $\mathbf{w}_{cd,j\ell}^{\pm}=\mathbf{W}_{cd,j\ell}^{\pm}(\varkappa_j,\varkappa_\ell)$ for  $j,\ell\in \{1,2,\ldots, \mu\}$.

Note that equations \eqref{eq4.58b} and \eqref{eq4.59b} have the over-estimate form
\begin{align*}
Y(t)\mathbf{Z}_{3,\pm}(t)=\left[\mathbf{U}_{re,\pm},Y_{cd,\pm}(t)\right]+O(t^{-1}),
\end{align*}
as $t\rightarrow \pm\infty$, where $Y_{cd,\pm}(t)$ is defined in \eqref{eq4.64} and
\begin{align}\label{eq4.72}
\mathbf{U}_{re,+}\equiv \left[\begin{array}{c} \mathbf{U}^{re,+}_1\\\hline \mathbf{U}^{re,+}_2\end{array}\right] = \left[\begin{array}{cc} U^{r}_1&U^e_{1}\\\hline U^{r}_2&U^{e}_2\end{array}\right], \ \ \ \mathbf{U}_{re,-}\equiv \left[\begin{array}{c} \mathbf{U}^{re,-}_1\\\hline \mathbf{U}^{re,-}_2\end{array}\right] = \left[\begin{array}{cc} V^{r}_1&U^e_{1}\\\hline V^{r}_2&U^{e}_2\end{array}\right] .
\end{align}
By  applying the same procedure of proof of Lemma \ref{lem4.22} to $Y_{cd,\pm}(t)$, there is a nonsingular  matrix $\Omega_{\pm}(t)=I_{n_r+n_e}\oplus \left({\rm diag}(e_{\varkappa_1}e_{\varkappa_1}^{H}, \cdots,e_{\varkappa_\mu}e_{\varkappa_\mu}^{H})+O(t^{-1})\right)$ as $t\rightarrow \pm\infty$ and a  permutation matrix $\mathbf{P}_{cd}=I_{n_r+n_e}\oplus P_{cd}$, where $P_{cd}$ is given in \eqref{eq4.71} such that
\begin{align*}
Y(t)\mathbf{Z}_{3,\pm}(t)\Omega_{\pm}(t)\mathbf{P}_{cd}=\left[\mathbf{U}_{re,\pm},\widehat{U}_{cd},\mathfrak{U}_{cd} \Delta_{\pm}(t) \right]+O(t^{-1}),
\end{align*}
where $\mathbf{U}_{re,\pm}$ is defined in \eqref{eq4.72}, $\widehat{U}_{cd}$, $\mathfrak{U}_{cd}$ and $\mathfrak{W}_{cd}^{\pm}$ are given in \eqref{eq4.71},
\begin{align}\label{eq4.73}
\Delta_{\pm}(t) \equiv\left[\begin{array}{c}\Delta_u^{\pm}(t)\\\Delta_v^{\pm}(t)  \end{array}\right]=\left[
\begin{array}{l}
{\rm diag}(\hat{\omega}_{11}^1,\hat{\omega}_{11}^2,\cdots,\hat{\omega}_{11}^\mu)\mathfrak{W}_{cd}^{\pm}+{\rm diag}(\hat{\omega}_{12}^1,\hat{\omega}_{12}^2,\cdots,\hat{\omega}_{12}^\mu)\\
{\rm diag}(\hat{\omega}_{21}^1,\hat{\omega}_{21}^2,\cdots,\hat{\omega}_{21}^\mu)\mathfrak{W}_{cd}^{\pm}+{\rm diag}(\hat{\omega}_{22}^1,\hat{\omega}_{22}^2,\cdots,\hat{\omega}_{22}^\mu)
\end{array}\right]
\end{align}
and $\hat{\omega}_{jk}^\ell$, for $j,k\in\{1,2\}$,  $\ell\in \{1,2\ldots,\mu\}$, are defined in \eqref{eq4.67}. Let
$\mathbf{Z}_{\pm}(t)=\mathbf{Z}_{3,\pm}(t)\Omega_{\pm}(t)\mathbf{P}_{cd}$. By using the asymptotic behaviors of $\mathbf{Z}_{3,+}(t)$ in \eqref{eq4.58a} and $\mathbf{Z}_{3,-}(t)$ in \eqref{eq4.59a}, we have
\begin{align}\label{eq4.74}
\begin{array}{ll}
\mathbf{Z}_{+}(t)=\mathfrak{Z}_{cd}^+\mathsf{E}_{\mu}^H+O(t^{-1}),&\text{as }t\rightarrow \infty,\\
\mathbf{Z}_{-}(t)=\mathfrak{Z}_{cd}^-\mathsf{E}_{\mu}^H+O(t^{-1}),&\text{as }t\rightarrow -\infty,
\end{array}
\end{align}
where $\mathfrak{Z}^{\pm}_{cd},\ \mathsf{E}_{\mu}\in \mathbb{C}^{n\times \mu}$ are defined in \eqref{eq4.71}. Hence we have the following theorem.

\begin{Theorem}\label{thm4.23}
With the same notations of Theorem \ref{thm4.19}. Then
\begin{itemize}
\item[(i)]  if {\bf Assumption} $\mathscr{A}_{+}$  holds, then there is a nonsingular matrix $\mathbf{Z}_{+}(t)$ of the form in \eqref{eq4.74} such that
\begin{align}\label{eq4.75}
Y(t)\mathbf{Z}_{+}(t)=\left[\mathbf{U}_{re,+},\widehat{U}_{cd},\mathfrak{U}_{cd} \Delta_{+}(t) \right]+O(t^{-1}),
\end{align}
as $t\rightarrow \infty$, where $\Delta_+(t)$ and $\mathbf{U}_{re,+}$ are defined  in \eqref{eq4.73} and \eqref{eq4.72}, respectively, and
$\widehat{U}_{cd}$, $\mathfrak{U}_{cd}$ are defined in \eqref{eq4.71};
\item[(ii)]  if {\bf Assumption} $\mathscr{A}_{-}$  holds, then there is a nonsingular matrix $\mathbf{Z}_{-}(t)$ of the form in \eqref{eq4.74} such that
\begin{align}\label{eq4.76}
Y(t)\mathbf{Z}_{-}(t)=\left[\mathbf{U}_{re,-},\widehat{U}_{cd},\mathfrak{U}_{cd} \Delta_{-}(t) \right]+O(t^{-1}),
\end{align}
as $t\rightarrow \infty$, where $\Delta_-(t)$ and $\mathbf{U}_{re,-}$ are defined  in \eqref{eq4.73} and \eqref{eq4.72}, respectively, and
$\widehat{U}_{cd}$, $\mathfrak{U}_{cd}$ are defined in \eqref{eq4.71}.
\end{itemize}
\end{Theorem}

Now, we are ready to analyze asymptotic behaviors of $W(t)$ and $Q(t)^{-1}$. Let
\begin{align}\label{eq4.77}
\begin{array}{l}
\Sigma_{\hat{\omega}_{11}}\equiv \Sigma_{\hat{\omega}_{11}}(t)={\rm diag}(\hat{\omega}_{11}^1,\hat{\omega}_{11}^2,\cdots,\hat{\omega}_{11}^\mu),\\
\Sigma_{\hat{\omega}_{12}}\equiv \Sigma_{\hat{\omega}_{12}}(t)={\rm diag}(\hat{\omega}_{12}^1,\hat{\omega}_{12}^2,\cdots,\hat{\omega}_{12}^\mu),\\
\Sigma_{\hat{\omega}_{21}}\equiv \Sigma_{\hat{\omega}_{21}}(t)={\rm diag}(\hat{\omega}_{21}^1,\hat{\omega}_{21}^2,\cdots,\hat{\omega}_{21}^\mu),\\
\Sigma_{\hat{\omega}_{22}}\equiv \Sigma_{\hat{\omega}_{22}}(t)={\rm diag}(\hat{\omega}_{22}^1,\hat{\omega}_{22}^2,\cdots,\hat{\omega}_{22}^\mu),\\
\Sigma_{\gamma}\equiv \Sigma_{\gamma}(t)={\rm diag}((-1)^{m_1}e^{i\gamma_1t},\cdots,(-1)^{m_\mu}e^{i\gamma_\mu t}),\\
\Sigma_{\delta}\equiv\Sigma_{\delta}(t)={\rm diag}((-1)^{n_1}e^{i\delta_1t},\cdots,(-1)^{n_\mu}e^{i\delta_\mu t}),\\
\Sigma_{\beta^{cd}}={\rm diag}(\beta^{cd}_1,\beta^{cd}_2,\ldots,\beta^{cd}_\mu),
\end{array}
\end{align}
where $\beta^{cd}_j=\beta^{c}_j\in\{-1,1\}$ if $j\leqslant \mu_c$ and  $\beta^{cd}_j=\beta^{d}_j\in\{-1,1\}$ if $\mu_c<j\leqslant \mu$ and $\hat{\omega}_{jk}^\ell$, for $j,k\in\{1,2\}$,  $\ell\in \{1,2\ldots,\mu\}$, are defined in \eqref{eq4.67}. Then we have the theorem and leave the proof in Appendix.

\begin{Theorem}\label{thm4.24}
With the same notations of Theorem \ref{thm4.19}, suppose that {\bf Assumptions} $\mathscr{A}_{+}$ and $\mathscr{A}_{-}$ hold. Let
\begin{align}\label{eq4.78}
\begin{array}{l}
\mathbf{U}_{1,\pm}=\left[\mathbf{U}^{re,\pm}_{1},\widehat{U}^{cd}_1,\frac{1}{2}\left(\mathfrak{U}^{cd}_{v,1}-i\mathfrak{U}^{cd}_{u,1}\Sigma_{\beta^{cd}}\right)\right],\\
\mathbf{U}_{2,\pm}=\left[\mathbf{U}^{re,\pm}_2,\widehat{U}^{cd}_2,\frac{1}{2}\left(\mathfrak{U}^{cd}_{v,2}-i\mathfrak{U}^{cd}_{u,2}\Sigma_{\beta^{cd}}\right)\right],\\
\Delta\mathbf{U}^{cd}_{1,\pm}(t)=[\mathfrak{U}^{cd}_{u,1}\Sigma_{\hat{\omega}_{11}}+\mathfrak{U}^{cd}_{v,1}\Sigma_{\hat{\omega}_{21}}]\mathfrak{W}_{cd}^{\pm}\Sigma_{\gamma}^{-1}+\frac{1}{2}[\mathfrak{U}^{cd}_{v,1}+i\mathfrak{U}^{cd}_{u,1}\Sigma_{\beta^{cd}}]\Sigma_{\delta}\Sigma_{\gamma}^{-1},\\
\Delta\mathbf{U}^{cd}_{2,\pm}(t)=[\mathfrak{U}^{cd}_{u,2}\Sigma_{\hat{\omega}_{11}}+\mathfrak{U}^{cd}_{v,2}\Sigma_{\hat{\omega}_{21}}]\mathfrak{W}_{cd}^{\pm}\Sigma_{\gamma}^{-1}+\frac{1}{2}[\mathfrak{U}^{cd}_{v,2}+i\mathfrak{U}^{cd}_{u,2}\Sigma_{\beta^{cd}}]\Sigma_{\delta}\Sigma_{\gamma}^{-1},
\end{array}
\end{align}
where for each $j,k\in \{1,2\}$, $\mathbf{U}^{re,\pm}_{j}$ is defined in \eqref{eq4.72}, $\widehat{U}^{cd}_j$, $\mathfrak{U}^{cd}_{u,j}$, $\mathfrak{U}^{cd}_{v,j}$, $\mathfrak{W}_{cd}^{\pm}$ are defined in \eqref{eq4.71} and  $\Sigma_{\hat{\omega}_{jk}}$, $\Sigma_{\gamma}$, $\Sigma_{\delta}$ are defined in \eqref{eq4.77}. Let $W(t)= P(t)Q(t)^{-1}$, where  $Y(t)=[Q(t)^{\top},  P(t)^{\top}]^{\top}$ is the solution of IVP \eqref{eq4.1}. If $\mathbf{U}_{1,+}$ and $\mathbf{U}_{1,-}$ are invertible, then
\begin{align*}
W(t)&=\mathbf{U}_{2,\pm}\mathbf{U}_{1,\pm}^{-1}+[\Delta\mathbf{U}^{cd}_{2,\pm}(t)-\mathbf{U}_{2,\pm} \mathbf{U}_{1,\pm}^{-1}\Delta\mathbf{U}^{cd}_{1,\pm}(t)+O(t^{-1})]\\&\ \ \ \ \ [I_{\mu}+\mathsf{E}_{\mu}^H\mathbf{U}_{1,\pm}^{-1}\Delta\mathbf{U}^{cd}_{1,\pm}(t)+O(t^{-1})]^{-1}\mathsf{E}_{\mu}^H[\mathbf{U}_{1,\pm}^{-1}+O(t^{-1})]+O(t^{-1}),\\
Q(t)^{-1}&=[\mathfrak{Z}_{cd}^{\pm}\Sigma_{\gamma}^{-1}+O(t^{-1})]\left[I_{\mu}+\mathsf{E}_{\mu}^H\mathbf{U}_{1,\pm}^{-1}\Delta\mathbf{U}^{cd}_{1,\pm}(t)+O(t^{-1})\right]^{-1}\\&\ \ \ \ \ \mathsf{E}_{\mu}^H[\mathbf{U}_{1,\pm}^{-1}+O(t^{-1})]+O(t^{-1}),
\end{align*}
as $t\rightarrow \pm\infty$, where $\mathfrak{Z}_{cd}^{\pm},\ \mathsf{E}_{\mu}\in \mathbb{C}^{n\times \mu}$ are defined in \eqref{eq4.71}.
\end{Theorem}

Note that the quasi-periodicity of $W(t)$ is driven by the terms $\Delta\mathbf{U}^{cd}_{1,\pm}(t)$ and $\Delta\mathbf{U}^{cd}_{2,\pm}(t)$ defined in \eqref{eq4.78}, in which $e^{i\gamma_jt}$ and $e^{i\delta_jt}$, $j=1,\ldots,\mu$, are involved;  and the matrices $\mathbf{U}_{1,\pm}$ and $\mathbf{U}_{2,\pm}$ in \eqref{eq4.78} are constant. Let
\begin{subequations}\label{eq4.79}
\begin{align}\label{eq4.79a}
\begin{array}{rr}
\mathbf{U}_{1,\pm}(t)=\mathbf{U}_{1,\pm}+\Delta\mathbf{U}_{1,\pm}(t), & \mathbf{U}_2(t)=\mathbf{U}_{2,\pm}+\Delta\mathbf{U}_{2,\pm}(t),
\end{array}
\end{align}
where
\begin{align}\label{eq4.79b}
\begin{array}{rr}
\Delta\mathbf{U}_{1,\pm}(t)=\left[\ 0, \ \Delta\mathbf{U}^{cd}_{1\pm}(t)\right]\}n, & \Delta\mathbf{U}_{2,\pm}(t)=\left[\ 0,\ \Delta\mathbf{U}^{cd}_{2,\pm}(t)\right]\}n,\vspace{-0.3cm}\\
\underbrace{}_{n-\mu}  \underbrace{\ \ \ \ \ \ \ \ \ \ }_{\mu} \hspace{0.8cm}&\underbrace{}_{n-\mu}  \underbrace{\ \ \ \ \ \ \ \ \ \ }_{\mu} \hspace{0.8cm}
\end{array}
\end{align}
\end{subequations}
Denote
\begin{align}\label{eq4.80}
W_{\infty,\pm}(t)=&\mathbf{U}_{2,\pm}(t)\mathbf{U}_{1,\pm}(t)^{-1}\nonumber\\
=&\mathbf{U}_{2,\pm}\mathbf{U}_{1,\pm}^{-1}+[\Delta\mathbf{U}^{cd}_{2,\pm}(t)-\mathbf{U}_{2,\pm} \mathbf{U}_{1,\pm}^{-1}\Delta\mathbf{U}^{cd}_{1,\pm}(t)]\\ &\ \ \ [I_{\mu}+\mathsf{E}_{\mu}^H\mathbf{U}_{1,\pm}^{-1}\Delta\mathbf{U}^{cd}_{1,\pm}(t)]^{-1}\mathsf{E}_{\mu}^H\mathbf{U}_{1,\pm}^{-1},\nonumber\\
Q^{-1}_{\infty,\pm}(t)
=&\mathfrak{Z}_{cd}^{\pm}\Sigma_{\gamma}^{-1}\left[I_{\mu}+\mathsf{E}_{\mu}^H\mathbf{U}_{1,\pm}^{-1}\Delta\mathbf{U}^{cd}_{1,\pm}(t)\right]^{-1} \mathsf{E}_{\mu}^H\mathbf{U}_{1,\pm}^{-1},\nonumber
\end{align}
for $t\in \{t\in \mathbb{R}\ |\ \mathbf{U}_{1,\pm}(t)\text{ is invertible}\}$,
where $\mathfrak{Z}_{cd}^{\pm}$, $\mathsf{E}_{\mu}$ and $\Sigma_{\gamma}$ are defined in \eqref{eq4.71} and \eqref{eq4.77}, respectively. Roughly speaking, Theorem \ref{thm4.24} shows that  $W(t)$ and $Q(t)^{-1}$ converge, respectively, to $W_{\infty,\pm}(t)$ and $Q_{\infty,\pm}^{-1}(t)$ with the rate  $O(t^{-1})$ as $t\rightarrow \pm\infty$. More precisely, for each $0<\rho\ll1$, this convergence with the rate $O(t^{-1})$ is taking $t\to\pm\infty$ along the unbounded set $\{t\in{\mathbb R}|~ \sigma_{\min}(\mathbf{U}_{1,\pm}(t))>\rho\}$, where $\sigma_{\min}(\mathbf{U}_{1,\pm}(t))$ means the smallest singular value of $\mathbf{U}_{1,\pm}(t)$.  For the elementary case $\mathfrak{J}=\mathfrak{J}_d$ as mentioned in Theorem~\ref{thm4.13} and comparing \eqref{eq4.80} to \eqref{eq4.37},
$\mathbf{U}_{2,\pm}\mathbf{U}_{1,\pm}^{-1}$, $[\Delta\mathbf{U}^{cd}_{2,\pm}(t)-\mathbf{U}_{2,\pm} \mathbf{U}_{1,\pm}^{-1}\Delta\mathbf{U}^{cd}_{1,\pm}(t)]$, $[I_{\mu}+\mathsf{E}_{\mu}^H\mathbf{U}_{1,\pm}^{-1}\Delta\mathbf{U}^{cd}_{1,\pm}(t)]^{-1}$ and $\mathsf{E}_{\mu}^H\mathbf{U}_{1,\pm}^{-1}$ play the roles of $\mathbf{U}_2\mathbf{U}_1^{-1}$, $e^{i\theta t}\left(\zeta_2-\mathbf{U}_2\mathbf{U}_1^{-1}\zeta_1\right)$, $(1+e^{i\theta t}e_n^H\mathbf{U}_1^{-1}\zeta_1)^{-1}$ and $e_n^{H}\mathbf{U}_1^{-1}$, respectively.

\begin{Remark}\label{rem4.4}
Suppose that $\mathscr{H}$ in \eqref{eq4.1} has Hamiltonian Jordan canonical form $\mathfrak{J}$ in \eqref{eq4.8} and all eigenvalues of $\mathscr{H}$ are pure imaginary, that is, the submatrix $R_r$ of $\mathfrak{J}$ is absent. Then {\bf Assumptions} $\mathscr{A}_{+}$ is equivalent to {\bf Assumptions} $\mathscr{A}_{-}$, and hence $\mathbf{U}_{j,+}=\mathbf{U}_{j,-}$, $\mathfrak{Z}_{cd}^{+}=\mathfrak{Z}_{cd}^{-}$, $\mathfrak{W}_{cd}^{+}=\mathfrak{W}_{cd}^{-}$ and $\Delta\mathbf{U}^{cd}_{j,+}(t)=\Delta\mathbf{U}^{cd}_{j,-}(t)$ for $j=1,2$. It follows from  \eqref{eq4.79} and \eqref{eq4.80} that $W_{\infty,+}(t)=W_{\infty,-}(t)$ and $Q^{-1}_{\infty,+}(t)=Q^{-1}_{\infty,-}(t)$.
\end{Remark}


\begin{Example}\label{ex4.2}
In this example, we show some numerical experiments to demonstrate above theorems.
Consider the Hamiltonian matrix $\mathscr{H}$ has a Jordan canonical form $\mathfrak{J}_{cd}=\left[\begin{array}{c|c}R_{cd} & D_{cd}\\\hline G_{cd}& -R_{cd}^{-H}\end{array}\right]$. Assume $\mathscr{H}=\mathcal{S}\mathfrak{J}_{cd}\mathcal{S}^{-1}$, where the symplectic matrix $\mathcal{S}$ is randomly generated and
\begin{align*}
R_{cd}&=\left[\begin{array}{cccc}
                 i\gamma_1 &  1          &        0        &          0          \\
            0           &       i\gamma_1   &     0          &  -\frac{\sqrt{2}}{2}          \\
             0           &       0             &     i\delta_1  & -\frac{\sqrt{2}}{2}             \\
             0           &       0              &    0               &    \frac{i}{2}(\gamma_1+\delta_1)\\
\end{array}\right]\oplus\left[\begin{array}{ccc}
       i\gamma_2   &     0          &  -\frac{\sqrt{2}}{2}          \\
     0             &     i\delta_2  & -\frac{\sqrt{2}}{2}             \\
      0              &    0               &    \frac{i}{2}(\gamma_2+\delta_2)\\
\end{array}\right]\oplus\left[\begin{array}{cc}
       i\gamma_3             &  -\frac{\sqrt{2}}{2}          \\
      0                          &    \frac{i}{2}(\gamma_3+\delta_3)\\
\end{array}\right],\\
D_{cd}&=\frac{\sqrt{2}i}{2}  \left(\left[\begin{array}{cccc}
               0          &         0           &        0          &         0        \\
               0         &          0           &        0          &          1\\
             0          &         0             &      0           &         -1\\
              0          &       -1      &   1   &- \frac{\sqrt{2}i(\gamma_1-\delta_1)}{2}    \\
\end{array}\right]\oplus\left[\begin{array}{ccc}
          0           &        0          &          1\\
          0             &      0           &         -1\\
        -1      &   1   &- \frac{\sqrt{2}i (\gamma_2-\delta_2) }{2}  \\
\end{array}\right]\oplus\left[\begin{array}{cc}
          0                 &          1\\
        -1      &- \frac{\sqrt{2}i(\gamma_3-\delta_3) }{2}   \\
\end{array}\right] \right),\\ G_x&=-\frac{1}{2}\left((\gamma_1-\delta_1) e_4e_4^{\top}\oplus(\gamma_2-\delta_2) e_3e_3^{\top}\oplus(\gamma_3-\delta_3) e_2e_2^{\top}\right).
\end{align*}
We also randomly generate a complex Hermitian matrix $W_0\in \mathbb{C}^{9\times 9}$ as the initial matrix of RDE \eqref{eq4.3}. Then the solution $Y(t)= [Q(t)^{\top},  P(t)^{\top}]^{\top}$ of IVP \eqref{eq4.1} can be computed by the formula $Y(t)=\mathcal{S}e^{\mathfrak{J}_{cd}t}\mathcal{S}^{-1}[I,W_0]^{\top}$. The extended solution of RDE can be obtained by the formula $W(t)=P(t)Q(t)^{-1}$ for $t\in \mathcal{T}_W$, where $\mathcal{T}_W$ is defined in \eqref{eq3.41}.

Let $\gamma_1= 5.8868$, $\gamma_2=4.8968$, $\gamma_3= 2.2337$, $\delta_1=9.2031$, $\delta_2=0.7449$ and $\delta_3=9.7818$. Since all eigenvalues of $\mathscr{H}$ are pure imaginary, from Remark \ref{rem4.4}, we have $\mathbf{U}_{1}(t)\equiv\mathbf{U}_{1,+}(t)=\mathbf{U}_{1,-}(t)$, $W_{\infty}(t)\equiv W_{\infty,+}(t)=W_{\infty,-}(t)$ and $Q_{\infty}^{-1}(t)\equiv Q_{\infty,+}^{-1}(t)=Q_{\infty,-}^{-1}(t)$, where $\mathbf{U}_{1,\pm}(t)$ is defined in \eqref{eq4.79}.  We note from \eqref{eq4.80} that the pole of $W_{\infty}(t)$ and $Q_{\infty}^{-1}(t)$ is the number $t$ such that  $\mathbf{U}_{1}(t)$ is singular. In Figure \ref{fig5},  we show  the smallest singular value of $\mathbf{U}_{1}(t)$  and $\|W_{\infty}(t)\|_F$, $\|W(t)\|_F$, $\|Q_{\infty}^{-1}(t)\|_F$ and $\|Q(t)^{-1}\|_F$  plotted by the log scale for $990\leqslant t\leqslant 1000$. This figure shows that $\|W_{\infty}(t)\|_F$ and $\|Q_{\infty}^{-1}(t)\|_F$ blow-up at each $t$, where $\mathbf{U}_{1}(t)$ is singular and that the behaviors of $\|W(t)\|_F$ and  $\|Q(t)^{-1}\|_F$ are similar to the behaviors of  $\|W_{\infty}(t)\|_F$ and $\|Q_{\infty}^{-1}(t)\|_F$, respectively. The differences, $\|W(t)-W_{\infty}(t)\|_F$ and $\|Q(t)^{-1}-Q_{\infty}^{-1}(t)\|_F$, for $-1000\leqslant t\leqslant 0$ and for $0\leqslant t\leqslant 1000$ are shown in Figure \ref{fig6}.  We see that for each $0<\rho\ll1$,  as $t\to\pm\infty$ along the set $\{t\in{\mathbb R}|~ \sigma_{\min}(\mathbf{U}_{1}(t))>\rho\}$, $W(t)$ and $Q(t)^{-1}$ converges to $W_\infty(t)$ and $Q_\infty^{-1}(t)$, respectively, with the rate $O(t^{-1})$. It turns out that the curves $\|W(t)-W_\infty(t)\|_F$ and $\|Q(t)^{-1}-Q_\infty^{-1}(t)\|_F$ match the curve $y=C/t$ on this set.  However, as $t\to\pm\infty$ along the set $\{t\in{\mathbb R}|~ \mathbf{U}_{1}(t)\text{ is singular}\}$, i.e., the poles of $W_\infty(t)$ and $Q_\infty^{-1}(t)$, $W(t)$ and $Q(t)^{-1}$ tend to infinity. This leads to the peaks appearing  in Figure~\ref{fig6}.  Therefore, the curves $\|W(t)-W_\infty(t)\|_F$ and $\|Q(t)^{-1}-Q_\infty^{-1}(t)\|_F$  blow up on this set.
\begin{figure}[htb]
\centering \resizebox{5in}{!}{\includegraphics{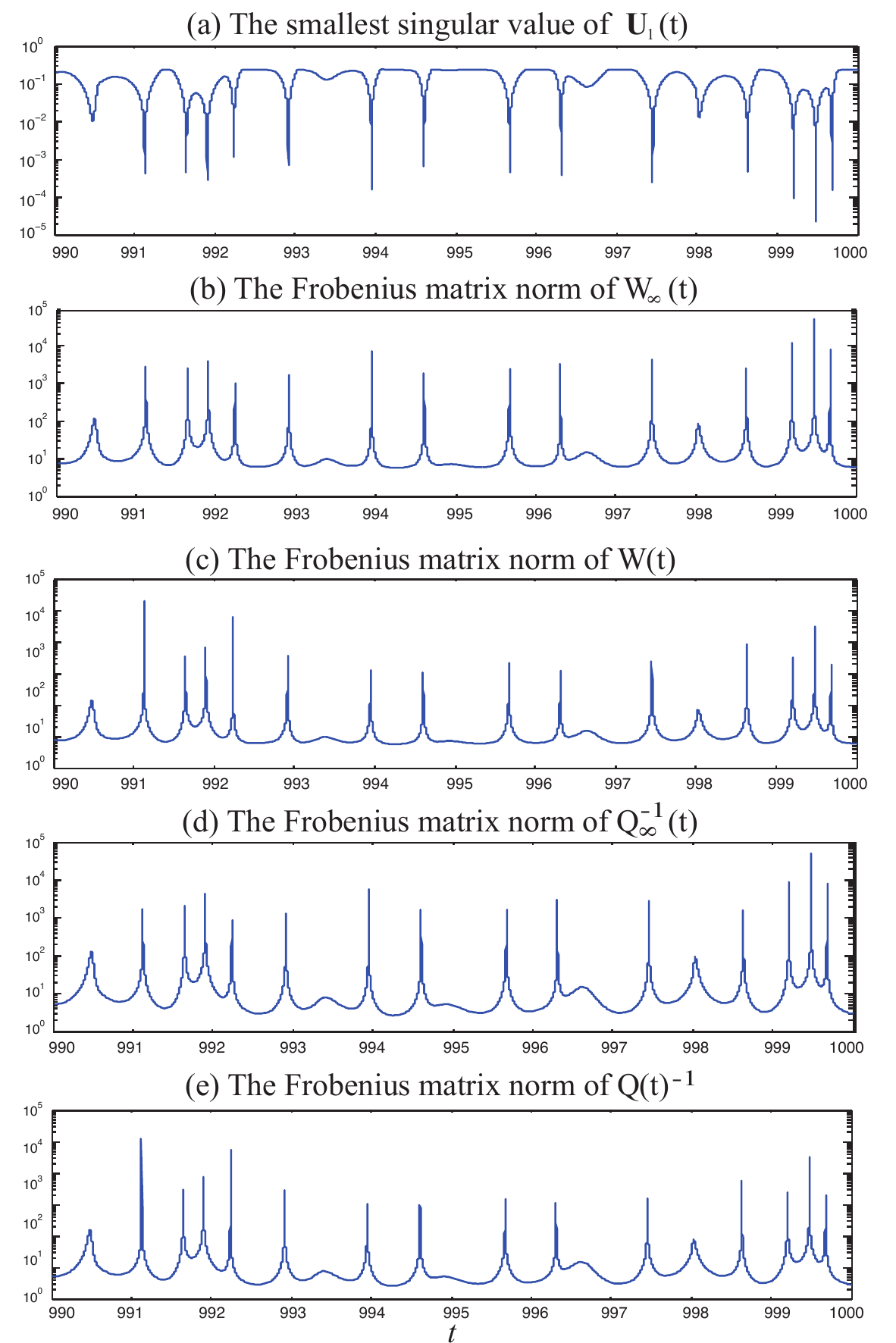}}
\caption{The smallest singular value of $\mathbf{U}_{1}(t)$, $\|W_{\infty}(t)\|_F$, $\|W(t)\|_F$, $\|Q_{\infty}^{-1}(t)\|_F$ and $\|Q(t)^{-1}\|_F$  plotted by the log scale.}\label{fig5}
\end{figure}
\begin{figure}[htb]
\centering \resizebox{6in}{!}{\includegraphics{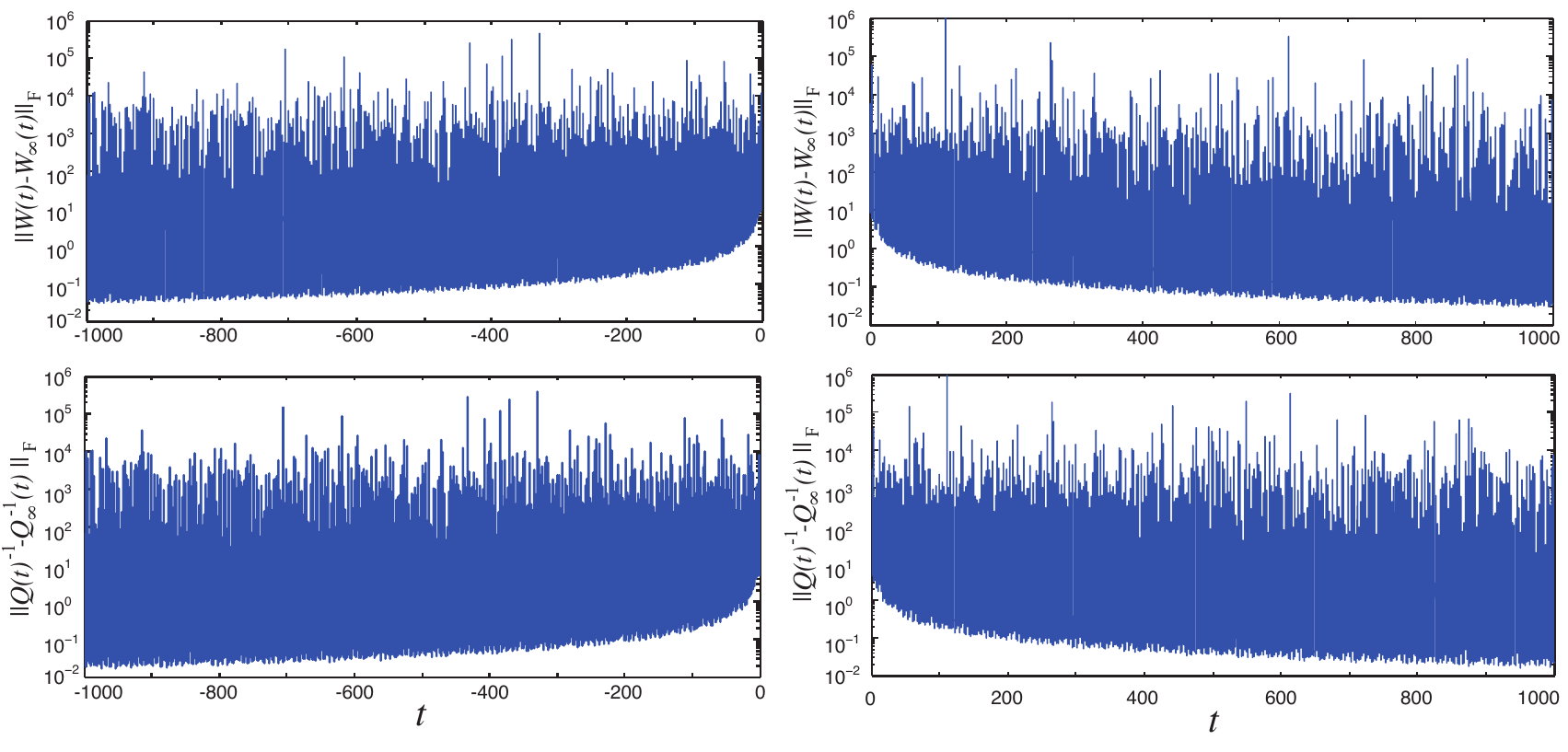}}
\caption{$\|W(t)-W_{\infty}(t)\|_F$  and $\|Q(t)^{-1}-Q_{\infty}^{-1}(t)\|_F$ plotted by the log scale  for $-1000\leqslant t\leqslant 0$ and for $0\leqslant t\leqslant 1000$.}\label{fig6}
\end{figure}
\end{Example}

\subsection{Application to the Convergence Analysis of SDA}\label{sec4.4}
In this subsection, we shall apply the asymptotic analysis of RDE \eqref{eq4.3} studied in previous subsections to the asymptotic behavior of SDA.  Throughout this subsection, we fix $(\mathcal{S}_1,\mathcal{S}_2)=(I,I)$ (the $\mathbb{S}_1$ class) or $(-I,\mathcal{J})$ (the $\mathbb{S}_2$ class) and let $X_1=[X_{ij}^{1}]_{1\leqslant i,j\leqslant 2}\in{\mathbb H}(2n)$ be given such that the pair $(\mathcal{M}_1, \mathcal{L}_1)=T_{\mathcal{S}_1,\mathcal{S}_2}(X_1)\in \mathbb{S}_1$ or $\mathbb{S}_2$ is regular with ${\rm ind}_{\infty} (\mathcal{M}_1,\mathcal{L}_1)\leqslant 1$. Let the idempotent matrices $\Pi_0=\Pi_0(\mathcal{M}_1,\mathcal{L}_1)$, $\Pi_\infty=\Pi_\infty(\mathcal{M}_1,\mathcal{L}_1)$ and the Hamiltonian matrix $\mathcal{H}=\mathcal{H}(\mathcal{M}_1,\mathcal{L}_1)$ be defined in Definition~\ref{def2.2}.  From Lemma~\ref{lem2.6} it follows that
\[
\mathcal{M}_1\Pi_0=\mathcal{L}_1\Pi_\infty e^{\mathcal{H}}.
\]
Suppose $(\mathcal{M}_k, \mathcal{L}_k)$, $k=1,2,\ldots$, is the sequence generated by the SDA and denote $X_k=[X^k_{ij}]_{1\le i,j\le 2}\equiv T_{\mathcal{S}_1,\mathcal{S}_2}^{-1}(\mathcal{M}_k, \mathcal{L}_k)$. It is shown in Theorem \ref{thm3.12} that $X_k=X(2^{k-1})$, where $X(t)$ is the extended solution of the IVP \eqref{eq3.9}.  Therefore, the asymptotic behaviors of the sequence $X_k$, as well as the sequence $(\mathcal{M}_k, \mathcal{L}_k)$, can be analyzed by using Lemma~\ref{lem4.1} as a connection to what we have studied on the RDE in previous subsections.

Suppose that $\mathcal{S}$ is a symplectic matrix such that  $\mathfrak{J}=\mathcal{S}^{-1}\mathcal{H}\mathcal{S}$ has the form in \eqref{eq4.8}. Partition $\mathcal{S}$ compatibly with $\mathfrak{J}$ being of the form
\begin{align}\label{eq4.81}
\mathcal{S}=\left[\begin{array}{ccc|ccc}U_r&U_e&U_{cd}&V_r&V_e&V_{cd}\end{array}\right]=\left[\begin{array}{ccc|ccc}U^r_1&U^e_1&U^{cd}_1&V^r_1&V^e_1&V^{cd}_1\\\hline U^r_2&U^e_2&U^{cd}_2&V^r_2&V^e_2&V^{cd}_2\end{array}\right].
\end{align}
Let
\begin{align}\label{eq4.82}
\begin{array}{ll}
\mathcal{S}_{-}=\mathcal{S}_2\mathcal{S},&\mathcal{S}_{+}=\mathcal{J}^{-1}\mathcal{S}_1\mathcal{S}
\end{array}
\end{align}
and
\begin{align*}
\begin{array}{rr}
\left[\begin{array}{c}
W_1^-\\
W_2^-\\
\end{array}%
\right]=\mathcal{S}^{-1}\mathcal{S}_2^{-1}\left[\begin{array}{c}
I\\
-X_{22}^{1}\\
\end{array}%
\right],& \left[\begin{array}{c}
W_1^+\\
W_2^+\\
\end{array}%
\right]=\mathcal{S}^{-1}\mathcal{S}_1^{-1}\mathcal{J}\left[\begin{array}{c}
I\\
X_{11}^{1}\\
\end{array}%
\right]
\end{array}\in{\mathbb C}^{2n\times n}.
\end{align*}
Partition   $W_j^\pm$ for $j=1,2$ as
\begin{align*}
\begin{array}{cc}W_j^{\pm}=\left[\begin{array}{cc}W^{j,\pm}_{1,1} & W^{j,\pm}_{1,2}   \\W^{j,\pm}_{2,1} & W^{j,\pm}_{2,2} \end{array}\right]&\!\!\!\!\!\!\!\!\begin{array}{l}\}n_r\\\} n_{ecd}\end{array}\vspace{-0.3cm}\\
\begin{array}{cc}\ \ \ \ \ \ \  \ \    \underbrace{}_{n_r}  &\ \ \underbrace{}_{n_{ecd}} \end{array}&\end{array}.
\end{align*}
Here $n_r$, $n_e$, $ n_c$ and $ n_d$   the sizes of $R_r$, $R_e$, $ R_c$ and $ R_d$ in \eqref{eq4.8}, respectively. We assume that
\begin{description}
 \item[{\bf Assumption}] SDA:  $\left[\begin{array}{cc}W^{1,+}_{1,1} & W^{1,+}_{1,2}   \\W^{2,+}_{2,1} & W^{2,+}_{2,2}  \end{array}\right]$ and $W_2^{-}$ are invertible.
 \end{description}
From Lemma \ref{lem4.1}, we see that the flows in \eqref{eq4.6} govern the sequence generated by SDA.
Under the {\bf Assumption} SDA, there exist invertible matrices $\mathbf{Z}_{1,\pm}$ in \eqref{eq4.38} such that
\begin{subequations}\label{eq4.83}
\begin{align}
&\left[\begin{array}{c}
Q(t;\mathcal{S}_2\mathcal{H}\mathcal{S}_2^{-1},-X_{22}^1)\\
P(t;\mathcal{S}_2\mathcal{H}\mathcal{S}_2^{-1},-X_{22}^1)\\
\end{array}%
\right]\mathbf{Z}_{1,-}=\mathcal{S}_-e^{\mathfrak{J}t}\left[\begin{array}{c}
\mathbf{ W}_1^-\\
\mathbf{ W}_2^-\\
\end{array}%
\right],\label{eq4.83a}\\
&\left[\begin{array}{c}
Q(t;\widetilde{\mathcal{H}}_{\star},X_{11}^1)\\
P(t;\widetilde{\mathcal{H}}_{\star},X_{11}^1)\\
\end{array}%
\right]\mathbf{Z}_{1,+}=\mathcal{S}_+e^{\mathfrak{J}t}\left[\begin{array}{c}
\mathbf{ W}_1^+\\
\mathbf{ W}_2^+\\
\end{array}%
\right],\label{eq4.83b}
\end{align}
\end{subequations}
where $\widetilde{\mathcal{H}}_{\star}=\mathcal{J}^{-1}\mathcal{S}_1\mathcal{H}\mathcal{S}_1^{-1}\mathcal{J}$ and $\mathbf{ W}_1^\pm$, $\mathbf{ W}_2^\pm$ have the form as in \eqref{eq4.40} and $\mathcal{S}_{-}$, $\mathcal{S}_{+}$ are defined in \eqref{eq4.82}. The asymptotic behaviors of
\begin{align*}
\begin{array}{lll}
W(t;\mathcal{S}_2\mathcal{H}\mathcal{S}_2^{-1},-X_{22}^1),&Q(t;\mathcal{S}_2\mathcal{H}\mathcal{S}_2^{-1},-X_{22}^1)^{-1},&\text{ as }t\rightarrow -\infty\\
W(t;\widetilde{\mathcal{H}}_{\star},X_{11}^1),&Q(t;\widetilde{\mathcal{H}}_{\star},X_{11}^1)^{-1},&\text{ as }t\rightarrow \infty
\end{array}
\end{align*}
have been studied in Subsection \ref{sec4.3}, and hence, can be used as a fundamental tool for the convergence analysis of SDA.

As a consequence of Lemma~\ref{lem4.1} and Corollary~\ref{cor4.17}, we see that the SDA exhibits a quadratic convergence whenever none of nonzero eigenvalues of $\mathcal{H}$ are pure imaginary. A similar convergence analysis has  been  carried out in \cite{ccghlx09,GuoLin:2010}.

\begin{Theorem}\label{thm4.25}
Suppose that $\mathcal{H}$ has no nonzero pure imaginary eigenvalue, that is, $U_{ecd}$ and $V_{ecd}$ are absent in \eqref{eq4.81} and $\mathfrak{J}=\left[\begin{array}{c|c}R_r & 0   \\\hline 0  & -R_r^H \end{array}\right]$.
Let  $\mathfrak{r}=\min\{\Re({\rm diag}(R_r))\}>0$ and
\begin{align*}
\left[\begin{array}{c}
\mathscr{U}_{1,-}\\
\mathscr{U}_{2,-}\\
\end{array}%
\right]=\mathcal{S}_2\left[\begin{array}{c}
V^r_1\\
V^r_2\\
\end{array}%
\right],\ \ \ \ \left[\begin{array}{c}
\mathscr{U}_{1,+}\\
\mathscr{U}_{2,+}\\
\end{array}%
\right]=\mathcal{J}^{-1}\mathcal{S}_1\left[\begin{array}{c}
U^r_1\\
U^r_2\\
\end{array}%
\right].
\end{align*}
If $\mathscr{U}_{1,-}$, $\mathscr{U}_{1,+}$ are invertible and {\bf Assumption} {\rm SDA} holds, then
\begin{align*}
\begin{array}{ll}
X^k_{22}=-\mathscr{U}_{2,-}\mathscr{U}_{1,-}^{-1}+O(e^{-\mathfrak{r}2^k}2^{2nk}),& X^k_{12}=O(e^{-\mathfrak{r}2^{k-1}}2^{nk}),\\
X^k_{11}=\mathscr{U}_{2,+}\mathscr{U}_{1,+}^{-1}+O(e^{-\mathfrak{r}2^k}2^{2nk}),& X^k_{21}=O(e^{-\mathfrak{r}2^{k-1}}2^{nk}),
\end{array}
\end{align*}
as $k\rightarrow \infty$. Here, $\mathscr{U}_{2,-}\mathscr{U}_{1,-}^{-1}$ and $\mathscr{U}_{2,+}\mathscr{U}_{1,+}^{-1}$ are Hermitian.
\end{Theorem}

\begin{proof}
We first prove assertions for $X_{22}^k$ and $X_{12}^k$. Note that  \eqref{eq4.83a} holds due to  {\bf Assumption} SDA.
Replacing the matrix $\mathcal{S}$ by $\mathcal{S}_{-}$ in Corollary~\ref{cor4.17}, it follows
\begin{align*}
\begin{array}{l}
W(t;\mathcal{S}_2\mathcal{H}\mathcal{S}_2^{-1},-X_{22}^1)=\mathscr{U}_{2,-}\mathscr{U}_{1,-}^{-1}+O(e^{-2\mathfrak{r}|t|}|t|^{2n}),\\
Q(t;\mathcal{S}_2\mathcal{H}\mathcal{S}_2^{-1},-X_{22}^1)^{-1}=O(e^{-\mathfrak{r}|t|}|t|^{n}),
\end{array}
\end{align*}
as $t\rightarrow -\infty$. Therefore we conclude from Lemma~\ref{lem4.1} that
\begin{align*}
X^k_{22}=-\mathscr{U}_{2,-}\mathscr{U}_{1,-}^{-1}+O(e^{-\mathfrak{r}2^k}2^{2nk})\text{ and } X^k_{12}=O(e^{-\mathfrak{r}2^{k-1}}2^{nk}),
\end{align*}
as $k\rightarrow \infty$.  Assertions for $X_{11}^k$ and $X_{21}^k$ can be accordingly obtained by using the matrix $\mathcal{S}_+$, \eqref{eq4.83b},  Corollary~\ref{cor4.17} and Lemma \ref{lem4.1}. The matrices $\mathscr{U}_{2,-}\mathscr{U}_{1,-}^{-1}$ and $\mathscr{U}_{2,+}\mathscr{U}_{1,+}^{-1}$ are Hermitian because $\mathcal{S}_-$ and $\mathcal{S}_+$ are symplectic, respectively.
\end{proof}

In a similar manner as the proof of Theorem \ref{thm4.25}, the following theorem can be obtained by applying Lemma~\ref{lem4.1} and Corollary~\ref{cor4.20}. We  see that the SDA exhibits a linear convergence whenever the sizes of Jordan blocks corresponding to nonzero  pure imaginary eigenvalues of $\mathcal{H}$ are even. A similar convergence analysis has been proven in \cite{HuangLin:2009}.

\begin{Theorem}\label{thm4.26}
Suppose that the sizes of Jordan blocks corresponding to nonzero  pure imaginary eigenvalues of $\mathcal{H}$ are even, that is, $U_{cd}$ and $V_{cd}$ are absent in \eqref{eq4.81} and $\mathfrak{J}$ has the form in \eqref{eq4.60}.
Let
\begin{align*}
\left[\begin{array}{c}
\mathscr{U}_{1,-}\\
\mathscr{U}_{2,-}\\
\end{array}%
\right]=\mathcal{S}_2\left[\begin{array}{cc}
V^r_1&U_1^e\\
V^r_2&U_2^e\\
\end{array}%
\right],\ \ \ \ \left[\begin{array}{c}
\mathscr{U}_{1,+}\\
\mathscr{U}_{2,+}\\
\end{array}%
\right]=\mathcal{J}^{-1}\mathcal{S}_1\left[\begin{array}{cc}
U^r_1&U_1^e\\
U^r_2&U_2^e\\
\end{array}%
\right].
\end{align*}
If $\mathscr{U}_{1,-}$, $\mathscr{U}_{1,+}$ are invertible and {\bf Assumption} {\rm SDA} holds, then
\begin{align*}
\begin{array}{ll}
X^k_{22}=-\mathscr{U}_{2,-}\mathscr{U}_{1,-}^{-1}+O(2^{-k}),& X^k_{12}=O(2^{-k}),\\
X^k_{11}=\mathscr{U}_{2,+}\mathscr{U}_{1,+}^{-1}+O(2^{-k}),& X^k_{21}=O(2^{-k}),
\end{array}
\end{align*}
as $k\rightarrow \infty$. Here, $\mathscr{U}_{2,-}\mathscr{U}_{1,-}^{-1}$ and $\mathscr{U}_{2,+}\mathscr{U}_{1,+}^{-1}$ are Hermitian.
\end{Theorem}

For the case that the Hamiltonian Jordan canonical form $\mathfrak{J}$ of $\mathcal{H}$ has the form in \eqref{eq4.62},
the following theorem can be obtained by applying Lemma~\ref{lem4.1} and Corollary~\ref{cor4.21}. We see that the sequences, $X_{22}^k$ and $X_{11}^k$, converge linearly to constant Hermitian matrices and that the sequences, $X_{12}^k$ and $X_{21}^k$, tend linearly to closed obits that consist of rank-one matrices.

\begin{Theorem}\label{thm4.27}
Suppose that {\bf Assumption} {\rm SDA} holds and  the Hamiltonian Jordan canonical form $\mathfrak{J}$ of $\mathcal{H}$ has the form in \eqref{eq4.62}, that is, $U_{cd}=U_{c}$ and $V_{cd}=V_{c}$ in \eqref{eq4.81}.
Let
\begin{align*}
\left[\begin{array}{c}
\mathscr{U}_{1,-}\\
\mathscr{U}_{2,-}\\
\end{array}%
\right]=\mathcal{S}_2\left[\begin{array}{c}
\mathbf{U}_{1,-}\\
\mathbf{U}_{2,-}\\
\end{array}%
\right],\ \ \ \ \left[\begin{array}{c}
\mathscr{U}_{1,+}\\
\mathscr{U}_{2,+}\\
\end{array}%
\right]=\mathcal{J}^{-1}\mathcal{S}_1\left[\begin{array}{c}
\mathbf{U}_{1,+}\\
\mathbf{U}_{2,+}\\
\end{array}%
\right],
\end{align*}
where $\mathbf{U}_{j,-}$ and $\mathbf{U}_{j,+}$ for $j=1,2$ are defined in \eqref{eq4.63}.
If $\mathscr{U}_{1,-}$, $\mathscr{U}_{1,+}$  are invertible, then as $k\rightarrow \infty$
\begin{align*}
\begin{array}{ll}
X^k_{22}=-\mathscr{U}_{2,-}\mathscr{U}_{1,-}^{-1}+O(2^{-k}),& X^k_{12}=e^{i\eta (2^{k-1}-1)}X^1_{12}\mathbf{Z}_{cd,-}^{1}e_{n_{cd}}e_n^{H}\mathscr{U}_{1,-}^{-1}+O(2^{-k}),\\
X^k_{11}=\mathscr{U}_{2,+}\mathscr{U}_{1,+}^{-1}+O(2^{-k}),& X^k_{21}=e^{-i\eta (2^{k-1}-1)}X^1_{21}\mathbf{Z}_{cd,+}^{1}e_{n_{cd}}e_n^{H}\mathscr{U}_{1,+}^{-1}+O(2^{-k}),
\end{array}
\end{align*}
where $\mathbf{Z}_{cd,\pm}^{1}$ is defined in \eqref{eq4.52}. Here, $\mathscr{U}_{2,-}\mathscr{U}_{1,-}^{-1}$  and $\mathscr{U}_{2,+}\mathscr{U}_{1,+}^{-1}$ are Hermitian.
\end{Theorem}

The following theorem can be obtained by applying Lemma~\ref{lem4.1} and Theorem~\ref{thm4.24}.
\begin{Theorem}\label{thm4.28}
Suppose that {\bf Assumption} {\rm SDA} holds and  the Hamiltonian Jordan canonical form $\mathfrak{J}$ of $\mathcal{H}$ is of the form in \eqref{eq4.8}.
Let
\begin{align*}
\begin{array}{ll}
\left[\begin{array}{c}
\mathscr{U}_{1,-}\\
\mathscr{U}_{2,-}\\
\end{array}%
\right]=\mathcal{S}_2\left[\begin{array}{c}
\mathbf{U}_{1,-}\\
\mathbf{U}_{2,-}\\
\end{array}%
\right],&\left[\begin{array}{c}
\Delta\mathscr{U}_{1,-}^{cd}(t)\\
\Delta\mathscr{U}_{2,-}^{cd}(t)\\
\end{array}%
\right]=\mathcal{S}_2\left[\begin{array}{c}
\Delta\mathbf{U}_{1,-}^{cd}(t)\\
\Delta\mathbf{U}_{2,-}^{cd}(t)\\
\end{array}%
\right], \\
\left[\begin{array}{c}
\mathscr{U}_{1,+}\\
\mathscr{U}_{2,+}\\
\end{array}%
\right]=\mathcal{J}^{-1}\mathcal{S}_1\left[\begin{array}{c}
\mathbf{U}_{1,+}\\
\mathbf{U}_{2,+}\\
\end{array}%
\right],&\left[\begin{array}{c}
\Delta\mathscr{U}_{1,+}^{cd}(t)\\
\Delta\mathscr{U}_{2,+}^{cd}(t)\\
\end{array}%
\right]=\mathcal{J}^{-1}\mathcal{S}_1\left[\begin{array}{c}
\Delta\mathbf{U}_{1,+}^{cd}(t)\\
\Delta\mathbf{U}_{2,+}^{cd}(t)\\
\end{array}%
\right],
\end{array}
\end{align*}
where $\mathbf{U}_{1,\pm}$, $\mathbf{U}_{2,\pm}$, $\Delta\mathbf{U}_{1,\pm}^{cd}(t)$ and $\Delta\mathbf{U}_{2,\pm}^{cd}(t)$  are defined in \eqref{eq4.78}.
If $\mathscr{U}_{1,-}$, $\mathscr{U}_{1,+}$  are invertible, then there exist four matrices
\begin{align*}
\mathscr{K}_{W_{\pm}}(k)=&\left[\Delta\mathscr{U}^{cd}_{2,\pm}(\pm2^{k-1}\mp 1)-\mathscr{U}_{2,\pm} \mathscr{U}_{1,\pm}^{-1}\Delta\mathscr{U}^{cd}_{1,\pm}(\pm2^{k-1}\mp 1)+O(2^{-k})\right]\\&[I_{\mu}+\mathsf{E}_{\mu}^H\mathscr{U}_{1,\pm}^{-1}\Delta\mathscr{U}^{cd}_{1,\pm}(\pm2^{k-1}\mp 1)+O(2^{-k})]^{-1}\mathsf{E}_{\mu}^H[\mathscr{U}_{1,\pm}^{-1}+O(2^{-k})],\\
\mathscr{K}_{Q_{\pm}}(k)=&[\mathfrak{Z}_{cd}^{\pm}\Sigma_{\gamma}(\pm2^{k-1}\mp 1)^{-1}+O(2^{-k})]\\&\left[I_{\mu}+\mathsf{E}_{\mu}^H\mathscr{U}_{1,\pm}^{-1}\Delta\mathscr{U}^{cd}_{1,\pm}(\pm2^{k-1}\mp 1)+O(2^{-k})\right]^{-1}\mathsf{E}_{\mu}^H[\mathscr{U}_{1,\pm}^{-1}+O(2^{-k})],
\end{align*}
such that as $k\rightarrow \infty$
\begin{align*}
\begin{array}{l}
X^k_{22}=-\mathscr{U}_{2,-}\mathscr{U}_{1,-}^{-1}-\mathscr{K}_{W_{-}}(k)+O(2^{-k}),\\
 X^k_{12}=X^1_{12}\mathscr{K}_{Q_{-}}(k)+O(2^{-k}),\\
X^k_{11}=\mathscr{U}_{2,+}\mathscr{U}_{1,+}^{-1}+\mathscr{K}_{W_{+}}(k)+O(2^{-k}),\\
X^k_{21}=X^1_{21}\mathscr{K}_{Q_{+}}(k)+O(2^{-k}),
\end{array}
\end{align*}
where $\mathfrak{Z}_{cd}^{\pm},\ \mathsf{E}_{\mu}\in \mathbb{C}^{n\times \mu}$ are defined in \eqref{eq4.71}. Here,  the ranks of $\mathscr{K}_{W_{\pm}}(k)$ and of $\mathscr{K}_{Q_{\pm}}(k)$ are at most $\mu$, where $\mu$ is the number of Jordan blocks in $R_{cd}$.
\end{Theorem}

\renewcommand{\thesection}{A}
\section{Appendix}\label{secA}

\subsection{Complementary of Section~\ref{sec2}}\label{secA.1}
Let $R>1$, $\theta\in [0,2\pi)$ and $$D_{R,\theta}=\{z\in \mathbb{C}| 1/R\leqslant |z|\leqslant R\}/\{z=re^{i\theta}|\ 1/R\leqslant r\leqslant R\}.$$ Let $\Gamma$ denote the boundary of $D_{R,\theta}$.
Suppose that $A\in \mathbb{C}^{n\times n}$ is an invertible matrix and $\sigma(A)\subseteq D_{R,\theta}$.
 We define Log$(A)$ by
\begin{align}\label{eq5.1}
{\rm Log}(A)=\frac{1}{2\pi i}\oint_{\Gamma}(zI-A)^{-1}(\log z)dz.
\end{align}
It has been shown that $e^{{\rm Log}(A)}=A$ for each invertible matrix $A\in \mathbb{C}^{n\times n}$ in \cite{Horn91}. Now, we show that if $\mathcal{S}\in Sp(n)$ then ${\rm Log}(\mathcal{S})$ is Hamiltonian and vice versa.

\begin{Theorem}\label{thm5.1}
Suppose that $\mathcal{S}\in Sp(n)$ is symplectic. Then ${\rm Log}(\mathcal{S})$ is Hamiltonian. Conversely, if $\mathcal{H}$ is Hamiltonian, then $e^{\mathcal{H}}$ is symplectic.
\end{Theorem}

\begin{proof}
Since $\mathcal{S}$ is symplectic, then so is $\mathcal{S}^H$. Let $\mathcal{S}$ have distinct eigenvalues $\lambda_1,\ldots,\lambda_k$, i.e., $\sigma(\mathcal{S})=\{\lambda_1,\ldots,\lambda_k\}$. Let $R>0$ and $\theta\in [0,2\pi)$ such that $\sigma(\mathcal{S})\cup \sigma(\mathcal{S}^H)\subseteq D_{R,\theta}$. Since $\mathcal{S}$ is symplectic, we know that for each $\lambda_j\in \sigma(\mathcal{S})$, $1/\bar{\lambda}_j\in \sigma(\mathcal{S})$. Let $\Gamma_1,\ldots,\Gamma_k\subseteq D_{R,\theta}$ be small nonintersecting circles with positive orientation in the complex plane centered at $\lambda_1,\ldots,\lambda_k$, respectively, which are symmetric with respect to the unit circle. 
Thus the transformation, $z\rightarrow 1/\bar{z}$, maps the set of circles $\Gamma_1,\ldots,\Gamma_k$ into the set of circles $-\Gamma_1,\ldots,-\Gamma_k$.

From \eqref{eq5.1}, we have
\begin{align}\label{eq5.2}
{\rm Log}(\mathcal{S})=\frac{1}{2\pi i}\sum_{j=1}^k\oint_{\Gamma_j}(zI-\mathcal{S})^{-1}(\log z)dz.
\end{align}
Make the change of variable $z=1/\bar{\xi}$ in the integrals in \eqref{eq5.2} and suppose that $\Gamma_s\rightarrow -\Gamma_j$.  Recall that $\mathcal{S}^{-1}=-\mathcal{J}\mathcal{S}^H\mathcal{J}$. Then we have
\begin{align*}
\frac{1}{2\pi i}\oint_{\Gamma_s}(zI-\mathcal{S})^{-1}(\log z)dz&=\frac{1}{2\pi i}\oint_{-\Gamma_j}[(1/\bar{\xi})I-\mathcal{S}]^{-1}(-\log \bar{\xi})(-d\bar{\xi}/\bar{\xi}^2)\\
&=\frac{-1}{2\pi i}\oint_{\Gamma_j}(I-\bar{\xi}\mathcal{S})^{-1}\bar{\xi}^{-1}\log \bar{\xi}d\bar{\xi}\\
&=\overline{\frac{1}{2\pi i}\oint_{{\Gamma}_j}(I-{\xi}\bar{\mathcal{S}})^{-1}{\xi}^{-1}\log {\xi}d{\xi}}\\
&=\overline{\frac{1}{2\pi i}\oint_{{\Gamma}_j}[\bar{\mathcal{S}}(I-{\xi}\bar{\mathcal{S}})^{-1}+{\xi}^{-1}I]\log {\xi}d{\xi}}\\
&=\overline{\frac{1}{2\pi i}\oint_{{\Gamma}_j}[(\bar{\mathcal{S}}^{-1}-{\xi}I)^{-1}+{\xi}^{-1}I]\log {\xi}d{\xi}}\\
&=\overline{\frac{1}{2\pi i}\oint_{{\Gamma}_j}[(-\mathcal{J}\mathcal{S}^{\top}\mathcal{J}+{\xi}\mathcal{J}\mathcal{J})^{-1}+{\xi}^{-1}I]\log {\xi}d{\xi}}\\
&=\mathcal{J}\overline{\left(\frac{1}{2\pi i}\oint_{{\Gamma}_j}({\xi}I-\mathcal{S}^{\top})^{-1}\log {\xi}d{\xi}\right)}\mathcal{J}+\overline{\frac{1}{2\pi i}\oint_{{\Gamma}_j}{\xi}^{-1}\log {\xi}d{\xi}}
\end{align*}
The circle $\Gamma_j$ does not enclose the origin, thus $\oint_{\Gamma_j}{\xi}^{-1}\log {\xi}d{\xi}=0$. From \eqref{eq5.2} and using the fact that ${\rm Log}(\mathcal{S}^{\top})={\rm Log}(\mathcal{S})^{\top}$, we have
${\rm Log}(\mathcal{S})=\mathcal{J}\overline{{\rm Log}(\mathcal{S})^{\top}}\mathcal{J}=\mathcal{J}{\rm Log}(\mathcal{S})^{H}\mathcal{J}$, and then $\mathcal{J}{\rm Log}(\mathcal{S})=-{\rm Log}(\mathcal{S})^{H}\mathcal{J}$. Therefore, ${\rm Log}(\mathcal{S})$ is Hamiltonian.

For the converse statement, suppose that $\mathcal{H}$ is Hamiltonian, then $-\mathcal{H}^H=\mathcal{J}^{-1}\mathcal{H}\mathcal{J}$.
By taking the matrix exponential at each sides of the resulting equation, it leads to $e^{\mathcal{H}}\mathcal{J}e^{\mathcal{H}^H}=\mathcal{J}$, and hence, $e^{\mathcal{H}}$ is symplectic.
\end{proof}

\subsection{Complementary of Subsection~\ref{sec4.2}}\label{secA.2}
In the following theorem, we show that $\digamma_{k_1}^{k_2}$ in \eqref{eq4.18} is invertible, where $k_1$, $k_2$ are positive integers with $0<k_1<k_2\leqslant 2k_1$. In order to prove this, we need a useful formula (Pascal's law):
\begin{align*}
P^{n}_r-P^{n-1}_r=rP^{n-1}_{r-1}\ \ \text{ for }n,r\in \mathbb{N}\text{ and }n\geqslant r,
\end{align*}
where $P^n_r=n(n-1)\cdots (n-r-1)=\frac{n!}{(n-r)!}$.
\begin{Theorem}\label{thm5.2}
Let $k_1$, $k_2$ be given positive integers satisfying $0<k_1<k_2\leqslant 2k_1$ and $\delta=k_2-k_1$. Then
\begin{align*}
{\rm det}(\digamma_{k_1}^{k_2})=\frac{\delta!(\delta-1)!\cdots 1!}{k_2!(k_2-1)!\cdots (k_1)!},
\end{align*}
where $\digamma_{k_1}^{k_2}$ is defined in \eqref{eq4.18}. Hence,  $\digamma_{k_1}^{k_2}$ is invertible.
\end{Theorem}
\begin{proof}
Let $D={\rm diag}(k_2!,(k_2-1)!,\ldots, (k_1)!)$. Denote
\begin{align*}
\widetilde{\digamma}_{k_1}^{k_2}\equiv D \digamma_{k_1}^{k_2}=\left[\begin{array}{cccc}P^{k_2}_\delta & P^{k_2}_{\delta-1} &\cdots & P^{k_2}_0 \\P^{k_2-1}_\delta & P^{k_2-1}_{\delta-1} &\cdots & P^{k_2-1}_0  \\\vdots & \vdots &  & \vdots \\P^{k_1}_\delta & P^{k_1}_{\delta-1} &\cdots & P^{k_1}_0 \end{array}\right]\in \mathbb{R}^{(\delta+1)\times (\delta+1)}.
\end{align*}
Let $e_j$ be the $j$th column vector of the identity matrix $I_{\delta+1}$ and $E_{i,j}=I_{\delta+1}-e_ie_j^H$. Using Pascal's law, we have
\begin{align*}
E_{\delta,\delta+1}\cdots E_{2,3}E_{1,2}\widetilde{\digamma}_{k_1}^{k_2}&=\left[\begin{array}{ccc|c}\delta P^{k_2-1}_{\delta-1} & \cdots&1P^{k_2-1}_0 & 0 \\\vdots & \vdots &  & \vdots \\\delta P^{k_1}_{\delta-1} & \cdots &1P^{k_1}_0 & 0\\\hline P^{k_1}_\delta & \cdots &P^{k_1}_{1} & 1 \end{array}\right]\\
&=\left[\begin{array}{c|c}\widetilde{\digamma}_{k_1}^{k_2-1} & 0 \\\hline * & 1\end{array}\right]{\rm diag}(\delta,(\delta-1),\cdots,1,1).
\end{align*}
It is easily seen that ${\rm det}(\widetilde{\digamma}_{k_1}^{k_2})=\delta!\cdot{\rm det}(\widetilde{\digamma}_{k_1}^{k_2-1})$. We then have
\begin{align*}
{\rm det}(\widetilde{\digamma}_{k_1}^{k_2})&=\delta!(\delta-1)!\cdots 1! \cdot{\rm det}(\widetilde{\digamma}_{k_1}^{k_1})\\&=\delta!(\delta-1)!\cdots 1! .
\end{align*}
Hence, we obtain $${\rm det}(\digamma_{k_1}^{k_2})=\frac{{\rm det}(\widetilde{\digamma}_{k_1}^{k_2})}{{\rm det}(D)}=\frac{\delta!(\delta-1)!\cdots 1!}{k_2!(k_2-1)!\cdots (k_1)!}.$$
\end{proof}
\begin{Theorem}\label{thm5.3}
Given $n\in \mathbb{N}$. Let $\kappa_n=\widehat{\psi}_{n}^{H}(\widehat{\Gamma}_{n+1}^{2n})^{-1}\phi_{n}$, where $\widehat{\psi}_{n}^{H}=\psi_{n}^HP_n$, $\phi_{n}$ and $\widehat{\Gamma}_{n+1}^{2n}$ are defined in \eqref{eq4.10}. Then
\begin{align}\label{eq5.3}
\kappa_n=\left\{\begin{array}{ll}0&\text{ if $n$ is even,}\\2&\text{ if $n$ is odd.}\end{array}\right.
\end{align}
\end{Theorem}
\begin{proof}
Using the definitions of  $\widehat{\psi}_{n}^{H}=\psi_{n}^HP_n$, $\phi_{n}$ and $\widehat{\Gamma}_{n+1}^{2n}$ in \eqref{eq4.10}, it follows from \eqref{eq4.19} that
$\kappa_n=\mathbf{x}^{H}_n(\digamma_{n+1}^{2n})^{-1}\mathbf{y}_n$,
where $\mathbf{x}_n=[1,\frac{1}{2!},\ldots,\frac{1}{n!}]^H$, $\mathbf{y}_n=[\frac{1}{n!},\frac{1}{(n-1)!},\ldots,1]^H$ and  $\digamma_{n+1}^{2n}$ is defined in \eqref{eq4.18}. It is easily seen that
\begin{align*}
\digamma_{n}^{2n}=\left[\begin{array}{cc}\mathbf{y}_n &  \digamma_{n+1}^{2n}\\1 & \mathbf{x}_n^{H}\end{array}\right].
\end{align*}
It follows from Theorem \ref{thm5.2} that $ \digamma_{n+1}^{2n}$ is invertible. Denote
\begin{align*}
E=\left[\begin{array}{c|c}1 &  0\\\hline-( \digamma_{n+1}^{2n})^{-1}\mathbf{y}_n & I\end{array}\right].
\end{align*}
Then we have
\begin{align*}
\digamma_{n}^{2n}E=\left[\begin{array}{c|c}0 &  \digamma_{n+1}^{2n}\\\hline1-\mathbf{x}_{n}^{H}( \digamma_{n+1}^{2n})^{-1}\mathbf{y}_n & \mathbf{x}_{n}^{H}\end{array}\right].
\end{align*}
From Theorem \ref{thm5.2}, we obtain that
\begin{align*}
\frac{n!(n-1)!\cdots 1!}{(2n)!(2n-1)!\cdots n!}&={\rm det}(\digamma_{n}^{2n})={\rm det}(\digamma_{n}^{2n}E)\\
&=(-1)^{n+2}(1-\kappa_n){\rm det}(\digamma_{n+1}^{2n})\\
&=(-1)^{n}(1-\kappa_n)\frac{(n-1)!(n-2)!\cdots 1!}{(2n)!(2n-1)!\cdots (n+1)!}.
\end{align*}
Hence, $(-1)^{n}(1-\kappa_n)=1$, that is, $\kappa_n$ satisfies \eqref{eq5.3}.
\end{proof}

\begin{Lemma}\label{lem5.4}
Let $n\in \mathbb{N}$. Then
\begin{itemize}
\item[(i)] $(\widehat{\Gamma}_{n}^{2n-1})^{-1}\Phi_n=O(t^{-1})$, $\widehat{\Phi}_n(\widehat{\Gamma}_{n}^{2n-1})^{-1}=O(t^{-1})$ and $\widehat{\Phi}_n(\widehat{\Gamma}_{n}^{2n-1})^{-1}\Phi_n=O(t^{-1})$ ;
\item[(ii)] $\widehat{\Phi}_n(\Phi_nW\pm \widehat{\Gamma}_{n}^{2n-1})^{-1}=O(t^{-1})$;
\end{itemize}
as $t\rightarrow \pm \infty$, where $W\in \mathbb{C}^{n\times n}$ is a constant matrix and $\Phi_n$,  $\widehat{\Gamma}_{n}^{2n-1}$ and $\widehat{\Phi}_n$ are given in \eqref{eq4.10}.
\end{Lemma}

\begin{proof}
It follows from \eqref{eq4.10} and \eqref{eq4.18} that $(\Xi_{n-1,0})^{-1}=O(1)$, $(\Xi_{0,n-1})^{-1}=O(1)$ and $(\Xi_{n-1,0})^{-1}\Phi_n=O(1)$, $\Phi_n(\Xi_{0,n-1})^{-1}=O(1)$ as $t\rightarrow \pm \infty$. By using \eqref{eq4.19}, we  have
\begin{align*}
&(\widehat{\Gamma}_{n}^{2n-1})^{-1}\Phi_n=t^{-1}P_n^{-1}(\Xi_{0,n-1})^{-1}(\digamma_{n}^{2n-1})^{-1}\left[(\Xi_{n-1,0})^{-1}\Phi_n\right]=O(t^{-1}),\\
&\widehat{\Phi}_n(\widehat{\Gamma}_{n}^{2n-1})^{-1}=t^{-1}\left[P_n^{-1}\Phi_n(\Xi_{0,n-1})^{-1}\right](\digamma_{n}^{2n-1})^{-1}(\Xi_{n-1,0})^{-1}=O(t^{-1}),\\
&\widehat{\Phi}_n(\widehat{\Gamma}_{n}^{2n-1})^{-1}\Phi_n=t^{-1}\left[P_n^{-1}\Phi_n(\Xi_{0,n-1})^{-1}\right](\digamma_{n}^{2n-1})^{-1}\left[(\Xi_{n-1,0})^{-1}\Phi_n\right]=O(t^{-1}),
\end{align*}
as $t\rightarrow \pm \infty$. This proves assertion (i). Now, we prove assertion (ii).  Note that
\begin{align*}
\widehat{\Phi}_n(\Phi_nW-\widehat{\Gamma}_{n}^{2n-1})^{-1}&=\widehat{\Phi}_n\left[(\widehat{\Gamma}_{n}^{2n-1})^{-1}\Phi_nW- I\right]^{-1}(\widehat{\Gamma}_{n}^{2n-1})^{-1}.
\end{align*}
Using the facts in assertion (i), we have
\begin{align*}
\widehat{\Phi}_n(\Phi_nW-\widehat{\Gamma}_{n}^{2n-1})^{-1}
&=-\widehat{\Phi}_n(\widehat{\Gamma}_{n}^{2n-1})^{-1}-\sum_{k=1}^{\infty}\widehat{\Phi}_n\left[(\widehat{\Gamma}_{n}^{2n-1})^{-1}\Phi_nW\right]^{k}(\widehat{\Gamma}_{n}^{2n-1})^{-1}\\
&=O(t^{-1}),
\end{align*}
as $t\rightarrow \pm \infty$. Similarly, the rest case $\widehat{\Phi}_n(\Phi_nW+\widehat{\Gamma}_{n}^{2n-1})^{-1}=O(t^{-1})$ can be proven.
\end{proof}

\begin{Lemma}\label{lem5.5}
Let $n_1,n_2\in \mathbb{N}$. Then,  we have
\begin{align*}
\begin{array}{ll}
(\Upsilon+\widehat{\Gamma}_{n_1+1,n_2+1}^{2n_1,2n_2})^{-1}=O(t^{-2}),&(\Upsilon+\widehat{\Gamma}_{n_1+1,n_2+1}^{2n_1,2n_2})^{-1}\phi_{n_1,n_2}^{j}=O(t^{-1}),\\
(\Upsilon+\widehat{\Gamma}_{n_1+1,n_2+1}^{2n_1,2n_2})^{-1}\Phi_{n_1,n_2}=O(t^{-2}),&\widehat{\psi}_{n_1,n_2}^{j^H}(\Upsilon+\widehat{\Gamma}_{n_1+1,n_2+1}^{2n_1,2n_2})^{-1}=O(t^{-1}),\\
\widehat{\Phi}_{n_1,n_2}(\Upsilon+\widehat{\Gamma}_{n_1+1,n_2+1}^{2n_1,2n_2})^{-1}=O(t^{-2}),&\widehat{\psi}_{n_1,n_2}^{j^H}(\Upsilon+\widehat{\Gamma}_{n_1+1,n_2+1}^{2n_1,2n_2})^{-1}\Phi_{n_1,n_2}=O(t^{-1}),\\
\widehat{\Phi}_{n_1,n_2}(\Upsilon+\widehat{\Gamma}_{n_1+1,n_2+1}^{2n_1,2n_2})^{-1}\Phi_{n_1,n_2}=O(t^{-2}),&\widehat{\Phi}_{n_1,n_2}(\Upsilon+\widehat{\Gamma}_{n_1+1,n_2+1}^{2n_1,2n_2})^{-1}\phi_{n_1,n_2}^{j}=O(t^{-1}),
\end{array}
\end{align*}
as $t\rightarrow \pm \infty$, where $j=1,2$, $\widehat{\Gamma}_{n_1+1,n_2+1}^{2n_1,2n_2}$, $\Phi_{n_1,n_2}$, $\widehat{\Phi}_{n_1,n_2}$, $\phi_{n_1,n_2}^{j}$ and $\widehat{\psi}_{n_1,n_2}^{j^H}$ are defined in \eqref{eq4.12} and $\Upsilon=\Phi_{n_1,n_2}W+\phi^1_{n_1,n_2}w^H$, $W\in \mathbb{C}^{(n_1+n_2)\times (n_1+n_2)}$ and $w\in\mathbb{C}^{n_1+n_2}$. Moreover, we also have
\begin{align*}
\widehat{\psi}_{n_1,n_2}^{j^H}(\Upsilon+\widehat{\Gamma}_{n_1+1,n_2+1}^{2n_1,2n_2})^{-1}\phi_{n_1,n_2}^{k}=\widehat{\psi}_{n_1,n_2}^{j^H}(\widehat{\Gamma}_{n_1+1,n_2+1}^{2n_1,2n_2})^{-1}\phi_{n_1,n_2}^{k}+O(t^{-1}),
\end{align*}
as $t\rightarrow \pm \infty$, where $j,k\in\{1,2\}$.
\end{Lemma}

\begin{proof}
Using the definition of $\widehat{\Gamma}_{n_1+1,n_2+1}^{2n_1,2n_2}$ in \eqref{eq4.12}, it follows from \eqref{eq4.19} that
\begin{align*}
(\widehat{\Gamma}_{n_1+1,n_2+1}^{2n_1,2n_2})^{-1}&=-i\beta\left[-e^{-i\gamma t}P_{n_1}^{-1}\left(\Gamma_{n_1+1}^{2n_1}\right)^{-1}\oplus e^{-i\delta t}P_{n_2}^{-1}\left(\Gamma_{n_2+1}^{2n_2}\right)^{-1}\right]\\
&=-i\beta t^{-2}\left[-e^{-i\gamma t}P_{n_1}^{-1}\left(\Xi_{0,n_1-1}\right)^{-1}\left(\digamma_{n_1+1}^{2n_1}\right)^{-1}\left(\Xi_{n_1-1,0}\right)^{-1}\right.\\
 &\ \ \ \ \ \ \ \ \ \ \ \ \ \ \ \left.\oplus e^{-i\delta t}P_{n_2}^{-1}\left(\Xi_{0,n_2-1}\right)^{-1}\left(\digamma_{n_2+1}^{2n_2}\right)^{-1}\left(\Xi_{n_2-1,0}\right)^{-1}\right].
\end{align*}
From \eqref{eq4.10}, \eqref{eq4.12} and \eqref{eq4.18}, we have
\begin{align}\label{eq5.4}
\begin{array}{ll}
(\widehat{\Gamma}_{n_1+1,n_2+1}^{2n_1,2n_2})^{-1}=O(t^{-2}),&(\widehat{\Gamma}_{n_1+1,n_2+1}^{2n_1,2n_2})^{-1}\phi_{n_1,n_2}^{j}=O(t^{-1}),\\
(\widehat{\Gamma}_{n_1+1,n_2+1}^{2n_1,2n_2})^{-1}\Phi_{n_1,n_2}=O(t^{-2}),&\widehat{\psi}_{n_1,n_2}^{j^H}(\widehat{\Gamma}_{n_1+1,n_2+1}^{2n_1,2n_2})^{-1}=O(t^{-1}),\\
\widehat{\Phi}_{n_1,n_2}(\widehat{\Gamma}_{n_1+1,n_2+1}^{2n_1,2n_2})^{-1}=O(t^{-2}),&\widehat{\psi}_{n_1,n_2}^{j^H}(\widehat{\Gamma}_{n_1+1,n_2+1}^{2n_1,2n_2})^{-1}\Phi_{n_1,n_2}=O(t^{-1}),\\
\widehat{\Phi}_{n_1,n_2}(\widehat{\Gamma}_{n_1+1,n_2+1}^{2n_1,2n_2})^{-1}\Phi_{n_1,n_2}=O(t^{-2}),&\widehat{\Phi}_{n_1,n_2}(\widehat{\Gamma}_{n_1+1,n_2+1}^{2n_1,2n_2})^{-1}\phi_{n_1,n_2}^{j}=O(t^{-1}),
\end{array}
\end{align}
as $t\rightarrow \pm \infty$. From \eqref{eq5.4}, we have $(\widehat{\Gamma}_{n_1+1,n_2+1}^{2n_1,2n_2})^{-1}\Upsilon=O(t^{-1})$  as $t\rightarrow \pm \infty$ and then
\begin{align*}
\left( \Upsilon+\widehat{\Gamma}_{n_1+1,n_2+1}^{2n_1,2n_2}\right)^{-1}= \left(\widehat{\Gamma}_{n_1+1,n_2+1}^{2n_1,2n_2}\right)^{-1}+\sum_{l=1}^{\infty}(-1)^l\left[\left(\widehat{\Gamma}_{n_1+1,n_2+1}^{2n_1,2n_2}\right)^{-1} \Upsilon \right]^{l} \left(\widehat{\Gamma}_{n_1+1,n_2+1}^{2n_1,2n_2}\right)^{-1}.
\end{align*}
Hence, the results of this lemma can be obtained accordingly from \eqref{eq5.4}.
\end{proof}

\subsection{Complementary of Subsection~\ref{sec4.3}}\label{secA.3}
\begin{Lemma}\label{lem5.6}
When $|t|$ is sufficiently large, the matrix $ \mathcal{T}_{e,\pm} $ in \eqref{eq4.55} is invertible and
\begin{align}\label{eq5.5}
\mathcal{T}_{e,\pm}^{-1}\mathcal{R}_e=O(t^{-1}),\ \ \  \mathcal{R}_e^{-H}\mathcal{T}_{e,\pm}^{-1}=O(t^{-1}),\ \ \
\mathcal{R}_e^{-H}\mathcal{T}_{e,\pm}^{-1}\mathcal{R}_e=O(t^{-1}),
\end{align}
as $t\rightarrow\pm \infty$.
\end{Lemma}
\begin{proof}
From Theorem \ref{thm4.8}, we have $\mathcal{R}_e=\oplus_{j=1}^{k}e^{i\alpha_j t}\Phi_{l_j}$ and $\mathcal{D}_e=-\oplus_{j=1}^{k}e^{i\alpha_j t}\beta_j^e\widehat{\Gamma}_{l_{j}}^{2l_{j}-1}$, where $\Phi_{l_j}$ and $\widehat{\Gamma}_{l_{j}}^{2l_{j}-1}$ are defined in \eqref{eq4.10}, $\alpha_j\in \mathbb{R}$ and $\beta_j^e\in \{-1,1\}$ for $j=1,\ldots,k$.
Since  each $\widehat{\Gamma}_{l_{j}}^{2l_{j}-1}$ is invertible, we obtain that $\mathcal{D}_e$ is invertible.
From  Table \ref{tab1}, we have $\mathcal{D}_e^{-1} \mathcal{R}_e=O(t^{-1})$ as $t\rightarrow \pm\infty$. Therefore,  $ \mathcal{T}_{e,\pm}=\mathcal{R}_e\mathbf{W}^{\pm}_{2,2}+\mathcal{D}_e $ is invertible for all sufficiently large values of $|t|$.

It follows from Lemma \ref{lem4.3} that $\mathcal{R}_e^{-H}=\oplus_{j=1}^{k}e^{i\alpha_j t}\widehat{\Phi}_{l_j}$, where $\widehat{\Phi}_{l_j}$ is defined in \eqref{eq4.10}. Then using the fact that $ \mathcal{T}_{e,\pm}^{-1}=\mathcal{D}_e^{-1}+\sum_{k=1}^{\infty}(-1)^k[\mathcal{D}_e^{-1}\mathcal{R}_e\mathbf{W}_{2,2}^{\pm}]^k\mathcal{D}_e^{-1}$, we obtain \eqref{eq5.5} directly by Table \ref{tab1}.
\end{proof}

\begin{Lemma}\label{lem5.7}
Let $ \mathcal{T}_{e,\pm}$ be the matrix defined in  \eqref{eq4.55}. Then
\begin{align}\label{eq5.6}
\begin{array}{l}
(\mathcal{R}_{cd}\mathbf{W}^{\pm}_{cd}+\mathcal{D}_{cd})^{-1}\mathcal{R}_{cd} \mathbf{W}^{\pm}_{3,2}\mathcal{T}_{e,\pm}^{-1} =O(t^{-1}),\\
 (\mathcal{G}_{cd}\mathbf{W}^{\pm}_{cd}+\mathcal{E}_{cd})(\mathcal{R}_{cd}\mathbf{W}^{\pm}_{cd}+\mathcal{D}_{cd})^{-1}\mathcal{R}_{cd} \mathbf{W}^{\pm}_{3,2}\mathcal{T}_{e,\pm}^{-1} =O(t^{-1}),
\end{array}
\end{align}
as $t\rightarrow\pm \infty$ and $\mathcal{R}_{cd}\mathbf{W}^{\pm}_{cd}+\mathcal{D}_{cd}$ is invertible.
\end{Lemma}
\begin{proof}
From Theorem \ref{thm4.8} and \eqref{eq4.54}, we assume
\begin{align}\label{eq5.7}
\mathcal{R}_{cd}=\oplus_{\ell=1}^{\mu}  \mathbf{B}_{\ell},\ \ \mathcal{D}_{cd}=\oplus_{\ell=1}^{\mu} \mathbf{D}_{\ell}, \ \ \mathcal{G}_{cd}=\oplus_{\ell=1}^{\mu} \mathbf{G}_{\ell} \text{ and }\mathcal{E}_{cd}=\oplus_{\ell=1}^{\mu} \mathbf{E}_{\ell},
\end{align}
where $\mathbf{B}_{\ell}$, $\mathbf{D}_{\ell}$, $\mathbf{G}_{\ell}$ and $\mathbf{E}_{\ell}$ have the forms in \eqref{eq4.66}. It follows from \eqref{eq4.12}, \eqref{eq4.18} and \eqref{eq4.19} that
\begin{align*}
\mathbf{D}^d_{\ell}&=\left[\begin{array}{cc}\Xi_{m_{\ell},1}\oplus \Xi_{n_{\ell},1}& 0  \\0 &1\end{array}\right]\mathfrak{D}_\ell\left[\begin{array}{cc}\Xi_{1,m_{\ell}}P_{m_{\ell}}\oplus \Xi_{1,n_{\ell}}P_{n_{\ell}}& 0  \\0 &1\end{array}\right],\\
&=:(\Xi_{m_\ell,1}^{n_\ell,1}\oplus 1)\mathfrak{D}_\ell(\widehat{\Xi}_{1,m_\ell}^{1,n_\ell}\oplus 1)
\end{align*}
where $\mathfrak{D}_\ell \equiv \mathfrak{D}_\ell(t)=O(1)$ and $\Xi_{j,1}$, $\Xi_{1,j}$ for $j=m_\ell, n_\ell$ are defined in \eqref{eq4.18}.
Using the definitions of $\mathbf{B}_{\ell}$, $\mathbf{E}_{\ell}$ in \eqref{eq4.66} and equations \eqref{eq4.10}, \eqref{eq4.12}, yields that
\begin{align}\label{eq5.8}
(\Xi_{m_\ell,1}^{n_\ell,1}\oplus 1)^{-1}\mathbf{B}_{\ell}=O(1),\ \ \
\mathbf{E}_{\ell}(\widehat{\Xi}_{1,m_\ell}^{1,n_\ell}\oplus 1)^{-1}=O(1).
\end{align}
Let $\mathcal{X}_{cd}=\oplus_{\ell=1}^{\mu}(\Xi_{m_\ell,1}^{n_\ell,1}\oplus 1)$ and $\widehat{\mathcal{X}}_{cd}=\oplus_{\ell=1}^{\mu}(\widehat{\Xi}_{1,m_\ell}^{1,n_\ell}\oplus 1)$. Then $\mathcal{X}_{cd}^{-1}\mathcal{R}_{cd}=O(1)$, $\mathcal{E}_{cd}\widehat{\mathcal{X}}_{cd}^{-1}=O(1)$ and $\widehat{\mathcal{X}}_{cd}^{-1}=O(1)$. Since $\mathcal{G}_{cd}=O(1)$, it follows from  \eqref{eq5.7} and \eqref{eq5.8} that
\begin{align*}
&(\mathcal{G}_{cd}\mathbf{W}^{\pm}_{cd}+\mathcal{E}_{cd})(\mathcal{R}_{cd}\mathbf{W}^{\pm}_{cd}+\mathcal{D}_{cd})^{-1}\mathcal{R}_{cd} \mathbf{W}^{\pm}_{3,2}\\&=[\mathcal{G}_{cd}\mathbf{W}^{\pm}_{cd}\widehat{\mathcal{X}}_{cd}^{-1}+\mathcal{E}_{cd}\widehat{\mathcal{X}}_{cd}^{-1}] [(\mathcal{X}_{cd}^{-1}\mathcal{R}_{cd})\mathbf{W}^{\pm}_{cd}\widehat{\mathcal{X}}_{cd}^{-1} +\oplus_{\ell=1}^{\mu} \mathfrak{D}_{\ell}]^{-1}(\mathcal{X}_{cd}^{-1}\mathcal{R}_{cd}) \mathbf{W}^{\pm}_{3,2}=O(1).
\end{align*}
From Lemma \ref{lem5.6}, we have $\mathcal{T}_{e,\pm}^{-1} =O(t^{-1})$ as $t\rightarrow \pm\infty$. Hence, \eqref{eq5.6} holds.
\end{proof}

\begin{proof}[\textbf{Proof of Lemma \ref{lem4.18}}]
Let  $Z_{1,\pm}(t)=\left[\begin{array}{cc} \mathcal{T}_{e,\pm}^{-1}&-\mathcal{T}_{e,\pm}^{-1}\mathcal{R}_e\mathbf{W}^{\pm}_{2,3}\\0&I_{n_{cd}}\end{array}\right]$ for $t\in \mathcal{I}_{\pm}$, where $ \mathcal{T}_{e,\pm}$ and $\mathcal{I}_{\pm}$ are defined in \eqref{eq4.55} and \eqref{eq4.56}, respectively. Using the fact that $\mathcal{G}_{cd}=O(1)$, it follows form  Lemma \ref{lem5.6} that $Z_{1,\pm}(t)=\left[\begin{array}{cc}O(t^{-1})&O(t^{-1})\\0&I_{n_{cd}}\end{array}\right]$ and
\begin{align}\label{eq5.9}
\left[\begin{array}{cc} \mathcal{R}_e\mathbf{W}^{\pm}_{2,2}+\mathcal{D}_e & \mathcal{R}_e\mathbf{W}^{\pm}_{2,3}\\\mathcal{R}_{cd} \mathbf{W}^{\pm}_{3,2} & \mathcal{R}_{cd}\mathbf{W}^{\pm}_{cd}+\mathcal{D}_{cd} \\\hline
\mathcal{R}_e^{-H} & 0  \\  \mathcal{G}_{cd}  \mathbf{W}^{\pm}_{3,2} & \mathcal{G}_{cd} \mathbf{W}^{\pm}_{cd}+\mathcal{E}_{cd}\end{array}\right]Z_{1,\pm}(t)=
\left[\begin{array}{cc} I_{n_e} & 0\\\mathcal{R}_{cd} \mathbf{W}^{\pm}_{3,2}\mathcal{T}_{e,\pm}^{-1} & \mathcal{R}_{cd}(\mathbf{W}^{\pm}_{cd}+O(t^{-1}))+\mathcal{D}_{cd} \\\hline
O(t^{-1}) & O(t^{-1})   \\  O(t^{-1}) & \mathcal{G}_{cd}( \mathbf{W}^{\pm}_{cd}+O(t^{-1}))+\mathcal{E}_{cd}\end{array}\right],
\end{align}
as $t\rightarrow \pm\infty$. We know that  $ \mathcal{R}_{cd}(\mathbf{W}^{\pm}_{cd}+O(t^{-1}))+\mathcal{D}_{cd}$ is invertible for $t\in \mathcal{I}_{\pm}$.  Let
\begin{align*}
Z_{2,\pm}(t)=\left[\begin{array}{cc} I_{n_e}&0\\-(\mathcal{R}_{cd}(\mathbf{W}^{\pm}_{cd}+O(t^{-1}))+\mathcal{D}_{cd})^{-1}\mathcal{R}_{cd} \mathbf{W}^{\pm}_{3,2}\mathcal{T}_{e,\pm}^{-1}&I_{n_{cd}}\end{array}\right],
\end{align*}
for $t\in \mathcal{I}_{\pm}$. Since $ \mathbf{W}^{\pm}_{cd}+O(t^{-1})=O(1)$, the consequences of Lemma \ref{lem5.7} also hold true whenever the matrix $ \mathbf{W}_{cd}^{\pm}$ in the statement is replaced by any $O(1)$ matrix. Hence, from Lemma \ref{lem5.7}, we have $Z_{2,\pm}(t)=\left[\begin{array}{cc}I_{n_{e}}&0\\O(t^{-1})&I_{n_{cd}}\end{array}\right]$, as $t\rightarrow \pm\infty$. Denote $ Z_{ecd,\pm}(t)=Z_{1,\pm}(t)Z_{2,\pm}(t)$ for $t\in \mathcal{I}_{\pm}$. It is easily seen that  the asymptotic behavior of $Z_{ecd,\pm}(t)$ has the form \eqref{eq4.57a}. Substituting \eqref{eq4.53} into \eqref{eq4.50}, it follows from \eqref{eq5.9} that we obtain \eqref{eq4.57b} as $t\rightarrow \pm\infty$.
\end{proof}

\begin{proof}[\textbf{Proof of Lemma \ref{lem4.22}}] Let $Y(t)\equiv\left[\begin{array}{c}Q(t)  \\\hline P(t)\end{array}\right]=\left[\begin{array}{c|c}\mathcal{R}_{cd} &\mathcal{D}_{cd} \\\hline \mathcal{G}_{cd} & \mathcal{E}_{cd}\end{array}\right]\left[\begin{array}{c}\mathbf{ W}\\\hline  I \end{array}\right]$.
From \eqref{eq4.68} and \eqref{eq4.69}, we have
\begin{align}\label{eq5.10}
\left[\begin{array}{c}Q(t)  \\\hline P(t)\end{array}\right]=\left[\begin{array}{cc}\mathbf{B}_{1}\mathbf{W}_{11}+\mathbf{D}_{1}&\mathbf{B}_{1}\mathbf{W}_{12}\\ \mathbf{B}_{2}\mathbf{W}_{21}&\mathbf{B}_{2}\mathbf{W}_{22}+\mathbf{D}_2\\\hline  \mathbf{G}_{1}\mathbf{W}_{11}+\mathbf{E}_{1}&\mathbf{G}_{1}\mathbf{W}_{12}\\ \mathbf{G}_{2}\mathbf{W}_{21}&\mathbf{G}_{2}\mathbf{W}_{22}+\mathbf{E}_2 \end{array}\right].
\end{align}
Partition $ \mathbf{ W}_{jk}$ as $\mathbf{ W}_{jk}=\left[\begin{array}{cc} \mathbf{ W}^{jk}_{11} &  \mathbf{ w}^{jk}_{12} \\\mathbf{ w}^{jk}_{21}  & \mathbf{ w}^{jk}_{22}\end{array}\right]$,
where $\mathbf{ w}^{jk}_{22}=\mathbf{ w}_{jk}:=\mathbf{W}_{jk}(\varkappa_j,\varkappa_k)\in \mathbb{C}$, for $j,k\in\{1,2\}$ and denote
$\mathbf{B}_j\equiv\left[\begin{array}{c}\mathbf{B}^j_{1} \\\hline \mathbf{b}^j_{2}\end{array}\right]:=\left[\begin{array}{cc}\Phi_{m_{j},n_{j}}& \phi_{m_{j},n_{j}}^{1}  \\\hline0&\omega_{11}^j\end{array}\right]$,
for $j=1,2$. Let
\begin{align}\label{eq5.11}
\begin{array}{l}
\Upsilon_1=\Phi_{m_{1},n_{1}} \mathbf{ W}^{11}_{11} +  \phi_{m_{1},n_{1}}^{1} \mathbf{ w}^{11}_{21},\\  p_{1}=\Phi_{m_{1},n_{1}} \mathbf{ w}^{11}_{12} +  \phi_{m_{1},n_{1}}^{1} \mathbf{ w}^{11}_{22}+ \phi_{m_{1},n_{1}}^{2},\\
\Omega_{11}(t)=\left[\begin{array}{c|c}\left( \Upsilon_1+\widehat{\Gamma}_{m_{1}+1,n_{1}+1}^{2m_{1},2n_{1}}\right)^{-1} &-\left( \Upsilon_1+\widehat{\Gamma}_{m_{1}+1,n_{1}+1}^{2m_{1},2n_{1}}\right)^{-1} p_1\\\hline0& 1\end{array}\right],\\
\Omega_1(t)=\left[\begin{array}{c|c}\Omega_{11}(t)&\left[\begin{array}{c}-\left( \Upsilon_1+\widehat{\Gamma}_{m_{1}+1,n_{1}+1}^{2m_{1},2n_{1}}\right)^{-1}\mathbf{B}^1_{1}\mathbf{W}_{12}\\0\end{array}\right]\\\hline0&I_{\varkappa_2}\end{array}\right],\\
\zeta^1_{jk}=\widehat{\psi}_{m_{1},n_{1}}^{j^H}\left( \widehat{\Gamma}_{m_{1}+1,n_{1}+1}^{2m_{1},2n_{1}}\right)^{-1} \phi_{m_{1},n_{1}}^{k},\ \ \text{for }j,k\in \{1,2\}.
\end{array}
\end{align}
From Table \ref{tab2}, we have
\begin{align}\label{eq5.12}
\begin{array}{l}
\Omega_{11}(t)=\left[\begin{array}{c|c}O(t^{-2})&O(t^{-1})\\\hline 0&1\end{array}\right],\ \ \Omega_{1}(t)=\left[\begin{array}{c|c}O(1)&O(t^{-1})\\\hline 0&I_{\varkappa_2}\end{array}\right],\\
\widehat{\Phi}_{m_{1},n_{1}}\left( \Upsilon_1+\widehat{\Gamma}_{m_{1}+1,n_{1}+1}^{2m_{1},2n_{1}}\right)^{-1}\mathbf{B}^1_{1}\mathbf{W}_{12}=O(t^{-1}), \\
 \widehat{\psi}_{m_{1},n_{1}}^{1^H}\left( \Upsilon_1+\widehat{\Gamma}_{m_{1}+1,n_{1}+1}^{2m_{1},2n_{1}}\right)^{-1}\mathbf{B}^1_{1}\mathbf{W}_{12}=\zeta^1_{11}[\mathbf{w}^{12}_{21},\mathbf{w}^{12}_{22}]+O(t^{-1}),\\
 \widehat{\psi}_{m_{1},n_{1}}^{2^H}\left( \Upsilon_1+\widehat{\Gamma}_{m_{1}+1,n_{1}+1}^{2m_{1},2n_{1}}\right)^{-1}\mathbf{B}^1_{1}\mathbf{W}_{12}=\zeta^1_{21}[\mathbf{w}^{12}_{21},\mathbf{w}^{12}_{22}]+O(t^{-1}),
\end{array}
\end{align}
as $t\rightarrow\pm \infty$, where $\zeta^1_{11}$ and $\zeta^1_{21}$ are defined in \eqref{eq5.11}. Post
multiplying $\Omega_1(t)$ to \eqref{eq5.10}, it follows from \eqref{eq5.11} and \eqref{eq5.12} that, as $t\rightarrow \pm\infty$,
\begin{subequations}\label{eq5.13}
\begin{align}
Q(t)\Omega_1(t)&=\left[\begin{array}{c|c}\mathbf{B}_{1}\mathbf{W}_{11}+\mathbf{D}_{1}&\mathbf{B}_{1}\mathbf{W}_{12}\\\hline \mathbf{B}_{2}\mathbf{W}_{21}&\mathbf{B}_{2}\mathbf{W}_{22}+\mathbf{D}_2\end{array}\right]\Omega_1(t)\nonumber\\&=\left[\begin{array}{c|c}\begin{array}{cc}I&0 \\O(t^{-1})&\varpi^u_{11}\end{array}&\begin{array}{c}0\\(\omega_{11}^1-\zeta^1_{11})[\mathbf{ w}^{12}_{21} , \mathbf{ w}^{12}_{22}]+O(t^{-1})\end{array}\\\hline \mathbf{B}_{2}\mathbf{W}_{21}\Omega_{11}(t)&\mathbf{B}_{2}(\mathbf{W}_{22}+\Delta)+\mathbf{D}_2\end{array}\right],\label{eq5.13a}\\
P(t)\Omega_1(t)&=\left[\begin{array}{c|c}\mathbf{G}_{1}\mathbf{W}_{11}+\mathbf{E}_{1}&\mathbf{G}_{1}\mathbf{W}_{12}\\\hline \mathbf{G}_{2}\mathbf{W}_{21}&\mathbf{G}_{2}\mathbf{W}_{22}+\mathbf{E}_2\end{array}\right]\Omega_1(t)\nonumber\\&=\left[\begin{array}{c|c}\begin{array}{cc}0&0 \\0&\varpi^d_{11}\end{array}&\begin{array}{c}0\\(\omega_{21}^1-\zeta^1_{21})[\mathbf{ w}^{12}_{21} , \mathbf{ w}^{12}_{22}]\end{array}\\\hline \mathbf{G}_{2}\mathbf{W}_{21}\Omega_{11}(t)&\mathbf{G}_{2}(\mathbf{W}_{22}+\Delta)+\mathbf{E}_2\end{array}\right]+\left[\begin{array}{c}O(t^{-1})\\\hline 0\end{array}\right],\label{eq5.13b}
\end{align}
\end{subequations}
where $\omega_{11}^1-\zeta^1_{11}=O(1)$, $\omega_{21}^1-\zeta^1_{21}=O(1)$, $\Delta=\mathbf{W}_{21}\left[\begin{array}{c}-\left( \Upsilon_1+\widehat{\Gamma}_{m_{1}+1,n_{1}+1}^{2m_{1},2n_{1}}\right)^{-1}\mathbf{B}^1_{1}\mathbf{W}_{12}\\0\end{array}\right]$ and
\begin{align*}
\left[\begin{array}{c}\varpi^u_{11}\\\hline\varpi^d_{11}\end{array}\right]&=\left[\begin{array}{cc}\omega^1_{11}&\omega^1_{12}\\\hline\omega^1_{21}&\omega^1_{22}\end{array}\right]\left[\begin{array}{c}\mathbf{w}^{11}_{22}\\1\end{array}\right]-\left[\begin{array}{c}\zeta^1_{11}\mathbf{w}^{11}_{22}+\zeta^1_{12}\\\hline \zeta^1_{21}\mathbf{w}^{11}_{22}+\zeta^1_{22}\end{array}\right]\nonumber\\
&=\left[\begin{array}{cc}\omega^1_{11}-\zeta^1_{11}&\omega^1_{12}-\zeta^1_{12}\\\hline\omega^1_{21}-\zeta^1_{21}&\omega^1_{22}-\zeta^1_{22}\end{array}\right]\left[\begin{array}{c}\mathbf{w}^{11}_{22}\\1\end{array}\right].
\end{align*}
From \eqref{eq5.12}, we have $\Delta=O(t^{-1})$ as $t\rightarrow \pm\infty$. Let $\widetilde{\mathbf{ W}}_{22}=\mathbf{W}_{22}+\Delta$. Then
\begin{align*}
\widetilde{\mathbf{ W}}_{22}=\left[\begin{array}{cc} \widetilde{\mathbf{ W}}^{22}_{11} &  \widetilde{\mathbf{ w}}^{22}_{12} \\\widetilde{\mathbf{ w}}^{22}_{21}  & \widetilde{\mathbf{ w}}^{22}_{22} \end{array}\right]
=\left[\begin{array}{cc} \mathbf{ W}^{22}_{11}&  \mathbf{ w}^{22}_{12} \\\mathbf{ w}^{22}_{21}  & \mathbf{ w}^{22}_{22} \end{array}\right]+O(t^{-1}).
\end{align*}
Similarly, let $\Upsilon_2=\Phi_{m_{2},n_{2}} \widetilde{\mathbf{ W}}^{22}_{11} +  \phi_{m_{2},n_{2}}^{1} \widetilde{\mathbf{ w}}^{22}_{21}$, $ p_{2}=\Phi_{m_{2},n_{2}} \widetilde{\mathbf{ w}}^{22}_{12} +  \phi_{m_{2},n_{2}}^{1}\widetilde{ \mathbf{ w}}^{22}_{22}+ \phi_{m_{2},n_{2}}^{2},$
\begin{align}\label{eq5.14}
\begin{array}{l}
\Omega_{22}(t)=\left[\begin{array}{c|c}\left( \Upsilon_2+\widehat{\Gamma}_{m_{2}+1,n_{2}+1}^{2m_{2},2n_{2}}\right)^{-1} &-\left( \Upsilon_2+\widehat{\Gamma}_{m_{2}+1,n_{2}+1}^{2m_{2},2n_{2}}\right)^{-1} p_2\\\hline0& 1\end{array}\right],\\
\Omega_2(t)=\left[\begin{array}{c|c}I_{\varkappa_1}&0\\\hline\left[\begin{array}{c}-\left( \Upsilon_2+\widehat{\Gamma}_{m_{2}+1,n_{2}+1}^{2m_{2},2n_{2}}\right)^{-1}\mathbf{B}^2_{1}\mathbf{W}_{21}\Omega_{11}(t)\\0\end{array}\right]&\Omega_{22}(t)\end{array}\right],\\
\zeta^2_{jk}=\widehat{\psi}_{m_{2},n_{2}}^{j^H}\left( \widehat{\Gamma}_{m_{2}+1,n_{2}+1}^{2m_{2},2n_{2}}\right)^{-1} \phi_{m_{2},n_{2}}^{k},\ \ \text{for }j,k\in \{1,2\}.
\end{array}
\end{align}
Then $\Omega_{22}(t)=\left[\begin{array}{c|c}O(t^{-2})&O(t^{-1})\\\hline 0&1\end{array}\right]$ and $\Omega_{2}(t)=\left[\begin{array}{c|c}I_{\varkappa_1}&0\\\hline O(t^{-1})&O(1)\end{array}\right]$, as $t\rightarrow \pm \infty$.
Denote $\Omega(t)=\Omega_1(t)\Omega_2(t)$. From \eqref{eq5.12} and \eqref{eq5.14}, we have
\begin{align*}
\Omega(t)=\left[\begin{array}{c|c}e_{\varkappa_1}e_{\varkappa_1}^{H}&0\\\hline0&e_{\varkappa_2}e_{\varkappa_2}^{H}\end{array}\right]+O(t^{-1}), \text{ as }t\rightarrow \pm\infty.
\end{align*}
Using the fact that $\omega_{11}^1-\zeta^1_{11}=O(1)$, $\omega_{21}^1-\zeta^1_{21}=O(1)$ and from \eqref{eq5.13} and Table \ref{tab2}, we have
\begin{align}\label{eq5.15}
\begin{array}{l}
Q(t)\Omega(t)=\left[\begin{array}{cc|cc}I&0&0&0 \\0&\varpi^u_{11}&0&(\omega_{11}^1-\zeta^1_{11}) \mathbf{ w}^{12}_{22}\\\hline 0&0&I&0\\0&(\omega_{11}^2-\zeta^2_{11})\mathbf{w}^{21}_{22}&0&\varpi^u_{22}\end{array}\right]+O(t^{-1}),\\
P(t)\Omega(t)=\left[\begin{array}{cc|cc}0&0&0&0 \\0&\varpi^d_{11}&0&(\omega_{21}^1-\zeta^1_{21}) \mathbf{ w}^{12}_{22}\\\hline 0&0&0&0\\0&(\omega_{21}^2-\zeta^2_{21})\mathbf{w}^{21}_{22}&0&\varpi^d_{22}\end{array}\right]+O(t^{-1}),\\
\left[\begin{array}{c}\varpi^u_{22}\\\hline\varpi^d_{22}\end{array}\right]=\left[\begin{array}{cc}\omega^2_{11}-\zeta^2_{11}&\omega^2_{12}-\zeta^2_{12}\\\hline\omega^2_{21}-\zeta^2_{21}&\omega^2_{22}-\zeta^2_{22}\end{array}\right]\left[\begin{array}{c}\mathbf{w}^{22}_{22}\\1\end{array}\right]+O(t^{-1}),
\end{array}
\end{align}
as $t\rightarrow \pm\infty$. Note that $\mathbf{w}^{jk}_{22}=\mathbf{w}_{jk}:=\mathbf{W}_{jk}(\varkappa_j,\varkappa_k)$ for each $j,k\in\{1,2\}$.

From the definitions of $\zeta^\ell_{jk}$, for $\ell,j,k\in\{1,2\}$, in \eqref{eq5.11} and \eqref{eq5.14}, we have
\begin{align*}
\left[\begin{array}{cc}\zeta^\ell_{11}&\zeta^\ell_{12}\\\zeta^\ell_{21}&\zeta^\ell_{22}\end{array}\right]=\Theta^H\left[\begin{array}{c|c}\widehat{\psi}_{m_\ell}^{H}&0\\\hline0&\widehat{\psi}_{n_\ell}^{H}\end{array}\right]\left[\begin{array}{c|c}\widehat{\Gamma}_{m_\ell+1}^{2m_\ell}&0\\\hline0&\widehat{\Gamma}_{n_\ell+1}^{2n_\ell}\end{array}\right]^{-1}
\left[\begin{array}{c|c}e^{i\gamma_\ell t}\phi_{m_\ell}&0\\\hline0&e^{i\delta_\ell t}\phi_{n_\ell}\end{array}\right]\Theta,
\end{align*}
where $\Theta$ in \eqref{eq4.27} is unitary. From Theorem \ref{thm5.3}, we obtain that
\begin{align*}
\left[\begin{array}{cc}\zeta^\ell_{11}&\zeta^\ell_{12}\\\zeta^\ell_{21}&\zeta^\ell_{22}\end{array}\right]&=\Theta^H\left[\begin{array}{cc}\kappa_{m_\ell}&0\\0&\kappa_{n_\ell}\end{array}\right]\left[\begin{array}{cc}e^{i\gamma_\ell t}&0\\0&e^{i\delta_\ell t}\end{array}\right]\Theta\\&=\frac{1}{2}\left[\begin{array}{cc}\kappa_{m_{\ell}}e^{i\gamma_\ell t}+\kappa_{n_{\ell}}e^{i\delta_\ell t}&-i\beta^d_\ell (\kappa_{m_{\ell}}e^{i\gamma_\ell t}-\kappa_{n_{\ell}}e^{i\delta_\ell t})\\i\beta^d_\ell (\kappa_{m_{\ell}}e^{i\gamma_\ell t}-\kappa_{n_{\ell}}e^{i\delta_\ell t})&\kappa_{m_{\ell}}e^{i\gamma_\ell t}+\kappa_{n_{\ell}}e^{i\delta_\ell t}\end{array}\right],
\end{align*}
where $\kappa_{m_\ell}$ and $\kappa_{n_\ell}$ satisfy  \eqref{eq5.3}. Using the definitions $\omega^{\ell}_{jk}$, for $\ell,j,k\in \{1,2\}$, in \eqref{eq4.12} yields that
\begin{align*}
\left[\begin{array}{cc}\hat{\omega}^{\ell}_{11}&\hat{\omega}^{\ell}_{12}\\\hat{\omega}^{\ell}_{21}&\hat{\omega}^{\ell}_{22}\end{array}\right]&\equiv \left[\begin{array}{cc}\omega^{\ell}_{11}-\zeta^\ell_{11}&\omega^{\ell}_{12}-\zeta^\ell_{12}\\\omega^{\ell}_{21}-\zeta^\ell_{21}&\omega^{\ell}_{22}-\zeta^\ell_{22}\end{array}\right]\nonumber\\=&\frac{1}{2}\left[\begin{array}{cc}(-1)^{m_{\ell}}e^{i\gamma_\ell t}+(-1)^{n_{\ell}}e^{i\delta_\ell t}&-i\beta^d_\ell ((-1)^{m_{\ell}}e^{i\gamma_\ell t}-(-1)^{n_{\ell}}e^{i\delta_\ell t})\\i\beta^d_\ell ((-1)^{m_{\ell}}e^{i\gamma_\ell t}-(-1)^{n_{\ell}}e^{i\delta_\ell t})&(-1)^{m_{\ell}}e^{i\gamma_\ell t}+(-1)^{n_{\ell}}e^{i\delta_\ell t}\end{array}\right].
\end{align*}From \eqref{eq5.10} and \eqref{eq5.15}, there exists an invertible matrix $\Omega(t)$ such that  \eqref{eq4.70} holds.
\end{proof}

\begin{proof}[\textbf{Proof of Theorem \ref{thm4.24}}]
Form \eqref{eq4.73} and \eqref{eq4.77}, we have $\Delta_u^{\pm}(t)=\Sigma_{\hat{\omega}_{11}}\mathfrak{W}_{cd}^{\pm}+\Sigma_{\hat{\omega}_{12}}$, $\Delta_v^{\pm}(t)=\Sigma_{\hat{\omega}_{21}}\mathfrak{W}_{cd}^{\pm}+\Sigma_{\hat{\omega}_{22}}$ and $\Sigma_{\beta^{cd}}\Sigma_{\beta^{cd}}=I$.
From \eqref{eq4.67}, \eqref{eq4.71} and \eqref{eq4.73}, we have
\begin{align*}
\mathfrak{U}_{cd} \Delta_{\pm}(t)&=\mathfrak{U}^{cd}_u(\Sigma_{\hat{\omega}_{11}}\mathfrak{W}_{cd}^{\pm}+\Sigma_{\hat{\omega}_{12}})+\mathfrak{U}^{cd}_v(\Sigma_{\hat{\omega}_{21}}\mathfrak{W}_{cd}^{\pm}+\Sigma_{\hat{\omega}_{22}})\\
&=\mathfrak{U}^{cd}_u\Sigma_{\hat{\omega}_{11}}\mathfrak{W}_{cd}^{\pm}+\mathfrak{U}^{cd}_u(i\Sigma_{\beta^{cd}})(-i\Sigma_{\beta^{cd}})\Sigma_{\hat{\omega}_{12}}+\mathfrak{U}^{cd}_v(\Sigma_{\hat{\omega}_{21}}\mathfrak{W}_{cd}^{\pm}+\Sigma_{\hat{\omega}_{22}})\\
&=[\mathfrak{U}^{cd}_u\Sigma_{\hat{\omega}_{11}}+\mathfrak{U}^{cd}_v\Sigma_{\hat{\omega}_{21}}]\mathfrak{W}_{cd}^{\pm}+\frac{i}{2}\mathfrak{U}^{cd}_u\Sigma_{\beta^{cd}}(-\Sigma_{\gamma}+\Sigma_{\delta})+\frac{1}{2}\mathfrak{U}^{cd}_v(\Sigma_{\gamma}+\Sigma_{\delta})\\
&=[\mathfrak{U}^{cd}_u\Sigma_{\hat{\omega}_{11}}+\mathfrak{U}^{cd}_v\Sigma_{\hat{\omega}_{21}}]\mathfrak{W}_{cd}^{\pm}+\frac{1}{2}[\mathfrak{U}^{cd}_v-i\mathfrak{U}^{cd}_u\Sigma_{\beta^{cd}}]\Sigma_{\gamma}+\frac{1}{2}[\mathfrak{U}^{cd}_v+i\mathfrak{U}^{cd}_u\Sigma_{\beta^{cd}}]\Sigma_{\delta}.
\end{align*}
Since $\Sigma_{\gamma}$ is invertible, we obtain
\begin{align*}
\mathfrak{U}_{cd} \Delta_{\pm}(t)\Sigma_{\gamma}^{-1}=&\frac{1}{2}[\mathfrak{U}^{cd}_v-i\mathfrak{U}^{cd}_u\Sigma_{\beta^{cd}}]+[\mathfrak{U}^{cd}_u\Sigma_{\hat{\omega}_{11}}+\mathfrak{U}^{cd}_v\Sigma_{\hat{\omega}_{21}}]\mathfrak{W}_{cd}^{\pm}\Sigma_{\gamma}^{-1}\nonumber\\&+\frac{1}{2}[\mathfrak{U}^{cd}_v+i\mathfrak{U}^{cd}_u\Sigma_{\beta^{cd}}]\Sigma_{\delta}\Sigma_{\gamma}^{-1}.
\end{align*}

From \eqref{eq4.75} and \eqref{eq4.76}, there exists an invertible matrix $\mathbf{Z}_{\pm}(t)$ such that
\begin{align}\label{eq5.16}
Y(t)\mathbf{Z}_{\pm}(t)(I_{n-\mu}\oplus\Sigma_{\gamma}^{-1})=\left[\begin{array}{c}\mathbf{U}_{1,\pm}+ \Delta\mathbf{U}_{1,\pm}(t)\\\mathbf{U}_{2,\pm}+ \Delta\mathbf{U}_{2,\pm}(t)\end{array}\right]+O(t^{-1}),
\end{align}
as $t\rightarrow \pm\infty$, where  for $j=1,2$,  $ \Delta\mathbf{U}_{j,\pm}(t)\equiv \left[0,\Delta\mathbf{U}^{cd}_{j,\pm}(t)\right]=\Delta\mathbf{U}^{cd}_{j,\pm}(t)\mathsf{E}_{\mu}^H$ and $\mathsf{E}_{\mu}$ and $\Delta\mathbf{U}^{cd}_{j,\pm}(t)$ are defined in \eqref{eq4.71} and \eqref{eq4.78}, respectively. It follows from the definition of $W(t)$ and \eqref{eq5.16} that there are matrix functions, $M^{\varepsilon}_1(t)=O(t^{-1})$ and $M^{\varepsilon}_2(t)=O(t^{-1})$ as $ t\rightarrow \pm\infty$, such that
\begin{align}
W(t)&=P(t)Q(t)^{-1}=(\mathbf{U}_{2,\pm}+ \Delta\mathbf{U}_{2,\pm}(t)+M^{\varepsilon}_2(t))(\mathbf{U}_{1,\pm}+ \Delta\mathbf{U}_{1,\pm}(t)+M^{\varepsilon}_1(t))^{-1}\notag\\
&=\left(\mathbf{U}_{2,\pm}+ \Delta\mathbf{U}^{cd}_{2,\pm}(t)\mathsf{E}_{\mu}^H\right)\left[(\mathbf{U}_{1,\pm}+M^{\varepsilon}_1(t))+ \Delta\mathbf{U}^{cd}_{1,\pm}(t)\mathsf{E}_{\mu}^H\right]^{-1}\notag\\
&\ \ \ \ +M^{\varepsilon}_2(t)\left[(\mathbf{U}_{1,\pm}+M^{\varepsilon}_1(t))+ \Delta\mathbf{U}^{cd}_{1,\pm}(t)\mathsf{E}_{\mu}^H\right]^{-1}\label{eq5.17}.
\end{align}
We assume that $\mathbf{U}^{\varepsilon}_{1,\pm}(t)\equiv \mathbf{U}_{1,\pm}+M^{\varepsilon}_1(t)$ is invertible.
Applying  the Sherman-Morrison-Woodbury formula, we have
\begin{align}\label{eq5.18}
(&\mathbf{U}^{\varepsilon}_{1,\pm}(t)+\Delta\mathbf{U}^{cd}_{1,\pm}(t)\mathsf{E}_{\mu}^H)^{-1}=\mathbf{U}^{\varepsilon}_{1,\pm}(t)^{-1}\nonumber\\&\hspace{1cm}-\mathbf{U}^{\varepsilon}_{1,\pm}(t)^{-1}\Delta\mathbf{U}^{cd}_{1,\pm}(t)(I_{\mu}+\mathsf{E}_{\mu}^H\mathbf{U}^{\varepsilon}_{1,\pm}(t)^{-1}\Delta\mathbf{U}^{cd}_{1,\pm}(t))^{-1}\mathsf{E}_{\mu}^H\mathbf{U}^{\varepsilon}_{1,\pm}(t)^{-1}.
\end{align}
Since $\mathbf{U}_{1,\pm}$ is invertible and $M^{\varepsilon}_1(t)=O(t^{-1})$ as $ t\rightarrow \pm\infty$, we see that
\begin{align}\label{eq5.19}
\mathbf{U}^{\varepsilon}_{1,\pm}(t)^{-1}=(\mathbf{U}_{1,\pm}+M^{\varepsilon}_1(t))^{-1}=\mathbf{U}_{1,\pm}^{-1}+O(t^{-1}),
\end{align}
as $t\rightarrow \pm\infty$.
Substituting \eqref{eq5.18} and \eqref{eq5.19} into \eqref{eq5.17}, it turns out
\begin{align*}
W(t)&=\mathbf{U}_{2,\pm}\mathbf{U}^{\varepsilon}_{1,\pm}(t)^{-1}+ \Delta\mathbf{U}^{cd}_{2,\pm}(t)\mathsf{E}_{\mu}^H\mathbf{U}^{\varepsilon}_{1,\pm}(t)^{-1}\\& - \Delta\mathbf{U}^{cd}_{2,\pm}(t)\mathsf{E}_{\mu}^H \mathbf{U}^{\varepsilon}_{1,\pm}(t)^{-1}\Delta\mathbf{U}^{cd}_{1,\pm}(t)(I_{\mu}+\mathsf{E}_{\mu}^H\mathbf{U}^{\varepsilon}_{1,\pm}(t)^{-1}\Delta\mathbf{U}^{cd}_{1,\pm}(t))^{-1}\mathsf{E}_{\mu}^H\mathbf{U}^{\varepsilon}_{1,\pm}(t)^{-1}\\& -\mathbf{U}_{2,\pm} \mathbf{U}^{\varepsilon}_{1,\pm}(t)^{-1}\Delta\mathbf{U}^{cd}_{1,\pm}(t)(I_{\mu}+\mathsf{E}_{\mu}^H\mathbf{U}^{\varepsilon}_{1,\pm}(t)^{-1}\Delta\mathbf{U}^{cd}_{1,\pm}(t))^{-1}\mathsf{E}_{\mu}^H\mathbf{U}^{\varepsilon}_{1,\pm}(t)^{-1}\\&+M^{\varepsilon}_2(t)\left[\mathbf{U}^{\varepsilon}_{1,\pm}(t)+\Delta\mathbf{U}^{cd}_{1,\pm}(t)\mathsf{E}_{\mu}^H\right]^{-1}\\ 
=&\mathbf{U}_{2,\pm}\mathbf{U}_{1,\pm}^{-1}+\Delta\mathbf{U}^{cd}_{2,\pm}(t) (I_{\mu}+\mathsf{E}_{\mu}^H\mathbf{U}^{\varepsilon}_{1,\pm}(t)^{-1}\Delta\mathbf{U}^{cd}_{1,\pm}(t))^{-1}\mathsf{E}_{\mu}^H\mathbf{U}^{\varepsilon}_{1,\pm}(t)^{-1}\\& -\mathbf{U}_{2,\pm} \mathbf{U}^{\varepsilon}_{1,\pm}(t)^{-1}\Delta\mathbf{U}^{cd}_{1,\pm}(t)(I_{\mu}+\mathsf{E}_{\mu}^H\mathbf{U}^{\varepsilon}_{1,\pm}(t)^{-1}\Delta\mathbf{U}^{cd}_{1,\pm}(t))^{-1}\mathsf{E}_{\mu}^H\mathbf{U}^{\varepsilon}_{1,\pm}(t)^{-1}\\& -M^{\varepsilon}_2(t)\mathbf{U}^{\varepsilon}_{1,\pm}(t)^{-1}\Delta\mathbf{U}^{cd}_{1,\pm}(t)(I_{\mu}+\mathsf{E}_{\mu}^H\mathbf{U}^{\varepsilon}_{1,\pm}(t)^{-1}\Delta\mathbf{U}^{cd}_{1,\pm}(t))^{-1}\mathsf{E}_{\mu}^H\mathbf{U}^{\varepsilon}_{1,\pm}(t)^{-1} +O(t^{-1})\\
=&\mathbf{U}_{2,\pm}\mathbf{U}_{1,\pm}^{-1}+[\Delta\mathbf{U}^{cd}_{2,\pm}(t)-\mathbf{U}_{2,\pm} \mathbf{U}_{1,\pm}^{-1}\Delta\mathbf{U}^{cd}_{1,\pm}(t)+O(t^{-1})]\\& [I_{\mu}+\mathsf{E}_{\mu}^H\mathbf{U}_{1,\pm}^{-1}\Delta\mathbf{U}^{cd}_{1,\pm}(t)+O(t^{-1})]^{-1}[\mathsf{E}_{\mu}^H\mathbf{U}_{1,\pm}^{-1}+O(t^{-1})]+O(t^{-1}),
\end{align*}
as $t\rightarrow \pm\infty$.

From \eqref{eq5.16} we have $Q(t)^{-1}=\mathbf{Z}_{\pm}(t)(I_{n-\mu}\oplus\Sigma_{\gamma}^{-1})(\mathbf{U}^{\varepsilon}_{1,\pm}(t)+ \Delta\mathbf{U}_{1,\pm}(t))^{-1}$ as $ t\rightarrow \pm\infty$. Using \eqref{eq4.74}, there exists matrix function, $M_3^{\varepsilon}(t)=O(t^{-1})$ as $t\rightarrow \pm\infty$, such that $\mathbf{Z}_{\pm}(t)=\mathfrak{Z}_{cd}^{\pm}\mathsf{E}_{\mu}^H+M_3^{\varepsilon}(t)$.
It follows from \eqref{eq5.18} and \eqref{eq5.19} that
\begin{align*}
Q(t)^{-1}&=(\mathfrak{Z}_{cd}^{\pm}\mathsf{E}_{\mu}^H+M_3^{\varepsilon}(t))(I_{n-\mu}\oplus\Sigma_{\gamma}^{-1})(\mathbf{U}^{\varepsilon}_{1,\pm}(t)+ \Delta\mathbf{U}_{1,\pm}(t))^{-1}\\
=&\mathfrak{Z}_{cd}^{\pm}\left([0,\Sigma_{\gamma}^{-1}]\mathbf{U}_{1,\pm}^{{\varepsilon}}(t)^{-1}\right)-\mathfrak{Z}_{cd}^{\pm}\left([0,\Sigma_{\gamma}^{-1}]\mathbf{U}_{1,\pm}^{{\varepsilon}}(t)^{-1}\Delta\mathbf{U}^{cd}_{1,\pm}(t)\right)\\&\left[I_{\mu}+\mathsf{E}_{\mu}^H\mathbf{U}_{1,\pm}^{{\varepsilon}}(t)^{-1}\Delta\mathbf{U}^{cd}_{1,\pm}(t)\right]^{-1}\left(\mathsf{E}_{\mu}^H\mathbf{U}_{1,\pm}^{{\varepsilon}}(t)^{-1}\right)+M_3^{\varepsilon}(t)(\mathbf{U}^{\varepsilon}_{1,\pm}(t)+ \Delta\mathbf{U}_{1,\pm}(t))^{-1}\\
=&\mathfrak{Z}_{cd}^{\pm}\Sigma_{\gamma}^{-1}\left[I_{\mu}-\mathsf{E}_{\mu}^H\mathbf{U}_{1,\pm}^{{\varepsilon}}(t)^{-1}\Delta\mathbf{U}^{cd}_{1,\pm}(t)\left(I_{\mu}+\mathsf{E}_{\mu}^H\mathbf{U}_{1,\pm}^{{\varepsilon}}(t)^{-1}\Delta\mathbf{U}^{cd}_{1,\pm}(t)\right)^{-1}\right]\\&\left(\mathsf{E}_{\mu}^H\mathbf{U}_{1,\pm}^{{\varepsilon}}(t)^{-1}\right) +M_3^{\varepsilon}(t)(\mathbf{U}^{\varepsilon}_{1,\pm}(t)+ \Delta\mathbf{U}_{1,\pm}(t))^{-1}\\
=&[\mathfrak{Z}_{cd}^{\pm}\Sigma_{\gamma}^{-1}+O(t^{-1})]\left[I_{\mu}+\mathsf{E}_{\mu}^H\mathbf{U}_{1,\pm}^{\varepsilon}(t)^{-1}\Delta\mathbf{U}^{cd}_{1,\pm}(t)\right]^{-1}\left(\mathsf{E}_{\mu}^H\mathbf{U}_{1,\pm}^{{\varepsilon}}(t)^{-1}\right)+O(t^{-1})\\
=&[\mathfrak{Z}_{cd}^{\pm}\Sigma_{\gamma}^{-1}+O(t^{-1})]\left[I_{\mu}+\mathsf{E}_{\mu}^H\mathbf{U}_{1,\pm}^{-1}\Delta\mathbf{U}^{cd}_{1,\pm}(t)+O(t^{-1})\right]^{-1}\mathsf{E}_{\mu}^H\left[\mathbf{U}_{1,\pm}^{-1}+O(t^{-1})\right]\\& +O(t^{-1}),
\end{align*}
as $t\rightarrow\pm \infty$, where $\mathfrak{Z}_{cd}^{\pm}\in \mathbb{C}^{n\times \mu}$ is defined in \eqref{eq4.71}. This completes the proof.
\end{proof}

\addcontentsline{toc}{section}{Bibliography}
\bibliographystyle{plain}
\bibliography{flow}

\end{document}